\documentclass[12 pt,reqno]{amsart}
\usepackage{tikz-cd}
\usepackage{pdfpages}
\usepackage{quiver}
\usepackage{amsmath,amsthm,amssymb,amscd}
\usepackage{setspace}
\usepackage{upgreek}
\usepackage{cite}
\usepackage{url}
\usepackage{mathtools}
\usepackage{fullpage}
\usepackage{float}
\usepackage{comment}
\usepackage[bookmarksopen,bookmarksdepth=3]{hyperref}
\usepackage{accents}
\newcommand{\hmt}{\widehat{\mkern1mu\mt\mkern0mu}}
\newcommand{\bhmtaux}[2]{\overline{\raisebox{0pt}[0.92\height]{$#1\hmt$}}}
\newcommand{\bhmt}{\mathpalette\bhmtaux\relax}
\usepackage{enumitem}
\newlength{\dhatheight}
\newcommand{\doublehat}[1]{%
    \settoheight{\dhatheight}{\ensuremath{\hat{#1}}}%
    \addtolength{\dhatheight}{-0.35ex}%
    \hat{\vphantom{\rule{1pt}{\dhatheight}}
    \smash{\hat{#1}}}}

\input xy
\xyoption{all}

\usepackage{microtype}

\allowdisplaybreaks[1]

\newtheorem{thmI}{Theorem}
\newtheorem{prop}{Proposition}[section]
\newtheorem{thm}[prop]{Theorem}
\newtheorem{lm}[prop]{Lemma}
\newtheorem{cor}[prop]{Corollary}

\newtheorem{cl}[prop]{Proposition}

\theoremstyle{definition}
\newtheorem{dfn}[prop]{Definition}

\theoremstyle{remark}
\newtheorem{rem}[prop]{Remark}
\newtheorem{ex}[prop]{Example}
\newtheorem{notn}[prop]{Notation}
\DeclareMathOperator{\Cone}{Cone}

\DeclareMathOperator{\Id}{Id}

\DeclareMathOperator{\im}{Im}
\DeclareMathOperator{\tr}{tr}

\DeclareMathOperator{\ogw}{OGW}

\newcommand{\fix}{\textit{fix}}

\newcommand{\lrarr}{\longrightarrow}

\renewcommand{\k}{\Bbbk}
\newcommand{\R}{\mathbb{R}}

\newcommand{\Z}{\mathbb{Z}}

\newcommand{\N}{\mathbb{N}}
\newcommand{\M}{\mathcal{M}}

\renewcommand{\P}{\mathbb{C}P}
\renewcommand{\L}{\Lambda}
\renewcommand{\l}{\lambda}
\newcommand{\mI}{\mathcal{I}}
\newcommand{\g}{\Gamma}

\newcommand{\Hh}{\widetilde{H}}
\newcommand{\J}{\mathcal{J}}

\newcommand{\m}{\mathfrak{m}}
\renewcommand{\d}{\partial}

\newcommand{\f}{\mathfrak{f}}

\newcommand{\at}{\tilde{\alpha}}

\newcommand{\xit}{\tilde{\xi}}
\newcommand{\etat}{\tilde{\eta}}

\newcommand{\bt}{\tilde{b}}
\newcommand{\gt}{\tilde{\gamma}}
\newcommand{\mt}{\tilde{\mathfrak{m}}}

\newcommand{\ct}{\tilde{c}}

\newcommand{\mC}{\mathfrak{C}}
\newcommand{\mCd}{\mathfrak{C}^{\diamond}}
\newcommand{\mCdh}{\hat{\mathfrak{C}}^{\diamond}}
\newcommand{\mCh}{\hat{\mathfrak{C}}}
\newcommand{\Ch}{\hat{C}}
\newcommand{\Rh}{\hat{R}}
\newcommand{\Rd}{R^{\diamond}}
\newcommand{\Rdh}{\hat{R}^{\diamond}}
\newcommand{\Cd}{C^{\diamond}}
\newcommand{\Cdh}{\hat{C}^{\diamond}}
\newcommand{\ah}{\hat{\alpha}}
\newcommand{\mD}{\mathfrak{D}}
\newcommand{\mR}{\mathfrak{R}}
\newcommand{\mB}{\mathfrak{B}}
\newcommand{\mRd}{\mathfrak{R}^{\diamond}}
\newcommand{\mRdh}{\hat{\mathfrak{R}}^{\diamond}}
\newcommand{\evbt}{\widetilde{evb}}
\newcommand{\evit}{\widetilde{evi}}

\newcommand{\mg}{\m^{\gamma}}
\newcommand{\mgh}{\mh^{\gamma}}
\newcommand{\mghb}{\mh^{\gamma,b}}
\newcommand{\mghbl}{\mh^{\gamma,b_{(l)}}}
\newcommand{\mghi}[1]{\mh^{\gamma,#1}}
\newcommand{\mgie}[2]{\mh^{\gamma,#1}_{0,#2}}
\newcommand{\mgiek}[3]{\mh^{\gamma,#1}_{#3,#2}}
\newcommand{\mgp}{\m^{\gamma'}}
\newcommand{\mgt}{\mt^{\gt}}
\newcommand{\mgtb}{\mt^{\gt,\tilde{b}}}
\newcommand{\mgti}[1]{\mt^{\gt,#1}}

\renewcommand{\ll}{\langle\!\langle}
\renewcommand{\gg}{\rangle\!\rangle}
\newcommand{\degev}[1]{|#1|_{\mathbb{Z}_2}}

\renewcommand{\Im}{\im}

\newcommand{\invs}{\kappa}

\newcommand{\RP}{\mathbb{R}P}

\newcommand{\sly}{\Pi}

\newcommand{\D}{\mathcal{D}}
\newcommand{\bb}{\bar{b}}

\newcommand{\Ups}{\Upsilon}

\newcommand{\Mt}{\widetilde{\M}}

\newcommand{\e}{{\bf{e}}}
\newcommand{\mbar}{\bar{\m}}
\newcommand{\mh}{\hat{\m}}
\newcommand{\lp}{{\prec}}
\newcommand{\rp}{{\succ}}
\newcommand{\cC}{\mathcal{C}}
\newcommand{\lpt}{{\preccurlyeq}}
\newcommand{\rpt}{{\succcurlyeq}}
\newcommand{\s}{\mathfrak{s}}

\renewcommand{\a}{\alpha}

\newcommand{\ssly}{{S}}
\newcommand{\spe}[2]{s_{#1} (#2)}
\newcommand{\gf}{\mathfrak{g}}

\title{Point-like non-commutative families of bounding cochains}
\author[E. Kosloff]{Elad Kosloff}
\address{Institute of Mathematics, The Hebrew University of Jerusalem}
\email{eladkosloff@gmail.com}
\author[J. Solomon]{Jake P. Solomon}
\address{Institute of Mathematics, The Hebrew University of Jerusalem}
\email{jake@math.huji.ac.il}
\keywords{$A_\infty$ algebra, bounding cochain, open Gromov-Witten invariant, Lagrangian submanifold, non-commutative family, $J$-holomorphic, stable map, superpotential}
\subjclass[2020]{53D45, 53D37 (Primary) 14N35, 14N10, 53D12 (Secondary)}
\date{December 2026}
\begin{document}

\begin{abstract}
    We define genus zero open Gromov-Witten invariants with boundary and interior constraints for a Lagrangian submanifold of arbitrary even dimension. The definition relies on constructing a canonical family of bounding cochains that satisfy the point-like condition of the second author and Tukachinsky. Since the Lagrangian is even dimensional, the parameter of the family is odd. Thus, to avoid the vanishing of invariants with more than one boundary constraint, the parameter must be non-commutative. The invariants are defined either when the Lagrangian is a rational cohomology sphere or when the Lagrangian is fixed by an anti-symplectic involution, has dimension $2$ modulo $4$, and its cohomology is that of a sphere aside from degree $1$ modulo $4$. In dimension $2$, these invariants recover Welschinger's invariants.
    
    We develop an obstruction theory for the existence and uniqueness of bounding cochains in a Fukaya $A_\infty$ algebra with non-commutative coefficients. The obstruction classes belong to twisted cohomology groups of the Lagrangian instead of the de Rham cohomology of the commutative setting. A spectral sequence is constructed to compute the twisted cohomology groups. The extension of scalars of an $A_\infty$ algebra by a non-commutative ring is treated in detail. A theory of pseudo-completeness is introduced to guarantee the convergence of the Maurer-Cartan equation, which defines bounding cochains, even though the non-commutative parameter is given zero filtration.    
\end{abstract}
\maketitle

\vspace{-2.5em}

\tableofcontents

\section{Introduction}
\subsection{Overview} 
An open Gromov-Witten invariant is a count of $J$-holomorphic curves in a symplectic manifold $X$ with boundary in a Lagrangian submanifold $L \subset X$ representing a homology class $d \in H_2(X, L)$. To obtain a finite count, the curves are constrained to pass through a collection of cycles on $X$ and $L$, called interior and boundary constraints respectively.
The count should not depend on the choice of tame almost complex structure $J$ and it should depend only on the homology class of the cycles used as constraints. The main difficulty in defining open Gromov-Witten invariants is that as one varies $J$ or the constraints in a one-parameter family, $J$-holomorphic curves can degenerate and disappear through the process of boundary bubbling. In some situations~\cite{Liu2002Counting,cho2008counting,solomon2006intersection,georgieva2016open}, boundary bubbles can be shown to cancel out using symmetries of the pair $(X,L)$. 
However, in the presence of boundary constraints, when $L$ has dimension greater than $3$, the cancellation of bubbles based on symmetries breaks down. The second author and Tukachinsky~\cite{solomon2016point,solomon2016differential,solomon2019relative} used the Fukaya-Oh-Ohta-Ono theory of bounding cochains~\cite{fukaya2010lagrangian,fukaya2010cyclic} to define open Gromov-Witten invariants without requiring any symmetry of $(X,L)$, in arbitrary dimension, with boundary and interior constraints. To this end, they define the notion of a ``point-like'' family of bounding cochains and prove its existence and uniqueness up to gauge-equivalence under appropriate conditions. However, in even dimensions, the parity of the parameter of the family is odd, and since it is graded commutative, invariants with more than one boundary constraint vanish.

The present work develops the theory of point-like families of bounding cochains with non-commutative parameter and uses it to define open Gromov-Witten invariants that need not vanish with an arbitrary number of boundary constraints. As in~\cite{solomon2016point,solomon2016differential,solomon2019relative}, these invariants are defined without requiring any symmetry of $(X,L)$, in arbitrary dimension, and interior constraints are allowed as well. In dimension $2$ the open Gromov-Witten invariants of the present work are shown to recover Welschinger's enumerative invariants~\cite{welschinger2005enumerative}. So, computations of Welschinger's invariants~\cite{Itenberg2003enumerationrealcurves,Itenberg2004LogEquiv,Itenberg2007newLogWel,Brugalle2018SurgeryWelsch} give examples where the open Gromov-Witten invariants of the present work with multiple boundary constraints are non-zero. 
In~\cite{Kosloff2026}, the first author extends the open WDVV equations of~\cite{solomon2019relative} to the invariants defined here and derives recursive formulas for their computation. While the present work assumes the Lagrangian submanifold is orientable, the techniques developed here applied to the Fukaya $A_\infty$ algebras of non-orientable Lagrangians developed in~\cite{kedar2022ainftyorientorcalculus,kedar2022fukaynonorientable} should give open Gromov-Witten invariants of non-orientable Lagrangians.
The existence and uniqueness of non-commutative families of bounding cochains is controlled by obstruction classes that belong to a twisted version of the de Rham cohomology of the Lagrangian submanifold. The twisted cohomology is computed from the ordinary cohomology of the Lagrangian by means of a spectral sequence. The non-triviality of the differentials of this spectral sequence forces the vanishing of certain obstructions that do not appear in the commutative setting. For the obstruction classes to belong to the twisted cohomology, we assign zero filtration to the non-commutative parameter. Thus, the Maurer-Cartan equation, which a bounding cochain satisfies by definition, a priori need not converge. To prove convergence, we show that Fukaya $A_\infty$ algebras satisfy a property called pseudo-completeness. 
\subsection{Statement of the results}

\subsubsection{\texorpdfstring{$A_{\infty}$}{A-infinity} algebras}\label{subsubsec:a_infty}

To formulate our results, we recall relevant notation from~\cite{solomon2016differential}. For an element $x$ of a graded abelian group we denote by $|x|$ the degree of $x$. Consider a closed symplectic manifold $(X,\omega)$ of complex dimension $n>0$, with $n$ even, and a connected closed Lagrangian submanifold $L$ with relative spin structure~$\s =\s_L$. Let $J$ be an $\omega$-tame almost complex structure on~$X$. Denote by $\mu:H_2(X,L) \to \Z$ the Maslov index. Denote by $A^*(L)$ the ring of differential forms on $L$ with coefficients in $\R$.
Let $\sly = \sly_L$ be the quotient of $H_2(X,L;\Z)$ by a possibly trivial subgroup $\ssly_L$ contained in the kernel of the homomorphism $\omega \oplus \mu : H_2(X,L;\Z) \to \R \oplus \Z$. Thus the homomorphisms $\omega$ and $\mu$ descend to $\sly$. Denote by $\beta_0$ the zero element of $\sly$. We use a Novikov ring $\L$ which is a completion of a subring of the group ring of $\sly$. The precise definition follows. Denote by $T^\beta$ the element of the group ring corresponding to $\beta \in \sly$, so $T^{\beta_1}T^{\beta_2} = T^{\beta_1 + \beta_2}$. Then,
\begin{equation}\label{eq:lambada_def}
    \L=\left\{\sum_{i=1}^\infty a_i T^{\beta_i}\bigg|a_i\in\R,\beta_i\in \sly,\omega(\beta_i)\ge 0,\omega(\beta_i)=0 \implies \beta_i= \beta_0, \; \lim_{i\to \infty}\omega(\beta_i)=\infty\right\}.
\end{equation}
A grading on $\L$ is defined by declaring $T^\beta$ to be of degree $\mu(\beta)$.
Denote also
\[
\L^+=\left\{\sum_{i=1}^\infty a_i T^{\beta_i} \in \L \bigg|\;\omega(\beta_i)> 0 \quad\forall i\right\}.
\]

We use a family of $A_\infty$ structures on $A^*(L)\otimes \L$ following~\cite{fukaya2010cyclic,fukaya2010lagrangian},
based on the results of~\cite{solomon2016differential}. Let $\M_{k+1,l}(\beta)$ be the moduli space of genus zero $J$-holomorphic open stable maps $u:(\Sigma,\d \Sigma) \to (X,L)$ of degree $[u_*([\Sigma,\d \Sigma])] = \beta \in \sly$ with one boundary component, $k+1$ boundary marked points, and $l$ interior marked points. The boundary points are labeled according to their cyclic order. The space $\M_{k+1,l}(\beta)$ carries evaluation maps associated to boundary marked points $evb_j^\beta:\M_{k+1,l}(\beta)\to L$, $j=0,\ldots,k$, and evaluation maps associated to interior marked points $evi_j^\beta:\M_{k+1,l}(\beta)\to X$, $j=1,\ldots,l$.

We assume that all $J$-holomorphic genus zero open stable maps with one boundary component are regular, the moduli spaces $\M_{k+1,l}(\beta;J)$ are smooth orbifolds with corners, and the evaluation maps $evb_0^\beta$ are proper submersions.
Examples include $(\P^n,\RP^n)$ with the standard symplectic and complex structures or, more generally, flag varieties, Grassmannians, and products thereof. See~\cite[Example 1.4, Remark 1.5]{solomon2016differential}.
Throughout the paper we fix a connected component $\mathcal{J}$ of the space of $\omega$-tame almost complex structures satisfying our assumptions. All almost complex structures are taken from $\J$. Using virtual fundamental classes constructed via Kuranishi structures~\cite{fukaya2010cyclic,Fukaya2011Counting,fukaya2010lagrangian,FukayaKurnishi2017I,Fukaya2010ToricI,Fukaya2011ToricII,Fukaya2016Miror,Fukaya2017involution,fukaya2024constructionkuranishistructuresmoduli,Fukaya2019Spectral,Fukaya2020Kuranishi,Fukaya1999Arnold,abouzaid2021complexcobordismhamiltonianloops,Abouzaid2024Gromov,hirschi2025opencloseddelignemumfordfieldtheory}, the definition of open Gromov-Witten invariants in the present paper extends to arbitrary targets $(X,\omega,L)$ and to the case in which $\mathcal{J}$ is the space of all $\omega$-tame almost complex structures. Alternatively, it should be possible to use the polyfold theory of~\cite{Hofer2007Fredholm,Hofer2009Fredholm,Hofer2010Integration,Hofer2009functors,li2014structures}. To make the present paper more accessible, we use the self-contained treatment of Fukaya $A_\infty$ algebras for Lagrangian submanifolds of~\cite{solomon2016differential}. The treatment suffices for the calculations of~\cite{Kosloff2026}. The treatment is sufficiently general to cover an infinite family of Lagrangian submanifolds for which no open Gromov-Witten invariants were defined previously, as mentioned above. The techniques introduced in this paper are no simpler for the Lagrangians covered by the treatment of~\cite{solomon2016differential} than in the general case.
The relative spin structure $\s$ determines an orientation on the moduli spaces $\M_{k+1,l}(\beta)$ as in~\cite[Chapter 8]{fukaya2010lagrangian}.
For $m>0$ denote by
$A^m(X,L)$ the space of differential $m$-forms on $X$ that vanish on $L$, and denote by $A^0(X,L)$ functions on $X$ that are constant on $L$. The exterior differential $d$ makes $A^*(X,L)$ into a complex. 
Let
\begin{equation*}
R:=\L [[t_0,\ldots,t_N]],\qquad Q:=\R[t_0,\ldots,t_N],
\end{equation*}
be graded commutative power series and polynomial rings, and let
\begin{equation*}
C:= R \otimes A^*(L),\qquad\text{and}\qquad D:= Q\otimes  A^*(X,L).
\end{equation*}
Here and throughout the paper, unless otherwise mentioned, tensor products are implicitly completed.
Let $S:=\R  \langle s\rangle$ denote the ring of polynomials in a non-commutative formal variable $s$ of degree $|s| = 1-n$, which is necessarily odd, and let
\begin{gather*}
 \hat{R}:=S \otimes R, \qquad \hat{C}:=S \otimes R \otimes A^*(L).
\end{gather*}
The grading on these complexes takes into account the degrees of $s,t_j,T^\beta$, and the degree of differential forms.
For an $\R$-algebra $\Upsilon$, write
\[
 \Hh^*(X,L;\Upsilon):=H^*(A^*(X,L)\otimes\Upsilon, d).
\]
Observe that
\[
\Hh^*(X,L;\Upsilon) \simeq \left(H^0(L;\R)\oplus H^{>0}(X,L;\R)\right)\otimes  \Upsilon,\qquad \Hh^*(X,L;Q)=H^*(D).
\]
Given a graded module $M$, we write $M_j$ or ${(M)}_j$ for the degree~$j$ part.
Let
\[
R^+:=\hat{R}\L^+\triangleleft \hat{R}
\]
be the ideal obtained by extending the positive-energy Novikov ideal $\L^+$ to $\hat R$. Let
\[
\mathcal{I}_Q:=\left<t_0,\ldots,t_N\right>\triangleleft Q
\]
be the ideal generated by the $t$-variables in $Q$, and let
\[
\mathcal{I}_{\hat{R}}:=\left<t_0,\ldots,t_N\right>+R^+ \triangleleft \hat{R}
\]
be the sum of the ideal generated by the $t$-variables in $\hat R$ and the positive-energy ideal $R^+$. Define $\hat{R}^*$ to be the units of the ring $\hat{R}$. For $x \in S$ we write
\[x= x^0 +x^1,\qquad x^0 \in \R\langle s^2 \rangle, \qquad x^1 \in s\cdot \R\langle s^2\rangle. \]
We can extend this definition to elements of tensor products $M$ of $S$ such as $M=\Rh$ and $M=\Ch$. An $x\in M$ will be called even (resp.\ odd) if it is a sum of terms of the form $s^i\otimes m$ with $|s^i|=i(1-n)$ even (resp.\ odd). The submodules of even and odd elements will be denoted by $M^{even}$ and $M^{odd}$ respectively. In the case treated in this paper, $n$ is even, so $|s|=1-n$ is odd, and thus $x$ is even (resp.\ odd) if and only if $x=x^0$ (resp.\ $x=x^1$).

Let $\gamma\in \mI_Q D$ be a closed form with $|\gamma|=2$. For example, given closed differential forms $\gamma_j\in A^*(X,L)$ for $j=0,\ldots,N$, take $t_j$ of degree $2-|\gamma_j|$ and $\gamma:=\sum_{j=0}^N t_j\gamma_j$.
Define structure maps
\[
\mgh_k:\Ch^{\otimes k}\lrarr \Ch
\]
by
\begin{multline*}
\mgh_k(\alpha_1,\ldots,\alpha_k):=\\
=\delta_{k,1}\cdot d\alpha_1+{(-1)}^{\sum_{j=1}^k j(|\alpha_j|+1)+1}
\sum_{\substack{\beta\in\sly\\l\ge0}}T^{\beta}\frac{1}{l!}{evb_0^\beta}_* (\bigwedge_{j=1}^l{(evi_j^\beta)}^*\gamma\wedge
\bigwedge_{j=1}^k {(evb_j^\beta)}^*\alpha_j
).
\end{multline*}
The push-forward ${(evb_0^\beta)}_*$ is defined by integration over the fiber; it is well-defined because $evb_0^\beta$ is a proper submersion. The condition $\gamma\in \mI_Q D$ ensures that the infinite sum converges.
Intuitively, $\gamma$ should be thought of as interior constraints, while $\alpha_j$ are boundary constraints. Then the output is a cochain on $L$ that is ``Poincar\'e dual'' to the image of the boundaries of disks that satisfy the given constraints.

Define the pairing $\langle\;,\;\rangle:\Ch\otimes \Ch \to \Rh^{odd}$ by
\[\langle \xi, \eta\rangle:= {(-1)}^{|(\eta)|}{\left(\int_L\xi\wedge\eta\right)}^1.\]
In Section~\ref{Fukaya algebras} we show that $(\Ch,{\{\mgh_k\}}_{k\ge 0},\langle\;,\,\rangle,1)$ is an $n$-dimensional curved cyclic unital $A_\infty$ algebra. 

\begin{dfn}\label{dfn_bd_pair}
A \textbf{bounding pair} with respect to $J$ is a pair $(\gamma,b)$ where $\gamma\in \mathcal{I}_Q D$ is closed with $|\gamma|=2$ and $b\in \left(\langle s\rangle +\mathcal{I}_{\hat{R}} \right)\hat{C}$ with
$
|b|=1,
$
such that
\begin{equation}\label{eq:bc}
\sum_{k\ge 0}\mh_k^\gamma(b^{\otimes k})=c\cdot 1, \qquad c\in \mI_{\Rh}\cap Z\left(\Rh\right),\qquad |c|=2. 
\end{equation}
In this situation, following~\cite{fukaya2010lagrangian}, $b$ is called a \textbf{bounding cochain} for $\mgh$. 
\end{dfn}
\begin{rem}
The condition that $c$ belongs to the center of $\Rh$ is crucial for bounding cochains in the present context to behave analogously to bounding cochains of Fukaya $A_\infty$ algebras over commutative rings.
\end{rem}
\begin{dfn}
A \textbf{unit bounding pair} $(\gamma,b)$ is a bounding pair  such that $ \int_L b^1 \in s \hat{R}^*  $.
\end{dfn}

The superpotential is defined by
\begin{equation}\label{eq:superpotential}
    \Omega(\gamma,b)
:=\big(\sum_{k\ge0}\frac{1}{(k+1)}\langle\mh_k^\gamma (b^{\otimes k}),b\rangle \big)
\end{equation}
\begin{rem}
The superpotential in~\cite{solomon2016point} includes an $\m_{-1}$ term to account for disks without boundary constraints, and the enhanced superpotential of~\cite{solomon2019relative} adds a term to correct for disks degenerating into spheres. These terms cannot be included in the superpotential in the non-commutative setting of the present paper for the following reason.
The pairing $\langle \cdot, \cdot\rangle$ is by definition the odd part of the Poincar\'e pairing. It is necessary to take the odd part in order for the pairing to be cyclically symmetric. See Lemma~\ref{lm:gradedcyclic}. Cyclic symmetry is used in the proof that the superpotential is invariant under gauge equivalence as asserted by Theorem~\ref{thm_inv} below. On the other hand, terms arising from disks without boundary marked points or spheres are of even parity. They are not invariant by themselves, but in the commutative setting they are augmented by contributions from other terms of the superpotential. In the non-commutative setting, the rest of the superpotential is odd, so it cannot contribute to the invariance of even terms.
\end{rem}
Definition~\ref{dfn_g_equiv} gives a notion of gauge equivalence between a bounding pair $(\gamma,b)$ with respect to $J$ and a bounding pair $(\gamma',b')$ with respect to another almost complex structure~$J'$. Let $\sim$ denote the resulting equivalence relation.
\begin{thmI}[Invariance of the superpotential]\label{thm_inv}
If $(\gamma,b)\sim(\gamma',b')$, then $\Omega(\gamma,b)=\Omega(\gamma',b')$.
\end{thmI}
To obtain invariants from $\Omega$, we must understand the space of gauge equivalence classes of bounding pairs.
Assume $n>0$.
Define a map
\[
\varrho:\{\text{unit bounding pairs}\}/\sim\;\;\lrarr \;{(\mI_Q\Hh^*(X,L;Q))}_2\oplus {(\hat{R}^*)}^{even}_{0}\oplus
\frac{(\mI_{\hat{R}}^{even})_{1-n}}{\left(s^2\mI_{\hat{R}}^{even}\right)_{1-n}}
\]
by
\begin{equation}\label{eqn_rho}
\varrho([\gamma,b]):=\left([\gamma]\,,\frac{\int_L b^1}{s},\int_L b^0 \right).
\end{equation}
We prove in Lemma~\ref{lm_rho} that $\varrho$ is well defined.
\begin{thmI}[Classification of bounding pairs -- rational cohomology spheres]\label{thm1}
Assume $H^*(L;\R)\cong H^*(S^n;\R)$. Then
$\varrho$ is bijective.
\end{thmI}
In the presence of an anti-symplectic involution, the assumptions on $L$ can be relaxed. A \textbf{real setting} is a quadruple $(X,L,\omega,\phi)$ where $\phi:X\to X$ is an anti-symplectic involution such that $L\subset \fix(\phi)$. Throughout the paper, whenever we discuss a real setting, we fix a connected subset $\J_\phi \subset \J$ consisting of $J \in \J$ such that $\phi^*J = -J$. All almost complex structures of a real setting are taken from $\J_\phi$. If we use virtual fundamental class techniques, we can treat any $\omega$-tame almost complex structure $J$ satisfying $\phi^*J = -J$. Whenever we discuss a real setting, we take $\ssly_L \subset  H_2(X,L;\Z)$ with $\Im(\Id+\phi_*) \subset \ssly_L$, so $\phi_*$ acts on $\sly_L = H_2(X,L;\Z)/\ssly_L$ as $-\Id$. Also, the formal variables $t_i$ have even degree. We denote by $\Hh_\phi^{even}(X,L;\R)$ (resp. $H_\phi^{even}(X;\R)$) the direct sum over $k$ of the ${(-1)}^{k}$-eigenspace of $\phi^*$ acting on $\Hh^{2k}(X,L;\R)$ (resp. $H^{2k}(X;\R)$). Note that the Poincar\'e duals of $\phi$-invariant almost complex submanifolds of $X$ disjoint from $L$ belong to $\Hh_\phi^{even}(X,L;\R)$.
Extend the action of $\phi^*$ to $\L,Q,R,\hat{R},S,C,\hat{C}$ and $D$, by taking
\begin{equation}\label{eq:phi*ext}
\phi^*T^\beta = {(-1)}^{\mu(\beta)/2}T^\beta, \qquad \phi^* t_i = {(-1)}^{|t_i|/2}t_i,\qquad \phi^*(s)={(-1)}^{\frac{|s|(|s| -1)}{2}}s.
\end{equation}
The action on $S$ is extended as a graded anti-homomorphism, that is, a homomorphism $\phi^*: S \to S^{op}$.
An element $a$ of $\L,Q,R,\hat{R},S,C,\hat{C}$, or $D$ is called \textbf{real} if
\begin{equation}\label{eq:relt}
\phi^* a = -a.
\end{equation}
A pair $(a_1,a_2)$ of such elements is called \textbf{real} if each component is real, that is, if
$\phi^*a_i=-a_i$ for $i=1,2$.
For a group $Z$ on which $\phi^*$ acts, let $Z^{-\phi^*}\subset Z$ denote the elements fixed by $-\phi^*$.  Let
\[
\varrho_\phi:\{\text{real unit bounding pairs}\}/\sim\;\;\lrarr\;{(\mI_Q\Hh^*(X,L;Q))}^{-\phi^*}_2\oplus {(\hat{R}^*)}^{even}_{0}
\]
be given by the first two coordinates of the formula for $\varrho$. As for $\varrho$, the unit condition $\int_L b^1\in s\hat{R}^*$ is what guarantees that the second coordinate lies in $\hat{R}^*$.
Then, we obtain the following variant of Theorem~\ref{thm1}.
\begin{thmI}[Classification of bounding pairs -- real spin case]\label{thm2}
Suppose $( X, L, \omega, \phi )$ is a real setting, $\mathfrak{s}$ is a spin structure, and $n \equiv 2\mod 4$. Moreover, 
assume $H^i(L;\R) \simeq H^i(S^n;\R)$ for $i \not \equiv 1 \mod 4$.
Then $\varrho_\phi$ is bijective.
\end{thmI}
\subsubsection{Open Gromov-Witten invariants}\label{sssec:intro_OGW}
When the hypothesis of either Theorem~\ref{thm1} or~\ref{thm2} is satisfied, we define open Gromov-Witten invariants as follows. In the case of Theorem~\ref{thm2}, take $\ssly_L$ containing $\Im(\Id+\phi_*)$. In the case of Theorem~\ref{thm1} (resp. Theorem~\ref{thm2}) let $W_L = \Hh^*(X,L;\R)$
(resp. $W_L = \Hh^{even}_\phi(X,L;\R)$). Fix
$\g_0,\ldots,\g_N$, a basis of $W_L$, set $|t_j|=2-|\g_j|$, and take
\[
\g:=\sum_{j=0}^N t_j\g_j \in {(\mI_Q\Hh^*(X,L;Q))}_2.
\]
By Theorem~\ref{thm1} (resp. Theorem~\ref{thm2}), choose a unit bounding pair $(\gamma,b)$ such that
\begin{equation}\label{eq:cbp}
\varrho([\gamma,b])=(\g,1,0) \qquad  \text{(resp. $\varrho_\phi([\gamma,b])=(\g,1))$.}
\end{equation}
By Theorem~\ref{thm_inv}, the superpotential $\Omega = \Omega(\gamma,b)$ is independent of the choice of $(\gamma,b)$.
\begin{dfn}
The \textbf{open Gromov-Witten invariants} of $(X,L)$, for every integer
$l\geq 0$,
\[
\ogw_{\beta,k}=\ogw_{\beta,k}^L : W_L^{\otimes l} \to \R,
\]
are defined by setting
	\[
	\ogw_{\beta,k}(\g_{i_1},\ldots,\g_{i_l})
	:=k!\cdot [T^{\beta}s^k]\left(
	\left.\d_{t_{i_1}}\cdots\d_{t_{i_l}}\Omega\right|_{t_j=0}
	\right)
	\]
	and extending linearly to general input.
	\end{dfn}
	
	\begin{rem}\label{rem:ogw_coeff_extraction}
	The coefficient extraction in $s$ replaces differentiation with respect to $s$. Indeed, since $s$ is odd, a formal derivative $\partial_s$ with $\partial_s s=1$ would be an odd derivation. Thus, the graded Leibniz rule gives
	\[
	\partial_s(s^k) = s^{k-1} - s \partial_s (s^{k-1}),
	\]
	and repeated differentiation in $s$ does not recover the coefficients of the powers $s^k$. The factor $k!$ is included so that in the commutative case, the definition agrees with differentiation with respect to $s$ by the identity
	\[
	\left.\partial_s^k f\right|_{s=0}=k![s^k]f.
	\]
	\end{rem}
	
	\begin{rem}\label{rem1}
	Theorems~\ref{thm1} and~\ref{thm2} say that a bounding cochain is determined up to equivalence by the cohomology class of its part that has degree $n$ in $A^*(L)$. In general, the degree $n$ part of $b$ must be ``corrected'' by non-closed forms of lower degrees in order to solve equation~\eqref{eq:bc}. Equation~\eqref{eq:cbp} says that the degree $n$ part of the bounding cochain parameterizes multiples of the point as the formal variable $s$ varies. Thus, following~\cite{solomon2016point}, such a bounding cochain is called \emph{point-like}.
	\end{rem}

\begin{rem}\label{rem:dim23}
In the special case that $n=2$, the cohomological assumption is always satisfied. This explains the significance of dimension $2$ in Welschinger's work~\cite{ welschinger2005invariants}. See Theorem~\ref{Welsch} for a comparison of Welschinger's invariants with the invariants of the present work.
\end{rem}
\subsubsection{Comparison with Welschinger invariants}\label{sssec:comps}
For a symplectic manifold $(X,\omega)$ of complex dimension $2$ with an anti-symplectic involution $\phi: X \to X$, we proceed to compare the invariants defined above with Welschinger's invariants~\cite{welschinger2005invariants}. We begin with some relevant notation. Let $Y \subset \fix(\phi)$ be closed and open. Define the doubling map
\begin{equation}\label{eq:chi}
\chi_Y: H_2(X,Y;\Z) \to H_2(X;\Z)
\end{equation}
as follows. For $\beta \in H_2(X,Y;\Z)$ represent $\beta$ by a singular chain $\sigma \in C_2(X,Y;\Z)$. Then $\chi_Y(\beta) = [\sigma - \phi_\#\sigma]$.  Observe that $\chi_{Y}$ descends to the quotient of $H_2(X,Y;\Z)$ by any subgroup contained in $\Im(\Id+\phi_*)$. We denote the map on the quotient by $\chi_Y$ as well. In particular, as long as $\ssly_L \subset \Im(\Id+\phi_*)$, we have $\chi_L : \sly_L \to H_2(X;\Z)$. Thus, for the following theorem, we take $\ssly_L \subset \Im(\Id+\phi_*)$. When we consider invariants arising from Theorem~\ref{thm2}, the opposite inclusion holds as well, so $\ssly_L = \Im(\Id+\phi_*)$.

For $d \in H_2(X;\Z), l \in \Z_{\geq 0}$, write
\[
k_{d,l} = c_1(X)(d) - 2l-1.
\]
Let $\mathcal{W}_{d,l}$ denote Welschinger's invariant counting real rational $J$-holomorphic curves in $(X,\phi)$ of degree $d$ with real locus in $L$ passing through $k_{d,l}$ real points and $l$ pairs of $\phi$-conjugate points. Recall from Remark~\ref{rem:dim23} that when $n = 2$, the hypothesis of Theorem~\ref{thm2} is always satisfied, so the invariants $\ogw_{\beta,k}$ are defined.
\begin{thmI}[Comparison with Welschinger's invariants]\label{Welsch}
Suppose $n=2$. Let $A \in \Hh^{2n}_\phi(X,L;\R)$ be the Poincar\'e dual of a point. Let $d \in H_2(X;\Z), l \in \Z_{\geq 0}$, such that  $k_{d,l} \geq 0$. Then,
\[
\sum_{\chi_L(\beta)=d}
\ogw_{\beta,\,k_{d,l}}(A^{\otimes l})
=\pm 2^{1-l}\cdot \mathcal{W}_{d,\,l}.
\]
\end{thmI}

\subsection{General notation}\leavevmode\label{ssec:gen_not}
We write $I:=[0,1]$ for the closed unit interval.
Use $i$ to denote the inclusion $i:L\hookrightarrow X$. By abuse of notation, we also use $i$ for $\Id\times i:I\times L\to I\times X$. The meaning in each case should be clear from the context.

Denote by $pt$ the map (from any space) to a point.

Define a filtration
\[
\nu: R \longrightarrow \mathbb{R}_{\geq 0}
\]
by
\[
\nu\left(\sum_{j=0}^{\infty} a_{j} T^{\beta_{j}} \prod_{r=0}^{N} t_{r}^{\ell_{j,r}}\right)
=\inf_{a_j\ne 0}\left(\omega\left(\beta_{j}\right)+\sum_{r=0}^{N}\ell_{j,r}\right).
\]
Let $\Upsilon^{\prime}$ be an $\mathbb{R}$ vector space and let $\Upsilon=\Upsilon^{\prime} \otimes R$. Then, equipping $\Upsilon^{\prime}$ with the trivial filtration, $\nu$ induces a filtration on $\Upsilon$ which we also denote by $\nu$. Whenever a tensor product (resp.\ direct sum) of modules with filtration is written, we mean the completed tensor product (resp.\ direct sum).

Write $A^*(L;R)$ for $R \otimes A^*(L) $. Similarly, $A^*(X;R)$ and $A^*(X,L;R)$ stand for $ R \otimes A^*(X)$ and $R \otimes A^*(X,L) $, respectively.           

We implicitly assume that elements of graded rings and modules are homogeneous and write $| \cdot |$ for the degree.
Given $\alpha$, a homogeneous differential form with coefficients in $R$, denote by $|\alpha|_{A}$ the degree of the differential form, ignoring the grading of $R$.

For a possibly non-homogeneous $\alpha$, denote by ${(\alpha)}_j$ the form that is the part of degree $j$ in $\alpha$. In particular, $|{(\alpha)}_j|_A=j$. For a graded module $M$, the notation $M_j$ or ${(M)}_j$ stands for the degree~$j$ part of the module, which in the present context involves degrees of forms as well as degrees of variables.

Let $\Upsilon'$ be an $\R$-vector space, $\Upsilon''=R, Q$, or $\L$, and let $\Upsilon=\Upsilon'\otimes \Upsilon''$. For $x\in \Upsilon$ and $\lambda\in\Upsilon''$, denote by $[\lambda](x)\in \Upsilon'$ the coefficient of $\lambda$ in $x$.

\subsection{AI declaration}
ChatGPT and Claude were used to proofread the paper and they helped correct minor errors.

\subsection{Acknowledgements}
The authors would like to thank P.~Seidel for suggestions that had a major impact on the present work. They would also like to thank O.~Kedar for helpful conversations. The authors were partially supported by ISF grants 569/18,~1127/22 and~3029/26. J.~S. was partially supported by the Miriam and Julius Vinik Chair in Mathematics.

\section{\texorpdfstring{$A_\infty$}{A-infinity} algebras over an associative algebra}\label{sec:Non_commutative_algebra}
\subsection{Graded and filtered algebras}
We recall that if $V$ is a vector space and $G$ an abelian group, then a $G$-grading of $V$ is a decomposition of $V$ into a direct sum of subspaces $V = \bigoplus_{\alpha \in G} V_\alpha$. If $v \in V_\alpha$, then we say that $v$ is homogeneous of degree $\alpha$ and we write $|v|=\alpha$. 
A $G$-graded algebra is an algebra $A$ with a $G$ grading for which $A_\alpha A_\beta \subset A_{\alpha+\beta}$. 
Let $A$ and $B$ be associative graded algebras, graded by $\Z$. Their tensor product $A \otimes B$ is the graded algebra whose space is the tensor product of the spaces of $A$ and $B$, with the induced $\Z$-grading and the product operation defined by 
\[
\left(a_{1} \otimes b_{1}\right)\left(a_{2} \otimes b_{2}\right)={(-1)}^{\left(| a_{2}|\right)\left(| b_{1}|\right)} a_{1} a_{2} \otimes b_{1} b_{2}, \qquad a_{i} \in A, \qquad b_{i} \in B.
\]
If $M$ is a graded $A$ bimodule and $N$ is a graded $B$ bimodule, then $M \otimes N$ is naturally a graded $A\otimes B$ bimodule. Explicitly, for $a \in A, b \in B, m \in M$ and $n \in N$, we have
\[\left(a \otimes b\right)\left(m \otimes n\right)={(-1)}^{\left(| b|\right)\left(| m|\right)} a m \otimes b n, \qquad \left(m \otimes n\right)\left(a \otimes b\right)={(-1)}^{\left(| a|\right)\left(| n|\right)}  m a \otimes n b.
\]
The bracket $[\cdot\, , \cdot ]$ in a graded associative algebra is defined by
\[
[a,b]=ab-{(-1)}^{{|a|}{|b|}}ba.
\]
If $A$ is an associative algebra with a bracket $[x, y]$, then the centralizer of a subset $S$ of $A$ is defined to be
\[
Z(S):=\{x \in A  \mid[x, s]=0 \text { for all } s \in S\}.
\]
For a graded vector space $V$, let  $V[k]$ be a copy of $V$ with the shifted grading ${V[k]}_i = {V}_{k+i}$.

\begin{dfn}
Let $M,N$ be differential graded associative algebras over a field $\k$. A differential graded $\k$ linear map $F:M \to N$ is said to be \textbf{graded cyclic} if for all $s_1,\dots, s_n\in  M$, we have
\[F(s_1 \cdots s_n)={(-1)}^{|s_n|\sum_{i=1}^{n-1}|s_i|}F(s_n \cdot s_1 \cdots s_{n-1}).\] 
\end{dfn}
\begin{ex}
The identity morphism of a graded commutative algebra is a graded cyclic morphism. Let $M_{n}(R)$ denote the algebra of $n\times n$ matrices with entries in a graded commutative $\k$ algebra $R$. Then the trace map $\tr : M_n(R) \to R$ is a graded cyclic morphism.
\end{ex}

\begin{dfn}\label{dfn:frm}
A \textbf{filtered ring} is a ring $\mathcal N$ with a map $\zeta_\mathcal{N}: \mathcal{N} \to \R_{\ge 0}\cup\{\infty\}$ such that
\[
\zeta_\mathcal{N}(x) = \infty \iff x = 0, \qquad \zeta_\mathcal{N}(xy) \ge \zeta_\mathcal{N}(x) + \zeta_\mathcal{N}(y), \qquad \zeta_\mathcal{N}(x + y) \geq \min(\zeta_\mathcal{N}(x),\zeta_\mathcal{N}(y)).
\]
We write $F_{\zeta_\mathcal{N}}^a\mathcal{N} = \{x \in \mathcal{N}| \zeta_\mathcal{N}(x) \ge a\}$ and $F_{\zeta_\mathcal{N}}^{> a}\mathcal{N} = \{x \in \mathcal{N}|\zeta_\mathcal{N}(x) > a\}$.
For $E\in \R$ abbreviate by $\mathcal{N}_E$ the $\overline{\mathcal{N}}$-modules $F_{\zeta_\mathcal{N}}^{ E}\mathcal{N}/F_{\zeta_\mathcal{N}}^{>E}\mathcal N$. The  \textbf{residue ring} is $\overline{\mathcal{N}} = \mathcal{N}_0$. 

A \textbf{filtered dga} over a field $\k$ is a dga $\mathcal{N}$ over $\k$ equipped with a filtration $\zeta_{\mathcal{N}}$ that makes $\mathcal{N}$ a filtered ring such that the differential $d$ satisfies $\zeta_{\mathcal{N}}(dx) \geq \zeta_{\mathcal{N}}(x)$ for $x \in \mathcal{N}$ and the projections to each graded part $\mathcal{N}_m$ do not decrease the filtration.

A \textbf{filtered module} over $\mathcal{N}$ is an $\mathcal{N}$-module $\mathcal{C}$ along with a map $\zeta_\mathcal{C} : \mathcal{C} \to \R_{\ge 0}\cup\{\infty\}$ such that
\[
\zeta_\mathcal{C}(x) = \infty \iff x = 0, \qquad \zeta_\mathcal{C}(xy) \ge \zeta_{\mathcal{N}}(x) + \zeta_\mathcal{C}(y), \qquad \zeta_\mathcal{C}(x + y) \geq \min(\zeta_{\mathcal{C}}(x), \zeta_\mathcal{C}(y)).
\]
A \textbf{filtered graded module} is a graded filtered module such that the projections to its graded parts $\mathcal{C}_m$ do not decrease the filtration.

We write $F_{\zeta_\mathcal{C}}^a\mathcal{C} = \{x \in \mathcal{C}| \zeta_\mathcal{C}(x) \ge a\}$ and $F_{\zeta_\mathcal{C}}^{> a}\mathcal C = \{x \in \mathcal{C}|\zeta_\mathcal{C}(x) > a\}$. For $E\in \R$ abbreviate by $\mathcal{C}_E$ the $\overline{\mathcal{N}}$ modules $F_{\zeta_\mathcal{C}}^{ E}\mathcal{C}/F_{\zeta_\mathcal{C}}^{>E}\mathcal C$. The \textbf{residue module} is the $\overline{\mathcal{N}}$-module $\overline{\mathcal{C}} = \mathcal{C}_0$. For $n \in \mathcal{N}$ we write $\bar{n} \in \overline{\mathcal{N}}$ for the residue. If $\mathcal{C}$ and $\mathcal{D}$ are filtered $\mathcal{N}$-modules and $f : \mathcal{C} \to \mathcal{D}$ is a filtration-preserving homomorphism of $\mathcal{N}$-modules, that is, $f(F_{\zeta_\mathcal{C}}^a\mathcal{C})\subseteq F_{\zeta_\mathcal{D}}^a\mathcal{D}$ for every $a\in\R_{\geq 0}$, we write $\bar f : \overline{\mathcal{C}} \to \overline{\mathcal{D}}$ for the induced homomorphism of $\overline{\mathcal{N}}$-modules.
\end{dfn}
\begin{dfn}\label{dfn:bar_structre}
Let $\mathcal{N}$ be a filtered graded-commutative unital dga. An $\overline{\mathcal{N}}$ \textbf{structure} on $\mathcal{N}$ is a differential graded $\overline{\mathcal{N}}$-algebra structure on $\mathcal{N}$ such that the residue map $\mathcal{N}\to\overline{\mathcal{N}}$ is a dga augmentation. Equivalently, it is a unital dga homomorphism $\overline{\mathcal{N}}\to\mathcal{N}$ such that the composition $\overline{\mathcal{N}}\to\mathcal{N}\to\overline{\mathcal{N}}$ is the identity.

Assume now that $\mathcal{N}$ has such a structure, and let $\mathcal{C}$ be a filtered graded $\mathcal{N}$-module whose
residue $\overline{\mathcal{C}}$ is flat over
$\overline{\mathcal{N}}$. An $\overline{\mathcal{N}}$ \textbf{structure} on $\mathcal{C}$ is a filtered graded $\mathcal{N}$-module isomorphism
\[
g:\mathcal{N}\otimes_{\overline{\mathcal{N}}}\overline{\mathcal{C}}\xrightarrow{\sim}\mathcal{C}
\]
that induces the identity on the residue.
\end{dfn}
\begin{lm}\label{lm:residue_id}
Let $\cC$ be a filtered graded $\mathcal{N}$-module with an $\bar{\mathcal{N}}$ structure given by $g:\mathcal{N} \otimes_{\bar{\mathcal{N}}} \overline{\mathcal{C}} \to \cC$. For all $b \in \bar{\mathcal{C}}$ we have
\[\overline{g(1\otimes b)}=b.\]
\end{lm}
\begin{proof}
By the definition of an $\bar{\mathcal{N}}$ structure, the composition
\begin{align*}
    & \bar{\cC}\to \bar{\mathcal{N}}\otimes_{\bar{\mathcal{N}}} \bar{\cC}\overset{\bar{g}}{\to}\bar{\cC}\\
    & b \,\mapsto \, \bar{1} \otimes b \; \mapsto\, \bar{g}(\bar{1}\otimes b)
\end{align*}
is the identity.
Hence $b=\bar{g}(\bar{1}\otimes b)=\overline{g(1\otimes b)}$.
\end{proof}
\begin{lm}\label{lm:residue_zero_equiv}
Let $\cC$ be a filtered graded $\mathcal{N}$-module with an $\bar{\mathcal{N}}$ structure given by
\[
g:\mathcal{T}:=\mathcal{N} \otimes_{\bar{\mathcal{N}}} \overline{\mathcal{C}} \longrightarrow \cC,
\]
and let $f$ be its inverse. Here $\mathcal{T}$ carries the filtration with respect to which $g$ is a filtered isomorphism. Suppose $\a \in \cC$. Then
\[
f(\a)\equiv 1 \otimes \bar{\a} \pmod{F^{>0}\mathcal{T}}.
\]
\end{lm}
\begin{proof}
It is enough to show that $\overline{f(\alpha)}=\bar{1}\otimes \bar{\a}$. Indeed by Lemma~\ref{lm:residue_id} 
\[\bar{g}(\bar{1}\otimes \bar{\a})=\overline{g\left(1\otimes \bar{\a}\right)}=\bar{\a}.\]
This implies 
\[\bar{1} \otimes \bar{\a}=\bar{f}\left(\bar{g}(\bar{1}\otimes \bar{\a})\right)=\bar{f}(\bar{\a}). \]
\end{proof}
Let $\mathcal{N}$ be a filtered graded-commutative unital dga with an $\bar{\mathcal{N}}$ structure given by  $h:\bar{\mathcal{N}}\to \mathcal{N}$. Let $\mathcal{C}$ be a filtered graded $\mathcal{N}$-module with a $\bar{\mathcal{N}}$ structure given by $g:\mathcal{N} \otimes_{\bar{\mathcal{N}}} \overline{\mathcal{C}}\overset{\sim}{\longrightarrow} \mathcal{C}$. Let $A \subset \mathcal{N}$ be a unital differential graded subalgebra with the induced filtration such that $h(\bar{\mathcal{N}})\subset A.$ Let $h_A : \bar{\mathcal{N}} \to A$ denote the corestriction.
\begin{lm}\label{lm:A_bar_A_str}
We have $\bar A = \bar{\mathcal{N}}$. It follows that $h_A : \bar{A} = \bar{\mathcal{N}} \to A$ is an $\bar{A}$ structure.
\end{lm}
\begin{proof}
Let $\pi:\mathcal{N} \to \bar{\mathcal{N}}$ be the residue map. As $h(\bar{\mathcal{N}})\subset A$ it follows that 
\[
\bar{\mathcal{N}}= \pi \circ h(\bar{\mathcal{N}}) \subset \bar{A} \subset \bar{\mathcal{N}}.
\]
So, $\bar A = \bar{\mathcal{N}}$. Let $\pi_A = \pi|_A: A \to \bar A $ denote the residue map. Then $\pi_A \circ h_A = \pi \circ h = \Id$. Since $h_A$ is a dga homomorphism, it is an $\bar{A}$ structure.
\end{proof}
Let 
\[
\mathcal{C}_A:=g(A \otimes_{\bar{A}} \overline{\mathcal{C}}).
\]
\begin{lm}\label{lm:bar_AC_bar_C}
We have $\overline{\mathcal{C}_A}=\overline{\cC}$.
\end{lm}
\begin{proof}
Keeping in mind Lemma~\ref{lm:A_bar_A_str}, since $g$ preserves filtration, we have 
\[ 
\overline{\mathcal{C}_A} = \overline{g(A \otimes_{\bar{A}}\overline{\cC})} = \bar g (\overline{A} \otimes_{\bar{A}} \overline{\cC})=\bar{g}(\bar{\mathcal{N}}\otimes_{\bar{\mathcal{N}}}\overline{\cC}  )=\overline{\cC}.
\]
\end{proof}
Let 
\[
g_A : A\otimes_{\bar{A}} \overline{\mathcal{C}_A} \to
\mathcal{C}_A
\]
be given by $g_A=g|_{A\otimes_{\bar{A}}\overline{\mathcal{C}} }.$
\begin{lm}
$g_A$ is an $\bar A$ structure on $\mathcal C_A$.
\end{lm}

\begin{proof}
For $E\ge0$, flatness and~\cite[\href{https://stacks.math.columbia.edu/tag/00HD}{Tag 00HD}]{stacks-project} give an exact sequence
\[
0\longrightarrow
\left(A/F^{>E}A\right)\otimes_{\overline{\mathcal N}}\overline{\mathcal C}
\longrightarrow
\left(\mathcal N/F^{>E}\mathcal N\right)\otimes_{\overline{\mathcal N}}\overline{\mathcal C}.
\]
Taking inverse limits and using~\cite[\href{https://stacks.math.columbia.edu/tag/02N1}{Tag 02N1}]{stacks-project}, we obtain
\[
0\longrightarrow A\otimes_{\overline{\mathcal N}}\overline{\mathcal C}
\longrightarrow\mathcal N\otimes_{\overline{\mathcal N}}\overline{\mathcal C}.
\]
The injections on filtration quotients imply that $g_A$ is a filtered isomorphism onto $\mathcal C_A$.

In Lemma~\ref{lm:bar_AC_bar_C} we showed $\overline{\cC_A}=\overline{\cC}$, and in Lemma~\ref{lm:A_bar_A_str} we showed that $\overline A=\overline{\mathcal N}$. Under these identifications, $\overline g_A=\overline g$. Hence the fact that $g$ is an $\overline{\mathcal{N}}$ structure on $\cC$ implies that $g_A$ induces the identity
\[
\begin{tikzcd}
\overline{\cC} \arrow[r, "\simeq "] \arrow[d, equal] & \overline{\mathcal N}\otimes_{\overline{\mathcal N}}\overline{\cC} \arrow[rr, "\overline g"] \arrow[d, equal] & & \overline{\cC} \arrow[d, equal] \\
\overline{\cC_A} \arrow[r, "\simeq "] & \overline A\otimes_{\overline A}\overline{\cC_A} \arrow[rr, "\overline g_A"] & & \overline{\cC_A}
\end{tikzcd}
\]
Since $g_A$ preserves degrees, it is an $\overline A$ structure on $\cC_A$.
\end{proof}

\begin{dfn}\label{dfn:restriction_of_scalars}
We call $h_A$ the \textbf{canonical} $\bar A$ structure on $A.$ 
We call $\mathcal{C}_A$ the \textbf{restriction of scalars} of $\mathcal{C}$. Under the hypotheses of the preceding lemma, we call $g_A$ the \textbf{canonical} $\bar A$ structure on $\mathcal{C}_A.$
\end{dfn}
\subsection{Cyclic unital \texorpdfstring{$A_{\infty}$}{A infinity} algebras}\label{ssec:Cyc_A_infty_alg}
Let $\mathcal{R}$ be a filtered dga over a field $\k$ of characteristic different then $2$ with filtration $\varsigma_{\mathcal{R}}$, let $\cC$ be a filtered $\mathcal{R}$ bimodule with filtration $\varsigma_{\cC}$, and let $M$ be a filtered differential graded $Z(\mathcal{R})$ module with differential $d_M$ and filtration $\varsigma_M$.

The following extends \cite[Definition~1.1]{solomon2016differential} for the case that $\mathcal{R}$ is not commutative.
\begin{dfn}\label{dfn:cycunit}
An $n$-dimensional (curved) \textbf{cyclic unital $A_\infty$ structure} over $\mathcal{R}$ with values in $M$ on $\cC$ is a triple $(\{\m_k\}_{k\ge 0},\lp\;,\,\rp,\e)$ of
maps $\m_k:\cC^{\otimes k}\to \cC[2-k]$, a $\k$-bilinear pairing $\lp\;,\,\rp:\cC\otimes_{\k}\cC\to M[-n]$, and an element $\e\in \cC$ with $|\e|=0$, satisfying the following properties. We denote by $\alpha,$ possibly with subscripts, an element of $\cC,$ and by $a$ an element of $\mathcal{R}.$
\begin{enumerate}[label=(\arabic*), ref=(\arabic*)]
		\item\label{it:lin}
		The operations $\m_k$ are $\mathcal{R}$-multilinear in the sense that
		\[\m_k(a\cdot\alpha_1,\ldots,\alpha_k)=
		(-1)^{|a|}
		a\cdot\m_k(\alpha_1,\ldots,\alpha_k)+\delta_{1,k}\cdot da\cdot\alpha_1,\]
		and
		\[\m_k(\alpha_1,\ldots,\alpha_k \cdot a)=
		\m_k(\alpha_1,\ldots,\alpha_k)\cdot a+(-1)^{|\a_1|+1}\delta_{1,k}\cdot \alpha_1\cdot da,\]
		and
		\[
		\m_k(\alpha_1,\ldots,\alpha_{i-1},a\cdot\alpha_i,\ldots,\alpha_k)=\m_k(\alpha_1,\ldots,\alpha_{i-1}\cdot a,\alpha_i,\ldots,\alpha_k).
			\]
    \item\label{it:a_infty}
    The $A_\infty$ relations hold:
    \[
    \sum_{\substack{k_1+k_2=k+1\\1\le i\le k_1}}(-1)^{\sum_{j=1}^{i-1}(|\alpha_j|+1)}
\m_{k_1}(\alpha_1,\ldots,\alpha_{i-1},\m_{k_2}(\alpha_i,\ldots,\alpha_{i+k_2-1}), \alpha_{i+k_2},\ldots,\alpha_k)=0.
    \]
    \item\label{it:val}
    $\varsigma_{\cC}(\m_k(\alpha))\ge \varsigma_{\cC}(\alpha)$ and $\varsigma_{\cC}(\m_0) > 0.$
	\item\label{it:unit1}
	$\m_k(\alpha_1,\ldots,\alpha_{i-1},\e,\alpha_{i+1},\ldots,\alpha_k)=0 \qquad \forall k\ne 0,2.$
	\item\label{it:unit2}
	$\m_2(\e,\alpha)=\alpha=(-1)^{|\alpha|}\m_2(\alpha,\e).$
	\item\label{it:plin}
		The pairing satisfies
		\[
		\lp\a_1,a\cdot\a_2\rp=\lp\a_1\cdot a,\a_2\rp,\qquad \forall a 
		\in \mathcal{R},
		\]
		and
		\[\lp b\cdot\a_1,\a_2\rp=b\cdot \lp\a_1,\a_2\rp,\qquad \forall b \in Z(\mathcal{R}).\]
    \item\label{it:val2}
    $\varsigma_M(\lp\alpha_1,\alpha_2\rp)\ge\varsigma_{\cC}(\alpha_1)+\varsigma_{\cC}(\alpha_2).$
    \item\label{it:symm}
    $\lp \alpha_1,\alpha_2\rp=(-1)^{(|\alpha_1|+1)(|\alpha_2|+1)+1}\lp\alpha_2,\alpha_1\rp$.
	\item\label{it:cyclic}
     The pairing is cyclic: for $k\geq 1$,
	\begin{multline*}
    \qquad\lp \m_k(\alpha_1,\ldots,\alpha_k),\alpha_{k+1}\rp=\\
	=(-1)^{(|\alpha_{k+1}|+1)\sum_{j=1}^k(|\alpha_j|+1)}
	\lp\m_k(\alpha_{k+1},\alpha_1,\ldots,\alpha_{k-1}),\alpha_k\rp+
	\delta_{1,k}\cdot d_M\lp\alpha_1,\alpha_2\rp.
    \end{multline*}
    \item\label{it:unit3}
    $\lp \m_0,\e\rp=0$.

\end{enumerate}
\end{dfn}
We omit the pairing target from the terminology when it is understood.
\begin{dfn}\label{dfn:A_infty_arises_dga}
A cyclic unital $A_\infty$ algebra $(\cC, \{\m_k\}_{k\ge 0},\lp\;,\,\rp,\e)$ \textbf{arises from a dga} if
		\[
		\qquad \m_k=0, \quad k \neq 1,2.
		\]
In this case the \textbf{underlying dga} structure on $\cC$ is given by
\[
d \a_1= \m_1(\a_1),\qquad \a_1 \a_2=(-1)^{|\a_1|}\m_2(\alpha_1,\alpha_2), \qquad \a_1, \a_2 \in \cC.
\]
Let $\Upsilon$ denote the algebra obtained by equipping $\cC$ with the product above. We say that $(\cC, \{\m_k\}_{k\ge 0},\lp\;,\,\rp,\e)$ arises from the dga $(\Upsilon,d).$
\end{dfn}
Recall the notion of the residue module from Definition~\ref{dfn:frm}.
By property~\ref{it:val}, the maps $\m_k$ descend to maps on the residue module,
\[
\mbar_k:\overline{\cC}^{\otimes k}\lrarr \overline{\cC}.
\]
For $\bar{\a}_1,\dots ,\bar{\a}_k\in \overline{\cC}$, we have 
\[\mbar_k\left(\bar{\a}_1,\dots ,\bar{\a}_k\right)=\overline{\m_k\left(\a_1,\dots, \a_k\right)},\]
where $\a_i \in \cC$ is a representative for $\bar{\a}_i$. By property~\ref{it:val2}, the pairing descends to $\overline{\cC}\otimes_{\k}\overline{\cC}\to\overline M[-n]$. Similarly, for $k\ge 1$ we define operators 
\[
\m_{k,E}:\left(\cC^{\otimes k}\right)_{E}\lrarr \cC_{E}.
\]
For $[\a_1\otimes\ldots\otimes\a_k]\in \left(\cC^{\otimes k}\right)_{E}$, we define 
\[\m_{k,E}\left([\a_1\otimes\ldots\otimes\a_k]\right)=\left[\m_k\left(\a_1\otimes\ldots\otimes\a_k\right)\right],\]
where $\a_1\otimes\ldots\otimes\a_k\in F^{E}\cC^{\otimes k}$ is a representative for $[\a_1\otimes\ldots\otimes\a_k].$ If $\m_0\in F^E\cC$, then we define $\m_{0,E}$ to be its image in $\cC_E.$ We sometimes extend the domain of the operators to $F^{E}\cC^{\otimes k}$ by precomposing them with the quotient map.
\begin{prop}\label{prop:r_structure_energy_dga}
Let $(\cC, \{\m_k\}_{k\ge 0},\lp\;,\,\rp,\e)$ be a cyclic unital $A_\infty$ algebra over a commutative dga $\mathcal{R}$. Assume that $\cC$ has an $\bar{\mathcal{R}}$ structure $g:\mathcal{R}\otimes_{\overline{\mathcal{R}}} \bar{\cC} \overset{\sim}{\to} \cC$ and let $f$ be its inverse. Assume also that $(\bar{\cC}, \{\bar{\m}_k\}_{k\ge 0},\lp\;,\,\rp,\e)$ arises from a dga. So, $\mathcal{R}\otimes_{\overline{\mathcal{R}}} \bar{\cC}$ inherits a dga structure as well. Let $\alpha_1,\dots, \a_k \in \cC$ be  homogeneous elements such that $\sum_{i=1}^k \varsigma_{\cC}(\alpha_i)=E$.
For $k=1$ we have
\[f\left(\m_1(\a_1)\right) \equiv df(\a_1) \pmod{F^{>E}(\mathcal{R}\otimes_{\overline{\mathcal{R}}}\bar{\cC})}.\]
If in addition the left and right $\mathcal{R}$ module structures on $\cC[1]$ coincide, for $k=2$ we have 
\[f\left(\m_2(\a_1, \a_2)\right)\equiv(-1)^{|\a_1|}f(\a_1)  f(\a_2) \pmod{F^{>E}(\mathcal{R}\otimes_{\overline{\mathcal{R}}}\bar{\cC})}. \]
If $k>2,$ then
\[\m_k(\a_1, \dots, \a_k)\equiv 0 \pmod{F^{>E}\cC}.\]
\end{prop}
\begin{proof}
 Assume without loss of generality there exist $a_j \in \mathcal{R}$ and $b_j \in \bar{\cC}$, such that $g(a_j\otimes b_j)=\a_j$ for $1\le j\le k$.
For $k=1$, we have $\varsigma_{\mathcal{R}}(a_1)=E$. We calculate
\begin{align}
    d f(\a_1)&=d(a_1 \otimes b_1) \nonumber\\
    &=(-1)^{|a_1|}a_1\otimes db_1+ da_1\otimes b_1\nonumber\\
    &= (-1)^{|a_1|}a_1\otimes \bar{m}_1(b_1)+ da_1\otimes b_1 \nonumber\\
    &=(-1)^{|a_1|}a_1(1 \otimes \mbar_1(b_1))+ da_1\otimes b_1. \nonumber\\
    \shortintertext{Using Lemma~\ref{lm:residue_id}, we proceed}
    &=(-1)^{|a_1|}a_1\left(1\otimes \overline{\m_1(g(1\otimes b_1))}\right)+da_1\otimes b_1
    \label{eq:dga_df}.
\end{align}
Using Lemma~\ref{lm:residue_zero_equiv}, we deduce that
\[1\otimes \overline{\m_1(g(1\otimes b_1))}\equiv f(\m_1(g(1\otimes b_1)))\pmod{F^{>0}(\mathcal{R}\otimes_{\overline{\mathcal{R}}}\bar{\cC})}.\] 
Remembering that $\cC$ is a filtered module, we combine the preceding equation with equation~\eqref{eq:dga_df} to deduce
\begin{multline*}
    df(\a_1)\equiv (-1)^{|a_1|}a_1f\left(\m_1(g(1\otimes b_1))\right)+ da_1\otimes b_1\equiv \\ \equiv f\left(\m_1(g(a_1\otimes b_1))\right)-f\left(da_1\cdot g(1\otimes b_1)\right)+ da_1\otimes b_1\equiv f(\m_1(\a_1)) \pmod{F^{>E}(\mathcal{R}\otimes_{\overline{\mathcal{R}}}\bar{\cC})}.
\end{multline*}

For $k=2$ we have
\begin{align}
f(\m_2(\a_1,\a_2))&=(-1)^{|a_1|}a_1f\left(\m_2(g(1\otimes b_1)a_2,g(1\otimes b_2))\right).\\
\shortintertext{
Remembering that the left and right $\mathcal{R}$ module structures on $\cC[1]$ coincide, we continue}
&=(-1)^{|a_1|+|a_2||b_1|}a_1 a_2f\left(\m_2(g(1\otimes b_1), g(1\otimes b_2))\right).
\end{align}
We use Lemma~\ref{lm:residue_zero_equiv} and Lemma~\ref{lm:residue_id} to deduce the equivalence
\begin{multline*}
    f\left(\m_2(g(1\otimes b_1), g(1\otimes b_2))\right)\equiv 1\otimes \overline{\m_2(g(1\otimes b_1), g(1\otimes b_2))}\equiv \\ \equiv 1\otimes \bar{\m}_2(\overline{g(1\otimes b_1)},\overline{g(1\otimes b_2)})\equiv 1\otimes \bar{\m}_2(b_1,b_2) \pmod{F^{>0}(\mathcal{R}\otimes_{\overline{\mathcal{R}}}\bar{\cC})}.
\end{multline*}
By the definition of the dga on $\bar{\cC}$ and the property $\varsigma_{\mathcal{R}}(a_1 a_2)\ge\varsigma_{\mathcal{R}}(a_1)+\varsigma_{\mathcal{R}}(a_2),$ we conclude
\begin{multline*}
    f(\m_2(\a_1,\a_2))\equiv (-1)^{|a_1|+|a_2||b_1|}(a_1 a_2) \left(1\otimes 
    \mbar_2(b_1, b_2)\right)\equiv \\\equiv (-1)^{|a_1|+|b_1|+|a_2||b_1|}(a_1 a_2) \otimes (b_1 b_2) \equiv  (-1)^{|\a_1|}f(\a_1)f(\a_2)\pmod{F^{>E}(\mathcal{R}\otimes_{\overline{\mathcal{R}}}\bar{\cC})}.
\end{multline*}
We omit the proof of the cases $k>2$ as it is similar.
\end{proof}

The following is a generalization of \cite[Proposition 4.20]{solomon2016differential}, which formulates the $A_\infty$ relations with respect to the cyclic pairing.
\begin{prop}[Cyclic structure equations]\label{prop:super_cyclic_structure_equation}
Let $(\mathcal{C}, \{\m_k\}_{k\ge 0},\lp\;,\,\rp,\e)$ be a cyclic unital $A_\infty$ algebra with values in $M$, then for $k\ge 0$,
\begin{multline*}
d_M\lp\m_k(\a_1,\ldots,\a_k),\a_{k+1}\rp=\\
=\sum_{\substack{k_1+k_2=k+1\\k_1\ge 1,k_2\ge 0\\1\le i\le k_1}}
(-1)^{\nu(\a;k_1,k_2,i)}\lp\m_{k_1}(\a_{i+k_2},\ldots,\a_{k+1},\a_1,\ldots,\a_{i-1}), \m_{k_2}(\a_i,\ldots,\a_{k_2+i-1})\rp
\end{multline*}
with
\[
\nu(\a;k_1,k_2,i):=\sum_{j=1}^{i-1}(|\a_j|+1)+
\sum_{j=i+k_2}^{k+1}(|\a_j|+1)\Big(\sum_{\substack{m\ne j\\1\le m\le k+1}}(|\a_m|+1)+1\Big)+1.
\]
\end{prop}
\begin{proof}
As our pairing $\lp \cdot, \cdot \rp$ is cyclic, the following equation holds for $\xi, \eta \in \cC$:
	\[
    \qquad\lp \m_1(\xi),\eta\rp
	=(-1)^{(|\eta|+1)(|\xi|+1)}
	\lp\m_1(\eta),\xi\rp+
	 d_M\lp\xi,\eta\rp.
    \]
We can plug in $\m_k(\alpha_1,\dots \alpha_k)$ and $\a_{k+1}$ instead of $\xi$ and $\eta$ in the equation above to get
\begin{align*}
&d_M\lp\m_k(\a_1,\ldots,\a_k),\a_{k+1}\rp=\\
&=\lp \m_1\left(\m_k(\a_1,\ldots,\a_k)\right),\a_{k+1}\rp-
(-1)^{(|\a_{k+1}|+1)(|\m_k(\a_1,\ldots,\a_k)|+1)}\lp \m_1(\a_{k+1}),\m_k(\a_1,\ldots,\a_k)\rp.
\end{align*}
Using the $A_\infty$ structure and cyclic symmetry we continue,
\begin{align*}
&d_M\lp\m_k(\a_1,\ldots,\a_k),\a_{k+1}\rp=\\
&=-\hspace{-1.5em}\sum_{\substack{k_1+k_2=k+1\\1<k_1\\ 1 \le i \le k_1}}\hspace{-1.5em} (-1)^{\sum_{j=1}^{i-1}(|\a_j|+1)}\lp\m_{k_1}(\a_1,\ldots,\a_{i-1},\m_{k_2}(\a_i,\ldots,\a_{i+k_2-1}),\a_{i+k_2},\ldots,\a_k),\a_{k+1}\rp+\\
&+(-1)^{(|\a_{k+1}|+1)(|\m_k(\a_1,\ldots,\a_k)|+1)+1}\lp \m_1(\a_{k+1}),\m_k(\a_1,\ldots,\a_k)\rp\\
&=\hspace{-0.5em}\sum_{\substack{k_1+k_2=k+1\\1<k_1\\ 1 \le i \le k_1}}\hspace{-1.5em} (-1)^{1+\sum_{j=1}^{i-1}(|\a_j|+1)+\mu}\lp\m_{k_1}(\a_{i+k_2},\ldots,\a_k,\a_{k+1},\a_1,\ldots,\a_{i-1}),\m_{k_2}(\a_i,\ldots,\a_{i+k_2-1})\rp+\\
&+(-1)^{(|\a_{k+1}|+1)(\sum_{j=1}^k(|\a_j|+1)+1)+1}\lp \m_1(\a_{k+1}),\m_k(\a_1,\ldots,\a_k)\rp,
\end{align*}
Let $\mu$ denote the cyclic-permutation sign in the preceding equality; using
\[
\left|\m_{k_2}(\a_i,\ldots,\a_{i+k_2-1})\right|
=\sum_{m=i}^{i+k_2-1}|\a_m|+2-k_2,
\]
we obtain
\[
\mu\equiv\sum_{j=i+k_2}^{k+1}(|\a_j|+1)
\left(\sum_{\substack{m\ne j\\1\le m\le k+1}}(|\a_m|+1)+1\right)
\pmod 2.
\]
Hence
\[
1+\sum_{j=1}^{i-1}(|\a_j|+1)+\mu
\equiv\nu(\a;k_1,k_2,i)\pmod 2,
\]
while the case $k_1=1$ follows directly from the definition of $\nu$.
\end{proof}

\subsection{Extension of scalars}\label{ssec:extension_of_scalar}
Let $\mathcal{R}$ be a differential graded algebra over a field $\k$ with filtration $\varsigma_{\mathcal{R}}$ and let $\cC$ be an $\mathcal{R}$ bimodule with filtration $\varsigma_{\cC}.$ Let $(\cC, \{\m_k\}_{k\ge 0},\lp\;,\,\rp,\e)$ be an $n$-dimensional curved cyclic unital $A_{\infty}$ algebra over $\mathcal{R}$ with values in $M$, and let $(S,e_S)$ be a unital associative differential graded algebra over $\k$, where $e_S \in S$ is the unit. Let $F:S \to S$ be a graded cyclic $\k$ linear map, that is also a $Z(S)$ module homomorphism. 

Define $\hat{\mathcal{R}}:=S \otimes \mathcal{R},$ which is naturally a differential graded algebra, and define $\hat{\cC} := S \otimes \cC$, which is naturally a graded bimodule over $\hat{\mathcal{R}}$. We use the grading of $\cC[1]$ to determine the bimodule structure. So, for $t,s\in S$, $\a \in \cC$ and $b \in \mathcal{R}$ we have
\[
(t \otimes b) (s \otimes \alpha) = (-1)^{|b||s|}ts \otimes b\alpha, \qquad (s\otimes \alpha) (t\otimes b) = (-1)^{|t|(|\alpha|+1)} st \otimes \alpha b.
\]
Equip $\hat{\mathcal{R}}$ with the filtration $\varsigma_{\hat{\mathcal{R}}}$ induced by $\varsigma_\mathcal{R}$ and the trivial filtration on $S$, and equip $\hat{\cC}$ with the filtration $\varsigma_{\hat{\cC}}$ induced by $\varsigma_{\cC}$ and the trivial filtration on $S$.
Set $\hat M:=S\otimes_{\k}M$, with the induced filtration and differential
\[
d_{\hat M}(s\otimes u)=d_Ss\otimes u+(-1)^{|s|}s\otimes d_Mu.
\]
It is a differential graded $Z(\hat{\mathcal R})=Z(S)\otimes_{\k}Z(\mathcal R)$ module under the tensor product action.
Define maps $\mh_k:\hat{\cC}^{\otimes k}\to \hat{\cC}[2-k]$ by 
\begin{multline}\label{eq:extscal}
    \mh_k\left(s_1 \otimes \alpha_1,\dots,s_k \otimes \alpha_k  \right)=\\
    =(-1)^{\sum_{j=1}^k |s_j|+\sum_{j=1}^k |s_j| \sum_{i=1}^{j-1}\left(|\alpha_i|+1\right)}s_1 \cdots s_k \otimes \m_k\left(\alpha_1,\dots,\alpha_k\right)+\delta_{1,k}\cdot ds_1 \otimes \a_1 ,
\end{multline}
and
\[\mh_0=e_S\otimes \m_0 .\]
Define a pairing $\lp\;,\,\rp_F:\hat{\cC}\otimes_{\k}\hat{\cC}\to \hat M[-n]$ by
\begin{equation}\label{eq:pairng_extension}
    \lp s_1 \otimes \alpha_1, s_2 \otimes \alpha_2\rp_F=(-1)^{|s_2|(|\a_1|+1)}F(s_1 s_2)\otimes\lp \alpha_1, \alpha_2\rp.
\end{equation}
\begin{prop}\label{prop:a_infity_extension}
The triple $(\{\mh_k\}_{k\ge 0},\lp\;,\,\rp_F,\e_S\otimes \e)$ is an $n$-dimensional cyclic unital $A_\infty$ structure over $\hat{\mathcal{R}}$ with values in $\hat M$ on $\hat{\cC}$. 
\end{prop}
\begin{dfn}\label{dfn:sfext}
The cyclic unital $A_\infty$ algebra $(\hat{\mathcal{C}}, \{\mh_k\}_{k \ge 0}, \lp\;,\,\rp_F,\e_S \otimes \e)$ is called the $(S,F)$ \textbf{extension of scalars} of the cyclic unital $A_{\infty}$ algebra $(\cC, \{\m_k\}_{k\ge 0},\lp\;,\,\rp,\e)$.
\end{dfn}

The proof of Proposition~\ref{prop:a_infity_extension} is given in the end of this section based on the series of lemmas stated below. 
\begin{lm}\label{lm:hm_lin}
The operations $\mh_k$ are $\hat{\mathcal{R}}$-multilinear in the sense that for $a\in \hat{\mathcal{R}}$,
\begin{gather*}
    \mh_k(a\cdot\ah_1,\ldots,\ah_k)=
		(-1)^{|a|}
		a\cdot\mh_k(\ah_1,\ldots,\ah_k)+\delta_{1,k}\cdot da\cdot\ah_1,\\
		\mh_k(\ah_1,\ldots,\ah_k \cdot a)=
		\mh_k(\ah_1,\ldots,\ah_k)\cdot a+(-1)^{|\ah_1|+1}\delta_{1,k}\cdot \ah_1\cdot da,
\end{gather*}
		and
		\[
		\mh_k(\ah_1,\ldots,\ah_{i-1},a\cdot\ah_i,\ldots,\ah_k)=\mh_k(\ah_1,\ldots,\ah_{i-1}\cdot a,\ah_i,\ldots,\ah_k).
			\]
\end{lm}
\begin{proof}
We will prove the first two equations only for $k=1$, as the other cases are similar but easier. Let $t,s\in S, \a \in \cC$ and $ b \in \mathcal{R}$ be homogeneous elements. We compute,
\begin{align*}
    \mh_1&\left((t\otimes b)\cdot (s\otimes \a)\right) = \\ &=(-1)^{|b||s|}\mh_1\left(ts\otimes b\cdot\a\right)\\
    &=(-1)^{|b||s|+|t|+|s|}ts\otimes\m_1(b\cdot\a)+(-1)^{|b||s|}d(ts)\otimes b\cdot \a\\
    &=(-1)^{|b||s|+|t|+|s|}ts\otimes \left((-1)^{|b|}b\cdot \m_1(\a)+db\cdot \a \right)+(-1)^{|b||s|}d(ts)\otimes b\cdot\a\\
    &=(-1)^{|b||s|+|t|+|s|}ts\otimes \left((-1)^{|b|}b\cdot \m_1(\a)+db\cdot \a \right)+\\
    &\quad+(-1)^{|b||s|}dt\cdot s\otimes b\cdot \a +(-1)^{|b||s|+|t|}tds \otimes b\cdot\alpha\\
    &=(-1)^{|t|+|s|+|b|}(t\otimes b) (s\otimes  \m_1(\a))+(-1)^{|t|}(t\otimes db)\cdot (s\otimes  \a) +\\
    &\quad+(dt\otimes b)\cdot  (s\otimes \a )+(-1)^{|b|+|t|}(t\otimes b) \cdot (ds \otimes \alpha)\\
    &=(-1)^{|t\otimes b|}(t\otimes b)\cdot \mh_1(s\otimes \a)+d(t\otimes b)\cdot (s\otimes \a).
\end{align*}
We also show the second equality by direct calculations:
\begin{align*}
    \mh_1\left( (s\otimes \a)\cdot (t\otimes b) \right)
    &=(-1)^{|t|(|\a|+1)}\mh_1\left(st\otimes \a\cdot b\right)\\
    &=(-1)^{|s|+|t||\a|}st\otimes\m_1(\a\cdot b)
      +(-1)^{|t|(|\a|+1)}d(st)\otimes \a\cdot b\\
    &=(-1)^{|s|+|t||\a|}st\otimes\m_1(\a)\cdot b
      +(-1)^{|s|+|t||\a|+|\a|+1}st\otimes\a\cdot db\\
    &\quad+(-1)^{|t|(|\a|+1)}(ds)t\otimes\a\cdot b
      +(-1)^{|s|+|t|(|\a|+1)}s(dt)\otimes\a\cdot b\\
    &=\left((-1)^{|s|}s\otimes\m_1(\a)+ds\otimes\a\right)\cdot(t\otimes b)\\
    &\quad+(-1)^{|s|+|\a|+1}(s\otimes\a)\cdot
      \left(dt\otimes b+(-1)^{|t|}t\otimes db\right)\\
    &=\mh_1(s\otimes \a)\cdot(t\otimes b)
      +(-1)^{|s|+|\a|+1}(s\otimes \a)\cdot d(t\otimes b).
\end{align*}

We will prove the third equality stated. Let $t,s_1,\ldots,s_k \in S$, $\a_1,\ldots, \a_k \in \cC$ and $ b \in \mathcal{R}$ be homogeneous elements. We compute
\begin{align*}
    \mh_k&(s_1\otimes \a_1,\ldots,t\otimes b\cdot s_i \otimes \a_i,\ldots, s_k\otimes \a_k) = \\
    &=(-1)^{|b||s_i|}\mh_k(s_1\otimes \a_1,\ldots, t s_i \otimes b \cdot\a_i,\ldots, s_k\otimes \a_k)\\
    &=(-1)^{|b||s_i|+|t|+\sum_{j=1}^k|s_j|
    +\sum_{j=1}^k |s_j| \sum_{r=1}^{j-1}\left(|\alpha_r|+1\right)
    +|t|\sum_{r=1}^{i-1}\left(|\alpha_r|+1\right)
    +|b|\sum_{j=i+1}^k|s_j|}s_1\cdots s_{i-1}ts_i \cdots s_k \otimes \\
    &\qquad\qquad\qquad\qquad\qquad\qquad\qquad\qquad\qquad\;\,\,\otimes \m_k(\a_1,\ldots, \a_{i-1}\cdot b,\a_i,\ldots,\a_k )\\
    &=(-1)^{|t|(|\a_{i-1}|+1)}\mh_k(s_1\otimes \a_1,\ldots, s_{i-1}t\otimes \a_{i-1}b,\ldots, s_k\otimes \a_k)\\
    &=\mh_k(s_1\otimes \a_1,\ldots, (s_{i-1}\otimes \a_{i-1})\cdot(t\otimes b),\ldots, s_k\otimes \a_k).
\end{align*}
\end{proof}
\begin{lm}\label{lm:a_infty_hm}
The operations $\{\mh_k\}_{k\ge 0}$ define an $A_\infty$ structure on $\hat{C}$. That is,
\begin{equation*}\label{eq:ahat-infty}
\sum_{\substack{k_1+k_2=k+1\\1\le i\le k_1}}(-1)^{\sum_{j=1}^{i-1}(|\ah_j|+1)}
\mh_{k_1}(\ah_1,\ldots,\ah_{i-1},\mh_{k_2}(\ah_i,\ldots,\ah_{i+k_2-1}), \ah_{i+k_2},\ldots,\ah_k)=0.
\end{equation*}
\end{lm}
\begin{proof}
The cases $k=0,1$ follow by a short computation.
Now suppose $k\ge 2$. Let $ s_1,\ldots,s_k \in S$, $\a_1,\ldots, \a_k \in \cC$, be homogeneous elements. For $k_1=k$ and $k_2=1$, we compute
\begin{align}
    &\mh_{k}(s_1\otimes \alpha_1,\ldots,s_{i-1} \otimes \alpha_{i-1},\mh_{1}(s_i \otimes\alpha_i), \ldots,s_k \otimes\alpha_k)= \nonumber\\&=\mh_{k}(s_1 \otimes\alpha_1,\ldots,s_{i-1}\otimes \alpha_{i-1},(-1)^{|s_i|}s_i\otimes \m_{1}(\alpha_i)+ds_i \otimes \a_i,\ldots,s_k\otimes\alpha_k)\nonumber\\
    &=(-1)^{\sum_{j=1}^{i-1}|s_j|+\sum_{j=1}^{k}|s_j|\sum_{s=1}^{j-1}(|\alpha_s|+1)} s_1\cdots s_k\otimes  \m_k(\a_1,\ldots,\m_1(\a_i),\ldots,\a_k)+\nonumber\\
    &\quad+ (-1)^{1+\sum_{j=1}^k|s_j|+\sum_{j=1}^k |s_j|\sum_{s=1}^{j-1}(|\a_s|+1)+\sum_{j=1}^{i-1}(|\a_j|+1)}s_1\cdots ds_i \cdots s_k \otimes \m_k(\a_1,\ldots, \a_k).\label{eq:der_k_one}
\end{align}
We will now compute for $k_1=1$ and $k_2=k$,
\begin{align}
    &\mh_1\left(\mh_k(s_1\otimes \a_1,\ldots s_k \otimes \a_k)\right)=\nonumber\\
    &= (-1)^{\sum_{j=1}^k |s_j|+\sum_{j=1}^k |s_j| \sum_{i=1}^{j-1}\left(|\alpha_i|+1\right)}\mh_1(s_1\cdots s_k \otimes \m_k(\a_1,\ldots, \a_k))=\nonumber\\
    &=(-1)^{\sum_{j=1}^k |s_j| \sum_{i=1}^{j-1}\left(|\alpha_i|+1\right)}s_1\cdots s_k \otimes\m_1( \m_k(\a_1,\ldots, \a_k))+\nonumber\\
    &\quad+(-1)^{\sum_{j=1}^k |s_j|+\sum_{j=1}^k |s_j| \sum_{i=1}^{j-1}\left(|\alpha_i|+1\right)}d(s_1\cdots s_k)\otimes \m_k(\a_1,\ldots, \a_k)\nonumber\\
    &=(-1)^{\sum_{j=1}^k |s_j| \sum_{i=1}^{j-1}\left(|\alpha_i|+1\right)}s_1\cdots s_k \otimes\m_1( \m_k(\a_1,\ldots, \a_k))+\nonumber\\
    &\quad+(-1)^{\sum_{j=1}^k |s_j|+\sum_{j=1}^k |s_j| \sum_{i=1}^{j-1}\left(|\alpha_i|+1\right)}\sum_{i=1}^k (-1)^{\sum_{j=1}^{i-1}|s_j|}s_1\cdots ds_i \cdots s_k \otimes \m_k(\a_1,\ldots, \a_k)=\nonumber\\
    &=(-1)^{\sum_{j=1}^k |s_j| \sum_{i=1}^{j-1}\left(|\alpha_i|+1\right)}s_1\cdots s_k \otimes\m_1( \m_k(\a_1,\ldots, \a_k))+\nonumber\\
    &\quad+\sum_{i=1}^k (-1)^{\sum_{j=i}^k |s_j|+\sum_{j=1}^k |s_j| \sum_{i=1}^{j-1}\left(|\alpha_i|+1\right)}s_1\cdots ds_i \cdots s_k \otimes \m_k(\a_1,\ldots, \a_k)\label{eq:der_one_k}.  
\end{align}
Adding up~\eqref{eq:der_k_one} and~\eqref{eq:der_one_k} with the sign coming from the $A_\infty$ relations, we obtain
\begin{align}
    &\sum_{i=1}^k(-1)^{\sum_{j=1}^{i-1}(|s_j\otimes \a_j|+1)}\mh_{k}(s_1\otimes \alpha_1,\ldots,s_{i-1} \otimes \alpha_{i-1},\mh_{1}(s_i \otimes\alpha_i), \ldots,s_k \otimes\alpha_k)+\nonumber\\
    &\quad+\mh_1\left(\mh_k(s_1\otimes \a_1,\ldots s_k \otimes \a_k)\right)=\nonumber\\
    &=(-1)^{\sum_{j=1}^k |s_j| \sum_{i=1}^{j-1}\left(|\alpha_i|+1\right)}s_1\cdots s_k \otimes\m_1( \m_k(\a_1,\ldots, \a_k))+\nonumber\\
    &\quad+\sum_{i=1}^k(-1)^{\sum_{j=1}^{i-1}(|\a_j|+1)+\sum_{j=1}^{k}|s_j|\sum_{s=1}^{j-1}(|\alpha_s|+1)} s_1\cdots s_k\otimes  \m_k(\a_1,\ldots,\m_1(\a_i),\ldots,\a_k).\label{eq:mh_one_a_infity_sign_compute}
\end{align}
For  $ s_1,\ldots,s_k \in S$, $\a_1,\ldots, \a_k \in \cC$, homogeneous elements, and $k_1,k_2 \ne 1$, compute
\begin{align}
        &\mh_{k_1}(s_1\otimes \alpha_1,\ldots,s_{i-1} \otimes \alpha_{i-1},\mh_{k_2}(s_i \otimes\alpha_i,\ldots, s_{i+k_2-1} \otimes \alpha_{i+k_2-1}), \ldots,s_k \otimes\alpha_k)=\nonumber\\
        &=\mh_{k_1}(s_1 \otimes\alpha_1,\ldots,s_{i-1}\otimes \alpha_{i-1},(-1)^*\prod_{j=i}^{i+k_2-1}s_j\otimes \m_{k_2}(\alpha_i,\ldots,\alpha_{i+k_2-1}),\ldots,s_k\otimes\alpha_k)\nonumber\\
        &=(-1)^*(-1)^{**}\prod_{j=1}^k s_j\otimes \m_{k_1}(\alpha_1,\ldots, \alpha_{i-1}, \m_{k_2}(\alpha_i,\ldots,\alpha_{i+k_2-1}),\alpha_{i+k_2} ,\ldots,\alpha_k),\label{eq:mh_sign_compute}
\end{align}
with 
\begin{align*}
    *&= \sum_{j=i}^{i+k_2-1}|s_j| + \sum_{j=i}^{i+k_2-1}|s_j|\sum_{r=i}^{j-1}(|\alpha_r|+1),\\
    **&\equiv\sum_{j=1}^k |s_j| + \sum_{j=1}^{i-1}|s_j|\sum_{r=1}^{j-1}(|\alpha_r|+1)+\left(\sum_{j=i}^{i+k_2-1}|s_j|\right)\sum_{r=1}^{i-1}(|\alpha_r|+1)+\\&\quad+\sum_{j=i+k_2}^k |s_j|\left(1+ \sum_{r=1}^{j-1}(|\alpha_r|+1)\right)\pmod{2}.
\end{align*}
Thus we get,
\begin{align*}
    *+** \equiv \sum_{j=1}^{i-1}|s_j|+\sum_{j=1}^{k}|s_j|\sum_{r=1}^{j-1}(|\alpha_r|+1)\pmod{2}.
\end{align*}
Note that $\sum_{j=1}^{k}|s_j|\sum_{r=1}^{j-1}(|\alpha_r|+1)$ is not a function of $i,k_1,k_2$. We also observe that
$$\sum_{j=1}^{i-1}(|s_j \otimes \alpha_j|+1)=\sum_{j=1}^{i-1}|s_j|+\sum_{j=1}^{i-1}(|\alpha_j|+1).$$
If we abbreviate $\hat{\alpha_i}=s_i\otimes \alpha_i$, we get that adding~\eqref{eq:mh_sign_compute} with the signs coming from the $A_\infty$~relations to~\eqref{eq:mh_one_a_infity_sign_compute} results in
\begin{multline*}
    \sum_{\substack{k_1+k_2=k+1\\1\le i\le k_1}}(-1)^{\sum_{j=1}^{i-1}(|\hat{\alpha}_j|+1)}
\mh_{k_1}(\hat{\alpha}_1,\ldots,\hat{\alpha}_{i-1},\mh_{k_2}(\hat{\alpha}_i,\ldots,\hat{\alpha}_{i+k_2-1}), \hat{\alpha}_{i+k_2},\ldots,\hat{\alpha_k})=\\=
(-1)^{\sum_{j=1}^{k}|s_j|\sum_{r=1}^{j-1}(|\alpha_r|+1)} s_1 \cdots s_k\otimes\sum_{\substack{k_1+k_2=k+1\\1\le i\le k_1}}(-1)^{\sum_{j=1}^{i-1}(|\alpha_j|+1)}\\
\m_{k_1}(\alpha_1,\ldots,\alpha_{i-1},\m_{k_2}(\alpha_i,\ldots,\alpha_{i+k_2-1}), \alpha_{i+k_2},\ldots,\alpha_k)=0.
\end{multline*}
The last equality follows from the $A_\infty$ relations of $\cC$.
\end{proof}
\begin{lm}\label{lm:phatlin}
The pairing satisfies
\[
		\lp\alpha_1,a\cdot\alpha_2\rp_F=
\lp\alpha_1\cdot a,\alpha_2\rp_F,\qquad \forall a \in \hat{\mathcal{R}},
\]
		and
\[
\lp a\cdot \a_1,\a_2\rp_F=a\cdot \lp\a_1,\a_2\rp_F,\qquad \forall a \in Z(\hat{\mathcal{R}}).
\]
\end{lm}
\begin{proof}
The first equation is deduced from the definition of $\lp\;,\;\rp_F$. We will prove the second equation. 
As both $\mathcal{R}$ and $S$ are algebras over the field $\k$, we have 
\[
Z\left(S \otimes_{\k} \mathcal{R}\right) \simeq Z(S) \otimes_{\k} Z(\mathcal{R}).
\]
Thus it suffices to compute for $s_1, s_2 \in S,  \a_1, \a_2 \in \cC$, $a \in Z(S)$ and $ b \in Z(\mathcal{R})$,
\begin{align*}
\lp a \otimes b\cdot s_1\otimes \alpha_1,s_2\otimes \alpha_2\rp_F&=\lp(-1)^{|b||s_1|} a  s_1\otimes b\alpha_1,s_2\otimes \alpha_2\rp_F=\\
&=(-1)^{|b||s_1|+|s_2|(|\a_1|+|b|+1)}F(as_1\cdot s_2)\otimes \lp b\a_1, \a_2 \rp=\\
&=(-1)^{|b||s_1|+|s_2|(|\a_1|+|b|+1)}aF(s_1\cdot s_2)\otimes b\lp \a_1, \a_2 \rp=\\
&=(-1)^{|s_2|(|\alpha_1|+1)}a\otimes b\cdot F(s_1\cdot s_2)\otimes \lp \a_1, \a_2 \rp=\\
&=a\otimes b\cdot  \lp s_1\otimes \alpha_1,s_2\otimes \alpha_2\rp_F.
\end{align*}
\end{proof}
\begin{lm}\label{lm:phatsim}
$$\lp \ah_1, \ah_2 \rp_F=(-1)^{(|\ah_1|+1)(|\ah_2|+1)+1}\lp \ah_2 , \ah_1 \rp_F.$$
\end{lm}
\begin{proof}
For $s_1, s_2 \in S$ and $\a_1, \a_2 \in \cC$ homogeneous, we compute,
\begin{align*}
    \lp s_1 \otimes \alpha_1,s_2 \otimes \alpha_2 \rp_F&=(-1)^{|s_2|(|\alpha_1|+1)}F(s_1 \cdot s_2)\otimes\lp \alpha_1 ,\alpha_2\rp=\\
    &=(-1)^{|s_1||s_2|+|s_2|(|\alpha_1|+1)+(|\alpha_1|+1)(|\alpha_2|+1)+1}F(s_2 \cdot s_1)\otimes\lp \alpha_2 ,\alpha_1\rp=\\
    &=(-1)^{(|\alpha_1|+|s_1|+1)(|\alpha_2|+|s_2|+1)+1}\lp s_2 \otimes \alpha_2,s_1 \otimes \alpha_1 \rp_F.
\end{align*}
\end{proof}

\begin{lm}\label{lm:phat_cyclic}
For $k\geq 1$ and any $\ah_1,\ldots,\ah_{k+1} \in \hat{\mathcal{C}}$ 
\begin{multline*}
\lp\mh_k(\ah_1,\ldots,\ah_k),\ah_{k+1}\rp_F=\\
=(-1)^{(|\ah_{k+1}|+1)\sum_{j=1}^{k}(|\ah_j|+1)}
\lp \mh_k(\ah_{k+1},\ah_1,\ldots,\ah_{k-1}),\ah_k\rp_F+\delta_{1,k}\cdot d_{\hat M}\lp \ah_1, \ah_2\rp_F.
\end{multline*}
\end{lm}
\begin{proof}
Let $s_1,\ldots,s_{k+1} \in S$ and $\a_1,\ldots,\a_{k+1} \in \cC$ be homogeneous elements. Assume first $k>1$, and calculate
\begin{align}
    \lp \mh_k (s_{1}\otimes&\alpha_1,\dots , s_{k}\otimes \alpha_k),s_{k+1}\otimes \alpha_{k+1}\rp_F=\nonumber\\
    &=(-1)^{\eta_1} \lp s_1\cdots s_k\otimes \m_k\left(\alpha_1,\dots , \alpha_k\right),s_{k+1}\otimes \alpha_{k+1}\rp_F\nonumber\\
    &=
    (-1)^{\eta_1+\eta_2}
    F(s_1 \cdots s_{k+1})\otimes\lp \m_k\left(\alpha_1,\dots , \alpha_k\right), \alpha_{k+1}\rp \label{eq:product_compute}\\
    &=(-1)^{\eta_1+\eta_2+\eta_3}
    F(s_{k+1}s_1\cdots s_k)\otimes\lp \m_k\left(\alpha_1,\dots , \alpha_k\right), \alpha_{k+1}\rp.\nonumber\\
    \shortintertext{Employing the cyclic symmetry property~\ref{it:cyclic} of $\mathcal{C}$, we have}
    &=(-1)^{\eta_1+\eta_2+\eta_3+\eta_4}
    F(s_{k+1}s_1\cdots s_k)\otimes\lp \m_k\left(\alpha_{k+1}, \a_1, \dots , \alpha_{k-1}\right), \alpha_{k}\rp \nonumber,
\end{align}
with
\begin{align*}
    \eta_1&=\sum_{j=1}^k |s_j| + \sum_{j=1}^k |s_j|\sum_{i=1}^{j-1}(|\alpha_i|+1),\\
    \eta_2 &=|s_{k+1}|\sum_{j=1}^{k}(|\alpha_j|+1)+|s_{k+1}|,\\
    \eta_3 &=|s_{k+1}|\sum_{j=1}^{k}|s_j|, \\
    \eta_4 &= (|\alpha_{k+1}|+1)\sum_{j=1}^k(|\alpha_j|+1).
\end{align*}
Similarly to~\eqref{eq:product_compute} we compute,
\begin{align*}
    \lp \mh_k (s_{k+1}&\otimes\alpha_{k+1}, s_1 \otimes \a_1 \dots , s_{k-1}\otimes \alpha_{k-1}),s_{k}\otimes \alpha_{k}\rp_F=\\
    &=(-1)^{\eta_5+\eta_6}
    F(s_{k+1}s_1\cdots s_k)\otimes\lp \m_k\left(\alpha_{k+1}, \a_1, \dots , \alpha_{k-1}\right), \alpha_{k}\rp,
\end{align*}
with
\begin{align*}
    \eta_5 &= \sum_{j=1}^{k-1}|s_j| +|s_{k+1}| +\sum_{j=1}^{k-1}|s_j|\left(|\alpha_{k+1}|+1+\sum_{i=1}^{j-1}(|\alpha_{i}|+1)\right).\\
    \eta_6 &= |s_k|\sum_{j=1}^{k-1}(|\alpha_j|+1)+|s_k|(|\alpha_{k+1}|+1)+|s_k|
\end{align*}
Those computations imply that 
\begin{multline*}
    \lp \mh_k (s_{1}\otimes\alpha_1,\dots , s_{k}\otimes \alpha_k),s_{k+1}\otimes \alpha_{k+1}\rp_F=\\=(-1)^{\eta_1+\eta_2+\eta_3+\eta_4+\eta_5+\eta_6}\lp \mh_k (s_{k+1}\otimes\alpha_{k+1}, s_1 \otimes \a_1 \dots , s_{k-1}\otimes \alpha_{k-1}),s_{k}\otimes \alpha_{k}\rp_F.
\end{multline*}
Observe that
\[
\eta_1 + \eta_2 + \eta_5 + \eta_6 = |s_{k+1}|\sum_{j=1}^{k}(|\alpha_j|+1)  + \left(|\alpha_{k+1}|+1\right)\sum_{j=1}^{k}|s_j|.
\]
So,
\[
\eta_1+\eta_2+\eta_3+\eta_4+\eta_5+\eta_6=(|\alpha_{k+1}|+|s_{k+1}|+1)\sum_{j=1}^k(|\alpha_j|+|s_j|+1).
\]
Thus we have showed the result for $k> 1$. 

For $k=1$ we compute
\begin{align*}
    \lp\mh_1&(s_1 \otimes \alpha_1), s_2 \otimes \alpha_2 \rp_F=\\
    &=\lp (-1)^{|s_1|}s_1\otimes \m_1( \alpha_1)+ds_1\otimes \a_1 ,s_2 \otimes \alpha_2 \rp_F\\
    &=(-1)^{|s_1|+|s_2||\a_1|}F(s_1\cdot s_2)\otimes \lp \m_1(\a_1), \a_2\rp +(-1)^{|s_2|(|\a_1|+1)}F(ds_1\cdot s_2)\otimes \lp \a_1,  \a_2\rp. 
\end{align*}
We proceed using cyclic symmetry property~\ref{it:cyclic} of $\cC$ to develop the first summand of the equation above.
\begin{align*}
    (-1)^{|s_1|+|s_2||\a_1|}&F(s_1\cdot s_2)\otimes \lp \m_1(\a_1), \a_2\rp=\\
    &=(-1)^{|s_1|+|s_2||\a_1|}F(s_1\cdot s_2)\otimes \left((-1)^{(|\a_1|+1)(|\a_2|+1)}\lp \m_1(\a_2),\a_1\rp+d_M\lp \a_1, \a_2 \rp  \right)\\
    &=(-1)^{|s_1|+|s_2||\a_1|+|s_1||s_2|+(|\a_1|+1)(|\a_2|+1)}F(s_2 \cdot s_1)\otimes \lp \m_1(\a_2),\a_1\rp+\\
    &\quad+(-1)^{|s_1|+|s_2||\a_1|}F(s_1\cdot s_2)\otimes d_M\lp \a_1, \a_2 \rp \\
    &=(-1)^{(|s_1|+|\a_1|+1)(|s_2|+|\a_2|+1)} \lp (-1)^{|s_2|}s_2\otimes \m_1(\a_2),s_1 \otimes \a_1\rp_F+\\
    &\quad+(-1)^{|s_1|+|s_2||\a_1|}F(s_1\cdot s_2)\otimes d_M\lp \a_1, \a_2 \rp. 
\end{align*}
We use the equality
\[(-1)^{|s_2|}s_2\otimes \m_1(\a_2)=\mh_1(s_2 \otimes \a_2)-ds_2 \otimes \a_2,\]
to derive
\begin{align*}
    (-1&)^{(|s_1|+|\a_1|+1)(|s_2|+|\a_2|+1)} \lp (-1)^{|s_2|}s_2\otimes \m_1(\a_2),s_1 \otimes \a_1\rp_F=\\
    &=(-1)^{(|s_1|+|\a_1|+1)(|s_2|+|\a_2|+1)}\lp \mh_1(s_2 \otimes \a_2)-ds_2 \otimes \a_2,s_1 \otimes \a_1 \rp_F\\
    &=(-1)^{(|s_1|+|\a_1|+1)(|s_2|+|\a_2|+1)}\left(\lp \mh_1(s_2 \otimes \a_2),s_1 \otimes \a_1 \rp_F-\lp ds_2 \otimes \a_2,s_1 \otimes \a_1 \rp_F \right).
\end{align*}
Using Lemma~\ref{lm:phatsim}, we develop the second term further,
\begin{align*}
-(-1)^{(|s_1|+|\a_1|+1)(|s_2|+|\a_2|+1)}&\lp ds_2 \otimes \a_2,s_1 \otimes \a_1 \rp_F =  \\
    &=(-1)^{|s_1|+|\a_1|+1}\lp s_1 \otimes \a_1,ds_2 \otimes \a_2 \rp_F \\
    &=(-1)^{1+|s_1|+|\a_1|+(|\a_1|+1)(|s_2|+1)}F(s_1 \cdot ds_2)\otimes \lp \a_1, \a_2 \rp\\
    &=(-1)^{|s_1|+|s_2|(|\a_1|+1)}F(s_1 \cdot ds_2)\otimes \lp \a_1, \a_2 \rp.
\end{align*}
Combining the computations above we get,
\begin{align*}
    \lp\mh_1(s_1 \otimes \alpha_1), s_2 \otimes \alpha_2 \rp_F&=(-1)^{(|s_1|+|\a_1|+1)(|s_2|+|\a_2|+1)}\lp \mh_1(s_2 \otimes \a_2),s_1 \otimes \a_1 \rp_F+\\
    &\quad+(-1)^{|s_2|(|\a_1|+1)}F(ds_1\cdot s_2)\otimes \lp \a_1,  \a_2\rp+\\
    &\quad+(-1)^{|s_1|+|s_2|(|\a_1|+1)}F(s_1 \cdot ds_2)\otimes \lp \a_1, \a_2 \rp+\\
    &\quad+(-1)^{|s_1|+|s_2|+|s_2|(|\a_1|+1)}F(s_1\cdot s_2)\otimes d_M\lp \a_1, \a_2 \rp\\
    &=(-1)^{(|s_1|+|\a_1|+1)(|s_2|+|\a_2|+1)}\lp \mh_1(s_2 \otimes \a_2),s_1 \otimes \a_1 \rp_F+\\
    &\quad+d_{\hat M}\lp s_1 \otimes \a_1, s_2\otimes \a_2 \rp_F.
\end{align*}
\end{proof}
\begin{proof}[Proof of Proposition~\ref{prop:a_infity_extension}] We will show that $(\{\mh_k\}_{k\ge 0},\lp\;,\,\rp_F,\e_S\otimes \e)$ satisfies the properties required to define a cyclic unital $A_\infty$ structure on $\hat{\cC}$. Lemma~\ref{lm:hm_lin} proves that the operators $\mh_k$ satisfy property~\ref{it:lin} of Definition~\ref{dfn:cycunit}. The operators $\mh_k$ satisfy the $A_\infty$ relation property~\ref{it:a_infty}, as shown in Lemma~\ref{lm:a_infty_hm}. The properties \ref{it:val},\ref{it:val2} hold as we chose the filtrations on $\hat{\cC}$ and $\hat M$ to be induced by those on $\cC$ and $M$ and the trivial filtration on $S$. Property~\ref{it:unit1} clearly follows from the corresponding property of $\cC$. Property~\ref{it:unit2} follows from the corresponding property of $\cC$. Property~\ref{it:plin} holds for $\lp\cdot,\cdot\rp_F$ by Lemma~\ref{lm:phatlin}. Property~\ref{it:symm} holds for $\lp\cdot,\cdot\rp_F$ by Lemma~\ref{lm:phatsim}. Property~\ref{it:cyclic} holds by Lemma~\ref{lm:phat_cyclic}. Property~\ref{it:unit3} follows from the definition of the extension and the corresponding property for $\cC$. 
\end{proof}
\begin{lm}\label{lm:cano_iso}
There is a canonical $\bar{\mathcal{R}}$ module isomorphism $ \bar{\hat{\cC}}\simeq S\otimes \bar{\cC} $ given by $  \overline{s \otimes \alpha}\mapsto s \otimes \bar\alpha$.
\end{lm}
\begin{proof}
We observe the exact sequence 
\[\begin{tikzcd}
0 \arrow[r] & F^{>0}\mathcal{C} \arrow[r] & \mathcal{C} \arrow[r] & \bar{\mathcal{C}} \arrow[r] & 0
.\end{tikzcd}\]
Thus we obtain an exact sequence
\[
\begin{tikzcd}
S\otimes_\k F^{>0}\mathcal{C} \arrow[r,"a"] & S\otimes_\k\mathcal{C} \arrow[r] & S\otimes_\k\bar{\mathcal{C}} \arrow[r] & 0.
\end{tikzcd}
\] 
To conclude we show that
\[
\im a = F^{>0}(S\otimes_\k\mathcal{C}).
\]
Indeed, by definition, the filtration on the tensor product is given by 
\[\nu(\beta)=\sup_{\beta=\sum_i s_i\otimes \alpha_i}{\min_i \,(\nu_S(s_i)+\nu_{\cC}(\alpha_i))}.\]  As $\nu_S$ is trivial, $\beta \in F^{>0}(S\otimes_\k\mathcal{C})$  if and only if we can write $\beta = \sum_i s_i \otimes \alpha_i$ with $\nu(\alpha_i) > 0,$ which is equivalent to $\beta \in \im(a).$ 
\end{proof}
\begin{prop}\label{prop:sf_extension_arises_dga}
Let $(\hat{\mathcal{C}}, \{\mh_k\}_{k \ge 0}, \lp\;,\,\rp_F,\e_S \otimes \e)$ be an $(S, F)$ extension of a cyclic unital $A_{\infty}$ algebra $(\cC, \{\m_k\}_{k\ge 0},\lp\;,\,\rp,\e)$. If $(\bar{\cC}, \{\bar{\m}_k\}_{k\ge 0},\lp\;,\,\rp,\e)$ arises from a dga $(\Upsilon,d)$, then $(\bar{\hat{\mathcal{C}}}, \{\bar{\mh}_k\}_{k \ge 0}, \lp\;,\,\rp_F,\e_S \otimes \e)$ arises from the dga $(S\otimes\Upsilon,d\otimes id+id\otimes d)$ using the canonical isomorphism $\bar{\hat{\cC}}\to  S\otimes\Upsilon $ given in Lemma \ref{lm:cano_iso}.
\end{prop}
\begin{proof}
Let $\a_1, \a_2 \in \mathcal{C}$ and $s_1, s_2 \in S$ be homogeneous elements. First, recall that
\[\mh_1(s_1 \otimes \a_1) = (-1)^{|s_1|}s_1 \otimes \m_1(\a_1)+ds_1 \otimes \a_1.\]
It follows that under the identification $\overline{s \otimes \alpha}\mapsto s \otimes \bar\alpha$ we have
\[\bar{\mh}_1(\overline{s_1 \otimes \a_1}) :=\overline{\mh_1(s_1 \otimes \a_1)}=(-1)^{|s_1|}s_1\otimes \overline{\m_1(\alpha_1)}+ds_1\otimes \bar{\a}_1
.\]
As $(\bar{\cC}, \{\mbar_k\}_{k\ge 0},\lp\;,\,\rp,\e)$ arises from the dga $(\Upsilon,d)$ we showed that
\[\bar{\mh}_1(s_1 \otimes \bar{\a}_1)=(-1)^{|s_1|}s_1\otimes d\bar{\alpha}_1+ds_1\otimes \bar{\a}_1.\]
In a similar way, we can show that 
\[(s_1 \otimes \bar{\a}_1)(s_2 \otimes \bar{\a}_2) =(-1)^{|s_1 \otimes \bar{\a}_1|}\bar{\mh}_2(s_1 \otimes \bar{\a}_1,s_2 \otimes \bar{\a}_2),\quad \qquad \bar{\mh}_k=0, \quad k \neq 1,2.\]
\end{proof}
\begin{dfn}\label{dfn:induced_R_structure}
Let $(\hat{\mathcal{C}}, \{\mh_k\}_{k \ge 0}, \lp\;,\,\rp_F,\e_S \otimes \e)$ be an $(S, F)$ extension of a cyclic unital $A_{\infty}$ algebra $(\cC, \{\m_k\}_{k\ge 0},\lp\;,\,\rp,\e)$ over a commutative dga $\mathcal{R}$ with values in $M$. Here $\hat{\cC}$ is regarded as an $\mathcal{R}$ module, while the pairing takes values in $\hat M$.
Suppose $\cC$ has an $\bar{\mathcal{R}}$ structure $g:\mathcal{R}\otimes_{\bar{\mathcal{R}}}\bar{\cC}\to \cC$ with inverse $f$. The \textbf{induced $\bar{\mathcal{R}}$ structure} on  $(\hat{\mathcal{C}}, \{\mh_k\}_{k \ge 0}, \lp\;,\,\rp_F,\e_S \otimes \e)$ is 
$$\hat{g}:\mathcal{R}\otimes_{\bar{\mathcal{R}}}\overline{\hat{\cC}}\to \hat{\cC}$$
given by the composition 
\begin{equation*}
\xymatrix{
 \mathcal{R}\otimes_{\bar{\mathcal{R}}}\overline{\hat{\cC}} = \mathcal{R}\otimes_{\bar{\mathcal{R}}}\overline{\left(S\otimes_{\k}\cC\right)} \ar[r]^(.6){\sim}
    & \mathcal{R}\otimes_{\bar{\mathcal{R}}}\left(S\otimes_{\k}\overline{\cC}\right)\ar[r]^\sim
    &  S\otimes_{\k}\left(\mathcal{R}\otimes_{\bar{\mathcal{R}}}\overline{\cC}\right)\ar[r]^(.7){\sim}_(.7){id\otimes g}
    & \hat{\cC}.
}
\end{equation*}
\end{dfn}

\subsection{Opposites and duality}
The following section is based on~\cite{solomon2020Involutions}. Let $\cC, \D$, be graded complete filtered vector spaces. An $A_{\infty}$ pre-homomorphism
\[
\mathfrak{f}: \cC \rightarrow \mathcal{D}
\]
is a collection of maps
\[
\mathfrak{f}_k: \cC^{\otimes k} \rightarrow \mathcal{D}, \quad k \geq 0,
\]
of degree $1-k$ such that
\[
\inf_{\substack{0\neq v\in\cC^{\otimes k}}}
\left(
\varsigma_{\mathcal{D}}\bigl(\mathfrak{f}_k(v)\bigr)
-
\varsigma_{\cC^{\otimes k}}(v)
\right)\geq 0,
\]
with strict inequality for $k=0$. An $A_\infty$ pre-homomorphism $\f$ is called \textbf{strict} if $\f_k = 0$ for $k \neq 1.$

Let $\f : C\to D$ and $\gf : D \to E$ be $A_\infty$ pre-homomorphisms. The composition $\gf\circ \f : C \to E$ is given by
\begin{equation}\label{eq:claw}
(\gf \circ \f)_k(\alpha_1,\ldots,\alpha_k)
=
\sum_{\substack{l\geq 0,\ r_1,\ldots,r_l\geq 0\\
r_1+\cdots+r_l=k}}
\gf_l\bigl(
\f_{r_1}(\alpha_1,\ldots,\alpha_{r_1}),\ldots,
\f_{r_l}(\alpha_{k-r_l+1},\ldots,\alpha_k)
\bigr).
\end{equation}
Here, an arity-zero map has no inputs, and the index conditions involving
$r_1,\ldots,r_l$ are vacuous when $l=0$.

Let $\left(\cC, \mathfrak{m}^\cC\right),\left(\D, \mathfrak{m}^\D\right)$, be $A_{\infty}$ algebras. An $A_{\infty}$ pre-homomorphism $\mathfrak{f}: \cC \rightarrow \D$ is called an $A_{\infty}$ homomorphism if
$$
\begin{aligned}
& \quad \sum_{\substack{l\geq 0,\ r_1,\ldots,r_l\geq 0 \\
r_1+\cdots+r_l=k}} \mathfrak{m}_l^D\left(\mathfrak{f}_{r_1}\left(\alpha_1, \ldots, \alpha_{r_1}\right), \ldots, \mathfrak{f}_{r_l}\left(\alpha_{k-r_l+1}, \ldots, \alpha_k\right)\right)= \\
& \quad=\sum_{\substack{k_1\geq 1,\ k_2\geq 0 \\
k_1+k_2=k+1 \\
1 \leq i \leq k_1}}(-1)^{\sum_{j=1}^{i-1}\left(\left|\alpha_j\right|+1\right)} \mathfrak{f}_{k_1}\left(\alpha_1, \ldots, \alpha_{i-1}, \mathfrak{m}_{k_2}^C\left(\alpha_i, \ldots, \alpha_{i+k_2-1}\right), \ldots, \alpha_k\right).
\end{aligned}
$$
Let $\alpha=\left(\alpha_1, \ldots, \alpha_k\right)$ be a list of homogeneous elements $\alpha_j \in \cC$. Let $\sigma$ be a permutation of the set $\{1, \ldots, k\}$. We use the notation

$$
s_\sigma(\alpha)=\sum_{\substack{i<j \\ \sigma(i)>\sigma(j)}}\left(\left|\alpha_{\sigma(i)}\right|+1\right)\left(\left|\alpha_{\sigma(j)}\right|+1\right) .
$$

For the rest of the paper, we fix $\tau$ to be the permutation $(1, \ldots, k) \mapsto(k, k-1, \ldots, 1)$.
Given an $A_{\infty}$ algebra
$(\cC,\{\mathfrak{m}_k\}_{k\geq 0})$, the opposite $A_{\infty}$ algebra
structure $\{\mathfrak{m}_k^{op}\}_{k\geq 0}$ on $\cC$ is defined by
\[
\mathfrak{m}_k^{op}\left(\alpha_1, \ldots, \alpha_k\right)=(-1)^{s_\tau(\alpha)+k+1} \mathfrak{m}_k\left(\alpha_k, \ldots, \alpha_1\right) .
\]
A straightforward calculation shows that
$(\cC,\{\mathfrak{m}_k^{op}\}_{k\geq 0})$ is indeed an $A_{\infty}$
algebra. Moreover, $(\mathfrak{m}_k^{op})^{op} = \mathfrak{m}_k.$

Let $\f : C \to D$ be an $A_\infty$ pre-homomorphism. The \textbf{opposite $A_\infty$ pre-homomorphism} $\f^{op} : C \to D$ is defined by
\[
\f^{op}_k(\alpha_1,\ldots,\alpha_k) = (-1)^{\spe{\tau}{\alpha} + k + 1} \f_k(\alpha_k,\ldots,\alpha_1).
\]
If $(C,\m^C),(D,\m^D),$ are $A_\infty$ algebras, and $\f : (C,\m^C) \to (D,\m^D)$ is an $A_\infty$ homomorphism, one checks that $\f^{op} : (C,\m^{op}) \to (D, \m^{op})$ is also an $A_\infty$ homomorphism.

\begin{dfn}\label{dfn:d_self_dual}
Let
\[
\mathfrak{d}:
(\cC,\{\m_k\}_{k\geq 0})
\longrightarrow
(\cC,\{\m_k^{op}\}_{k\geq 0})
\]
be a strict $A_\infty$ homomorphism. We say that
$(\cC,\{\m_k\}_{k\geq 0})$ is \textbf{$\mathfrak{d}$ self-dual} if
$\mathfrak{d}^{op} \circ \mathfrak{d} = \Id_\cC.$ Let
$\phi : S \to S^{op}$ be a homomorphism of dgas. We say that $S$ is
\textbf{$\phi$ self-dual} if the $A_\infty$ algebra that arises from $S$ is
$\phi$ self-dual.
\end{dfn}

\begin{rem}\label{rem:self_dual_explicit}
The map
\[
\mathfrak{d}:
(\cC,\{\m_k\}_{k\geq 0})
\longrightarrow
(\cC,\{\m_k^{op}\}_{k\geq 0})
\]
is a strict $A_\infty$ homomorphism if and only if for all homogeneous
$\alpha_1,\dots,\alpha_k\in\cC$,
\[
\mathfrak{d}\,\mathfrak{m}_k(\alpha_1,\dots,\alpha_k)=(-1)^{1+k+s_\tau(\alpha)}\mathfrak{m}_k(\mathfrak{d}\alpha_k,\dots,\mathfrak{d}\alpha_1).
\]
\end{rem}
 \begin{rem}\label{rem:dga_self_dual_hom}
When $S$ is regarded as an $A_\infty$ algebra arising from a dga, our convention is
\[
xy=(-1)^{|x|}\m_2(x,y).
\]
Thus, although the opposite $A_\infty$ structure uses the shifted Koszul sign
\[
(|x|+1)(|y|+1),
\]
the induced opposite dga product is the usual graded opposite product
\[
x\cdot_{op}y=(-1)^{|x||y|}yx.
\]
Therefore a strict $A_\infty$ homomorphism $S\to S^{op}$ is precisely a chain map which reverses the ordinary dga product with the usual graded sign.
\end{rem}
\begin{prop}\label{prop:sd_extension}
Let $(\cC,\{\m_k\}_{k\ge 0},\lp\;,\,\rp,\e)$ be an $n$-dimensional curved cyclic unital $A_\infty$ algebra, and suppose $(\cC,\m)$ is $\mathfrak{d}$-self-dual. Let
$
(\hat{\cC},{\mh_k}_{k\ge 0},\lp\;,\,\rp_F,\e_S\otimes\e)
$
be an $(S,F)$ extension. Suppose the differential graded algebra $S$ is $\sigma$-self-dual. Then $(\hat{\cC},\mh)$ is $\hat{\mathfrak{d}}$-self-dual, where
$
\hat{\mathfrak{d}}:=\sigma\otimes\mathfrak{d}.
$
\end{prop}

\begin{proof}
The condition
\[
\hat{\mathfrak{d}}^{op}\circ\hat{\mathfrak{d}}=\Id_{\hat{\cC}}
\]
follows immediately from the corresponding identities for $\sigma$ and $\mathfrak{d}$.

It remains to prove the compatibility with the $A_\infty$ operations. By multilinearity, it is enough to check the identity on pure tensors
\[
\hat\alpha_i=x_i\otimes\alpha_i.
\]
For $k=0$, using Definition~\ref{dfn:sfext} and Remark~\ref{rem:self_dual_explicit}, we have
\[
\hat{\mathfrak{d}}\mh_0
= \sigma(e_S)\otimes\mathfrak{d}\m_0 =e_S\otimes(-\m_0)=-\mh_0.
\]
Since $s_\tau(\emptyset)=0$, this is
\[
\hat{\mathfrak{d}}\mh_0
=(-1)^{1+0+s_\tau(\emptyset)}\mh_0.
\]
Now let $k\ge 1$. By equation~\eqref{eq:extscal}, 
\begin{equation*}
      \mh_k\left(\hat\alpha_1,\dots,\hat\alpha_k  \right)=(-1)^A x_1 \cdots x_k \otimes \m_k\left(\alpha_1,\dots,\alpha_k\right)+\delta_{1,k}\cdot dx_1 \otimes \a_1  
\end{equation*}
where
\[
A=\sum_{i=1}^k |x_i|+\sum_{i=1}^k |x_i|\sum_{j=1}^{i-1}(|\alpha_j|+1).
\]
Since $S$ is $\sigma$-self-dual,
\[
\sigma(x_1\cdots x_k)
=
(-1)^B\sigma(x_k)\cdots\sigma(x_1),
\qquad
B=\sum_{i<j}|x_i||x_j|.
\]
Using Remark~\ref{rem:self_dual_explicit}, we get
\begin{align*}
\hat{\mathfrak{d}}
\left(
(-1)^A x_1\cdots x_k\otimes \m_k(\alpha_1,\ldots,\alpha_k)
\right)
&=
(-1)^{A+B}
\sigma(x_k)\cdots\sigma(x_1)
\otimes
\mathfrak{d}\m_k(\alpha_1,\ldots,\alpha_k)\\
&=
(-1)^{A+B+1+k+s_\tau(\alpha)}
\sigma(x_k)\cdots\sigma(x_1)
\otimes
\m_k(\mathfrak{d}\alpha_k,\ldots,\mathfrak{d}\alpha_1).
\end{align*}

On the other hand, by equation~\eqref{eq:extscal},
\[
\mh_k(\hat{\mathfrak{d}}\hat\alpha_k,\ldots,\hat{\mathfrak{d}}\hat\alpha_1)
=
(-1)^C
\sigma(x_k)\cdots\sigma(x_1)
\otimes
\m_k(\mathfrak{d}\alpha_k,\ldots,\mathfrak{d}\alpha_1) + \delta_{1,k} d\sigma(x_1) \otimes \mathfrak{d}\alpha_1
\]
where
\[
C=\sum_{i=1}^k |x_i|+\sum_{i=1}^k |x_i|\sum_{j=i+1}^{k}(|\alpha_j|+1).
\]
Since $|\hat\alpha_i|=|x_i|+|\alpha_i|$, we compute modulo $2$,
\begin{align*}
A+B+s_\tau(\alpha)+C
&\equiv
\sum_{i<j}|x_i||x_j|
+
\sum_{i<j}|x_j|(|\alpha_i|+1)
+
\sum_{i<j}|x_i|(|\alpha_j|+1)\\
&\qquad
+
\sum_{i<j}(|\alpha_i|+1)(|\alpha_j|+1)\\
&\equiv
\sum_{i<j}
\bigl(|x_i|+|\alpha_i|+1\bigr)
\bigl(|x_j|+|\alpha_j|+1\bigr)\\
&=
s_\tau(\hat\alpha).
\end{align*}
Thus,
\[
A+B+1+k+s_\tau(\alpha)
\equiv
1+k+s_\tau(\hat\alpha)+C
\mod 2.
\]
Moreover, since $\sigma$ is a chain map,
\[
\hat{\mathfrak{d}}(dx_1\otimes\alpha_1)
=\sigma(dx_1)\otimes\mathfrak{d}\alpha_1=d\sigma(x_1)\otimes\mathfrak{d}\alpha_1.
\]
Therefore,
\[
\hat{\mathfrak{d}}\mh_k(\hat\alpha_1,\ldots,\hat\alpha_k)=(-1)^{1+k+s_\tau(\hat\alpha)}
\mh_k(\hat{\mathfrak{d}}\hat\alpha_k,\ldots,\hat{\mathfrak{d}}\hat\alpha_1).
\]
and $(\hat{\cC},\mh)$ is $\hat{\mathfrak{d}}$-self-dual.
\end{proof}

\section{Pseudo-completeness and the sababa property}
\subsection{Definitions}
In the following section, we introduce the property of pseudo-completeness of an $A_\infty$ algebra and show that relevant classes of $A_\infty$ algebras are pseudo-complete.
Before plunging into the definition, a word of explanation is in order.
The proof of Theorems~\ref{thm1} and~\ref{thm2} relies on an inductive argument based on obstruction theory applied to an $(S,F)$ extension of the de Rham Fukaya $A_\infty$ algebra of the Lagrangian submanifold $L.$ As explained in Remarks~\ref{rem:maurer_cartan_converge}, and~\ref{rem:pseudo_sababa} below, the analysis of the obstruction classes requires us to consider $S$ with the trivial filtration. Consequently, unit bounding cochains have trivial filtration, and a priori there is no reason the Maurer-Cartan equation should converge.  Pseudo-completeness guarantees convergence of the Maurer-Cartan equation and related infinite sums that appear in the obstruction theory computation. Precise statements are formulated in Remark~\ref{rem:psedu_converge_uniformly} Lemma~\ref{lm:m_bar_complex_converge}, and Corollary~\ref{cor:mgh_expo_converge} below.
\begin{dfn}
Let $(\cC,\m_k)$ be an $A_\infty$ algebra over $\mathcal{R}$, complete with respect to its filtration $\varsigma_1$. Let $\varsigma_2$ be another filtration with $\varsigma_2 \ge \varsigma_1$. We say that $(\cC,\m_k)$ is \textbf{pseudo-complete} with respect to $\varsigma_2$, if for all $E>0$ and $D \in \Z$ there exists $K>0$ such that for $\a_1,\,\ldots,\a_k \in \cC$, we have
\[\sum_{i=1}^k \varsigma_2(\a_i)>K,\, |\m_k(\a_1,\ldots,\a_k)|\ge D \implies \varsigma_1(\m_k(\a_1,\ldots,\a_k))> E. \]
\end{dfn}
\begin{rem}\label{rem:psedu_converge_uniformly}
For the topologies induced by $\varsigma_2^{\otimes k}$ and $\varsigma_1$, pseudo-completeness implies that the maps $\m_k:(\cC^{\otimes k})_{j+k}\to\cC$ are continuous, uniformly in $k$ for each fixed $j$.
\end{rem}
As mentioned above, we will use pseudo-completeness to show convergence of the Maurer-Cartan equation in an $(S,F)$ extension of the Fukaya $A_\infty$ algebra of the Lagrangian submanifold $L.$ However, as explained in Remark~\ref{rem:disaster_example} below, if we $(S,F)$ extend the whole algebra, the result need not be pseudo-complete. Instead, we focus on a subalgebra that satisfies the sababa property in the following definition. Proposition~\ref{prop:saba_decend_psedu_comp} below guarantees that the $(S,F)$ extension of the sababa subalgebra is pseudo-complete. The sababa property for a subalgebra is closely related to the sababa property for a monoid used in~\cite{solomon2016point}.
\begin{dfn}\label{dfn:sababa_algebra}
A filtered differential graded $\k$-algebra $\mathcal{R}$ is \textbf{sababa} if it is generated as a topological algebra by elements $\{\lambda_i\}_{i\in \N}$ with $\lim_{i \to \infty }\varsigma_{\mathcal{R}}(\lambda_i)=\infty$. A cyclic unital $A_{\infty}$ algebra $(\cC, \{\m_k\}_{k\ge 0},\lp\;,\,\rp,\e)$ over $\mathcal{R}$ is called \textbf{sababa} if $\mathcal{R}$ is sababa, $\cC$ has a $\bar{\mathcal{R}}$ structure, and there exists $h \in \Z$ such that if $\bar{\a} \in \bar{\cC}$ then $|\bar{\a}|\le h$. The constant $h$ is called the sababa constant.
\end{dfn}

\subsection{Preliminary lemmas}
\begin{lm}\label{lm:enrgy_degree_bound_sababa}
Let $(\cC, \{\m_k\}_{k\ge 0},\lp\;,\,\rp,\e)$ be a sababa cyclic unital $A_{\infty}$ algebra over a sababa $\mathcal{R}$. For all $E>0$ there exists $m\in \Z$ such that if $\a \in \cC$ is homogeneous, then 
\[\varsigma_\cC(\a)< E \implies |\a|\le m.\]  
\end{lm}
\begin{proof}
We will use the $\bar{\mathcal R}$ structure to identify $\cC$ with $\mathcal R\otimes_{\bar{\mathcal R}}\bar{\cC}$ and regard $\bar{\mathcal R}$ as a subalgebra of $\mathcal R$. Pick generators as in Definition~\ref{dfn:sababa_algebra}. Replacing each zero-filtration generator $\lambda_i$ by $\lambda_i-\bar\lambda_i$ and discarding zeros gives generators over $\bar{\mathcal R}$ of positive filtration. We continue to denote them by $\{\lambda_i\}$. Only finitely many original generators were changed, so there are still only finitely many below any fixed energy.

Let $\alpha\in \cC$ be homogeneous, and assume that $\nu_{\cC}(\alpha)<E$. There exist $\{r_i\}_{i\in \N}\subset \mathcal{R}$ and $\{\bar\a_i\}_{i\in \N} \subset \overline{\cC}$ such that 
\[\alpha=\sum_{i\in \N} r_i \otimes \bar\a_i.\]
Thus there exists $i\in \N$ such that $\nu_{\mathcal{R}}(r_i)<E.$
Write $r_i\equiv \sum_{j=1}^{N} r^j_i\bar c^j \pmod{F^E\mathcal{R}}$, where the $r^j_i$ are monomials in the generators $\{\lambda_\ell\}_{\ell \in \N}$ and $\bar c^j\in\bar{\mathcal R}$. Then
\[
r_i\otimes\bar\a_i
=\sum_{j=1}^{N}(r^j_i\bar c^j)\otimes\bar\a_i
=\sum_{j=1}^{N} r^j_i\otimes(\bar c^j\bar\a_i).
\]
The set
\[
A:=\left\{r=\lambda_{i_1}\cdots\lambda_{i_k}\ne0:
 k\ge0,\ \varsigma_{\mathcal R}(r)<E\right\},
\]
including the empty monomial $1$, is finite: there are finitely many possible factors, and their positive filtrations bound the length of the product. Choose an integer $a$ such that $|r|\le a$ for all $r\in A$. It follows that for any $1 \leq j \leq N,$
\[
|\alpha|= |r_i\otimes\bar\a_i| = |r^j_i\otimes(\bar c^j\bar\a_i)| \le  a+h,
\]
where $h$ is the sababa constant. Hence $m=a+h$ is the desired bound.
\end{proof}

\begin{lm}\label{lm:product_filitiration_it_val}
Let $(S,e_S)$ be a filtered unital associative differential graded algebra over $\k$ with a filtration $\varsigma_S$. Let $(\hat{\mathcal{C}}, \{\mh_k\}_{k \ge 0}, \lp\;,\,\rp_F,\e_S \otimes \e)$ be an $(S, F)$ extension of a cyclic unital $A_{\infty}$ algebra $(\cC, \{\m_k\}_{k\ge 0},\lp\;,\,\rp,\e)$. Then for all $\a_1,\ldots \a_k \in \hat{\cC}$ we have
\[\varsigma_{S}\otimes \varsigma_{\cC}(\mh_k(\a_1,\ldots \a_k ))\ge \sum_{i=1}^k \varsigma_{S}\otimes \varsigma_{\cC}(\a_i). \]
\end{lm}
\begin{proof}
Assume by contradiction that the lemma is false. Then there exist $\a_1, \ldots, \a_k \in \hat{\cC}$ such that 
\[\sum_{i=1}^k \varsigma_{S}\otimes \varsigma_{\cC}(\a_i) > \varsigma_{S}\otimes \varsigma_{\cC}(\mh_k(\a_1,\ldots \a_k )). \]
In particular, there exist  $
\{s_{ij}\}_{j \in \N}\subset S$, and $\{\eta_{ij}\}_{j \in \N}\subset \cC$, such that 
\[\a_i = \sum_{j\in \N}s_{ij}\otimes \eta_{ij},\]
and that
\[\sum_{i=1}^k \min_{j\in \N} \bigl(\varsigma_{S}(s_{ij})+\varsigma_{\cC}(\eta_{ij})\bigr)>\varsigma_{S}\otimes \varsigma_{\cC}(\mh_k(\a_1,\ldots \a_k )).\] 
As $S\otimes \cC$ is a filtered module we have \[\varsigma_{S}\otimes \varsigma_{\cC}(x+y)\ge \min(\varsigma_{S}\otimes \varsigma_{\cC}(x),\varsigma_{S}\otimes \varsigma_{\cC}(y)),\qquad x,y \in S\otimes \cC.\]
Thus we have 
\begin{align*}
    \varsigma_{S}\otimes \varsigma_{\cC}(\mh_k(\a_1,\ldots \a_k ))&\ge \min_{(j_1,\ldots,j_k)\in \N^k}\varsigma_{S}\otimes \varsigma_{\cC}(\mh_k(s_{1{j_1}}\otimes \eta_{1{j_1}},\ldots s_{k{j_k}}\otimes \eta_{k{j_k}} ))\\
    &=\min_{(j_1,\ldots,j_k)\in \N^k}\varsigma_{S}\otimes \varsigma_{\cC}\Bigl(\pm s_{1{j_1}}\ldots s_{k{j_k}}\otimes  \m_k(\eta_{1{j_1}},\ldots, \eta_{k{j_k}})+\delta_{1,k}\,d s_{1{j_1}}\otimes\eta_{1{j_1}}\Bigr)\\
    &\ge \min_{(j_1,\ldots,j_k)\in \N^k}\Bigl(\sum_{i=1}^k \varsigma_{S}(s_{i{j_i}})+\sum_{i=1}^k \varsigma_{\cC}(\eta_{i{j_i}})\Bigr), 
\end{align*}
where the last inequality follows by bounding each summand using $S$ being a filtered ring, property~\ref{it:val} of the $A_\infty$ algebra $\cC$, and, for $k=1$, $\varsigma_S(ds_{1j_1})\ge\varsigma_S(s_{1j_1})$. Thus we have a contradiction as
\[\sum_{i=1}^k \min_{j\in \N} \bigl(\varsigma_{S}(s_{ij})+\varsigma_{\cC}(\eta_{ij})\bigr)\le \min_{(j_1,\ldots,j_k)\in \N^k}\Bigl(\sum_{i=1}^k \varsigma_{S}(s_{i{j_i}})+\sum_{i=1}^k \varsigma_{\cC}(\eta_{i{j_i}})\Bigr)\le  \varsigma_{S}\otimes \varsigma_{\cC}(\mh_k(\a_1,\ldots \a_k )).\]
\end{proof}

\subsection{Descending filtrations}
\begin{dfn}
A filtered unital associative differential graded algebra $(S,e_S)$ over $\k$ with a filtration $\varsigma_S$ is called \textbf{descending} if every nonzero homogeneous \(s\in S\) satisfies
\[|s|\le -\varsigma_S(s).\]
\end{dfn}
Let $(\hat{\cC}, \{\mh_k\}_{k\ge 0},\lp\;,\,\rp_F,\e_S\otimes\e)$ be an $(S,F)$ extension of a cyclic unital $A_{\infty}$ algebra over $\mathcal{R}$. For $m \in \Z$, we define the map $\pi_{S,m}:\hat{\cC}\to \bigoplus_{i\le m}S_{i}\otimes \cC$ to be the projection. Explicitly, for $s\in S_j,$ and $\a \in \mathcal{C}$, the map is defined by
\[\pi_{S,m}(s\otimes \alpha)=\begin{cases}
s\otimes \alpha, & j\le m\\
0, & otherwise,
\end{cases} \]
and extended linearly for arbitrary elements.
All filtration thresholds in the following lemma are finite. When an assertion at infinite filtration is needed, it is obtained by proving the corresponding assertions at every finite threshold \(E\) and passing to the limit as \(E\to\infty\).
\begin{lm}\label{lm:decending_proj_non_trivial}
Let $(S, e_S)$ be a descending unital differential graded algebra over $\k$. Let \[(\hat{\mathcal{C}}, \{\mh_k\}_{k \ge 0}, \lp\;,\,\rp_F,\e_S \otimes \e)\]
be an $(S, F)$ extension of a cyclic unital $A_{\infty}$ algebra $(\cC, \{\m_k\}_{k\ge 0},\lp\;,\,\rp,\e)$  over $\mathcal{R}$. Let $E\in \R_{\ge 0}$, and
let $(\a_n)_{n\in \N}$ be a sequence in $\hat{\cC}$ such that $\varsigma_{\hat{\cC}}(\a_n)\le E$ and
\[\varsigma_S \otimes \varsigma_{\cC}(\a_n)-\varsigma_{\hat{\cC}}(\a_n) \xrightarrow{n \to \infty}\infty. \]
Then for any $m \in \Z $, we have for $n$ large enough 
\[\pi_{S,m}(\a_n)\not\equiv 0 \pmod{F^{>E}\hat{\cC}},\]
\end{lm}
\begin{proof}
Without loss of generality we can assume $m<0$. Pick $N$ such that for all $n>N$ we have 
\[\varsigma_S \otimes\varsigma_{\cC}(\a_n)-\varsigma_{\hat{\cC}}(\a_n) \geq -m + 2.\]
Thus there exist $\{\eta_i\}_{i\in \N}\subset \cC$, and $\{s_i\}_{i\in \N} \subset S$ such that  \[\varsigma_{\cC}(\eta_i) \xrightarrow{i\to \infty }\infty,\qquad \a_n=\sum_{i=1}^{\infty}s_i \otimes \eta_i,  \]
and
\[\min_{1\le i < \infty }\varsigma_{S}(s_i)+\varsigma_{\cC}(\eta_i) \geq \varsigma_S \otimes\varsigma_{\cC}(\a_n) -1.
\]
Then,
\[
\min_{1\le i < \infty }\varsigma_{S}(s_i)+\varsigma_{\cC}(\eta_i) \geq-m+1+\varsigma_{\hat{\cC}}(\a_n)>-m+\varsigma_{\hat{\cC}}(\a_n) \geq -m + \min_{1\le j < \infty }\varsigma_{\cC}(\eta_j).\]
Let $i_0$ and $j_0$ be indices that realize the minima above.
Then
\[\varsigma_{S}(s_{j_0}) +\varsigma_{\hat{\cC}}(\a_n)\ge \varsigma_{S}(s_{j_0}) +\varsigma_{\cC}(\eta_{j_0})\ge \varsigma_{S}(s_{i_0})+\varsigma_{\cC}(\eta_{i_0})>-m+\varsigma_{\hat{\cC}}(\a_n).\]
From the preceding inequality we get that $\varsigma_{S}(s_{j_0})>-m$. Since $S$ is descending we deduce $|s_{j_0}|<m$.
In the case that the projection of $\a_n$ to $S_{|s_{j_0}|}\otimes \cC $ is non-trivial modulo $F^{>E}\hat{\cC}$ we are done. Otherwise, we have 
\[\a_n=\sum_{ \substack{i\in \N:\\
|s_i|\ne|s_{j_0}|}}s_i \otimes \eta_i\pmod{F^{>E}\hat{\cC}}.\]
In the same manner as before, we can find another index $j_1\ne j_0$ such that $\varsigma_{\hat{\cC}}(\a_n)\ge \varsigma_{\cC}(\eta_{j_1}) $ and $|s_{j_1}|<m$. If the projection of $\a_n$ to $S_{|s_{j_1}|}\otimes \cC$ is trivial modulo $F^{>E}\hat{\cC}$ we repeat the argument and find another index $j_2$. Because $\varsigma_{\cC}(\eta_i) \xrightarrow{i\to \infty }\infty$ there has to be some index $j_l$, such that $\varsigma_{\hat{\cC}}(\a_n)\ge \varsigma_{\cC}(\eta_{j_l}) $, $|s_{j_l}|<m$ and the projection of $\a_n$ to $S_{|s_{j_l}|}\otimes \cC$ is modulo $F^{>E}\hat{\cC}$ non-trivial. We find that $\pi_{S,m}(\a_n)\ne 0 \pmod{F^{>E}\hat{\cC}}$ for all $n>N$. 
\end{proof}

\subsection{Pseudo-completeness of extension}
\begin{prop}\label{prop:saba_decend_psedu_comp}
Let $(S, e_S)$ be a descending unital differential graded algebra over $\k$.
Let $(\hat{\mathcal{C}}, \{\mh_k\}_{k \ge 0}, \lp\;,\,\rp_F,\e_S \otimes \e)$ be an $(S, F)$ extension of an $A_{\infty}$ algebra $(\cC, \{\m_k\}_{k\ge 0},\lp\;,\,\rp,\e)$  over $\mathcal{R}$. Assume that $(\cC, \{\m_k\}_{k\ge 0},\lp\;,\,\rp,\e)$ is a sababa cyclic unital $A_{\infty}$ algebra over some differential graded algebra $\mathcal{R}^\prime$.  Then the $A_\infty$ algebra $(\hat{\mathcal{C}}, \{\mh_k\}_{k \ge 0}, \lp\;,\,\rp_F,\e_S \otimes \e)$ is pseudo-complete with respect to $\varsigma_S\otimes \varsigma_{\cC}$. 
\end{prop}
\begin{proof}
Assume by contradiction that the proposition is false. Then there exist $E>0$ and $D \in \Z$ such that for all $K>0$ there exist $\a_{1}^K,\ldots, \a_{k_K}^K \in \hat{\cC}$ with
\[\sum_{i=1}^{k_K}\varsigma_S\otimes \varsigma_{\cC}(\a^K_i)>K,\quad |\mh_{k_K}(\a^K_1,\ldots,\a^K_{k_K})|\ge D,\quad \varsigma_{\hat{\cC}}(\mh_{k_K}(\a^K_1,\ldots,\a^K_{k_K}))\le E.\]
It follows from our assumption and Lemma~\ref{lm:product_filitiration_it_val} that
\[\varsigma_S\otimes \varsigma_{\cC}(\mh_{k_K}(\a^K_1,\ldots,\a^K_{k_K}))-\varsigma_{\hat{\cC}}(\mh_{k_K}(\a^K_1,\ldots,\a^K_{k_K}))\xrightarrow{K\to \infty} \infty.\] 
For all $m\in \Z$ we can use Lemma~\ref{lm:decending_proj_non_trivial} to find some $N,$ such that for $K>N$ we have 
\[\pi_{S,m}(\mh_{k_K}(\a^K_1,\ldots,\a^K_{k_K}))\not\equiv 0 \pmod{F^{>E}\hat{\cC}}.\]
Write $q=|\mh_{k_K}(\a^K_1,\ldots,\a^K_{k_K})|\ge D$ and choose a total-degree-homogeneous representation
 \[\pi_{S,m}(\mh_{k_K}(\a^K_1,\ldots,\a^K_{k_K}))=\sum_{i\in \N}s_i \otimes \eta_i,\]
where $s_i\in S$ and $\eta_i\in\cC$ are homogeneous and $|s_i\otimes\eta_i|=q$ for every nonzero term. Then there exists $j\in\N$ such that $\varsigma_{\cC}(\eta_j)\le E$.
By Lemma~\ref{lm:enrgy_degree_bound_sababa}, applied with $E+1$, there exists a bound $M$ on the degree of such homogeneous elements $\eta_j$. Thus we get
\[|s_j \otimes \eta_j|\le m +M.\]
Choosing $m$ sufficiently small, this contradicts the lower bound
\[D \le q=|s_j \otimes \eta_j|.\]  
\end{proof}

\subsection{Convergence of deformations by bounding cochains}\label{sssec:convergence}
Let $(\cC,\{\m_k\}_{k\ge 0})$ be an $A_\infty$ algebra over $\mathcal{R}$, complete with respect to its filtration $\varsigma_1$. Let $\varsigma_2$ be another filtration defined on $\cC$, and assume that $(\cC,\{\m_k\}_{k\ge 0})$ is pseudo-complete with respect to $\varsigma_2$. 
\begin{lm}\label{lm:m_bar_complex_converge}
    Let $b \in \cC$ with $\varsigma_2(b)>0$, and $|b|\ge 1$ then for any $x_1,\ldots, x_k \in\cC$ homogeneous,
    \[
   \lim_{\ell \to \infty} \min_{\ell_0 + \cdots + \ell_k = \ell} \varsigma_1( \m_{k+\ell_{0}+\cdots+\ell_{k}}(\underbrace{b, \cdots, b}_{\ell_{0}}, x_{1}, \underbrace{b, \cdots, b}_{\ell_{1}}, \cdots, \underbrace{b, \cdots, b}_{\ell_{k-1}}, x_{k}, \underbrace{b, \cdots, b}_{\ell_{k}})) = \infty.
    \]
\end{lm}
\begin{proof}
  We will show that for $E>0$ there are only finitely many instances of $\ell_{0}, \cdots, \ell_{k}\in \mathbb{Z}_{\ge 0}$ with 
    \[\varsigma_1\left(\m_{k+\ell_{0}+\cdots+\ell_{k}}(\underbrace{b, \cdots, b}_{\ell_{0}}, x_{1}, \underbrace{b, \cdots, b}_{\ell_{1}}, \cdots, \underbrace{b, \cdots, b}_{\ell_{k-1}}, x_{k}, \underbrace{b, \cdots, b}_{\ell_{k}})\right)\le E.\]
    By pseudo-completeness of $(\cC, \m_k)$ with respect to $\varsigma_2$, for $D:=2-k+\sum_{i=1}^k |x_i|$ there exists $K>0$ such that for $\a_1,\ldots,\a_m \in \cC$, we have
    \[\sum_{i=1}^m \varsigma_2(\a_i)>K,\, |\m_m(\a_1,\ldots,\a_m)|\ge D \implies \varsigma_1(\m_m(\a_1,\ldots,\a_m))> E. \]
    From the assumption on $b$ there exists $\ell \in \N$ such that $\sum_{i=1}^\ell \varsigma_2(b)>K,$ and
    \[\left|\m_{k+\ell_{0}+\cdots+\ell_{k}}(\underbrace{b, \cdots, b}_{\ell_{0}}, x_{1}, \underbrace{b, \cdots, b}_{\ell_{1}}, \cdots, \underbrace{b, \cdots, b}_{\ell_{k-1}}, x_{k}, \underbrace{b, \cdots, b}_{\ell_{k}})\right|\ge D.\]
    Thus, if $\sum_{i=0}^k \ell_i \ge \ell$ then
    \[\varsigma_1\left(\m_{k+\ell_{0}+\cdots+\ell_{k}}(\underbrace{b, \cdots, b}_{\ell_{0}}, x_{1}, \underbrace{b, \cdots, b}_{\ell_{1}}, \cdots, \underbrace{b, \cdots, b}_{\ell_{k-1}}, x_{k}, \underbrace{b, \cdots, b}_{\ell_{k}})\right)> E.\]
\end{proof}

Let $\overset{\circ}{\cC}$ be the filtered non-completed direct sum $\bigoplus_{m\in \Z}(\mathcal{C})_m$. For $b\in \cC$ with $\varsigma_2(b)>0$ and $|b|=1$, the above lemma, completeness of $\cC$, and property \ref{it:val} of the $A_\infty$ algebra allow us to define a bounded $k$-linear map $\overset{\circ}{\m}_k^b: \prod_{i=1}^k \overset{\circ}{\cC} \to \cC$ by 
\[\overset{\circ}{\m}_k^b\left(x_{1}, \cdots, x_{k}\right)=\sum_{\ell_{0}, \cdots, \ell_{k}\in \mathbb{Z}_{\ge 0}} \m_{k+\ell_{0}+\cdots+\ell_{k}}(\underbrace{b, \cdots, b}_{\ell_{0}}, x_{1}, \underbrace{b, \cdots, b}_{\ell_{1}}, \cdots, \underbrace{b, \cdots, b}_{\ell_{k-1}}, x_{k}, \underbrace{b, \cdots, b}_{\ell_{k}}) .\]
\begin{dfn}\label{dfn:bdef}
    For $b\in \cC$ with $\varsigma_2(b)>0$ and $|b|=1$, a \textbf{$b$ deformation} of the operator $\m_k:\cC^{\otimes k} \to \cC[2-k]$ is a bounded $\k$ linear map $\m_k^b:\cC^{\otimes k} \to \cC[2-k]$ such that the following diagram commutes
    \[\begin{tikzcd}
	{\prod_{i=1}^k\overset{\circ}{\mathcal{C}}} & {\mathcal{C}^{\otimes k}} \\
	& {\mathcal{C}}
	\arrow["{\overset{\circ}{\mathfrak{m}}_k^b}"', from=1-1, to=2-2]
	\arrow["\tau", from=1-1, to=1-2]
	\arrow["{\mathfrak{m}_k^b}", dashed, from=1-2, to=2-2]
\end{tikzcd},\]
    where $\tau:\prod_{i=1}^k \cC \to \cC^{\otimes k} $ is given by $(x_1,\ldots,x_k)\mapsto x_1\otimes\ldots \otimes x_k$. 
\end{dfn}
\begin{lm}\label{lm:unique_b_deform}
     For $b\in \cC$ with $\varsigma_2(b)>0$ and $|b|=1$, there exists a unique $b$ deformation of $\m_k:\cC^{\otimes k}\to \cC$ for each $k\ge 0$.
\end{lm}
\begin{proof}
    For $k=0$, the statement follows from Lemma~\ref{lm:m_bar_complex_converge} and completeness of $\cC$. For $k=1$ the statement follows from continuity of $\overset{\circ}{\m}^b_1,$ and completeness of $\cC.$ For $k=2$, the statement follows by Proposition 1 in \cite[Section 2.1.7]{bosch1984non},  concerning bounded bilinear maps. The case $k>2$ is obtained by adjusting Proposition 1 in \cite[Section 2.1.7]{bosch1984non} to  bounded multilinear maps. 
\end{proof} 
\begin{lm}\label{lm:deformed_operator_explicit_formula}
    Let $x_1,\ldots x_k \in \cC$ then 
    \[\m_k^b\left(x_{1}, \cdots, x_{k}\right)=\sum_{\ell_{0}, \cdots, \ell_{k}\in \mathbb{Z}_{\ge 0}} \m_{k+\ell_{0}+\cdots+\ell_{k}}(\underbrace{b, \cdots, b}_{\ell_{0}}, x_{1}, \underbrace{b, \cdots, b}_{\ell_{1}}, \cdots, \underbrace{b, \cdots, b}_{\ell_{k-1}}, x_{k}, \underbrace{b, \cdots, b}_{\ell_{k}}) .\]
\end{lm}
\begin{proof}
    Write $x_i=\sum_{j=1}^\infty \a_j^i $ with $\a_j^i \in \overset{\circ}{\cC}$. As boundedness implies continuity for filtered $k$-linear maps, we have
    \begin{align*}
        \m_k^b(x_{1}, &\cdots, x_{k}) = \\
        &= \sum_{j_1,\ldots j_k \in \N}\m^b_k(\a_{j_1}^1,\ldots ,\a_{j_k}^k)\\
        &=\sum_{j_1,\ldots j_k \in \N}\sum_{\ell_{0}, \cdots, \ell_{k}\in \mathbb{Z}_{\ge 0}} \m_{k+\ell_{0}+\cdots+\ell_{k}}(\underbrace{b, \cdots, b}_{\ell_{0}}, \a_{j_1}^1, \underbrace{b, \cdots, b}_{\ell_{1}}, \cdots, \underbrace{b, \cdots, b}_{\ell_{k-1}}, \a_{j_k}^k, \underbrace{b, \cdots, b}_{\ell_{k}}). 
    \shortintertext{Lemma~\ref{lm:m_bar_complex_converge} ensures that the assumptions of Proposition 2 in \cite[Section 1.1.8]{bosch1984non} are satisfied. Thus, we can change the order of summation and get}\\
    &=\sum_{\ell_{0}, \cdots, \ell_{k}\in \mathbb{Z}_{\ge 0}} \sum_{j_1,\ldots j_k \in \N}\m_{k+\ell_{0}+\cdots+\ell_{k}}(\underbrace{b, \cdots, b}_{\ell_{0}}, \a_{j_1}^1, \underbrace{b, \cdots, b}_{\ell_{1}}, \cdots, \underbrace{b, \cdots, b}_{\ell_{k-1}}, \a_{j_k}^k, \underbrace{b, \cdots, b}_{\ell_{k}})\\
    &=\sum_{\ell_{0}, \cdots, \ell_{k}\in \mathbb{Z}_{\ge 0}} \m_{k+\ell_{0}+\cdots+\ell_{k}}(\underbrace{b, \cdots, b}_{\ell_{0}}, x_{1}, \underbrace{b, \cdots, b}_{\ell_{1}}, \cdots, \underbrace{b, \cdots, b}_{\ell_{k-1}}, x_{k}, \underbrace{b, \cdots, b}_{\ell_{k}}).
\end{align*}
To obtain the last equality, we used the boundedness of the operators $\{\m_k\}_{k\ge 0}$.      
\end{proof}

Lemma~\ref{lm:deformed_operator_explicit_formula} allows us to use the axioms of the $n$-dimensional curved cyclic unital $A_\infty$ algebra $(\cC, \{\m_k\}_{k\ge 0},\lp\;,\,\rp,\e)$ and deduce the corollary below. 
\begin{cor}\label{cor:cyclic_b_deformation}
    Let $(\cC, \{\m_k\}_{k\ge 0},\lp\;,\,\rp,\e)$ be an $n$-dimensional curved cyclic unital $A_{\infty}$ algebra over $\mathcal R$ with values in $M$, and assume that $\operatorname{char}\k\ne2$. Assume that the left and right module structures of $\cC[1]$ over $Z(\mathcal{R})$ coincide. Let $b\in\cC$ satisfy
    \[
    \varsigma_2(b)>0,\qquad |b|=1,\qquad
    \overline{\m_0^b}=0,\qquad d_M\lp b,\e\rp=0.
    \]
    Then $(\cC, \{\m_k^b\}_{k\ge 0},\lp\;,\,\rp,\e)$ is an $n$-dimensional curved cyclic unital $A_{\infty}$ algebra over $Z(\mathcal{R})$ with values in the same $M$.
\end{cor}
\begin{proof}
The properties of Definition~\ref{dfn:cycunit}, except for property~\ref{it:unit3}, follow from Lemma~\ref{lm:deformed_operator_explicit_formula} and the assumption $\overline{\m_0^b}=0$. By cyclicity and the unit property,
\begin{equation}\label{eq:deformed_curvature_unit_pairing}
\begin{aligned}
\lp\m_0^b,\e\rp
&=\lp\m_1(b),\e\rp+\lp\m_2(b,b),\e\rp\\
&=d_M\lp b,\e\rp+\lp b,b\rp\\
&=d_M\lp b,\e\rp.
\end{aligned}
\end{equation}
Here $\lp b,b\rp=0$ by property~\ref{it:symm}, since $|b|=1$ and $\operatorname{char}\k\ne2$. The assumption $d_M\lp b,\e\rp=0$ proves property~\ref{it:unit3}.
\end{proof}
\begin{rem}
    We do not exclude the case that $\varsigma_1(b)=0$. So,
    the condition $\overline{\m_0^b}=0$ is not trivial.
    Properties~\ref{it:lin}, \ref{it:a_infty}, \ref{it:unit1}, \ref{it:unit2}, and~\ref{it:cyclic}, as well as equation~\eqref{eq:deformed_curvature_unit_pairing}, hold without the assumptions $\overline{\m_0^b}=0$ and $d_M\lp b,\e\rp=0$.
\end{rem}
Let $(S, e_s)$ be a descending unital differential graded algebra over $\k$.
Let 
\[(\hat{\mathcal{C}}, \{\mh_k\}_{k \ge 0}, \lp\;,\,\rp_F,\e_S \otimes \e)\] 
be an $(S, F)$ extension of a cyclic unital $A_{\infty}$ algebra $(\cC, \{\m_k\}_{k\ge 0},\lp\;,\,\rp,\e)$  over $\mathcal{R}$. Assume also that $(\cC, \{\m_k\}_{k\ge 0},\lp\;,\,\rp,\e)$ is a sababa cyclic unital $A_{\infty}$ algebra over some differential graded algebra $\mathcal{R}^\prime$.
By Proposition~\ref{prop:saba_decend_psedu_comp} $(\hat{\mathcal{C}}, \{\mh_k\}_{k \ge 0}, \lp\;,\,\rp_F,\e_S \otimes \e)$ is pseudo-complete with respect to $\varsigma_{S}\otimes \varsigma_{\cC}$. Thus, given $b \in \hat{\cC}$ with $\varsigma_{S}\otimes\varsigma_{\cC}(b)>0$ and $|b|=1$, Lemma~\ref{lm:unique_b_deform} guarantees the existence and uniqueness of the $b$ deformed operators $\{\mh_k^b\}_{k\ge 0}$.
\begin{cor}\label{cor:mkbh_converges}
Let $(S, e_s)$ be a descending unital differential graded algebra over $\k$, with $\operatorname{char}\k\ne2$.
Let $(\hat{\mathcal{C}}, \{\mh_k\}_{k \ge 0}, \lp\;,\,\rp_F,\e_S \otimes \e)$ be an $(S, F)$ extension of a cyclic unital $A_{\infty}$ algebra $(\cC, \{\m_k\}_{k\ge 0},\lp\;,\,\rp,\e)$  over $\mathcal{R}$ with values in $M$ that is sababa over a subalgebra $\mathcal{R}' \subset \mathcal{R}$. Assume that the left and right module structures of $\cC[1]$ over $\mathcal{R}$ coincide. Let $b \in \hat{\cC}$ satisfy
\[
\varsigma_{S}\otimes\varsigma_{\cC}(b)>0,\qquad |b|=1,\qquad
\overline{\mh_0^b}=0,\qquad d_{\hat M}\lp b,\e_S\otimes\e\rp_F=0.
\]
Then $(\hat{\mathcal{C}}, \{\mh_k^b\}_{k \ge 0}, \lp\;,\,\rp_F,\e_S \otimes \e)$ is a cyclic unital $A_{\infty}$ algebra over $Z\left(\hat{\mathcal{R}}\right)$ with values in $\hat M$.
\end{cor}
\begin{prop}\label{prop:deformed_boundrary_operator}
Let $(S, e_S)$ be a descending unital differential graded algebra over $\k$. Let $(\cC, \{\m_k\}_{k\ge 0},\lp\;,\,\rp,\e)$ be a cyclic unital $A_{\infty}$ algebra over $\mathcal{R}$ with values in $M$ such that the residue algebra $(\bar{\cC}, \{\bar{\m}_k\}_{k\ge 0},\lp\;,\,\rp,\e)$ arises from a dga $(\Upsilon,d)$.  Assume that the left and right module structures of $\cC[1]$ over $\mathcal{R}$ coincide. Assume also that $\cC$ has an $\bar{\mathcal{R}}$ structure given by $g:\mathcal{R}\otimes_{\bar{\mathcal{R}}} \bar{\cC} \to \cC$ and let $f$ be its inverse.
Let $(\hat{\cC}, \{\hat{\m}_k\}_{k\ge 0},\lp\;,\,\rp_F,\e_S\otimes\e)$ be an $(S,F)$ extension of $\cC$, let $\hat{g}$ be the induced $\bar{\mathcal{R}}$ structure on it, and let $\hat{f}$ be its inverse. Let $b\in \hat{\cC}$ with $|b|=1$, and $\varsigma_{S}\otimes\varsigma_{\cC}(b)>0$.
Let $\alpha \in F^E\hat{\cC}$ be a homogeneous element, then
\[
\hat{f}\left(\mh^b_1(\alpha)\right)
\equiv
d\hat{f}(\a)-[\hat{f}(b),\hat{f}(\a)]
\pmod{F^{>E}\left(\mathcal{R}\otimes_{\bar{\mathcal{R}}}\overline{\hat{\cC}}\right)},
\]
where the product on the right hand side is induced by the algebra structure of $\mathcal{R} \otimes_{\bar{\mathcal{R}}} \hat\Upsilon$ given by the isomorphism $\bar{\hat{\cC}} \to \hat{\Upsilon}$ of Lemma~\ref{lm:cano_iso}.
\end{prop}
\begin{proof}
We first note that $$(\hat{\cC}, \{\hat{\m}_k\}_{k\ge 0},\lp\;,\,\rp_F,\e_S\otimes\e)$$ is a cyclic $A_\infty$ algebra over $\mathcal{R}$ with values in $\hat M$, by restriction of scalars. As the left and right  $\mathcal{R}$ module structures on $\cC[1]$ coincide, so do the left and right $\mathcal{R}$ module structures on $\hat{\cC}[1]$. As $(\bar{\cC}, \{\bar{\m}_k\}_{k\ge 0},\lp\;,\,\rp,\e)$ arises from the dga $(\Upsilon,d)$, Proposition~\ref{prop:sf_extension_arises_dga} implies that $(\bar{\hat{\mathcal{C}}}, \{\bar{\mh}_k\}_{k \ge 0}, \lp\;,\,\rp_F,\e_S \otimes \e)$ arises from the dga $(\hat{\Upsilon},d_S\otimes\operatorname{id}+\operatorname{id}\otimes d)$. We continue to denote the induced total differential by $d$. Thus the conditions of Proposition~\ref{prop:r_structure_energy_dga} are satisfied, and we can deduce
\[\hat{f}(\mh_1^b(\a))\equiv d\hat{f}(\a)+ (-1)^{|b|}\hat{f}(b)\hat{f}(\a)+(-1)^{|\a|}\hat{f}(\a)\hat{f}(b)\pmod{F^{>E} \mathcal{R}\otimes_{\bar{\mathcal{R}}}\overline{\hat{\cC}}}
\]
where the product is given by the natural product on $\mathcal{R}\otimes_{\bar{\mathcal{R}}} \bar{\hat{\mathcal{C}}}$. The equation above is equivalent to the desired formula as $|b|=1$.
\end{proof}

\subsection{Duality and deformations}\label{ssec:opposite_duality}

For a set automorphism $\invs : Z \to Z,$ let $Z^\invs \subset Z$ denote the subset fixed pointwise by $\invs.$ 
Let $\mathfrak{d}: (\cC, \m) \rightarrow (\cC,\m^{op})$ be a strict $A_\infty$ homomorphism such that $(\cC, \mathfrak{m})$ is $\mathfrak{d}$ self-dual as in Definition~\ref{dfn:d_self_dual}.
\begin{lm}\label{lm:sd_m_zero}
    Let $b \in (\mathcal{C}_1)^{-\mathfrak{d}}$ with $\varsigma_2(b) > 0$. Then $\m_0^b \in (\mathcal{C}_2)^{-\mathfrak{d}}.$
\end{lm}
\begin{proof}
The fact that $(\mathcal{C}, \mathfrak{m})$ is $\mathfrak{d}$-self-dual combined with the assumption that $|b|=1$ implies that 
\[
\mathfrak{d} \circ \m_k(b,\dots,b)=(-1)^{k+1+s_\tau(b,\dots,b)}\m_k\left(\mathfrak{d}(b),\dots,\mathfrak{d}(b)\right)=-\m_k(b,\dots,b),
\]
and that $\mathfrak{d}\m_0=-\m_0.$ 
Thus, the result follows from Lemma~\ref{lm:deformed_operator_explicit_formula}.
\end{proof}
\begin{lm}\label{lm:sd_equivariance}
Let $(\mathcal{C}, \mathfrak{m})$ be a $\mathfrak{d}$-self-dual $A_\infty$ algebra, and let $b \in (\mathcal{C}_1)^{-\mathfrak{d}}$ with $\varsigma_2(b) > 0$. Then the deformed operator $\mathfrak{m}_1^b$ is $\mathfrak{d}$-equivariant, i.e.,
\[\mathfrak{d} \circ \mathfrak{m}_1^b = \mathfrak{m}_1^b \circ \mathfrak{d}.\]
\end{lm}

\begin{proof}
For any $x \in \mathcal{C}$, using Lemma~\ref{lm:deformed_operator_explicit_formula} and the fact that $(\mathcal{C}, \mathfrak{m})$ is $\mathfrak{d}$-self-dual combined with the assumption that $|b|=1$:
\begin{align*}
\mathfrak{d} \circ \mathfrak{m}_1^b(x) &= \sum_{\ell_0, \ell_1 \geq 0} \mathfrak{d} \circ \mathfrak{m}_{1+\ell_0+\ell_1}(\underbrace{b, \ldots, b}_{\ell_0}, x, \underbrace{b, \ldots, b}_{\ell_1})\\
&= \sum_{\ell_0, \ell_1 \geq 0} (-1)^{\ell_0+\ell_1} \mathfrak{m}_{1+\ell_0+\ell_1}(\underbrace{-b, \ldots, -b}_{\ell_1}, \mathfrak{d}(x), \underbrace{-b, \ldots, -b}_{\ell_0})\\
&= \mathfrak{m}_1^b(\mathfrak{d}(x)).
\end{align*}
\end{proof}

\begin{cor}\label{cor:preserved_subspace}
The subspace $(\mathcal{C})^{-\mathfrak{d}}$ is preserved by $\mathfrak{m}_1^b$.
\end{cor}

\begin{proof}
If $\mathfrak{d}(\eta) = -\eta$, then by Lemma~\ref{lm:sd_equivariance}:
\[\mathfrak{d}(\mathfrak{m}_1^b(\eta)) = \mathfrak{m}_1^b(\mathfrak{d}(\eta)) = \mathfrak{m}_1^b(-\eta) = -\mathfrak{m}_1^b(\eta).\]
\end{proof}

\begin{cor}\label{cor:primitive_exists}
Let $D:\mathcal{C}\to\mathcal{C}$ be a linear map that commutes with an involution $\mathfrak{d}$. Then
\[
D(\mathcal{C})^{-\mathfrak{d}}=D(\mathcal{C}^{-\mathfrak{d}}).
\]
In particular, this applies to $D=\mathfrak{m}_1^b$.
\end{cor}

\begin{proof}
If $D(x)=\eta$ and $\mathfrak{d}(\eta)=-\eta$, set
\[
y=\frac{x-\mathfrak{d}(x)}2.
\]
Then $\mathfrak{d}(y)=-y$, and since $D\mathfrak{d}=\mathfrak{d}D$,
\[
D(y)=\frac{D(x)-D(\mathfrak{d}(x))}{2}
=\frac{\eta-\mathfrak{d}(\eta)}2
=\eta.
\]
\end{proof}

\subsection{Obstruction theory}\label{ssec:obs_theory}
Let $(\cC, \{\m_k\}_{k\ge 0},\lp\;,\,\rp,\e)$ be an $n$-dimensional curved cyclic unital $A_{\infty}$ algebra with values in $M$, complete with respect to its filtration $\varsigma_1$. Let $\varsigma_2$ be another filtration defined on $\cC$, and assume that $(\cC, \{\m_k\}_{k\ge 0},\lp\;,\,\rp,\e)$ is pseudo-complete with respect to $\varsigma_2$. Recall the definitions of the  modules $\mathfrak{\cC}_E$ from Definition~\ref{dfn:frm}. Assume that the left and right module structures of $\cC[1]$ over $Z(\mathcal{R})$ coincide. 
For $b$ of degree $1$ with $\varsigma_2(b)>0$, the operator $\m^b_1$ is well-defined, has degree $1$, and acts on $\cC$. It follows from the $A_\infty$ relation that for $\a\in \cC$ we have $-\m^b_1\circ\m^b_1(\a)=\m^b_2(\m_0^b,\a)+(-1)^{(|\a|+1)(|\m_0^b|+1)}\m^b_2(\a,\m_0^b).$ Thus, if $\m_0^b=0$, then $\m^b_1\circ\m^b_1=0$, so $\m^b_1$ is a differential. It is therefore natural to consider the \textbf{Maurer-Cartan equation},
\[
\m_0^b=0,\qquad b\in \cC^1, \quad \varsigma_2(b)>0.
\]
 Define the $\left(Z(\mathcal{R})\cap\ker d_{\mathcal{R}}\right)$-module 
\begin{equation}\label{eq:defnote}
\cC_{\not \e}:= \frac{\cC}{\left(Z(\mathcal{R})\cap\ker d_{\mathcal{R}}\right)\cdot e}.
\end{equation}
 
\begin{lm}\label{lm:C_not_e_inf_alg}
        Let $(\cC, \{\m_k\}_{k\ge 0},\lp\;,\,\rp,\e)$ be an $n$-dimensional curved cyclic unital $A_{\infty}$ algebra. Then the induced operators $\m_k:\cC_{\not \e}^{\otimes k}\to \cC_{\not \e}$ given by
        \[\m_k\left([\a_1],\ldots,[\a_k]\right)=[\m_k\left(\a_1,\ldots,\a_k\right)],\]
        are well defined for $k\ne 2.$
\end{lm}
\begin{proof}
    This follows from the properties of a cyclic $A_\infty$ algebra given in Definition~\ref{dfn:cycunit}. 
\end{proof}
\begin{dfn}\label{dfn:bounding_cochain}
    Let $(\cC, \{\m_k\}_{k\ge 0},\lp\;,\,\rp,\e)$ be an $n$-dimensional curved cyclic unital $A_{\infty}$ algebra. An element $b \in \cC$ of degree $1$ with $\varsigma_2(b)>0$ is a \textbf{bounding cochain} if the image of $\m_0^b$ in the quotient modules $\cC_{\not \e}$ and $\overline{\cC}$ vanishes. 
\end{dfn}
\begin{rem}
    In the literature the term bounding cochain is sometimes reserved for a solution of the Maurer-Cartan equation in $\cC$ and the definition above corresponds to weak bounding cochains. 
    For $b\in \cC^1$ a bounding cochain, we do not assume that $\varsigma_1(b)> 0$. So, the condition that $b$ solves the Maurer-Cartan equation in the quotient module $\overline{\cC}$ is not automatic. 
    If the filtration $\varsigma_1$ is trivial, then the requirement for a bounding cochain to solve the Maurer-Cartan equation in the quotient module $\cC_{\not \e}$ is trivial.
\end{rem}
Let $\mathfrak{d}: \cC \rightarrow \cC$ be an involution, and assume that $(\cC, \mathfrak{m})$ is $\mathfrak{d}$ self-dual. Let $\invs$ denote $-\mathfrak{d}$ or the identity map on $\cC$.
Let $b \in \cC^\invs$ of degree $1$ with $\varsigma_2(b)>0$ be a solution to the Maurer-Cartan equation in the quotient module $\overline{\cC}$. Recall the definitions of $\m^b_{k,E}$ from Section~\ref{ssec:Cyc_A_infty_alg}.
\begin{lm}\label{lm:bounding_chain_bar_indu_boundary_op}
    Let $b\in \cC^\invs$ be such that $\bar{b}$ is a bounding cochain for the induced $A_\infty$ algebra $(\overline{\cC},\{\bar{\m}_k\}_{k\ge 0})$. Then, the induced operator $\m^b_{1,E}$ on $\cC_{E}^\invs$ is a differential.
\end{lm}
\begin{proof}
    By the assumption $\overline{\m_0^b}=0.$ Let $\alpha$ be an element of $F^{E}\cC$ representing an element of $\cC_{E}^\invs.$ It follows from properties~\ref{it:lin},\ref{it:a_infty},\ref{it:unit2} of the cyclic $A_\infty$ algebra that
    \[
        -\m_1^b(\m_1^b(\a))=\m^b_2(\m_0^b,\a)+(-1)^{(|\a|+1)}\m^b_2(\a,\m_0^b).
\]
    Observe that $\varsigma_1(\m_0^b)>0$. So it follows from property~\ref{it:val} of the cyclic $A_\infty$ structure of $\cC$ that 
    \[\varsigma_1\left(\m^b_2(\m_0^b,\a)+(-1)^{(|\a|+1)}\m^b_2(\a,\m_0^b)\right)>E.\] 
    The claim now follows from the definition of $\cC_E,$ when $\invs$ is the identity. When $\invs=-\mathfrak{d}$, we also need to show that $\m_1^b$ preserves $\cC_E^{\invs}$, and this follows from Corollary~\ref{cor:preserved_subspace}.
\end{proof}
Observe that property~\ref{it:val} implies that the induced operators $\{\m_k\}_{k\ge 0}$ on the quotient module $\frac{\mathcal{C}}{F^{>E_{\ell}}\mathcal{C}}$ are well defined, and they define an $A_\infty$ algebra on $\frac{\mathcal{C}}{F^{>E_{\ell}}\mathcal{C}}$. Property~\ref{it:unit2} combined with property~\ref{it:val} of Definition~\ref{dfn:cycunit} implies that the induced $A_\infty$ algebra on $\frac{\mathcal{C}}{F^{>E_{\ell}}\mathcal{C}}$ is cyclic unital with values in $M/F^{>E_{\ell}}M$.
Assume that $\im \varsigma_1$ can be ordered
$\{E_0<E_1<E_2 <\ldots\}\subset \R$.
Let $b\in \cC^\invs$ be a bounding cochain for the induced $A_\infty$ algebra $\frac{\mathcal{C}}{F^{>E_{\ell}}\mathcal{C}}$. It follows from the discreteness of $\im\varsigma_1$ and Lemma~\ref{lm:sd_m_zero} that $\m_0^b\in F^{ E_{\ell+1}}\cC^\invs_{\not \e}.$ Let $\m^b_{0,E_{\ell+1}} \in \cC^\invs_{\not \e,E_{\ell+1}}$ denote the image of $\m_0^b.$
\begin{lm}\label{lm:o_is_class_HCe}
We have $\m_{1,E_{\ell+1}}^b(\m^b_{0,E_{\ell+1}})=0.$ 
\end{lm}
\begin{proof}
    It follows from the $A_\infty$ relations that $\m_1^b(\m_0^b)=0.$ So, by the definition of the induced operator $\m_{1,E_{\ell+1}}^b(\m^b_{0,E_{\ell+1}})=0.$ 
\end{proof}
Let $b\in \cC^{\invs}$ be a bounding cochain for the induced $A_\infty$ algebra $\frac{\mathcal{C}}{F^{>E_{\ell}}\mathcal{C}}$. Abbreviate by $\mathfrak{o}$ the class of $\m_{0,E_{\ell+1}}^b$ in $H^2\left(\cC_{\not \e, E_{\ell+1}}^{\invs},\m^b_1\right)$.
\begin{lm}\label{lm:obs_induction}
    Suppose $\mathfrak{o}=0$. Then there exists $\xi \in F^{E_{\ell+1}}\cC^{\invs}$ of degree $1$ such that if we set 
\[
b' = b + \xi,\]
then $b'\in \cC^{\invs}$ is a bounding cochain for the induced $A_\infty$ algebra $\frac{\cC}{F^{>E_{\ell+1}}\cC}$.
\end{lm}
\begin{proof}
Since $\mathfrak{o} = 0,$ we can choose $\xi \in F^{E_{\ell+1}}\cC^{\invs}$ such that $\m^b_1(\xi) \equiv -\m_0^b \pmod{F^{>E_{\ell+1}}\cC_{\not\e}}.$ It follows from Property~\ref{it:val} of Definition~\ref{dfn:cycunit} that
\[
\m^{b + \xi}_0 \equiv \m^{b}_0 + \m^{b}_1(\xi) \equiv 0 \pmod{F^{>E_{\ell+1}}\cC_{\not\e}}.
\]
\end{proof}
\begin{lm}\label{lm:twisted_f_chain_iso}
    Let $(\cC, \{\m_k\}_{k\ge 0},\lp\;,\,\rp,\e)$ be a cyclic unital $A_{\infty}$ algebra over $\mathcal{R}$ with values in $M$ such that the residue algebra $(\bar{\cC}, \{\bar{\m}_k\}_{k\ge 0},\lp\;,\,\rp,\e)$ arises from a dga.   Assume that $\cC$ has an $\bar{\mathcal{R}}$ structure given by $g:\mathcal{R}\otimes_{\overline{\mathcal{R}}} \bar{\cC} \to \cC$ and let $f$ be its inverse. Let $b\in\cC$ be such that $\bar{b}$ is a bounding cochain for the induced $A_\infty$ algebra $(\overline{\cC},\{\bar{\m}_k\}_{k\ge 0})$. Then, $f$ induces an isomorphism of cochain complexes between $\left(\mathcal{C}_E,\m_{1,E}^b\right)$ and $\left(\mathcal{R}_E\otimes_{\overline{\mathcal{R}}} \overline{\cC},d_{\mathcal{R}}\otimes \Id+\Id\otimes (\bar{\m}_1)^{\bar{b}}\right).$
\end{lm}
\begin{proof}
It follows from the definition of an $\bar{\mathcal{R}}$ structure, Definition~\ref{dfn:bar_structre}, that $f$ induces an isomorphism of $\bar{\mathcal{R}}$-modules 
\[
\mathcal{C}_E \to \mathcal{R}_E\otimes_{\overline{\mathcal{R}}} \overline{\cC}.
\]
Since $(\bar{\cC},\{\bar\m_k\}_{k\ge0})$ arises from a dga, Proposition~\ref{prop:r_structure_energy_dga} and Lemma~\ref{lm:deformed_operator_explicit_formula} give
\[
f\circ\m^b_{1,E}
=\left(d_{\mathcal R}\otimes\Id+\Id\otimes(\bar\m_1)^{\bar b}\right)\circ f.
\]
Hence $f$ is a chain map, and the result follows.
\end{proof}

\subsection{The ring \texorpdfstring{$\mathbb{R} \langle s \rangle $}{R<s>}}\label{ssec:the_ring_S}
All $(S,F)$ extensions used in the main construction of the present paper arise from the following choices of $S$ and $F,$  the basic properties of which we establish here.
Define $S:= \mathbb{R}\langle s  \rangle$ to be the ring of non-graded-commutative polynomials in one variable where the grading on $S$ is given by $|s|=1-n$. Define a filtration $\nu_S:S\to \R_{\ge 0}\cup\{\infty \}$ by 
\[\nu_S\left(\sum_{i=0}^m a_i s^i\right)=\inf_{a_i\ne 0}-|s^i|=\inf_{a_i\ne 0}i(n-1).\]
Recall that $n>1$, so this is not the trivial filtration.

Define a $\Z_2$ grading on $S$ by, for homogeneous elements $a \in S,$  
\[
\degev{a}=|a| \mod 2.
\]
Let $\Upsilon$ be a graded commutative algebra, and write $\hat{\Upsilon}:=S\otimes\Upsilon$. We equip $\hat{\Upsilon}$ with the $\mathbb{Z}_2$ grading induced by the $\mathbb{Z}_2$ grading of $S,$ and we denote it also by $\degev{\cdot}$. For example if $s$ is a homogeneous element in $S$, then for any $r \in \Upsilon$ we have
$$\degev{s \otimes r}\equiv |s| \mod 2.$$

\begin{lm}\label{lm:sign_multipication_order}
 Let $s_1, s_2 \in S$ and $\alpha, \xi \in \Upsilon$ homogeneous elements. Then the multiplication on $\hat{\Upsilon}$ satisfies,
\[(s_1 \otimes \alpha ) (s_2 \otimes \xi)=(-1)^{(|s_1\otimes \alpha|)(|s_2\otimes \xi|)+\degev{s_1}\degev{s_2}}(s_2 \otimes \xi )( s_1 \otimes \alpha).\]
\end{lm}
\begin{proof}
We compute the product of the elements $s_1 \otimes \alpha, s_2 \otimes \xi \in S \otimes \Upsilon$,
\[(s_1 \otimes \alpha) (s_2 \otimes \xi)=(-1)^{|\alpha||s_2|}s_1 s_2 \otimes(\alpha  \xi)=(-1)^{|\alpha||s_2|+|\xi||\alpha|}s_1 s_2\otimes(\xi  \alpha).\]
If we change the order of multiplication, we get
\[(s_2 \otimes \xi )( s_1 \otimes \alpha)=(-1)^{|\xi||s_1|}s_2  s_1\otimes(\xi  \alpha).\]
Remembering that in $S$ we have $s_1 s_2=s_2 s_1$, we conclude that
\[(s_1 \otimes \alpha ) (s_2 \otimes \xi)=(-1)^{(|s_1\otimes \alpha|)(|s_2\otimes \xi|)+\degev{s_1}\degev{s_2}}(s_2 \otimes \xi )( s_1 \otimes \alpha).\]
\end{proof}
\begin{cor}\label{cor:center_of_hat_algebra}
For $\alpha_1,\alpha_2 \in \hat{\Upsilon}$
\begin{equation}\label{eq:bracket_C_hat}
    [\alpha_1,\alpha_2]=\begin{cases}
    0, & \degev{\alpha_1}\cdot\degev{\alpha_2}=0\\
    2\alpha_1 \cdot\alpha_2, & otherwise.
    \end{cases}
\end{equation}
It follows that the center of $\hat{\Upsilon}$ is given by
\[Z(\hat{\Upsilon})=\{x\in \hat{\Upsilon}\big| \;\degev{x}=0 \}.\]
\end{cor}
We define the $\R$ module homomorphism $F:S \to  S$ to be the projection to the odd part. Concretely, the morphism is given on homogeneous elements by
$$F(a s^i)= \begin{cases}
0, & i \equiv 0 \mod{2},\\
a s^i, & otherwise
.\end{cases}$$
\begin{lm}\label{lm:gradedcyclic}
The homomorphism $F:S\to  S$ is a graded cyclic $Z(S)$ module homomorphism.
\end{lm}
\begin{proof}
Let $s_1,\dots, s_n \in S$ be homogeneous elements. If $s_1 \cdots s_n$ is an odd element in $S$ then $\degev{s_1 \cdots s_{n-1}}\degev{s_n}=0$. Thus 
\[F(s_1 \cdots s_n)=s_1 \cdots s_n =s_n \cdot s_1 \cdots s_{n-1}=(-1)^{|s_n|\sum_{i=1}^{n-1}|s_i|}F(s_n \cdot s_1 \cdots s_{n-1}).\]
If $s_1 \cdots s_n$ is an even element in $S$ then the equality above also holds. Hence  $F:S\to  S$ is a graded cyclic homomorphism.
We use Corollary~\ref{cor:center_of_hat_algebra} to deduce that $F:S \to S$ is a $Z(S)-\operatorname{mod}$ homomorphism.
\end{proof}

\section{Fukaya \texorpdfstring{$A_\infty$}{A-infinity} algebras}\label{Fukaya algebras}
In this section we recall definitions and results concerning Fukaya $A_\infty$-algebras from ~\cite{solomon2016differential}. The notation, as well as sign and orientation conventions, are the same as in~\cite{solomon2016differential}. Then, using the results of Section~\ref{ssec:extension_of_scalar}, we extend scalars by the non-commutative ring $S = \mathbb{R}\langle s\rangle$ in Section~\ref{ssec:fukaya_scalar_ext}.
In Section~\ref{sssec:sababa_property}, we restrict scalars to obtain a sababa $A_\infty$ algebra so that the scalar extension by $S$ becomes pseudo-complete (Proposition~\ref{prop:saba_decend_psedu_comp}). This pseudo-completeness is the key property that allows us in Section~\ref{ssec:fukaya_bounding_gauge} to use the notions of bounding cochains developed in Section~\ref{ssec:obs_theory} in the setting of Fukaya $A_\infty$-algebras. When an anti-symplectic involution is present, Section~\ref{ssec:real} establishes the self-duality properties required for the real case. Section~\ref{ssec:discrete_filtration} then introduces the discrete filtration and finite energy quotients, which provide the technical framework needed to adapt the abstract obstruction theory from Section~\ref{ssec:obs_theory} to the Fukaya $A_\infty$ context of Section~\ref{sec:bd_chains}.

\subsection{Construction}\label{ssec:constr}
Throughout this work $(X,\omega)$ is a symplectic manifold, $L\subset X$ is a connected Lagrangian submanifold with relative spin structure $\s$, and $J$ is an $\omega$-tame almost complex structure. Let $\dim_\R X=2n$.
Write $\beta_0:=0\in \sly$.

A $J$-holomorphic genus-$0$ open stable map to $(X,L)$ of degree $\beta \in \sly$ with one boundary component, $k+1$ boundary marked points, and $l$ interior marked points, is a quadruple $(\Sigma, u,\vec{z},\vec{w})$ as follows. The domain $\Sigma$ is a genus-$0$ nodal Riemann surface with boundary consisting of one connected component. The map
\[
u: (\Sigma,\d\Sigma) \to (X,L)
\]
is continuous, $J$-holomorphic on each irreducible component of $\Sigma,$ and $[u_*([\Sigma,\partial\Sigma])] = \beta.$
The marked points are denoted by
\[
\vec{z} = (z_0,\ldots,z_k), \qquad \vec{w} = (w_1,\ldots,w_l),
\]
with $z_j \in \partial \Sigma, \, w_j \in int(\Sigma),$ distinct. The labeling of the marked points $z_j$ respects the cyclic order given by the orientation of $\partial \Sigma$ induced by the complex orientation of $\Sigma.$ Stability means that if $\Sigma_i$ is an irreducible component of $\Sigma,$ then either $u|_{\Sigma_i}$ is non-constant, or the number of marked points and nodal points on $\Sigma_i$ is at least $3,$ with interior points counted twice if $\Sigma_i$ is a disk component. An isomorphism of open stable maps $(\Sigma,u,\vec{z},\vec{w})$ and $(\Sigma',u',\vec{z}',\vec{w}')$ is a homeomorphism $\theta : \Sigma \to \Sigma'$, biholomorphic on each irreducible component, such that
\[
u = u' \circ \theta, \qquad\qquad  z_j' = \theta(z_j), \quad j = 0,\ldots,k, \qquad w_j' = \theta(w_j), \quad j = 1,\ldots,l.
\]

Let $\M_{k+1,l}(\beta) = \M_{k+1,l}(\beta;J)$ denote the moduli space of $J$-holomorphic genus zero open stable maps to $(X,L)$ of degree $\beta$ with one boundary component, $k+1$ boundary marked points, and $l$ internal marked points.
Let the evaluation maps
\begin{gather*}
evb_j^\beta:\M_{k+1,l}(\beta)\to L, \qquad  \qquad j=0,\ldots,k, \\
evi_j^\beta:\M_{k+1,l}(\beta) \to X, \qquad \qquad j=1,\ldots,l,
\end{gather*}
be given by $evb_j^\beta((\Sigma,u,\vec{z},\vec{w}))=u(z_j)$ and $evi_j^\beta((\Sigma,u,\vec{z},\vec{w}))= u(w_j).$
We may omit the superscript $\beta$ when the omission does not create ambiguity.

We assume that $\M_{k+1,l}(\beta)$ is a smooth orbifold with corners and that $evb_0^{\beta}$ is a proper submersion.

Let $\gamma\in \mI_QD$ be a closed form with $|\gamma|=2$. For example, given closed differential forms $\gamma_j\in A^*(X,L)$ for $j=0,\ldots,N,$ take $t_j$ of degree $2-|\gamma_j|$ and $\gamma:=\sum_{j=0}^Nt_j\gamma_j$.
Define
\[
\mg_k:C^{\otimes k} \lrarr C
\]
by
\begin{equation}\label{eq:fukaya_operator}
\begin{aligned}
\mg_k(\alpha_1,\ldots,\alpha_k)
:={}&\delta_{k,1}\cdot d\alpha_1+(-1)^{\varepsilon(\alpha)}
\sum_{\substack{\beta\in\sly\\l\ge0}}T^{\beta}\frac{1}{l!}{evb_0^\beta}_*\left(
\bigwedge_{j=1}^l(evi_j^\beta)^*\gamma\wedge
\bigwedge_{j=1}^k (evb_j^\beta)^*\alpha_j
\right)
\end{aligned}
\end{equation}
with
\[
\varepsilon(\alpha)
:=\sum_{j=1}^kj(|\alpha_j|+1)+1.
\]
The push-forward $(evb_0)_*$ is defined by integration over the fiber; it is well-defined because $evb_0$ is a proper submersion. 

Denote by $\langle\;,\;\rangle$ the signed Poincar\'e pairing
\begin{equation}\label{eq:pairing}
\langle\xi,\eta\rangle:=(-1)^{|\eta|}\int_L\xi\wedge\eta.
\end{equation}
Equip $C$ with the diagonal bimodule structure over $R$ with respect to grading of~$C[1].$
\begin{thm}[{\cite[Theorem 1]{solomon2016differential}}]\label{thm:solomon}
    $(C,\{\mg_k\}_{k\ge 0},\langle\;,\,\rangle,1)$ is a cyclic unital $A_\infty$ algebra over $R$ with values in $R$ in the sense of Definition~\ref{dfn:cycunit}.
\end{thm}
For $\alpha \in C$ and $k \in \N$ recall from Section~\ref{ssec:gen_not} the notation $(\alpha)_k$ for the component of differential form degree~$k$. 
\begin{cl}[Top degree {\cite[Proposition 3.12]{solomon2016differential}}]\label{no_top_deg}
    Let $\a_1,\dots \a_k\in C$ then
\[\left(\mg_k(\a_1,\dots,\a_k)\right)_n=\begin{cases}
   - \left(\gamma|_{L}\right)_n,\qquad& k=0,\\
    \left(d\a_1\right)_n,\qquad& k=1,\\
    \left((-1)^{|\a_1|}\a_1\wedge \a_2\right)_n, \qquad& k=2,\\
    0,\qquad& otherwise.
\end{cases}\]
\end{cl}
 In light of Definition~\ref{dfn:A_infty_arises_dga}, we can reformulate~\cite[Theorem 3-(3)]{solomon2016differential} as follows.
\begin{prop}\label{prop:fukaya_infity_arises_dga}
The $A_\infty$ algebra $(\overline{C},\{\bar{\m}^{\gamma}_k\}_{k\ge 0},\langle\;,\,\rangle,1)$  arises from the dga of differential forms $(A^*(L), d)$.
\end{prop}
\subsection{Pseudoisotopies}\label{pseudoisot}
Set
\begin{equation}\label{eq:mR}
\mC:=A^*(I\times L;R), \qquad \mD:=A^*(I\times X,I\times L;Q), \qquad \mR := A^*(I;R).
\end{equation}
The filtration $\nu$ induces filtrations on $\mR, \mathfrak{C}$, and $\mathfrak{D}$, which we still denote by $\nu$. For $t \in[0,1]$ and letting $M$ be either $L$ or the point, denote by
$$
j_t: M \rightarrow[0,1] \times M
$$
the inclusion $j_t(p)=(t, p)$.
We construct a family of $A_\infty$ structures on $\mC$. Let $\{J_t\}_{t\in I}$ be a path in $\J$ from $J = J_0$ to $J'=J_1.$
For each $\beta, k, l,$ set
\[
\Mt_{k+1,l}(\beta):=\{(t,u,\vec{z},\vec{w})\,|\, (u,\vec{z},\vec{w})\in\M_{k+1,l}(\beta;J_t)\}.
\]
We have evaluation maps
\begin{gather*}
\evbt_j:\Mt_{k+1,l}(\beta)\lrarr I\times L, \quad j\in\{0,\ldots,k\},\\
\evbt_j(t,u,\vec{z},\vec{w}):=(t,u(z_j)),
\end{gather*}
and
\begin{gather*}
\evit_j:\Mt_{k+1,l}(\beta)\lrarr I\times X, \quad j\in\{1,\ldots,l\}\\
\evit_j(t,u,\vec{z},\vec{w}):=(t,u(w_j)).
\end{gather*}
It follows from the assumption on $\J$ that all $\Mt_{k+1,l}(\beta)$ are smooth orbifolds with corners, and $\evbt_0$ is a proper submersion.
\begin{ex}
In the special case when $J_t=J = J'$ for all $t\in I$, we have $\Mt_{k+1,l}(\beta)=I\times\M_{k+1,l}(\beta;J)$. The evaluation maps in this case are $\evbt_j=\Id \times evb_j$ and $\evit_j=\Id\times evi_j.$
\end{ex}

Let
\[
p:I\times L\lrarr I,\qquad p_\M: \Mt_{k+1,l}(\beta)\lrarr I
\]
denote the projections.

For each closed $\gt\in \mI_Q\mD$ with $|\gt|=2,$ define structure maps
\[
\mgt_k:\mC^{\otimes k}\lrarr \mC
\]
by
\[
\mgt_k(\at_1,\ldots,\at_k):=\delta_{k,1}\cdot d\at_1+(-1)^{\varepsilon(\at)}
\sum_{\substack{\beta\in\sly\\l\ge0}}T^{\beta}\frac{1}{l!}\left({{\evbt}_{0}^{\beta}}\right)_* (\bigwedge_{j=1}^l(\evit_j^\beta)^*\gt\wedge
\bigwedge_{j=1}^k (\evbt_j^\beta)^*\at_j
).
\]

Define a pairing on $\mC$:
\[
\ll\,,\,\gg:\mC\otimes\mC\lrarr \mathfrak{R},\qquad
\ll\tilde{\xi},\tilde{\eta}\gg:=(-1)^{|\etat|}p_*(\tilde{\xi}\wedge\tilde{\eta}).
\]
Note that
\begin{equation}\label{eq:pseudopair}
\ll\xit,\etat\gg=(-1)^{|\etat|}p_*(\xit\wedge\etat)
=(-1)^{|\etat|+|\etat|\cdot|\xit|}p_*(\etat\wedge\xit)
=(-1)^{(|\etat|+1)(|\xit|+1)+1}\ll\etat,\xit\gg.
\end{equation}
Equip $\mC$ with the diagonal bimodule structure over $\mR$ with respect to grading of~$\mC[1].$
In the proof of \cite[Theorem 2]{solomon2016differential}, the following theorem is proved.
\begin{thm}\label{thm:mt_cyclic_A_infty}
     $(\mC,\{\mgt_k\}_{k\ge 0},\ll\;,\,\gg,1)$ is a cyclic unital $A_\infty$ algebra over $\mR$ with values in $\mR$ for each closed $\gt\in \mI_Q\mD$ with $|\gt|=2.$
\end{thm}

\begin{dfn}
Let $\mathcal A_0=(\m,\lp\;,\,\rp,\e)$ and $\mathcal A_1=(\m',\lp\;,\,\rp',\e')$ be cyclic unital $A_\infty$ structures over $R$ with values in $R$ on $C$.
A cyclic unital \textbf{pseudoisotopy} from $\mathcal A_0$ to $\mathcal A_1$ is a cyclic unital $A_\infty$ structure $(\mt,\lpt\;,\rpt,\tilde{\e})$ over $\mathfrak R$ with values in $\mathfrak R$ on $\mC$ such that
 for all $\at_j\in \mC$ and all $k\ge 0$,
\begin{gather*}
j_0^*\tilde{\m}_k(\at_1,\ldots,\at_k)=\m_k(j_0^*\at_1,\ldots,j_0^*\at_k),\\
j_1^*\tilde{\m}_k(\at_1,\ldots,\at_k)=\m'_k(j_1^*\at_1,\ldots,j_1^*\at_k),
\end{gather*}
and
\begin{gather*}
j_0^*\lpt\at_1,\at_2\rpt=\lp j_0^*\at_1,j_0^*\at_2\rp,\qquad j_0^*\tilde{\e}=\e,\\
j_1^*\lpt\at_1,\at_2\rpt=\lp j_1^*\at_1,j_1^*\at_2\rp',\qquad j_1^*\tilde{\e}=\e'.
\end{gather*}
\end{dfn}
Write $\mgp$ for the operations defined in Section~\ref{ssec:constr} using the almost complex structure $J'$ and a closed form $\gamma' \in \mI_QD$ with $|\gamma'|=2.$
\begin{thm}[{\cite[Theorem 2]{solomon2016differential}}]\label{thm:pseduoisotopy}
    Assume $[\gamma]=\left[\gamma^{\prime}\right] \in \Hh^*(X, L ; Q)$. Let $\xi\in (\mI_QD)_1$ with $(i^*\xi)_0 = 0$
be such that
\[
\gamma'-\gamma=d\xi
\]
and define
\begin{equation}\label{eq:gt_psedu_def}
    \gt:=\gamma+t(\gamma'-\gamma)+dt\wedge\xi\in (\mI_Q\mD)_2.
\end{equation}
Then, $(\mgt, \ll\,,\,\gg, 1)$ is a  pseudoisotopy from $\left(\mg,\langle\,,\,\rangle, 1\right)$ to $\left(\mgp,\langle\,,\,\rangle, 1\right)$.
\end{thm}
\begin{cl}[Top degree {\cite[Proposition 4.17]{solomon2016differential}}]\label{cl:qt_no_top_deg}
Let $\at_1,\dots,\at_k\in \mC$. Then
\[\left(\mgt_k(\at_1,\dots,\at_k)\right)_{n+1}=\begin{cases}
    -\left(\gt|_{I\times L}\right)_{n+1},\qquad& k=0,\\
    \left(d\at_1\right)_{n+1},\qquad& k=1,\\
    \left((-1)^{|\at_1|}\at_1\wedge \at_2\right)_{n+1}, \qquad& k=2,\\
    0,\qquad& otherwise.
\end{cases}\]
\end{cl}
\begin{prop}[{\cite[Proposition 4.15]{solomon2016differential}}]\label{prop:fukaya_infity_arises_dga_dt}
The $A_\infty$ algebra $(\overline{\mathfrak{C}},\{\bar{\mt}^{\gt}_k\}_{k\ge 0},\ll\;,\,\gg,1)$ arises from the dga of differential forms  $(A^*(I \times L),d)$.
\end{prop}
\begin{lm}[{\cite[Lemma 4.9]{solomon2016differential}}]\label{lm:d_ll_gg}
For any $\xit,\etat\in A^*(I\times L),$
\[
(-1)^{\deg \xit+\deg \etat +n}
\int_I d\ll\xit,\etat\gg=\langle j_1^*\xit,j_1^*\etat\rangle- \langle j_0^*\xit,j_0^*\etat\rangle.
\]
\end{lm}

\subsection{Extension of scalars}\label{ssec:fukaya_scalar_ext}
\subsubsection{Definitions}\label{sssec:fukaya_scalar_ext_dfn}
Recall the definitions of $R, C,\mR,\mC$ from Section~\ref{subsubsec:a_infty}, and~\eqref{eq:mR}. Also recall the Definition~\ref{dfn:bar_structre} of an $\bar{\mathcal{N}}$ structure on a filtered graded-commutative unital dga $\mathcal{N}$, and an $\bar{\mathcal{N}}$ structure on a filtered graded $\mathcal{N}$-module $\cC$. In order to define an $\bar{R}$ structure on $R$, observe that we have an isomorphism $\R \overset{\sim}\to \bar{R}$ given by $a\mapsto \overline{aT^{\beta_0}}$. As $R$ is a unital dga over $\R$ with zero differential, there exists a canonical dga map $k:\R \to R$. It is straightforward to check that 
$$\bar{R}\simeq \R \overset{k}{\rightarrow} R$$ 
defines a $\bar{R}$ structure on $R$. 
We have an isomorphism $A^*(L) \overset{\sim}\to \overline{C}$ given by the inclusion, and the composition 
\[R\otimes_{\bar{R}}\overline{C}\simeq R \otimes_{\R}A^* (L)\simeq C \]
defines a $\bar{R}$ structure 
\[
g:R\otimes_{\bar{R}}\overline{C}\to C.
\]
We may also think of $(\mC,\{\mgt_k\}_{k\ge 0},\langle\langle\;,\,\rangle\rangle,1)$ as a curved cyclic $A_\infty$ algebra over $R$ with values in $\mR$, by restriction of scalars along $R\to\mR$. Thus, with the same $\bar R$ structure on $R$, we define the $\bar{R}$ structure 
\[
g:R\otimes_{\bar{R}} \overline{\mC}\to \mC
\]
by the composition
\[R\otimes_{\bar{R}} \overline{\mC}\simeq R \otimes_{\R} A^*(I\times L)\simeq \mC.  \]
We abbreviate $f = g^{-1}$ in both cases.
\subsubsection{Extension of scalars}\label{sssec:Extension_scalars_fukaya}
Recall the Definition~\ref{dfn:sfext} concerning $(S,F)$ extensions and take $S$ and $F$ as in Section~\ref{ssec:the_ring_S}. Recall the definition of the cyclic unital $A_\infty$ algebra $(\hat{C},\{\mgh_k\}_{k\ge 0},\langle\;,\,\rangle_F,1\otimes 1)$ from Section~\ref{subsubsec:a_infty} and observe it is the $(S,F)$ extension of $(C,\{\mg_k\}_{k\ge 0},\langle\;,\,\rangle,1)$ from Theorem~\ref{thm:solomon}. Similarly, define $(\hat{C},\{\mh_k^{\gamma'}\}_{k\ge 0},\langle\;,\,\rangle_F,1\otimes 1)$ to be the $(S,F)$ extension of $(C,\{\mgp_k\}_{k\ge 0},\langle\;,\,\rangle,1)$ from Theorem~\ref{thm:pseduoisotopy}, and define $(\hat{\mC},\{ \hmt_k^{\gt} \}_{k\ge 0},\ll\;,\,\gg_F,1\otimes 1)$ to be the $(S,F)$ extension of $(\mC,\{ \mt_k^{\gt} \}_{k\ge 0},\ll\;,\,\gg,1)$ from Theorem~\ref{thm:mt_cyclic_A_infty}. Then, the following is a corollary of Proposition~\ref{prop:a_infity_extension}.
\begin{prop}
$(\hat{C},\{\mgh_k\}_{k\ge 0},\langle\;,\,\rangle_F,1\otimes 1)$ and $(\hat{\mC},\{ \hmt_k^{\gt} \}_{k\ge 0},\ll\;,\,\gg_F,1\otimes 1)$ are $n$-dimensional cyclic unital $A_\infty$ algebras over $\Rh$ and $\hat{\mR}$, with values in $\Rh$ and $\hat{\mR}$, respectively. 
\end{prop}
The following is a corollary of Proposition~\ref{prop:a_infity_extension} and Theorem~\ref{thm:pseduoisotopy}.
\begin{lm}\label{lm:pseudo}
    $(\hat{\mC},\{ \hmt_k^{\gt} \}_{k\ge 0},\ll\;,\,\gg_F,1\otimes 1)$ is a pseudoisotopy from $(\hat{C},\{\mgh_k\}_{k\ge 0},\langle\;,\,\rangle_F,1\otimes 1)$ to $(\hat{C},\{\mh_k^{\gamma'}\}_{k\ge 0},\langle\;,\,\rangle_F,1\otimes 1).$
\end{lm}

The $\bar R$-structures on $C$ and $\mC$ induce $\bar R$-structures
on their $(S,F)$ extensions $\hat C$ and $\hat{\mC}$, as in
Definition~\ref{dfn:induced_R_structure}.
We denote each induced structure map by $\hat g$ and its inverse
by $\hat f:=\hat g^{-1}$.

\subsubsection{Sababa property}\label{sssec:sababa_property}

Let $\Theta$ be a compact metric space and $\Theta \to \mathcal{J}(X,\omega):t \mapsto J_t$ be a continuous function with respect to the $C^\infty$-topology on $\mathcal{J}(X,\omega)$.
By Gromov compactness, for a fixed $E \in \mathbb{R}$ the set
\[
\sly_E:=\{\beta\;|\;\omega(\beta)\le E,\,\exists t \in \Theta :\M_{3,0}(\beta,J_t)\ne \emptyset\}
\]
is finite. Thus we can order $\sly_\infty = \cup_E \sly_E$ as a list $\beta_0,\beta_1,\ldots,$ where $i<j$ implies $\omega(\beta_i)\le \omega(\beta_j).$ As above $\beta_0$ denotes the zero element. Recall Definition \ref{dfn:sababa_algebra} concerning the sababa property. Let $\Rd \subset R$ be a sababa unital differential graded subalgebra of $R$, and assume also that $\{T^{\beta}\}_{\beta \in \sly_\infty}, \{t_i\}_{i=0}^N \subset \Rd$. Since $\bar{R}$ is isomorphic to $\R$, it follows that $\bar{\Rd}$ is also isomorphic to $\R$. Recall Definition~\ref{dfn:restriction_of_scalars}, and define $\Cd:=C_{\Rd}$ and $\mCd:=\mC_{\Rd}$ to be the $\Rd$ restrictions of scalars of $C$ and $\mC$, respectively. Their canonical $\bar{\Rd}$ structure maps from that definition are
\[
g_{\Rd}:\Rd\otimes_{\bar{\Rd}}\overline{\Cd}\longrightarrow\Cd,
\qquad
g_{\Rd}:\Rd\otimes_{\bar{\Rd}}\overline{\mCd}\longrightarrow\mCd.
\]
By abuse of notation, we omit the subscript $\Rd$ and write $g$
for either map. We denote the induced structure maps on $\Cdh$
and $\mCdh$ by $\hat g$, and write $f:=g^{-1}$ and
$\hat f:=\hat g^{-1}$ for the respective inverses.
\begin{lm}\label{lm:R_daimond_sababa_A_infty}
$(\Cd,\{\mg_k\}_{k\ge 0},\langle\;,\,\rangle,1)$ (resp. $(\mCd,\{\mgt_k\}_{k \ge 0},\ll\;,\,\gg,1)$) is a sababa cyclic unital $A_\infty$ algebra over $\Rd$ with values in $\Rd$ (resp. $\Rd\otimes A^*(I)$).
\end{lm}
\begin{proof}
Since, $\{T^{\beta}\}_{\beta \in \sly_\infty}, \{t_i\}_{i=0}^N \subset \Rd$ we have $\mg_k({\Cd}^{\otimes k}) \subset \Cd$. Thus $(\Cd,\{\mg_k\}_{k\ge 0},\langle\;,\,\rangle,1)$ is a cyclic unital $A_\infty$ algebra. We observe that the degrees of elements in $\overline{C}=A^*(L)$ are bounded by $n$. Thus, we have
\[|\alpha|\le n.\]
In particular, $n$ is a sababa constant for $\Cd$. The proof for $(\mC,\{\mgt_k\}_{k \ge 0},\ll\;,\,\gg,1)$ is the same but with sababa constant $n+1.$
\end{proof}
Recall the definition of $\nu$ from Section~\ref{ssec:gen_not}, and also denote by $\nu$ the induced filtration on $\Ch$ and $\hat{\mC}$. Let $\Ups =\Ch$ or $\Ups = \Rh$. Decompose $\Ups$ as 
$\Ups = S \otimes_{\R} R \otimes_{\R} \Ups'$ where $\Ups'= A^*(L)$ or $\Ups'= \R$, respectively. Let $\alpha_1,\dots,\alpha_n \in \hat{C}.$  For $i = 1,\ldots,n,$ write
\begin{equation}\label{eq:n_decomp}
\alpha_i =\sum_{j=0}^\infty s^{k_{ij}} \lambda_{ij} \alpha_{ij},\qquad \alpha_{ij}\in \Ups',\quad\lambda_{ij}=T^{\beta_{ij}}\prod_{a=0}^Nt_a^{l_{aij}},\quad \lim_{j\to\infty}\nu(\lambda_{ij})=\infty.
\end{equation}
\begin{notn}\label{notn:n_decomp}
Denote by $R(\alpha_1,\dots,\alpha_n)$ the sababa subalgebra of $R$ generated by
\[
\{\lambda_{ij}\}_{i \in [n], \,j\in \N},\;\{T^{\beta}\}_{\beta \in \sly_\infty}, \{t_i\}_{i=0}^N. \]
\end{notn}

\subsubsection{Pseudo-completeness}
Recall the definition of $\nu_S$ from Section~\ref{ssec:the_ring_S}.
\begin{cor}\label{cor:fukaya:sf:pseudocomplete}
The $A_\infty$ algebras $(\Cdh, \{\mgh_k\}_{k \ge 0}, \langle\;,\,\rangle_F, 1 \otimes 1)$ and $(\mCdh,\{ {\hmt}_{k}^{\gt} \}_{k\ge 0},\ll\;,\,\gg_F,1)$ are pseudo-complete with respect to $\nu_S\otimes\nu$ when $n > 1.$
\end{cor}
\begin{proof}
If $n>1$, then $(S,1)$ is a descending unital differential graded algebra over $\R$. Apply Proposition~\ref{prop:saba_decend_psedu_comp}.
\end{proof}
Define $\mRd:=\Rd\otimes A^*(I)$ and $\mRdh:=S\otimes\mRd$. By Theorem~\ref{thm:mt_cyclic_A_infty}, Proposition~\ref{prop:a_infity_extension}, and restriction of scalars, $(\mCdh,\{\hmt_k^{\gt}\}_{k\ge0},\ll\;,\,\gg_F,1)$ is a cyclic unital $A_\infty$ algebra over $\mRd$ with values in $\mRdh$.
Recall Definition~\ref{dfn:bdef} concerning $b$ deformations of an $A_\infty$ algebra. 
\begin{cor}\label{cor:mgh_expo_converge}
Let $b\in \Cdh$ (resp. $b\in \mCdh$) with $\nu_S\otimes \nu(b)>0$ and $|b|=1$. Then $\mghi{b}_k(\eta_1, \ldots, \eta_k)$ (resp. $\hmt_k^{\gt, b} (\eta_1, \ldots, \eta_k)$ ) converges with respect to $\nu$ for all $\eta_1, \ldots, \eta_k \in  \Cdh $ (resp. $\mCdh$). Thus, $\mghi{b}_k:{\Cdh}^{\otimes k}\to \Cdh$ (resp. $\hmt_k^{\gt, b}:{\mCdh}^{\otimes k}\to \mCdh$) is well defined.
Assume in addition that $\overline{\mghb_0}=0$
in the first case, or
\[
\overline{\hmt_0^{\gt,b}}=0,\qquad d_{\mRdh}\ll b,1\gg_F=0
\]
in the pseudoisotopy case. Then $(\Cdh, \{\mghb_k\}_{k \ge 0}, \langle\;,\,\rangle_F, 1 \otimes 1)$ (resp. $(\mCdh,\{ \hmt_k^{\gt,b} \}_{k\ge 0},\ll\;,\,\gg_F,1)$) is a cyclic unital $A_\infty$ algebra over $Z(\Rdh)$ (resp. $Z(\mRdh)$) with values in $\Rdh$ (resp. $\mRdh$). 
\end{cor}
\begin{proof}
Convergence follows from Corollary~\ref{cor:fukaya:sf:pseudocomplete} and Lemma~\ref{lm:unique_b_deform}. The last assertion follows from Corollary~\ref{cor:mkbh_converges}.
\end{proof}
\begin{rem}\label{rem:disaster_example}
If $\Rd$ is not sababa then $\Cdh$ is not necessarily
pseudo-complete as seen in the following example.
Let $\beta_1,\beta_2,\beta_3\in\sly$ and assume they satisfy
$\omega(\beta_1)=-\omega(\beta_2)$ and $\omega(\beta_3)>0$.
Also assume that $\mu(\beta_1)=2n-2$, $\mu(\beta_2)=0$
and $\mu(\beta_3)=2$.
Then for $m\in\N$ we have
$T^{m(\beta_1+\beta_2)+\beta_3}$ is an element of
$\mI_{\Rh}$.
For each $m$ we have
$\nu(T^{m(\beta_1+\beta_2)+\beta_3})=\omega(\beta_3)$,
so $\{T^{m(\beta_1+\beta_2)+\beta_3}\}_{m\in\N}$
is not a subset of any sababa algebra.
Define
$\lambda_m:=s^{2m}T^{m(\beta_1+\beta_2)+\beta_3}\in\Rh$,
and observe that
$\lim_{m\to\infty}\nu_S\otimes\nu(\lambda_m)=\infty$
and $|\lambda_m|=2$.
Suppose $\eta'_k$ is an element of $C^{\otimes k}$ with
$|\eta'_k|\ge k-2$ and
$\nu(\mg_k(\eta'_k))\le M$ for some $M\in\R$.
We define $\eta_k:=\lambda_k\otimes\eta'_k$, so
\[
\lim_{k\to\infty}\nu_S\otimes\nu(\eta_k)=\infty,
\qquad
|\eta_k|\ge k,
\qquad
\nu(\mgh_k(\eta_k))\le M+\omega(\beta_3).
\]
This contradicts pseudo-completeness.
Moreover, the set $\{\mgh_k(\eta_k)\}_{k\in\N}$
is linearly independent because
$s^{2m}\mid\mgh_k(\eta_k)$ if and only if $m\le k$.
Thus, $\mgh(\eta)$ does not converge with respect to $\nu$.
\end{rem}
\begin{rem}\label{rem:maurer_cartan_converge}
Without Corollary~\ref{cor:mgh_expo_converge}, we would have $\mghbl_0 \in \doublehat{\Cd},$ where $\doublehat{\Cd}$ is defined to be the completion with respect to $\nu_S \otimes \nu$. Then the inductive construction in the proof of Proposition~\ref{prop:exist} would give $b \in \doublehat{\Cd},$ so $(\gamma, b)$ would not be a unit bounding pair. If we were to relax the definition of unit bounding pairs and allow $b\in \doublehat{\Cd}$, the inductive constructions of Section~\ref{ssec:classification} would not work. See Remark~\ref{rem:pseudo_sababa}.
\end{rem}
\subsection{Bounding cochains and gauge equivalence}\label{ssec:fukaya_bounding_gauge}
\subsubsection{Bounding pairs and \texorpdfstring{$b$-deformations}{b-deformations}}
Recall Definition~\ref{dfn_bd_pair} from Section~\ref{subsubsec:a_infty}, and Definition~\ref{dfn:bounding_cochain}. Observe that for a bounding pair $(\gamma,b)$ where $b \in \Cdh$ we have $\nu_S\otimes \nu(b)>0,$ and $|b|=1.$ So, Corollary~\ref{cor:mgh_expo_converge} implies that the deformed operator $\mghi{b}_0$ is well defined. We can use Lemma~\ref{lm:deformed_operator_explicit_formula}  to rewrite equation~\eqref{eq:bc} as
\[\mghi{b}_0=c\cdot 1, \qquad c\in \mI_{\Rh}\cap Z(\Rh),\qquad |c|=2. \]
All the generators of $\mI_{\Rh}$ have strictly positive $\nu$ filtration. So,  $\nu(\mghi{b}_0)>0.$ Hence, $b$ is a bounding cochain for the $A_\infty$ algebra $(\Cdh,\{\mgh_k\}_{k\ge 0},\langle\;,\,\rangle_F,1).$

In the other direction, if $b$ is a bounding cochain for the $A_\infty$ algebra $(\Cdh,\{\mgh_k\}_{k\ge 0},\langle\;,\,\rangle_F,1)$, then $(\gamma,b)$ is a bounding pair as in Definition~\ref{dfn_bd_pair}.
\begin{rem}
    The assumption that $b\in \Cdh$ is not restrictive. We can always write $\Rdh=R(b)$ in our construction of $\Cdh.$
\end{rem}

\subsubsection{Gauge equivalence}
\begin{dfn}\label{dfn_g_equiv}
We say a bounding pair $(\gamma,b)$ with respect to $J$ is \textbf{gauge equivalent} to a bounding pair $(\gamma',b')$ with respect to $J'$, if there exist
a path $\{J_t\}$ in $\mathcal{J}$ from $J$ to $J'$,
$\gt\in (\mI_Q\mD)_2$ closed, and $\bt\in\left((\langle s\rangle+\mI_{\hat{R}})\hat{\mC}\right)_1$ such that
\begin{enumerate}
    \item the triple $(\hmt^{\gt},\ll\,,\,\gg_F,1)$ is a pseudoisotopy from $\left(\mgh,\langle\,,\,\rangle_F,1\right)$ to $\left(\mh^{\gamma'},\langle\,,\,\rangle_F,1\right)$;
    \item $j_0^*\gt=\gamma,\quad j_1^*\gt=\gamma',\quad j_0^*\bt=b,\quad j_1^*\bt=b'$;
    \item $\bt$ is a bounding cochain for $\hmt^{\gt}$.
\end{enumerate}
In this case, we say that $(\hmt^{\gt},\bt)$ is a \textbf{pseudoisotopy} from
$(\mgh,b)$ to $(\mh^{\gamma'},b')$ and write
$(\gamma,b)\sim(\gamma',b')$. In the special case $J_t=J=J'$,
$\gamma=\gamma'$ and $\gt=\pi^*\gamma$, we say $b$ is gauge equivalent
to $b'$ as a bounding cochain for $\mgh$.
\end{dfn}
\begin{rem}
Suppose $(\hmt^{\gt},\bt)$ is a pseudoisotopy from $(\mgh,b)$ to $(\hat{\m}^{\gamma^\prime},b^\prime)$ with $\mghb_0=c\cdot 1,$ $\hat{\m}^{\gamma^\prime,b^\prime}_0=c'\cdot 1$ and $\mgtb_0=\ct\cdot 1$. The assumption $d\ct=0$ implies $\ct=\hat{c}+f(t) \,dt$ with $\hat{c}$ constant.
We get
\[
j_0^* \hmt^{\gt,\bt}_0 = \mghb_0,\qquad j_1^*\hmt^{\gt,\bt}_0=\hat{\m}_0^{\gamma^\prime,b^\prime}.
\]
Therefore, $c=j_0^*\ct =\hat{c} = j_1^*\ct =c'$.
\end{rem}

The next lemmas prove that the map $\varrho$ given by~\eqref{eqn_rho} is well defined.
The following is Lemma 3.15 of~\cite{solomon2016point}.
\begin{lm}\label{lm:homotopy}
Let $M$ be a manifold with $\d M=\emptyset$ and let $\tilde{\xi}\in A^*(I\times M)$ such that $d\tilde{\xi}=0.$ Then
\[
[j_0^*\xit]=[j_1^*\xit]\in H^*(M).
\]
\end{lm}

Define the integral $\int_L:R\otimes\overline{\Ch}\to \Rh$ by the composition \[\begin{tikzcd}
{R\otimes\overline{\Ch}} && {R\otimes S\otimes A^*(L)} && {R\otimes S\otimes \R}\cong \Rh	
	\arrow["\text{Lemma~\ref{lm:cano_iso}}", from=1-1, to=1-3]
	\arrow["{\Id\otimes\int_L}", from=1-3, to=1-5]
\end{tikzcd}.\] 
Define the integral $\int_{I\times L}:R\otimes\overline{\mCh}\to \Rh$ by the composition
\[\begin{tikzcd}
{R\otimes\overline{\mCh}} && {R\otimes S\otimes A^*(I\times L)} && {R\otimes S\otimes \R}\cong \Rh	
	\arrow["\text{Lemma~\ref{lm:cano_iso}}", from=1-1, to=1-3]
	\arrow["{\Id\otimes\int_{I\times L}}", from=1-3, to=1-5]
\end{tikzcd}.\]  

\begin{lm}\label{lm_rho}
Assume $n>0$. If $(\gamma,b)\sim(\gamma',b')$, then 
$[\gamma]=[\gamma']$ and 
\[
\int_L\hat{f}(b)\equiv\int_L \hat{f}(b^\prime) \pmod{\mI_{\hat{R}}\cap \left(\hat{R}^{odd}\right)^2}.
\]
\end{lm}
\begin{proof}
By definition of gauge equivalence, there exists a form $\gt\in \mI_Q\mD$ with $d\gt = 0$ such that $j_0^*\gt=\gamma$ and $j_1^*\gt=\gamma'$.
Lemma~\ref{lm:homotopy} therefore implies that $[\gamma]=[\gamma'].$

Choose $\bt$ as in the definition of gauge equivalence. Then Proposition~\ref{cl:qt_no_top_deg} with $(\gt|_{I\times L})_{n+1}=0$ and the definition of a bounding cochain imply
\[
(d\hat{f}(\bt)-\hat{f}(\bt)\cdot\hat{f}(\bt))_{n+1}=(\hat{f}(\hmt^{\gt,\bt}_0))_{n+1}=(\ct \cdot 1)_{n+1},
\]
where $\ct \in (\mI_{\hat{R}}\hat{\mR})_2$. Thus, Lemma~\ref{lm:sign_multipication_order} gives
\[
\begin{aligned}
(d\hat{f}(\bt))_{n+1}&=(\hat{f}(\ct)+(\hat{f}(\bt^0)+\hat{f}(\bt^1))\cdot(\hat{f}(\bt^0)+\hat{f}(\bt^1)))_{n+1}\\
&=(\ct\cdot 1+\hat{f}(\bt^1)\cdot\hat{f}(\bt^1))_{n+1}.
\end{aligned}
\]
Recalling the sign convention of~\cite[Remark 2.3]{solomon2016differential}, we compute
\[
\begin{aligned}
(-1)^{n+1}\left(\int_L\hat{f}(b')-\int_L\hat{f}(b)\right)
&=(-1)^{n+1}\int_{\partial(I\times L)}\hat{f}(\bt)
=\int_{I\times L}d\hat{f}(\bt)\\
&=\int_{I\times L}(\ct\cdot 1+\hat{f}(\bt^1)\cdot\hat{f}(\bt^1))_{n+1}=\int_{I\times L}\hat{f}(\bt^1)\cdot\hat{f}(\bt^1).
\end{aligned}
\]
For the last equality, we have used the fact that for  $n>0$, we have $(\ct\cdot 1)_{n+1}=0$.
By the definition of gauge equivalence, $\bt\in(\langle s\rangle+\mI_{\hat{R}})\hat{\mC}$. For $n$ odd, the underlying dga is graded commutative by Lemma~\ref{lm:sign_multipication_order}, so $\hat{f}(\bt^1)\cdot\hat{f}(\bt^1)=0$ because $|\hat{f}(\bt^1)|$ is odd. 
This proves the result for $n$ odd. For $n$ even, the coefficients of $\hat{f}(\bt^1)$ belong to $\hat{R}^{odd}$ and hence $\int_{I\times L}\hat{f}(\bt^1)\cdot\hat{f}(\bt^1)\in\left(\hat{R}^{odd}\right)^2$. Applying Lemma~\ref{lm:cano_iso} with \(\mathcal C=R\), viewed as a module over \(\mathcal R=R\), gives a graded $S$-module isomorphism
\[
\hat R/\mI_{\hat R}
=\overline{\hat R}
\simeq S\otimes\bar R
\simeq S.
\]
We deduce that the image of $\int_L\hat{f}(b')-\int_L\hat{f}(b)$ under the quotient map lies in $(S^{even})_{1-n}$.
Since $1-n$ is odd, we have $(S^{even})_{1-n}=0$. Thus $\int_L\hat{f}(b')-\int_L\hat{f}(b) \in \mI_{\hat{R}}\cap \left(\hat{R}^{odd}\right)^2$.
\end{proof}
\subsubsection{Bounding cochains of the residue algebra.}\label{sssec: equivalence}
Assume $n>1$, so $\nu_S$ is non-trivial.
\begin{lm}\label{lm:energy_zero_degree_1}
    Let $b\in \Cdh$ be a degree $1$ element such that $\nu_S \otimes \nu(b)>0$. Then there exists an element $\alpha \in A^n(L)$ such that $\bar{b}=s\otimes \a.$
\end{lm}
\begin{proof}
    Notice that $\bar{b} \in s\cdot \overline{\Cdh}.$ Indeed, in any homogeneous expansion $b=\sum_i s_i \otimes \alpha_i$, every term surviving in the residue must have positive $\nu_S$, and hence is divisible by $s$. Therefore, by Proposition~\ref{prop:fukaya_infity_arises_dga} and Lemma~\ref{lm:cano_iso}, there exists an element $\alpha \in A^n(L)$ such that $\bar{b}=s\otimes \alpha$.
\end{proof}
\begin{lm}\label{lm:energy_zero_bounding_chain}
    Let $b\in \Cdh$ be a degree $1$ element such that $\nu_S \otimes \nu(b)>0$. Then, $\bar{b}$ is a bounding cochain for the $A_\infty$ algebra $(\overline{\Cdh},\{\bar{\mh}^{\gamma}_k\}_{k\ge 0})$,  and $\bar{\hat{\m}}^{\gamma,b}_1$ is a boundary operator on $\overline{\Cdh}$. 
\end{lm}
\begin{proof}
      It follows from Proposition~\ref{prop:fukaya_infity_arises_dga} and Proposition \ref{prop:sf_extension_arises_dga} that
    $\bar{\hat{\m}}^{\gamma,b}_0=d\bar{b}-\bar{b}\cdot \bar{b}.$ By Lemma~\ref{lm:energy_zero_degree_1} $\bar{b}=s\otimes \a$ for some $\alpha\in A^n(L)$. It follows that
    \[d\bar{b}-\bar{b}\cdot \bar{b}=0.\]
    So, $\bar{b}$ is a bounding cochain for the $A_\infty$ algebra $(\overline{\Cdh},\{\bar{\mh}^{\gamma}_k\}_{k\ge 0})$, and by Lemma~\ref{lm:bounding_chain_bar_indu_boundary_op} $\bar{\hat{\m}}^{\gamma,b}_1$ is a boundary operator on $\overline{\Cdh}$. 
\end{proof}

\begin{cor}\label{cor:bar_fukaya_deformed_boundry}
    Assume $b\in \Cdh$ (resp. $b\in \mCdh$) is such that $\bar{b}$ is a bounding cochain for the $A_\infty$ algebra $(\overline{\Cdh},\{\bar{\mh}^{\gamma}_k\}_{k\ge 0})$ (resp. $(\overline{\mCdh},\{\bhmt^{\gt}_k\}_{k\ge 0})$). Then, for homogeneous elements $\eta \in \overline{\Cdh}$ (resp. $\eta \in \overline{\mCdh}$), we have $\bar{\hat{\m}}^{\gamma,b}_1(\eta)=\Id \otimes d\eta -[\bar{b},\eta]$ (resp. $\bhmt^{\gt,b}_1(\eta)=\Id \otimes d\eta -[\bar{b},\eta]$) is a boundary operator.
\end{cor}
\begin{proof}
    By Proposition~\ref{prop:fukaya_infity_arises_dga}, the $A_\infty$ algebra $(\overline{\Cd},\{\bar{\m}^{\gamma}_k\}_{k\ge 0},\langle\;,\,\rangle,1)$ arises from the dga of differential forms $(A^*(L),d)$, and the assumptions of Proposition~\ref{prop:deformed_boundrary_operator} hold. That is
    \begin{enumerate}
        \item $(S,1)$ is a descending unital differential graded algebra over $\R$.
        \item $(\overline{\Cd},\{\bar{\m}^{\gamma}_k\}_{k\ge 0},\langle\;,\,\rangle,1)$ is a cyclic unital $A_\infty$ algebra that arises from the dga $(A^*(L),d)$.
        \item The left and right module structure of $\Cd[1]$ over $\Rd$ coincide.
        \item $\Cd$ has an $\bar{\Rd}$ structure $g$ with inverse $f$.
    \end{enumerate}
    By definition of an $\bar{\Rd}$  structure, $\hat{f}$ induces the identity on $\overline{\Cdh}$. So, the result follows. The proof for $(\overline{\mCdh},\{\bhmt^{\gt}_k\}_{k\ge 0})$ is the same.
\end{proof}

\subsubsection{Bounding cochains of the residue algebra of a pseudoisotopy}\label{sssec:psedu_residu_obstruct}

Throughout this section, we assume that $(\mgt, \ll\,,\,\gg, 1)$ is a pseudoisotopy from $(\mg, \langle\,,\,\rangle, 1)$ to $(\mgp, \langle\,,\,\rangle, 1)$. 

We let $b, b^\prime \in \Cdh$ denote bounding cochains for the $A_\infty$ algebras $(\hat{C}, \{\mgh_k\}_{k\ge 0}, \langle\;\,,\,\rangle_F, 1\otimes 1)$ and $(\hat{C}, \mh^{\gamma^\prime}, \langle\,,\,\rangle_F, 1\otimes 1)$, respectively, and assume that $\int_L \bar{b} = \int_L \bar{b}^\prime$.
\begin{lm}\label{lm:bt_exist_base_case1}
 There exists an element $\bt_0 \in \left(\mCdh\right)_1$, such that $\bar{\bt}_0$ is a bounding cochain with respect to the residue $A_\infty$ algebra $(\overline{\mCdh},\{ \bhmt^{\gt}_k \}_{k\ge 0},\ll\;,\,\gg_F,1)$, and $j_0^*(\bt_0)=b$, $j_1^*(\bt_0)=b^{\prime}$. If $\langle b,1\rangle_F=\langle b',1\rangle_F$, then $\bt_0$ can also be chosen so that $d_{\mRdh}\ll\bt_0,1\gg_F=0$.
\end{lm}
\begin{proof}
    It follows from Proposition~\ref{prop:fukaya_infity_arises_dga_dt} and Proposition \ref{prop:sf_extension_arises_dga} that
    $\bhmt^{\gt,\bt}_0=d\bar{\bt}-\bar{\bt}\cdot \bar{\bt}$ for all $\bt\in \left(\mCdh\right)_1.$  Lemma~\ref{lm:energy_zero_degree_1} gives $\bar{b}=s\otimes \a$ for some $\alpha\in A^n(L)$ and $\bar{b}^\prime=s\otimes \a^\prime$ for some $\alpha^\prime \in A^n(L)$. By Poincar\'e duality and the assumption $\int_L \bar{b}= \int_L \bar{b}^\prime,$ there exists $\eta \in A^{n-1}(L)$ such that $ d\eta =(-1)^{1-n}(\a^\prime-\a).$ 
Define 
    \[
    \bt_0:=(1-t)b+t b^\prime +\hat{g}\left(1 \otimes s\otimes \eta\wedge dt\right).
    \]
Clearly, $j_0^*(\bt_0)=b$ and $j_1^*(\bt_0)=b^\prime$. Using Lemma~\ref{lm:residue_id} and recalling that $n>1$, it follows that
    \[
    d\bar{\bt}_0-\bar{\bt}_0\cdot\bar{\bt}_0
    =s\otimes((\a-\a^\prime)\wedge dt)
    +d\left(s\otimes\eta\wedge dt\right)=0.
    \]
Since $\eta\in A^{n-1}(L)$, we have $p_*(\eta\wedge dt)=0$. Hence,
\[
\ll\bt_0,1\gg_F=(1-t)\langle b,1\rangle_F+t\langle b',1\rangle_F.
\]
If $\langle b,1\rangle_F=\langle b',1\rangle_F$, it follows that $d_{\mRdh}\ll\bt_0,1\gg_F=0$.
\end{proof}

Define the submodule $\mB \subset\mCdh$ over the dga $Z(\mRdh)$ by
\begin{equation}\label{eq:defmB}
\mB:=\left\{a\in \mCdh: j_i^*(a)\in Z(\Rdh)\cdot 1 \text{ for } i=0,1\right\}.
\end{equation} 

\begin{lm}\label{lm:bt_exist_base_case}
For $\bt_{0}\in \left(\mCdh\right)_1$ with $\nu_S\otimes\nu(\bt_0)>0$, the operators $\{\hmt^{\gt,\bt_0}_k\}_{k\ge1}$ preserve $\mB$. If $\bt_0$ is as in Lemma~\ref{lm:bt_exist_base_case1}, then $\hmt^{\gt,\bt_0}_0\in\mB$ and its residue vanishes. If, in addition,
\[
d_{\mRdh}\ll\bt_0,1\gg_F=0,
\]
then $(\mB,\{ \hmt^{\gt,\bt_0}_k\}_{k\ge0},\ll\;,\,\gg_F,1)$ is a cyclic unital $A_\infty$ algebra over the dga $Z(\mRdh)$ with values in $\mRdh$.
\end{lm}

\begin{proof}
The operators are well defined by Corollary~\ref{cor:mgh_expo_converge}. Let $a_1,\ldots,a_k \in \mB.$
    By Lemma~\ref{lm:pseudo}, 
\begin{equation}\label{eq:j*mgti}
    j_0^{*}\hmt^{\gt,\bt_0}_k(a_1,\ldots,a_k)=\mghi{j_0^*\bt_0}_k\left(j_0^*(a_1),\ldots,j^*_0(a_k)\right).
\end{equation}
By the definition of $\mB$, properties~\ref{it:lin} and~\ref{it:unit1}, and the fact that $d_{\Rdh}=0$, we have
\[
j_0^*\hmt^{\gt,\bt_0}_k(a_1,\ldots,a_k) = 0, \qquad k\ne 0, 2.
\]
Also, the unit property implies 
\[
j^*_0\hmt^{\gt,\bt_0}_2(a_1, a_2)=j_0^*(a_1)j_0^*(a_2) \in Z(\Rdh)\cdot 1.
\]
Similarly, $j_1^{*}\hmt^{\gt,\bt_0}_k(a_1,\ldots,a_k)\in Z(\Rdh) \cdot 1.$

If $\bt_0$ is as in Lemma~\ref{lm:bt_exist_base_case1}, since $b$ is a bounding cochain, it follows from~\eqref{eq:j*mgti} that
\[
j^*_0\hmt^{\gt,\bt_0}_0 = \mghi{b}_0 \in Z(\Rdh)\cdot 1.
\]
Similarly, $j_1^*\hmt^{\gt,\bt_0}_0\in Z(\Rdh)\cdot1$. Thus $\hmt^{\gt,\bt_0}_0\in\mB$, and its residue vanishes by Lemma~\ref{lm:bt_exist_base_case1}. The last assertion follows from Corollary~\ref{cor:mgh_expo_converge}.
\end{proof}

\subsection{Anti-symplectic involutions and self-duality}\label{ssec:real}
Let $\phi:X\to X$ be an anti-symplectic involution, that is, $\phi^*\omega=-\omega.$ Let $L\subset \fix(\phi)$ and $J\in\J_\phi$. In particular, $\phi^*J=-J$.
For the entire Section~\ref{ssec:real}, we take $\ssly_L \subset  H_2(X,L;\Z)$ with $\Im(\Id+\phi_*) \subset \ssly_L,$ so $\phi_*$ acts on $\sly_L = H_2(X,L;\Z)/\ssly_L$ as $-\Id.$ Also, we take $\deg t_j\in 2\Z$ for all $j=0,\ldots,N$.

Recall the action of $\phi^*$ on $\L,Q,R,C,D,$ defined in equation~\eqref{eq:phi*ext} and the definition of real elements~\eqref{eq:relt}. For a group $Z$ on which a function $\invs$ acts, let $Z^{\invs} \subset Z$ denote the elements fixed by $\invs$. Recall  Definition~\ref{dfn:d_self_dual} of a self-dual $A_\infty$ algebra and a self-dual dga.
\begin{prop}[{\cite[Proposition 4.5]{solomon2016point}}]\label{prop:sd}
Suppose the relative spin structure $\s$ on $L$ is in fact a spin structure, and $\gamma \in \mathcal{I}_Q D$ is real. Then, $(C,\mg)$ is $\phi^*$ self-dual.
\end{prop}

Recall that $\phi^*$ was extended to $S$ in equation~\eqref{eq:phi*ext}.
\begin{lm}\label{lm:phi_star_op_hom}
The algebra $S$ is $\phi^*$ self-dual. 
\end{lm}
\begin{proof}
Observe $(\phi^*|_S)^{op}\circ\phi^*|_S=\Id$ by definition. Since $S$ carries the zero differential, $\phi^*$ is a chain map. It remains to show that for homogeneous $a,b\in S$,
\[
\phi^*(ab)=(-1)^{|a||b|}\phi^*(b)\,\phi^*(a).
\]
For this, we use the algebraic identity
\[
\frac{(|a|+|b|)(|a|+|b|-1)}{2}-\frac{|a|(|a|-1)}{2}-\frac{|b|(|b|-1)}{2}=|a||b|.
\]
\end{proof}

\begin{prop}\label{prop:sd_super}
Suppose the relative spin structure $\s$ on $L$ is a spin structure, and $\gamma \in \mathcal{I}_Q D$ is real. Then, $(\Ch,\mgh)$ is $\phi^*$ self-dual.
\end{prop}
\begin{proof}
This follows from Proposition~\ref{prop:sd}, Lemma~\ref{lm:phi_star_op_hom} and Proposition~\ref{prop:sd_extension}.
\end{proof}

\begin{prop}[{\cite[Proposition 4.13]{solomon2016point}}]\label{prop:sd_dt}
Suppose the relative spin structure $\s$ on $L$ is a spin structure, and $\gt \in \mathcal{I}_Q \mathfrak{D}$ is real. Then, $\left(\mC,\mgt\right)$ is $\phi^*$ self dual
\end{prop}

\begin{prop}\label{prop:sd_super_dt}
Suppose the relative spin structure $\s$ on $L$ is a spin structure, and $\gt \in \mathcal{I}_Q \mathfrak{D}$ is real. Then, $\left(\mCh, \hmt^{\gt}\right)$ is $\phi^*$ self dual.
\end{prop}
\begin{proof}
This follows from Proposition~\ref{prop:sd_dt}, Lemma~\ref{lm:phi_star_op_hom} and Proposition~\ref{prop:sd_extension}.
\end{proof}
For $E\in\mathbb{R}$, recall the notation $\Cdh_E$ and $\mCdh_E$ from Definition~\ref{dfn:frm}.
\begin{cor}\label{cor:real_subcomp}
    Suppose the relative spin structure $\s$ on $L$ is a spin structure, and $\gamma \in \mathcal{I}_Q D$ is real.
    Let $b\in \left(\Cdh\right)^{-\phi^*}$ with $\nu_S\otimes \nu(b)>0$ and $|b|=1$. Then, $\mghi{b}_1$ preserves the invariant subspace $\left(\Cdh\right)^{-\phi^*} \subset \Cdh.$ So, if $\bar{b}$ is a bounding cochain for the residue $A_\infty$ algebra $(\overline{\Cdh},\{\bar{\mh}^{\gamma}_k\}_{k\ge 0})$, then the real elements $\left(\Cdh_E\right)^{-\phi^*}$ form a sub-complex of $\left(\Cdh_E,\mgiek{b}{1}{E}\right)$. 
\end{cor}
\begin{proof}
Preservation of $\left(\Cdh\right)^{-\phi^*}$ follows from Proposition~\ref{prop:sd_super} and Corollary~\ref{cor:preserved_subspace}. If $\bar{b}$ is a bounding cochain, Lemma~\ref{lm:bounding_chain_bar_indu_boundary_op} shows that $\mgiek{b}{1}{E}$ is a boundary operator on $\Cdh_E$. Restricting to the preserved fixed part gives the subcomplex.
\end{proof}

\begin{cor}\label{cor:real_subcomp_dt}
   Suppose the relative spin structure $\s$ on $L$ is a spin structure, and $\gt \in \mathcal{I}_Q \mathfrak{D}$ is real.
   Let $\bt\in \left(\mCdh\right)^{-\phi^*}$ with $\nu_S\otimes \nu(\bt)>0$ and $|\bt|=1$. Then, $\hmt^{\gt, \bt}_1$ preserves the invariant subspace $\left(\mCdh\right)^{-\phi^*} \subset \mCdh.$ So, if $\bar{\bt}$ is a bounding cochain for the residue $A_\infty$ algebra $(\overline{\mCdh},\{\bhmt^{\gt}_k\}_{k\ge 0})$, then the real elements $\left(\mCdh_E\right)^{-\phi^*}$ form a sub-complex of $\left(\mCdh_E,\hmt^{\gt, \bt}_{1,E}\right)$.
\end{cor}
\begin{proof}
Preservation of $\left(\mCdh\right)^{-\phi^*}$ follows from Proposition~\ref{prop:sd_super_dt} and Corollary~\ref{cor:preserved_subspace}. If $\bar{\bt}$ is a bounding cochain, Lemma~\ref{lm:bounding_chain_bar_indu_boundary_op} shows that $\hmt^{\gt,\bt}_{1,E}$ is a boundary operator on $\mCdh_E$. Restricting to the preserved fixed part gives the subcomplex.
\end{proof}

\begin{lm}\label{lm:reR}
A homogeneous element $a \in \hat{R}$ is real if and only if $\deg a \equiv 2$ or $3 \mod 4.$
\end{lm}
\begin{proof}
For $t \in R$ and $x \in S$ homogeneous we compute
\[
\phi^*(x)\otimes \phi^*(t)=(-1)^{\frac{|x|(|x|-1)}{2}}x\otimes \left((-1)^{\frac{|t|(|t|-1)}{2}}t\right )=(-1)^{\frac{|x|(|x|-1)}{2}+\frac{|t|(|t|-1)}{2}}x\otimes t.
\]
Since
\[\frac{|x|(|x|-1)}{2}+\frac{|t|(|t|-1)}{2}=\frac{(|x|+|t|)(|x|+|t|-1)}{2}-|t||x|,
\]
and the elements of $R$ are of even degree, we deduce that
\[
\phi^* \otimes \phi^* (x \otimes t)=(-1)^\frac{(|x\otimes t|)(|x \otimes t|-1)}{2}x \otimes t.
\]
Thus $x \otimes t$ is real if and only if $|x \otimes t| \equiv 2$ or $3 \mod 4 $
\end{proof}

For $v \in \Z$ and  $a\in S$ homogeneous, denote by $a_{v} \in S[v]$ the element corresponding to $a.$ 
\begin{dfn}\label{dfn:inv_shift}
Define $\phi^*:S[v]\to S[v]$ by
\[\phi^*(a_{v})=(-1)^{\frac{(|a|-v)(|a|-v-1)}{2}}a_{v}=(-1)^{\frac{(|a_{v}|)(|a_{v}|-1)}{2}}a_{v}.\]
For $\R$ algebras $\Ups$ define $\phi^*:\hat{\Ups}[v]\to \hat{\Ups}[v]$ by identifying $\hat{\Ups}[v]\cong S[v]\otimes \Ups$ and tensoring $\phi^*$ on $S[v]$ with the identity map on $\Ups.$
\end{dfn}
\begin{lm}\label{lm:phi_star_even_shift}
Let $v\in 2\Z$. Under the identification $S[v]\simeq S$, the shifted action of $\phi^*$ differs from the unshifted action by the constant sign $(-1)^{v/2}.$
\end{lm}
\begin{proof}
For homogeneous $a\in S$, Definition~\ref{dfn:inv_shift} gives
\[
\frac{(|a|-v)(|a|-v-1)}{2}
=\frac{|a|(|a|-1)}{2}-v|a|+\frac{v(v+1)}{2}.
\]
Since $v$ is even, the term $v|a|$ is even and $\frac{v(v+1)}2 \equiv \frac{v}2 \mod 2.$ The claim follows.
\end{proof}
\begin{lm}\label{lm:phi_star_cochainmap}
    Suppose the relative spin structure $\s$ on $L$ is a spin structure.
    Let $v\in 2\Z$. Assume that $\gamma \in \mathcal{I}_Q D$ is real, $b\in (\Cdh)_1$ is real and $\bar{b}$ is a bounding cochain of the residue $A_\infty$ algebra $\left(\overline{\Cdh}, \{\bar{\mh}^{\gamma}_k\}_{k\ge 0}\right)$. Then
    \[
    \phi^*:(\overline{\Cdh}[v],\bar{\mh}_1^{\gamma,\bar{b}}[v])\to (\overline{\Cdh}[v],\bar{\mh}_1^{\gamma,\bar{b}}[v])
    \]
    is a cochain map. Similarly, assume that $\gt \in \mathcal{I}_Q \mathfrak{D}$ is real, $\bt \in (\mCdh)_1$ is real and $\bar{\bt}$ is a bounding cochain of the residue $A_\infty$ algebra $\left(\overline{\mCdh}, \{\bhmt^{\gt}_k\}_{k\ge 0}\right)$. Then $\phi^*:(\overline{\mCdh}[v],\bhmt^{\gt,\bar{\bt}}_1[v])\to(\overline{\mCdh}[v],\bhmt^{\gt,\bar{\bt}}_1[v])$ is a cochain map.
\end{lm}
\begin{proof}
We prove the statement for $\overline{\Cdh}$; the proof for $\overline{\mCdh}$ is the same using Proposition~\ref{prop:sd_super_dt} in place of Proposition~\ref{prop:sd_super}. By Proposition~\ref{prop:sd_super}, the $A_\infty$ algebra $(\Cdh,\mgh)$ is $\phi^*$ self-dual, and hence so is its residue. Since $b$ is real, $\bar{b}\in(\overline{\Cdh}_1)^{-\phi^*}$. Lemma~\ref{lm:sd_equivariance} applied to the residue algebra gives
\[
\phi^*\circ \bar{\mh}^{\gamma,\bar{b}}_1=\bar{\mh}^{\gamma,\bar{b}}_1\circ \phi^*
\]
on $\overline{\Cdh}$. By Lemma~\ref{lm:phi_star_even_shift}, for even $v$ the action of $\phi^*$ on $\overline{\Cdh}[v]$ differs from the unshifted action by a constant sign. Therefore it also commutes with the shifted differential $\bar{\mh}^{\gamma,\bar{b}}_1[v]$.
\end{proof}

\subsection{Discrete filtration and finite energy quotients}\label{ssec:discrete_filtration}
Throughout this section, either $\kappa = \Id$ or $(X,L,\omega,\phi)$ is a real setting, $\s$ is a spin structure, $R$ is concentrated in even degree and $\kappa = -\phi^*.$ 
For the remainder of the paper, whenever $(X,L,\omega,\phi)$ is a real setting, we assume that $\s$ is a spin structure and $\phi^*\gamma=-\gamma$. When considering a pseudoisotopy in the real setting, we also assume that $\phi^*\gt=-\gt$.

Because $(\Cd,\{\mg_k\}_{k\ge 0},\langle\;,\,\rangle,1)$ (resp. 
 $(\mCd,\{\mgt_k\}_{k \ge 0},\ll\;,\,\gg,1)$) is sababa, the set $\nu(\Cdh) = \nu(\Cd)\subset \R$ ($\nu(\mCdh) = \nu(\mCd)\subset \R$) can be written 
\begin{equation}\label{eq:energy_filtration}
\nu(\Cdh)=\{0=E_0<E_1<E_2< \dots \},    
\end{equation}
with $\lim_{i \to \infty}E_i=\infty$.  
\begin{rem}\label{rem:induc_cons_desc}
The above discreteness of $\nu$ is important for inductive constructions. Consider the following general problem. Let $\mathfrak{p}:\mathcal{N}\to \mathcal{M}$ be a morphism of complete $\R$ filtered modules. Suppose we are searching for solutions to the  equation
\begin{equation}\label{eq:genral_operator_equation}
    \mathfrak{p}(\eta)=c.
\end{equation}
Assume that equation~\eqref{eq:genral_operator_equation} has the property that if $\eta$ is a solution modulo $F^{>E} \mathcal{M}$ 
for all $E<E'$, then there exists $\eta' \in F^{E'}\mathcal{N}$ such that 
\[
\mathfrak{p}(\eta+\eta' )\equiv c \pmod{F^{>E'} \mathcal{M}}.
\]
If $\mathcal{M}$ has a discrete filtration $\{0=E_0<E_1<E_2< \dots \}$, this property of $\mathfrak{p}$ assures us that if we are given $\eta_0 \in \mathcal{N}$ that is a solution to equation~\eqref{eq:genral_operator_equation} modulo $F^{>0} \mathcal{M},$ then we can find a genuine solution to the equation. Indeed the property of $\mathfrak{p}$ implies the existence of elements $\eta_i \in F^{E_i}\mathcal{N}$ such that $\eta_{(l)}=\sum_{i=0}^{l}\eta_i$ is a solution to equation~\eqref{eq:genral_operator_equation} modulo $F^{>E_l} \mathcal{M}$, for all $l\in \N$. Hence, since $\mathcal{N}$ is complete, we have that $\eta_{\infty}:=\lim_{l \to \infty}\eta_{(l)}$ is a genuine solution to equation~\eqref{eq:genral_operator_equation}.
\end{rem}

For $E > 0$ recall the definitions of $\Rd_{E}, \Rdh_{E}, \Cdh_{E}$ (resp. $\mRd_{E}, \mRdh_{E}, \mCdh_{E}$) from Definition~\ref{dfn:frm}.
Let $\hat{g}:\Rd \otimes_{\bar{\Rd}}\overline{\Cdh}\to \Cdh $ (resp. $\hat{g}:\Rd \otimes_{\bar{\Rd}}\overline{\mCdh}\to \mCdh $) be the induced $\bar{\Rd}$  structure on $\Cdh$ (resp. $\mCdh$) as in Definition~\ref{dfn:induced_R_structure}, and let $\hat{f}$ be its inverse.
Let $b\in \Cdh$ (resp. $b\in \mCdh$) be such that $\bar{b}$ is a bounding cochain for the residue $A_\infty$ algebra $(\overline{\Cdh},\{\bar{\mh}^{\gamma}_k\}_{k\ge 0})$ (resp. $(\overline{\mCdh},\{\bhmt^{\gt}_k\}_{k\ge 0})$).
By Corollary~\ref{cor:mgh_expo_converge} the map $\mghb_1:\Cdh \to \Cdh$ (resp. $\hmt_1^{\gt, b}:{\mCdh} \to \mCdh$) is well defined.
By Lemma~\ref{lm:bounding_chain_bar_indu_boundary_op} the induced map $\mghb_1: \Cdh_{E} \to \Cdh_{E}$ (resp. $\hmt_1^{\gt, b}:{\mCdh}_{E} \to \mCdh_{E}$) is a boundary operator. 

In what follows let $\cC=\Cdh$ or $\mCdh$, and let $\m_k=\mg_k$ or $\hmt^{\gt}_k$, respectively. Consider $\R$ as a graded module concentrated in degree $0$. Then, the module $\hom_{\mathbb{R}}(\Rd_E,\R)$ has a natural grading.  Denote the module $\hom_{\mathbb{R}}(\Rd_E,\R)$ by ${\Rd_{E}}^{\vee}$. For a homogeneous element $r \in {\Rd_{E}}^{\vee},$ let $\hat{\pi}_r : \Rd_{E}\otimes \overline{\cC} \to \overline{\cC}[|r|]$ be given by $\lambda \otimes \alpha \mapsto r(\lambda)\alpha.$
Let
\begin{equation}\label{eq:pir}
\pi_r : \cC_{E} \to \overline{\cC}[|r|]
\end{equation}
be given by $\hat\pi_r\circ \hat f.$ Observe that the map descends naturally to a map on the quotients $\pi_r : \cC_{\not \e , E} \to \overline{\cC_{\not \e}}[|r|].$

\begin{lm}\label{lm:pi_r_phi_equivariant}
Assume we are in the real setting.
The map $\pi_r:\cC_E\to \overline{\cC}[|r|]$ is $\phi^*$ equivariant.
\end{lm}
\begin{proof}
We prove the statement for $\cC=\Cdh$. The proof for $\cC=\mCdh$ is the same, replacing $A^*(L)$ by $A^*(I\times L)$.
By the definition of $\Cdh$, for a homogeneous element $\a \in \Cdh_{E}$, we can write 
\[\a=\left[\sum_{j=0}^{m} s^{k_j}\otimes \lambda_j \otimes \eta_j \right],\]
where $s$ is the generator of $S$, and  $\lambda_j\in \Rd$ and $\eta_j \in A^*(L)$ are homogeneous elements. Recalling that $\phi^*$ acts trivially on $A^*(L)$, we compute 
\begin{align*}
    \pi_r\left(\phi^*(\a)\right)&=\pi_r\left(\left[\sum_{j=0}^{m} \phi^*(s^{k_j})\otimes \phi^*(\lambda_j) \otimes \eta_j\right]\right)\\
    &=\sum_{j=0}^{m}(-1)^{|s^{k_j}||\lambda_j|}r\left(\left[\phi^*(\lambda_j) \right] \right)\phi^*(s^{k_j})_{|r|}\otimes  \eta_j\\
    &=\sum_{j=0}^{m}(-1)^{|s^{k_j}||\lambda_j|+\frac{|s^{k_j}|(|s^{k_j}|-1)}{2}+\frac{|\lambda_j|}{2}}r\left(\left[\lambda_j \right] \right)(s^{k_j})_{|r|}\otimes  \eta_j\\
    &=\sum_{j=0}^{m}(-1)^{\frac{(|s^{k_j}|+|\lambda_j|)(|s^{k_j}|+|\lambda_j|-1)}{2}}r\left(\left[\lambda_j \right] \right)(s^{k_j})_{|r|}\otimes  \eta_j.\\
    \shortintertext{Observe that for $j$ such that $r([\lambda_j])\ne 0$ then $|\lambda_j|=-|r|.$ So, $r\left(\left[\lambda_j \right] \right)(s^{k_j})_{|r|}\in S[|r|]_{|s^{k_j}|+|\lambda_j|}.$ Hence, we continue}
    &=\sum_{j=0}^{m}\phi^*\left(r\left(\left[\lambda_j \right] \right)(s^{k_j})_{|r|}\otimes  \eta_j\right)\\
    &=\phi^*(\pi_r(\a)).
\end{align*}      
\end{proof}

\begin{lm}\label{lm:pi_r_chain}\label{lm:Real_res_chain_map}\label{lm:kappa_pi_r_chain}
Let $r\in{\Rd_E}^{\vee}$ be homogeneous. Then the following hold.
\begin{enumerate}
    \item Let $b\in({\Cdh}^{\invs})_1$ be such that $\bar{b}$
    is a bounding cochain for the residue $A_\infty$ algebra
    $(\overline{\Cdh},\{\bar{\mh}^{\gamma}_k\}_{k\ge 0})$.
    Then $\pi_r$ gives a cochain map
    \[
    \pi_r:
    \left((\Cdh_E)^{\invs},\mgiek{b}{E}{1}\right)
    \longrightarrow
    \left(\overline{\Cdh}[|r|]^{\invs},
    \bar{\mh}^{\gamma,b}_1[|r|]\right).
    \]
    It also descends to a cochain map
    \[
    \pi_r:
    \left((\Cdh_{\not \e,E})^{\invs},\mgiek{b}{E}{1}\right)
    \longrightarrow
    \left(\overline{\Cdh_{\not \e}}[|r|]^{\invs},
    \bar{\mh}^{\gamma,b}_1[|r|]\right).
    \]

    \item Let $\bt\in({\mCdh}^{\invs})_1$ be such that
    $\bar{\bt}$ is a bounding cochain for the residue
    $A_\infty$ algebra induced on $\overline{\mCdh}$.
    Then $\pi_r$ gives a cochain map
    \[
    \pi_r:
    \left((\mB_E)^{\invs},\hmt^{\gt,\bt}_{1,E}\right)
    \longrightarrow
    \left(\overline{\mB}[|r|]^{\invs},
    \bhmt^{\gt,\bt}_1[|r|]\right).
    \]
    It also descends to a cochain map
    \[
    \pi_r:
    \left((\mB_{\not \e,E})^{\invs},\hmt^{\gt,\bt}_{1,E}\right)
    \longrightarrow
    \left(\overline{\mB_{\not \e}}[|r|]^{\invs},
    \bhmt^{\gt,\bt}_1[|r|]\right).
    \]
\end{enumerate}
Moreover, if $r\ne0$, then each induced map on quotients above
is surjective.
\end{lm}

\begin{proof}
Let $\cC$ denote either $\Cdh$ or $\mB$, with the corresponding
differentials in the statement. The induced $\bar{\Rd}$
structure identifies $(\cC_E,d_{\cC})$ with
$(\Rd_E\otimes\overline{\cC},\Id\otimes\bar{d}_{\cC})$;
for $\Cdh$ this is Lemma~\ref{lm:twisted_f_chain_iso},
and for $\mB$ it is the same identification restricted to the
preserved submodule $\mB$.
Viewing $r$ as the cochain map
$\rho_r:\Rd_E\to\mathbb{R}[|r|]$, the map $\hat{\pi}_r$ is
\[
\Rd_E\otimes\overline{\cC}
\xrightarrow{\rho_r\otimes\Id}
\mathbb{R}[|r|]\otimes\overline{\cC}
\simeq\overline{\cC}[|r|].
\]
Since $d_{\Rd_E}=0$, this is a cochain map; hence
$\pi_r=\hat{\pi}_r\circ\hat{f}$ is a cochain map as well.
It is compatible with the unit summands, so it descends to the
$\not\e$ quotients.

For $\invs=-\phi^*$, the required fixed subcomplexes are given
by Corollaries~\ref{cor:real_subcomp},~\ref{cor:real_subcomp_dt}
and Lemma~\ref{lm:phi_star_cochainmap};
by Lemma~\ref{lm:pi_r_phi_equivariant}, $\pi_r$ preserves them,
and the assertion for $\mB$ follows by restriction from $\mCdh$.
Thus the maps restrict to the fixed parts and to the quotients.
For $\invs=\Id$ there is nothing to check.

Finally assume $r\ne0$, and choose homogeneous
$r^\vee\in\Rd_E$ with $r(r^\vee)=1$.
Let $[\eta]$ be an element of either target quotient, and set
$x=\hat{g}(r^\vee\otimes\eta)$ for a representative $\eta$.
Then $[x]$ maps to $[\eta]$.
If $\invs=-\phi^*$ and $[\eta]$ is fixed, set
\[
y:=\frac{x+\invs x}{2}.
\]
Then $[y]$ is fixed and, by the $\invs$-equivariance of $\pi_r$,
maps to $\frac{[\eta]+\invs[\eta]}{2}=[\eta]$.
\end{proof}

\begin{lm}\label{lm:basis_exact}\label{lm:dt_basis_exact}
The following exactness criteria hold.
\begin{enumerate}
    \item Let $b\in({\Cdh}^{\invs})_1$ be such that $\bar{b}$
    is a bounding cochain for
    $(\overline{\Cdh},\{\bar{\mh}^{\gamma}_k\}_{k\ge 0})$.
    A cochain
    $\a\in\left(({\Cdh_{\not \e}}^{\invs})_E\right)_*$
    is $\mghb_1$ exact in
    $\left((\Cdh_{\not \e,E})^{\invs},\mgiek{b}{E}{1}\right)$
    if and only if, for all homogeneous $r\in{\Rd_E}^{\vee}$,
    the cochain
    $\pi_r(\a)\in\overline{\Cdh_{\not \e}}[|r|]^{\invs}$
    is $\bar{\mh}^{\gamma,b}_1[|r|]$ exact.

    \item Let $\bt\in({\mCdh}^{\invs})_1$ be such that
    $\bar{\bt}$ is a bounding cochain for the residue
    $A_\infty$ algebra induced on $\overline{\mCdh}$.
    A cochain $\a\in\left((\mB_{\not \e}^{\invs})_E\right)_*$
    is $\hmt^{\gt,\bt}_1$ exact in
    $\left((\mB_{\not \e,E})^{\invs},
    \hmt^{\gt,\bt}_{1,E}\right)$
    if and only if, for all homogeneous $r\in{\Rd_E}^{\vee}$,
    the cochain
    $\pi_r(\a)\in\overline{\mB_{\not \e}}[|r|]^{\invs}$
    is $\bhmt^{\gt,\bt}_1[|r|]$ exact.
\end{enumerate}
\end{lm}

\begin{proof}
Pick a homogeneous basis $\{r_i\}_{i=1}^m$ for $\Rd_E$ as a
real vector space and let $\{r_i^\vee\}_{i=1}^m$ be its dual.
If each $\pi_{r_i^\vee}(\a)$ has a primitive $\eta^i$ in its
shifted target complex, then
$\sum_i r_i\otimes\eta^i$ is a primitive for
$\hat{f}(\a)$.
By Lemma~\ref{lm:twisted_f_chain_iso}, the cochain $\a$ is exact.
The converse follows from Lemma~\ref{lm:pi_r_chain}.

For $\invs=\Id$, this proves the first statement.
In the real case, Corollary~\ref{cor:primitive_exists},
applied to the induced finite-energy quotient, gives a
primitive in the fixed subcomplex.
The second statement is identical, using that the induced
$\bar{\Rd}$ structure sends $\mB$ to
$\Rd\otimes\overline{\mB}$.
\end{proof}

\begin{rem}\label{rem:obstruction_reduction}
     Define the operator $\mathfrak{p}:\left(\mathcal{I}_{\Rdh}\Cdh\right)_1\to \left(\Cdh_{\not \e}\right)_2$ by 
    \[\eta \mapsto \mghi{b+\eta}_0. \]
    In Section~\ref{sec:bd_chains}, our goal is to establish a bounding cochain by resolving $\mathfrak{p}(\eta)=0.$ As described in Remark~\ref{rem:induc_cons_desc}, we aim to demonstrate that the problem operator $\mathfrak{p}:\left(\mathcal{I}_{\Rdh}\Cdh\right)_1\to \left(\Cdh_{\not \e}\right)_2$ exhibits the characteristic that, if $\eta$ is a solution modulo $F^{>E} \left(\Cdh_{\not \e}\right)_2$
for all $E<E'$, then there exists $\eta' \in F^{E'}\left(\mathcal{I}_{\Rdh}\Cdh\right)_1$ such that 
\[
\mathfrak{p}(\eta+\eta' )\equiv 0\pmod{F^{>E'}  \left(\Cdh_{\not \e}\right)_2}.
\]
Given Lemma~\ref{lm:obs_induction}, this property ultimately becomes a cohomological issue within the cochain complexes $(\Cdh_{\not \e,E},\mghb_{1,E})$. The findings presented in this section indicate that our primary concern is to understand $H^{*}\left(\overline{\Cdh_{\not \e}},\bar{\hat{\m}}^{\gamma,b}_1\right)$.
\end{rem}

\section{Spectral sequence}\label{sec:spectral_sequence}
In this section, we introduce spectral sequences associated with the filtered differential graded modules derived from our $A_\infty$ algebras. These spectral sequences are crucial tools for analyzing the obstruction theory in Section~\ref{sec:bd_chains}, both for constructing bounding pairs (Theorem~\ref{thm:spectral_conv}) and for constructing gauge equivalences via pseudoisotopies (Theorem~\ref{thm:spectral_conv_psedu}). 

We recall several definitions from earlier sections. The modules $\Cdh$ over $\Rdh$ and $\mCdh$ over $\mRdh$ are defined in Section~\ref{sssec:sababa_property}, and the $Z(\mRdh)$ module $\mB$ is defined in Section~\ref{sssec:psedu_residu_obstruct}.
Corollary~\ref{cor:bar_fukaya_deformed_boundry} endows $\overline{\Cdh}$, and $\overline{\mCdh}$ with a cochain complex structure over $\R$, if they are given bounding cochains.  For an $\mathcal{R}$ module $\mathcal{C}$, the notation $\cC_{\not \e}$ was introduced in Section~\ref{ssec:obs_theory}. Note that while the unit elements of both $\mB$ equal $1$, we will continue to denote them by $\e$ for consistency with earlier notation.

Throughout this section, we maintain the same assumptions as in Section~\ref{sssec:psedu_residu_obstruct}: the triple $(\hmt^{\gt}, \ll\,,\,\gg, 1)$ is a pseudoisotopy from $(\mg, \langle\,,\,\rangle, 1)$ to $(\mgp, \langle\,,\,\rangle, 1)$; $b, b^\prime \in \Cdh$ are bounding cochains for the $A_\infty$ algebras $(\hat{C}, \{\mgh_k\}_{k\ge 0}, \langle\;\,,\,\rangle_F, 1\otimes 1)$ and $(\hat{C}, \mh^{\gamma^\prime}, \langle\,,\,\rangle_F, 1\otimes 1)$, respectively; and $\int_L \hat{f}(\bar{b}) = \int_L \hat{f}(\bar{b}^\prime)$. We also choose $\bt=\bt_{0}\in \left(\mCdh\right)_1$ as in Lemma~\ref{lm:bt_exist_base_case1}. By Lemma~\ref{lm:bt_exist_base_case}, the operators preserve $\mB$. Since $\bhmt^{\gt,\bt}_0=0$, the $A_\infty$ relations imply $(\bhmt^{\gt,\bt}_1)^2=0$. Thus $H^*\left(\overline{\mB},\bhmt^{\gt,\bt}_1\right)$ is well-defined. By abuse of notation, we write $\mgt$ for $\hmt^{\gt}$; since below these operators are applied only to the extension-of-scalars modules, this should cause no confusion.

Understanding the cohomology groups $H^*\left(\overline{\Cdh_{\not \e}},\bar{\mh}^{\gamma,b}_1\right)$ and $H^*\left(\overline{\mB},\bar{\mt}^{\gt,\bt}_1\right)$ computed by these spectral sequences allows us to determine whether obstructions to constructing bounding cochains or constructing pseudoisotopies vanish, particularly by examining specific components identified by the filtration (see Corollaries~\ref{cor:even_top_deg} and~\ref{cor:dt_even_top_deg}).

\begin{dfn}
    A spectral sequence (of cohomological type) is a collection of differential bigraded modules $\left\{E_r^{*, *}, d_r\right\}$, where $r=0,1, \ldots$ The differentials are of bidegree $(r,1-r)$ and for all $p, q, r,$ we have $E_{r+1}^{p, q} \simeq H^{p, q}\left(E_r^{*, *}, d_r\right)$.
\end{dfn}

We state the main theorems first and defer their proofs to Section~\ref{subsec:proofs_of_main_theorems}, after presenting the necessary supporting lemmas and general spectral sequence machinery.

\begin{thm}\label{thm:spectral_conv}
Let $b\in \Cdh$ be a degree $1$ element such that $\nu_S \otimes \nu(b)>0$. There exists a spectral sequence $(E_r^{p,q}\left(\overline{\Cdh_{\not \e}}\right),d_r)$ that converges to $H^*\left(\overline{\Cdh_{\not \e}},\bar{\mh}^{\gamma,b}_1\right)$ and satisfies
\[
E_1^{p,q} \simeq \begin{cases}
0,&\left(p\in 2\N \bigwedge q+2p=0\right),\\
S_{-p}\otimes H^{q+2p}(L,d),& \text{otherwise.}
\end{cases}
\]
If $\int_L \bar{b}=\overline{\int_L \hat{f}(b)}\ne 0$ and $n$ is even, then
\[
E_\infty^{p,q} =\begin{cases}
0,&\left( q+2p=0 \right)\bigvee \left(p\in 2\N+2\bigwedge q+2p=n\right),\\
E_1^{p,q},& \text{otherwise.}
\end{cases}
\]
Otherwise (if $\int_L \bar{b} = 0$ or $n$ is odd), the spectral sequence degenerates at the $E_1$ page, so
\[
E_\infty^{p,q}=E_1^{p,q}.
\]
\end{thm}
\begin{thm}\label{thm:spectral_conv_psedu}
There exists a spectral sequence  $(E_r^{p,q}(\overline{\mB}),d_r)$ that converges to $H^*\left(\overline{\mB},\bar{\mt}^{\gt,\bt}_1\right)$. Its $E_r$ page, for $r\ge 1$, is related to the $E_r$ page of the spectral sequence for $(\overline{\Cdh_{\not \e}},\bar{\hat{\m}}^{\gamma,b}_1)$ (from Theorem~\ref{thm:spectral_conv}) by:
\[
E_r^{p,q}\left(\overline{\mB}\right)\cong \begin{cases}
E_r^{p,q-1}\left(\overline{\Cdh_{\not \e}}\right), & 2p+q\ne 0, \text{ or } p\in 2\N+1,\\
S_{-p}, & \text{otherwise.}
\end{cases}
.\]
In particular, if $\int_L \bar{b}=\overline{\int_L \hat{f}(b)}\ne 0$ and $n$ is even, then
\[
E_\infty^{p,q}\left(\overline{\mB}\right)\cong \begin{cases}
    0,&\left(q+2p=1\right)\bigvee\left( q+2p=n+1\bigwedge p\in 2\N+2\right),\\
    E_1^{p,q}\left(\overline{\mB}\right), & \text{otherwise.}
\end{cases}\]
Otherwise (if $\int_L \bar{b}=\overline{\int_L \hat{f}(b)}= 0$ or $n$ is odd), the spectral sequence degenerates at the $E_1$ page.
\end{thm}

\subsection{Spectral sequence generalities}
\subsubsection{Spectral sequence of a filtered complex}
Let $\left(A, d, F^\bullet\right)$ be a filtered differential graded module. Following the proof of \cite[Theorem 2.6]{mccleary2001user}, for~$r \geq -1$, let
\begin{align}
Z_r^{p,q} &= F^pA_{p+q} \cap d^{-1}(F^{p+r}A_{p+q+1}) \notag\\
B_r^{p,q} &= F^pA_{p+q}\cap d(F^{p-r}A_{p+q-1}) \label{eq:ssnot_repeat}\\
Z_\infty^{p,q} &= \ker d \cap F^pA_{p+q} \notag\\
B_\infty^{p,q} &= \im d \cap F^pA_{p+q}. \notag
\end{align}
For $r \geq 0$, let
\begin{equation}\label{eq:ssd}
E_r^{p,q} = Z_r^{p,q}/(Z_{r-1}^{p+1,q-1} + B_{r-1}^{p,q})
\end{equation}
and let $d_r : E_r^{p,q} \to E_r^{p+r,q-r+1}$ be the map induced by $d$.

The following is a consequence of the proof of~\cite[Theorem 2.6]{mccleary2001user}.
\begin{thm}\label{thm:filtered_spec_conv}
The sequence of  differential bigraded modules $\left\{E_r^{*, *}, d_r\right\}$ above is a spectral sequence. Moreover,
\[
E_0^{p, q} = F^p A_{p+q} / F^{p+1} A_{p+q},
\]
and $d_0 : E_0^{p,q} \to E_0^{p,q+1}$ is the map induced by $d$.
Suppose further that the filtration is bounded, that is, for each degree $n$, there are values $s=s(n)$ and $t=t(n)$ such that \[
\{0\} = F^s A_n \subset F^{s-1} A_n \subset \cdots \subset F^{t+1} A_n \subset F^t A_n=A_n.
\]
Then, the spectral sequence converges to $H(A, d)$, that is,
\[
E_{\infty}^{p, q} \cong F^p H^{p+q}(A, d) / F^{p+1} H^{p+q}(A, d).
\]
\end{thm}
Let $(A^\prime, d^\prime, F^{\prime\bullet}), (A, d, F^\bullet)$ be filtered differential graded modules. A map $f:(A^\prime, d^\prime, F^{\prime\bullet})\to (A, d, F^\bullet)$ is a morphism of filtered differential graded modules if $f$ is a cochain map and $f(F^{\prime p} A^\prime) \subseteq F^p A$.
Let $\{(E_r^{*,*}, d_r)\}$ and $\{(\bar{E}_r^{*,*}, \bar{d}_r)\}$ be two spectral sequences. A morphism of spectral sequences is a collection of maps of differential bigraded modules
\[
f_r: E_r^{*,*} \longrightarrow \bar{E}_r^{*,*}, \quad r = 0, 1, 2, \ldots,
\]
such that for each $r \geq 0$, the map $f_{r+1}:E_{r+1}^{*,*}\cong H^{*,*}(E_r^{*,*},d_r)\to \bar E_{r+1}^{*,*}\cong H^{*,*}(\bar{E}^{*,*}_r,\bar d_r)$ is the map induced by $f_r$ on cohomology.

\begin{thm}[{\cite[Theorem 3.5]{mccleary2001user}}]\label{thm:spectralquasi_iso_iso}
A morphism of filtered differential graded modules
\[
\phi:(A^\prime, d^\prime, F^{\prime\bullet})\to (A, d, F^\bullet),
\]
determines a morphism of the associated spectral sequences. If, for some $n$, $\phi_n: E_n^{*,*}(A^\prime) \rightarrow E_n^{*,*}(A)$ is an isomorphism of bigraded modules, then for $n\le \ell\le \infty$ the bigraded morphism $\phi_\ell: E_\ell^{*,*}(A^\prime) \rightarrow E_\ell^{*,*}(A)$ is an isomorphism. If the filtrations are bounded, then $\phi$ induces an isomorphism $H(\phi): H(A^\prime, d^\prime) \rightarrow H(A, d)$.
\end{thm}
\begin{prop}\label{prop:spec_exact_sequ}
     Let $(A^\prime, d^\prime, F^{\prime\bullet})$, $(A, d, F^\bullet)$, and $(A^{\prime\prime}, d^{\prime\prime}, F^{\prime\prime\bullet})$ be filtered differential graded modules. Suppose there exists a short exact sequence of filtered differential graded modules:
    \[
    0 \to (A^\prime, d^\prime, F^{\prime\bullet}) \xrightarrow{f} (A, d, F^\bullet) \xrightarrow{h} (A^{\prime\prime}, d^{\prime\prime}, F^{\prime\prime\bullet}) \to 0.
    \]
    Assume $f$ and $h$ are strictly filtered; in particular, $f^{-1}(F^p A) = F^{\prime p} A^\prime$ for all $p$.
    Then we obtain a short exact sequence of bigraded differential modules
    \[
    0 \to (E_0^{*,*}(A^\prime),d_0^\prime) \xrightarrow{f_0} (E_0^{*,*}(A),d_0) \xrightarrow{h_0} (E_0^{*,*}(A^{\prime\prime}),d_0^{\prime\prime}) \to 0.
    \]
    In particular, the assumptions hold if the filtration on $A^\prime$ is the induced filtration from $A$ and the filtration on $A^{\prime\prime}$ is the quotient filtration.

    If in addition for all $r \geq 1$ and all $p, q$, the map induced by $h$ on the cycles $Z_r^{p,q}(A) \to Z_r^{p,q}(A^{\prime\prime})$ is surjective, then the induced maps $f_r$ and $h_r$ on the $E_r$-pages of the spectral sequences form a short exact sequence of bigraded differential modules:
    \[
    0 \to (E_r^{*,*}(A^\prime),d_r^\prime) \xrightarrow{f_r} (E_r^{*,*}(A),d_r) \xrightarrow{h_r} (E_r^{*,*}(A^{\prime\prime}),d_r^{\prime\prime}) \to 0.
    \]
\end{prop}
\begin{proof}
By the assumption that $f$ and $h$ are strictly filtered and~\cite[\href{https://stacks.math.columbia.edu/tag/0124}{Tag 0124},\href{https://stacks.math.columbia.edu/tag/0127}{Tag 0127}]{stacks-project}, we obtain a short exact sequence
\begin{equation*}
0 \to E_{0}^{p,q}(A^\prime) \xrightarrow{f_0^{p,q}} E_{0}^{p,q}(A) \xrightarrow{h_0^{p,q}} E_{0}^{p,q}(A^{\prime\prime}) \to 0.
\end{equation*}
This is the case $r = 0$ of the proposition.

For $r \geq 1$, if the map induced by $h$ on the cycles $Z_{r}^{p,q}(A) \to Z_{r}^{p,q}(A^{\prime\prime})$ is surjective, then so is the induced map $h_{r} : E_{r}^{p,q}(A) \to E_{r}^{p,q}(A^{\prime\prime})$. By induction we have an exact sequence,
\begin{equation*}
0 \to (E_{r-1}^{*,*}(A^\prime),d_{r-1}^\prime) \xrightarrow{f_{r-1}} (E_{r-1}^{*,*}(A),d_{r-1}) \xrightarrow{h_{r-1}} (E_{r-1}^{*,*}(A^{\prime\prime}),d_{r-1}^{\prime\prime}) \to 0.
\end{equation*}
Using the long exact sequence in cohomology with respect to $d_{r-1}$  together with surjectivity of $h_r = H(h_{r-1})$ gives case $r$ of the proposition.
\end{proof}

\subsubsection{Twisted cohomology}\label{sssec:twistedcoh}
Let $A$ be a unital dga. Define
\[
[\cdot,\cdot] : A \times A \to A,
\]
by
$
[a,x] = ax -(-1)^{|a||x|} xa$.
For $a \in A_1,$  define 
\[
d_{twist} : A \to A[1]
\]
by
$d_{twist}(x)=dx-[a,x]$. 
\begin{lm}\label{lm:twist_d_defiend}
If $da-\frac{[a,a]}{2}\in Z(A)$, then $d_{twist}$ is a boundary map. 
\end{lm}
    \begin{proof}
          For any $x\in A$ the graded Jacobi identity implies
    \[(-1)^{|x|}[a,[a,x]] - [a,[x,a]]+(-1)^{|x|}[x,[a,a]]=0.\]
    By graded anti-symmetry this can be rewritten
    \[(-1)^{|x|}[a,[a,x]]+(-1)^{|x|}[a,[a,x]]+(-1)^{|x|}[x,[a,a]]=0.\]
    The above equation implies that 
    \[[a,[a,x]]=-[x,\frac{[a,a]}{2}].\]
    We compute
    \begin{align*}
        d_{twist}\circ d_{twist}(x)&=d_{twist}\left( dx-[a,x]\right)\\
        &=-[a,dx]-d[a,x]+[a,[a,x]]\\
        &=-(da) x+x (da)+[a,[a,x]]\\
        &=-[da,x]+[a,[a,x]]\\
        &=-[da,x]+[\frac{[a,a]}{2},x]\\
        &=0.
    \end{align*}
    \end{proof}
    \begin{lm}
        If $da-\frac{[a,a]}{2}\in Z(A)$, then $(A,\wedge, d_{twist})$ is a unital dga. 
    \end{lm}
    \begin{proof}
        Let $\a,\beta \in A$, then by definition
        \begin{multline*}
            d_{twist}(\a\wedge \beta )= d\left(\a\wedge \beta \right) -[a,\a\wedge \beta ]\\=d\alpha \wedge \beta + (-1)^{|\a|} \a \wedge d\beta - \left( a \wedge (\alpha \wedge \beta) + (-1)^{|\a|+|\beta|+1} (\a \wedge \beta) \wedge a \right).
        \end{multline*}
        We proceed with the computations
        \begin{align*}
            &=d\alpha\wedge\beta +(-1)^{|\alpha|}\alpha\wedge d\beta-\left([a,\alpha]\wedge\beta +(-1)^{|\alpha|}\alpha \wedge \left(a\wedge \beta\right)+ (-1)^{|\alpha|+|\beta|+1}\alpha\wedge(\beta\wedge a)\right)\\
            &=d\alpha\wedge\beta +(-1)^{|\alpha|}\alpha\wedge d\beta-\left([a,\alpha]\wedge\beta +(-1)^{|\alpha|}\alpha \wedge [a,\beta]\right)\\
            &=d_{twist}(\alpha)\wedge \beta +(-1)^{|\alpha|}\alpha \wedge d_{twist}(\beta)
        \end{align*}
        
    \end{proof}
Let $(S,\e)$ be a unital graded $\k$ algebra with $S_0=\k\cdot \e$ and $S_i=\{0\}$ for $i>0.$ Observe that any element in $s\in S$ can be written $s=\sum_{i=n}^{0}s_i$ where $n$ is a negative integer and $s_i\in S_{i}.$ Define a filtration $\nu_S$ on $S$ by
\[\nu_S\left(\sum_{i\le 0}s_i\right)=-\max_{s_i\ne 0}i,\qquad \nu_S(0)=\infty.\]
So, if $S\not\cong \k$ then $\nu_S$ defines a non trivial filtration. For a $\k$ vector space $\mathcal{N}$, also denote by $\nu_S$ the filtration on $S\otimes \mathcal{N}$ defined by tensoring $\nu_S$ with the trivial filtration on $\mathcal{N}.$
Let $A$ be a differential graded algebra over $\k$. Let $\hat{A}$ be the $S$ extension of scalars of $A$. Endow $\hat{A}$ with the filtration $\nu_S$. We give $\hat{A}$ the obvious filtered differential graded  structure over $\k$.  Observe that $Z(\hat{A})=Z(S)\otimes_\k Z(A)$.

Let $b\in \hat{A}_1$ be an element such that $\nu_S(b)>0$ and $db-\frac{[b,b]}{2}\in Z(\hat{A})$. By Lemma~\ref{lm:twist_d_defiend} the differential $d_{twist}$ defined by $b$ is a boundary map. 
Assume that $A$ is bounded from above by $m\in \N$. So, $\hat{A}_1=\bigoplus_{i=0}^{m-1}S_{-i}\otimes A_{1+i}$, and we can write $b=\sum_{i=0}^{m-1}b_i$ with $b_i\in S_{-i}\otimes A_{1+i}.$ The assumption $\nu_S(b)>0$ implies that $b_0=0$.

\begin{thm}\label{thm:hat_spec_seq}
There exists a spectral sequence that converges to $H^*(\hat{A},d_{twist})$ with $E_{1}^{p,q}\cong S_{-p}\otimes H^{2p+q}(A)$. Moreover, for $r\ge1$, each element $\alpha \in E_r^{p,q}$ can be represented by an element $x\in Z_r^{p,q}$ such that \[x=\sum^{\infty}_{i=0}x_i,\qquad x_i\in S_{-p-i}\otimes A_{2p+i+q},\] 
\[
d\sum_{i=0}^{r-1}x_i-\sum_{\substack{1\le j\le m-1\\ 0\le i < r-j}}[b_j,x_i]=0.
\]
Furthermore,
\[
d_r\left(\a\right)=d_r\left([x]\right)=\left[-\sum_{0 \le i \le r-1}[b_{r-i},x_i]+F^{p+r+1} \hat{A}\right],
\]
where the displayed $E_1$-class is taken in the subquotient $E_r^{p+r,q-r+1}$, and $b_{r-i}=0$ if $r-i<1$ or $r-i>m-1.$
\end{thm}
\begin{proof}
    Observe the long exact sequence
\[\begin{tikzcd}[column sep=small]
	\ldots & {H^{p+q}\left(F^{p+1} \hat{A}\right) } && {H^{p+q}\left(F^p \hat{A}\right) } && {H^{p+q}\left(F^p \hat{A} / F^{p+1} \hat{A}\right)} \\
	& {H^{p+q+1}\left(F^{p+1} \hat{A}\right)} && {H^{p+q+1}\left(F^{p} \hat{A}\right)} && {H^{p+q+1}\left(F^p \hat{A} / F^{p+1} \hat{A}\right)} & \ldots
	\arrow["k", from=1-1, to=1-2]
	\arrow["i", from=1-2, to=1-4]
	\arrow["j", from=1-4, to=1-6]
	\arrow["k"', from=1-6, to=2-2]
	\arrow["i", from=2-2, to=2-4]
	\arrow["j", from=2-4, to=2-6]
	\arrow["k", from=2-6, to=2-7]
\end{tikzcd}\]
where $k$ is the connecting homomorphism. 
Define the bigraded modules 
\[E_{1}^{p, q}=H^{p+q}\left(F^p \hat{A}/F^{p+1} \hat{A}\right),\]
and $D^{p, q}=H^{p+q}\left(F^p \hat{A}\right)$. This gives an exact couple from the long exact sequences:
\[\begin{tikzcd}
	{D^{p+1,q-1}} && {D^{p,q}} \\
	& {D^{p+1,q}} \\
	&& {E_1^{p,q}}
	\arrow["i", from=1-1, to=1-3]
	\arrow["j", from=1-3, to=3-3]
	\arrow["k", from=3-3, to=2-2]
\end{tikzcd}.\]
By~\cite[Theorem 2.6]{mccleary2001user}, this yields a spectral sequence $(E_{r}^{*,*},d_r)$. By~\cite[Proposition 2.11]{mccleary2001user} the spectral sequence is the same as Theorem~\ref{thm:filtered_spec_conv}. By Theorem~\ref{thm:filtered_spec_conv} $d_0$ is induced from $d_{twist}$. We observe that $d_0=d$ as $\nu_S(b)>0$. So, $E_{1}^{p,q}\cong S_{-p}\otimes H^{2p+q}(A),$ and the spectral sequence converges. 

By~\cite[Proposition 2.9]{mccleary2001user} 
\[E_r^{p, *}\cong \frac{k^{-1}\left(\operatorname{im} i^{r-1}: D^{p+r, *} \longrightarrow D^{p+1, *}\right)}{j\left(\operatorname{ker} i^{r-1}: D^{p, *} \longrightarrow D^{p-r+1, *}\right)}\]
with differential induced on the representation by $j \circ k.$

Suppose $z$ is in $H^{p+q}\left(F^p \hat{A} / F^{p+1} \hat{A}\right)=E_1^{p, q}$. Then  $z=\left[x+F^{p+1} \hat{A}\right]$ with $x$ in $F^p \hat{A}$ and $d_{twist}(x)$ in $F^{p+1} \hat{A}$. The boundary homomorphism $k$ in the long exact sequence that is part of the exact couple can be described explicitly by

$$
k\left(\left[x+F^{p+1} \hat{A}\right]\right)=[d_{twist}(x)] \quad \text { in } \quad H^{p+q+1}\left(F^{p+1} \hat{A}\right) .
$$

Thus $\left[x+F^{p+1} \hat{A}\right]$ is in $k^{-1}\left(\operatorname{im} i^{r-1}\right)$ if and only if $[d_{twist}(x)]$ is in im $i^{r-1}$ if and only if, after changing representatives, $d_{twist}(x)$ is in $F^{p+r} \hat{A}_{p+q+1}$. For $x\in F^{p} \hat{A}$ a representative of an element in $k^{-1}\left(\operatorname{im} i^{r-1}\right)$ we decompose as 
\[x=\sum^{\infty}_{i=0}x_i,\qquad x_i\in S_{-p-i}\otimes A_{2p+i+q}.\]  
We have
\[
d_{twist}(x)=d\sum_{i=0}^{\infty}x_i-[\sum_{j=1}^{m-1}b_j,\sum_{i=0}^{\infty}x_i].
\]
So, the condition that $d_{twist}(x) \in F^{p+r} \hat{A}_{p+q+1}$ is equivalent to
\[
d\sum_{i=0}^{r-1}x_i-\sum_{\substack{1\le j\le m-1\\ 0\le i < r-j}}[b_j,x_i]=0.
\]
It follows that
\[
d_{twist}(x) \in d x_r-\sum_{0 \le i \le r-1}[b_{r-i},x_i] + F^{p+r+1}\hat A.
\]
Thus,
\[
j\circ k \left(\left[x+F^{p+1} \hat{A}\right] \right)=\left[-\sum_{0 \le i \le r-1}[b_{r-i},x_i]+F^{p+r+1} \hat{A}\right],
\]
where $b_{r-i}=0$ if $r-i<1$ or $r-i>m-1.$ 
\end{proof}
\begin{dfn}
    A spectral sequence of algebras over $R$ is a spectral sequence, $\left\{E_r^{*, *}, d_r\right\}$ together with algebra structures $\psi_r: E_r \otimes_R E_r \rightarrow E_r$ for each $r$, such that $\psi_{r+1}$ can be written as the composite

$$
\begin{aligned}
\psi_{r+1}: E_{r+1} \otimes_R E_{r+1} \xrightarrow{\cong} & H\left(E_r\right) \otimes_R H\left(E_r\right) \\
& \xrightarrow{p} H\left(E_r \otimes E_r\right) \xrightarrow[H\left(\psi_r\right)]{ } H\left(E_r\right) \underset{\cong}{\longrightarrow} E_{r+1},
\end{aligned}
$$
where the homomorphism $p$ is given by $p([u] \otimes[v])=[u \otimes v]$.
\end{dfn}
The following result is \cite[Theorem~2.14]{mccleary2001user}.
\begin{thm}\label{thm:filtered_algebra_spectral_sequence}
     Suppose $(A, d, F)$ is a filtered differential graded algebra over $R$ with product $\psi: A \otimes_R A \rightarrow A$. Suppose that the product satisfies the condition for all $p, q$,  
$$
\psi\left(F^p A \otimes_R F^q A\right) \subset F^{p+q} A.
$$
Then, the spectral sequence associated with $(A, d, F)$ is a spectral sequence of algebras. If the filtration on $A$ is bounded, then the spectral sequence converges to $H(A, d)$ as an algebra.
\end{thm}
\begin{lm}\label{lm:differential_formula_ell_page}
    Assume $b\in Z_{1}^{\ell,1-\ell}$, then $E_{\ell}=E_1$. Moreover, the differential $d_{\ell}:E_{\ell}^{p,q}\to E_{\ell}^{p+\ell,q+1-\ell}$ is given by
\[d_\ell (\a) = -[[b], \a],\]
where $[b]$ denotes the class represented by $b$ in $E_\ell^{\ell, 1-\ell} = E_1^{\ell, 1-\ell}$, and the commutator $[\cdot, \cdot]$ is taken with respect to the algebra structure on the $E_1$ page (induced by the product $\psi$ on $A$).
\end{lm}
\begin{proof}
    We observe that Theorem~\ref{thm:hat_spec_seq} implies that if $b \in F^{\ell} \hat{A}_1$, then $d_i = 0$ for $1 \leq i \leq \ell - 1$. It follows that an element $\alpha \in E_{\ell}^{p,q}$ can be represented by $x \in S_{-p} \otimes A_{2p + q}$ such that $dx = 0$. Therefore, 
    \begin{align*} 
    d_{\ell}(\alpha) = d_{\ell}([x]) &= \left[dx - [b,x]\right] \\ &= \left[-\psi(b,x) + (-1)^{|x|}\psi(x,b)\right] \\ &= -\psi_1([b],[x]) + (-1)^{|x|}\psi_1([x],[b]), 
    \end{align*}
    where the last equality follows from Theorem~\ref{thm:filtered_algebra_spectral_sequence}.
\end{proof}
\subsubsection{Mapping cones}
\begin{dfn}\label{dfn:abstract_mapping_cone}
Let $(C, d_C)$ and $(D, d_D)$ be two cochain complexes. Let $f: C \to D$ be a cochain map.
The \textbf{mapping cone} of $f$, denoted by $\Cone(f)$, is a cochain complex defined as follows. For each integer $n$, the module in degree $n$ is the direct sum:
    \[
    \Cone(f)_n = D_{n-1} \oplus C_{n} .
    \]
The differential $d_{\Cone(f)}: \Cone(f)_n \to \Cone(f)_{n+1}$ is defined for an element $(d, c) \in D_{n-1}\oplus C_{n} $ by,
    \[
    d_{\Cone(f)}(d, c) = (-d_D^{n-1}(d) + f^{n}(c), d_C^{n}(c)).
    \]
The mapping cone fits into a short exact sequence of cochain complexes:
\begin{equation}\label{eq:sescone}
0 \longrightarrow D[-1] \xrightarrow{h} \Cone(f) \xrightarrow{p} C \longrightarrow 0,
\end{equation}
where $h(d) = (d, 0)$ for $d \in D[-1]_n$, $p(d, c) = c$ for $(d, c) \in \Cone(f)_n = D_{n-1} \oplus C_n$.
\end{dfn}
For filtered cochain complexes $(A,d,F)$ and $(B,d^\prime,F^\prime)$ we define the filtration on $A\oplus B$ by $F_{A\oplus B}^p(A\oplus B)_n=F^p A\oplus F^{\prime p} B$. We also define $F^p A[n]_m=F^{p}A_{n+m}$. For $f:A\to B$
define the filtration   on 
$\Cone(f)$ by
\[F^p\Cone(f)=F^{\prime p}B[-1]\oplus F^{p}
A.\]
Observe that 
\begin{equation}\label{eq:sssh}
E_r^{p,q}(A[n]) = E_r^{p,q+n}(A)
\end{equation}
because $Z_r^{p,q}(A[n])=Z_r^{p,q+n}(A)$ and $B_r^{p,q}(A[n])=B_r^{p,q+n}(A)$.
\begin{lm}
\label{lm:cone_sequ_exact_generalized}
Let $f: C \to D$ be a morphism of filtered cochain complexes. Let $f_r: E_r^{p,q}(C) \to E_r^{p,q}(D)$ be the map induced by $f$ on the $r$-th pages of the associated spectral sequences. If $f_r$ is injective for all $r \geq 1$, then there exists a short exact sequence of complexes:

\[\begin{tikzcd}[column sep=scriptsize]
	0 & {E_{r}^{p,q}(C)} && {E_{r}^{p,q}(D)} && {E_{r}^{p,q}(\Cone(f)[1])} & 0.
	\arrow[from=1-1, to=1-2]
	\arrow["{f_r}", from=1-2, to=1-4]
	\arrow["{h[1]_r}", from=1-4, to=1-6]
	\arrow[from=1-6, to=1-7]
\end{tikzcd}\]
\end{lm}
\begin{proof}
By definition of the filtration on $\Cone(f)$, the morphisms in the exact sequence~\eqref{eq:sescone} are strictly filtered.
By Proposition~\ref{prop:spec_exact_sequ},
we obtain a long exact sequence
\[
\begin{tikzcd}[column sep=scriptsize]
	\ldots & {E_{1}^{p,q}\left(C\right)} & {E_{1}^{p,q+1}\left(D[-1]\right)} & {E_{1}^{p,q+1}\left(\Cone(f)\right)} & {E_{1}^{p,q+1}\left(C\right)} & \ldots,
	\arrow[from=1-1, to=1-2]
	\arrow["{\partial}", from=1-2, to=1-3]
	\arrow["{h_1}", from=1-3, to=1-4]
	\arrow["{p_1}", from=1-4, to=1-5]
	\arrow[from=1-5, to=1-6]
\end{tikzcd}
\]
By definition of $\Cone(f)$ and a short computation, the boundary homomorphism $\partial$ coincides with the induced map $f_1 : E_1^{p,q}(C) \to E_1^{p,q}(D) = E_1^{p,q+1}(D[-1])$ given by Theorem~\ref{thm:spectralquasi_iso_iso}.
Since $f_1$ is injective by assumption, keeping in mind equation~\eqref{eq:sssh}, we obtain a short exact sequence 
\[\begin{tikzcd}[column sep=scriptsize]
	0 & {E_{1}^{p,q}\left(C\right)} & {E_{1}^{p,q}\left(D\right)} & {E_{1}^{p,q}\left(\Cone(f)[1]\right)} & 0,
	\arrow[from=1-1, to=1-2]
	\arrow["{f_1}", from=1-2, to=1-3]
	\arrow["{h[1]_1}", from=1-3, to=1-4]
	\arrow[from=1-4, to=1-5]
\end{tikzcd}\]
and the base of the induction holds.

By the induction hypothesis for $r-1$, there arises a long exact sequence from the short exact sequence of the $r-1$ pages, 
\[\begin{tikzpicture}[descr/.style={fill=white,inner sep=1.0pt}]
        \matrix (m) [
            matrix of math nodes,
            row sep=3.0em,
            column sep=1.5em,
            text height=1.5ex, text depth=0.25ex
        ]
        {& & &\ldots  \\ & {E_{r}^{p,q}\left(C\right)} & {E_{r}^{p,q}\left(D\right)} & {E_{r}^{p,q}\left(\Cone(f)[1]\right)} \\
            & {E_{r}^{p+r-1,q+2-r}\left(C\right)} & {E_{r}^{p+r-1,q+2-r}\left(D\right)} & {E_{r}^{p+r-1,q+2-r}\left(\Cone(f)[1]\right)}  \\ 
             &\ldots & & \\
        };

        \path[overlay,->, font=\scriptsize,>=latex]
        (m-1-4) edge[out=355,in=175] node[descr,yshift=0.3ex] {$\partial$} (m-2-2)
        (m-2-2) edge node[above,yshift=0.0ex]{${f_r}$} (m-2-3)
        (m-2-3) edge node[above,yshift=0.0ex]{${h[1]_r}$} (m-2-4)
        (m-2-4) edge[out=355,in=175] node[descr,yshift=0.3ex] {$\partial$} (m-3-2)
        (m-3-2) edge node[above,yshift=0.3ex]{${f_r}$} (m-3-3)
        (m-3-3) edge node[above,yshift=0.3ex]{${h[1]_r}$} (m-3-4)
        (m-3-4) edge[out=355,in=175] node[descr,yshift=0.3ex] {$\partial$} (m-4-2)
        ;
\end{tikzpicture}\]
By assumption, $f_r$ is injective. Thus, the long exact sequence splits into short exact sequences, and the result follows by induction.
\end{proof}
\subsection{Application to \texorpdfstring{$\nu_S$}{nu\_S} filtration}
We apply the general theory to the filtration on $\overline{\Cdh}$ and $\overline{\mCdh}$ induced by the filtration $\nu_S$ defined in Section~\ref{ssec:the_ring_S}. This is a special case of the setup in Section~\ref{sssec:twistedcoh}.
The relevant differentials for the spectral sequence are $\bar{\mh}^{\gamma,b}_1$ and $\bar{\mt}^{\gt,\bt}_1$ respectively. By Corollary~\ref{cor:bar_fukaya_deformed_boundry}, we have 
\begin{equation}\label{eq:m1bdt}
\bar{\mh}^{\gamma,b}_1(\eta) = (\Id \otimes d)\eta -[\bar{b},\eta]
\end{equation}
and similarly for $\bar{\mt}^{\gt,\bt}_1$. Thus, these differentials are a special case of the differential $d_{twist}$ from Section~\ref{sssec:twistedcoh}.

\subsubsection{Base calculations}
\begin{lm}\label{lm:Fpexplicit}
The canonical isomorphism of Lemma~\ref{lm:cano_iso} induces an isomorphism
\begin{align*}
    F_{\nu_S}^p \overline{\Cdh}_h&\simeq\bigoplus_{\substack{\ell \ge p\\-\ell+m=h\\0\le m\le n}}S_{-\ell}\otimes A^m(L)
= \bigoplus_{\substack{\ell \ge p\\-h\le \ell \le n-h}}S_{-\ell}\otimes A^{h+\ell}(L)\\
F_{\nu_S}^p \overline{\mCdh}_h&\simeq\bigoplus_{\substack{\ell \ge p\\-\ell+m=h\\0\le m\le n+1}}S_{-\ell}\otimes A^m(I\times L)
= \bigoplus_{\substack{\ell \ge p\\-h\le \ell \le n+1-h}}S_{-\ell}\otimes A^{h+\ell}(I\times L)
\end{align*}

\end{lm}
\begin{proof}
We calculate,
\[
F_{\nu_S}^p \overline{\Cdh}_h\simeq \bigoplus_{\substack{\ell \ge p\\-\ell+m=h\\0\le m\le n}}S_{-\ell}\otimes A^m(L)=\bigoplus_{\substack{\ell \ge p\\-h\le \ell \le n-h}}S_{-\ell}\otimes A^{h+\ell}(L).
\]
\end{proof}

\begin{lm}\label{lm:S_filration_modu}
For every $h\in \Z$,
\[F_{\nu_S}^{-h}\overline{\Cdh}_h=\overline{\Cdh}_h,\qquad F_{\nu_S}^{n-h}\overline{\Cdh}_h\simeq S_{h-n}\otimes A^{n}(L),\qquad F_{\nu_S}^{n-h+1}\overline{\Cdh}_h=0, \]
and
\[F_{\nu_S}^{-h}\left(\overline{\Cdh_{\not \e}}\right)_h=\left(\overline{\Cdh_{\not \e}}\right)_h,\qquad F_{\nu_S}^{n-h}\left(\overline{\Cdh_{\not \e}}\right)_h \simeq S_{h-n}\otimes A^{n}(L),\qquad F_{\nu_S}^{n-h+1}\left(\overline{\Cdh_{\not \e}}\right)_h=0.\]
In particular, the filtration induced by $\nu_S$ on $\overline{\Cdh_{\not \e}}$ is bounded in all degrees. Similar statements hold for $\overline{\mCdh}$ and $\overline{\mB}$.
\end{lm}
\begin{proof}
The first line is obtained by inserting $p=-h,n-h,n-h+1$ in the isomorphism of Lemma~\ref{lm:Fpexplicit} above. The second line follows by doing the same computation in the quotient~$\Cdh_{\not \e}$.
\end{proof}
Observe that $Z(S)\cdot e\subseteq \ker\left(\bar{\mh}^{\gamma,b}_1\right)$ by the Properties \ref{it:a_infty} and ~\ref{it:unit1} of Definition~\ref{dfn:cycunit}.
\begin{lm}\label{lm:E_infty_page_center}
    For $0\le\ell\le \infty$ we have $E_{\ell}^{p,q}\left(Z(S)\cdot 1\right)=S_{-p}\cdot 1$, if $2p+q=0$ and $p\in 2\N.$ Otherwise, it is trivial.
\end{lm}
\begin{proof}
    By equation~\eqref{eq:m1bdt} the differential on $Z(S)\cdot 1$ is trivial. By Corollary~\ref{cor:center_of_hat_algebra} $Z(S)=\bigoplus_{i\in \N}S_{-2i}$. Theorem~\ref{thm:filtered_spec_conv} implies that 
    \[E_{\ell}^{p,q}\left(Z(S)\cdot 1\right)=E_{0}^{p,q}\left(Z(S)\cdot 1\right)=\frac{F_{\nu_S}^p \left(Z(S)\cdot 1\right)_{p+q}}{F_{\nu_S}^{p+1} \left(Z(S)\cdot 1\right)_{p+q}}\cong Z(S)_{-p}\otimes (\R)_{2p+q}.\]
    The claim follows by combining the above observations.
\end{proof}

Let $\overline{\mB}_0$ be the submodule of $\overline{\mCdh}$ defined by
\[\left\{a\in \overline{\mCdh}: j_i^*(a)=0 \textit{ for } i=0,1\right\}.\] 
The submodule $\overline{\mB}_0$ is preserved by $\overline{\mgtb_1}$ and thus is a subcomplex. Recall Definition~\ref{dfn:inv_shift} concerning the action of an anti-symplectic involution $\phi$ on shifted complexes. In particular, taking the subcomplex fixed by $-\phi^*$ does not commute with shifts.
For $v \in \Z$ and $\invs=Id$, or for $v\in 2\Z$ and $\invs=-\phi^*$, let $\mathcal{A}$  denote the cochain complex $\overline{\Cdh}[v]^{\invs}, \overline{\mCdh}[v]^{\invs}$ or $\overline{\mB}_0[v]^{\invs}$ with boundary operator $\m_1=\bar{\mh}^{\gamma,b}_1[v], \bar{\mt}^{\gt,\bt}_1[v]$ or $\bar{\mt}^{\gt,\bt}_1[v]$. When $\invs = -\phi^*$, we additionally assume that $b$ and $\bt$ are real, so that Lemma~\ref{lm:phi_star_cochainmap} ensures $\mathcal{A}$ is indeed a complex.  Let $\mathcal{L} = L, I \times L$ or $(I \times L,\partial I \times L)$. Let $(E_0^{p,q}(\mathcal{A}), d_0)$ be the $E_0$ page of the spectral sequence associated with $(\mathcal{A}, \m_1, F_{\nu_S})$. Define 
\[
A^{p,q}(\mathcal{L}) := (S[v]^{\invs})_{-p-v} \otimes A^{2p+q+v}(\mathcal{L}). 
\]
Let 
\[
\iota_{p,q}:A^{p,q}(\mathcal{L})\to E_0^{p,q}(\mathcal{A})
\]
denote the map induced by the inverse of the canonical isomorphism  of Lemma~\ref{lm:cano_iso} composed with the quotient map $\mathcal{A} \to E_0(\mathcal A)$. 

\begin{lm}\label{lm:d_zero_induced_mbar_one}
The operator $d_0 : E_0^{p,q}(\mathcal{A}) \to E_0^{p,q+1}(\mathcal{A})$ agrees with the map induced from $\m_1:F^{p}_{\nu_S}(\mathcal{A})_{p+*}\to F^{p}_{\nu_S}(\mathcal{A})_{p+*+1}.$ 
\end{lm}
\begin{proof}
     We write the proof for $\mathcal{A}=\overline{\Cdh}[v]^{\invs}$. The other cases are similar. Let $\a^\prime \in \Cdh[v]^{\invs}$ represent an element $\a \in E_0$. As $\nu_S\otimes \nu(b)>0$, it follows from Lemma~\ref{lm:deformed_operator_explicit_formula} and property~\ref{it:val} of Definition~\ref{dfn:cycunit} that 
     \[[\overline{\mh^{\gamma,b}_1[v](\a^\prime)}]=\left[\bar{\mh}^{\gamma}_1[v](\overline{\alpha^{\prime}}) \right].\] 
     By Theorem~\ref{thm:filtered_spec_conv} $d_0(\a)=[\overline{\mh^{\gamma,b}_1[v](\a^\prime)}]$, and the result follows.
\end{proof}
\begin{lm}\label{lm:E_0_isom_diff_forms}
The map $\iota_{p,\bullet}: (A^{p,\bullet}(\mathcal{L}), \Id\otimes d)\to (E_0^{p,\bullet}(\mathcal{A}),d_0)$ is a cochain isomorphism.
\end{lm}
\begin{proof}
     We write the proof for $\mathcal{A}=\overline{\Cdh}[v]^{\invs}$. The other cases are similar. By Theorem~\ref{thm:filtered_spec_conv},
\[
\begin{aligned}
E_0^{p,q}\left(\overline{\Cdh}[v]^{\invs}\right)
&=
\frac{F_{\nu_S}^{p}\left(\overline{\Cdh}[v]^{\invs}\right)_{p+q}}
     {F_{\nu_S}^{p+1}\left(\overline{\Cdh}[v]^{\invs}\right)_{p+q}}\\
&\overset{\ref{lm:Fpexplicit}}{\cong}
(S[v]^{\invs})_{-p-v}\otimes A^{2p+q+v}(L)
=
A^{p,q}(L).
\end{aligned}
\]
Thus $\iota_{p,q}$ is an isomorphism of modules. By Lemma~\ref{lm:d_zero_induced_mbar_one}, the differential $d_0:E_0^{p,*}(\mathcal{A})\to E_0^{p,*+1}(\mathcal{A})$ is the map induced from $\bar{\mh}^{\gamma}_1[v]:F^{p}_{\nu_S}\left(\overline{\Cdh}[v]^{\invs}\right)_{p+*}\to F^{p}_{\nu_S}\left(\overline{\Cdh}[v]^{\invs}\right)_{p+*+1}$. It follows from Proposition~\ref{prop:sf_extension_arises_dga} that 
\[
d_0 \circ \iota_{p,\bullet}=\iota_{p,\bullet+1} \circ \left(\Id \otimes d \right) .
\]
So, $\iota_{p,\bullet}$ is a cochain map.
\end{proof}

\subsubsection{The spectral sequence for twisted relative cohomology}
Here, we collect results that are used to analyze the spectral sequence for $(\overline{\mB}_0, \bar{\mt}^{\gt,\bt}_1)$ culminating in Corollary~\ref{cor:B_zero_Cdh_spectral_iso}.
For $i\in\{0,1\}$, set $\gamma_0=\gamma$ and $\gamma_1=\gamma'$.

\begin{cor}\label{cor:C_mC_retract}
    Let $i=1$ or $i=0$. The map $j_i^*:\left(\overline{\mCdh}[v]^{\invs},\bar{\mt}^{\gt,\bt}_1\right)\to \left(\overline{\Cdh}[v]^{\invs},\bar{\mh}_1^{\gamma_i,j_i^*(\bt)}\right)$ induces an isomorphism $E_1^{p,q}(\overline{\mCdh}[v]^{\invs})\to E_1^{p,q}(\overline{\Cdh}[v]^{\invs}).$
\end{cor}

\begin{proof}
    Consider the commutative diagram 
\[\begin{tikzcd}
	{H^q\left(A^{p,\bullet}(I\times L)\right)} && {H^q\left(A^{p,\bullet}( L)\right)} \\
	\\
	{E_1^{p,q}(\overline{\mCdh})} && {E_1^{p,q}(\overline{\Cdh})}
	\arrow["{j_i^*}"', from=1-1, to=1-3]
	\arrow["{(\iota_{p,q})_*}", from=1-1, to=3-1]
	\arrow["{(\iota_{p,q})_*}"', from=1-3, to=3-3]
	\arrow["{j_i^*}", from=3-1, to=3-3]
\end{tikzcd}\]
The vertical maps are isomorphisms by Lemma~\ref{lm:E_0_isom_diff_forms}. Observe that $j_i$ is a deformation retract of topological spaces. Thus, the top horizontal map is an isomorphism. It follows that the bottom horizontal map is also an isomorphism.
\end{proof}
\begin{cor}\label{cor:spec_isom}
    For $\ell\ge 1$, the map $j_i^*:\left(\overline{\mCdh}[v]^{\invs},\bar{\mt}^{\gt,\bt}_1\right)\to \left(\overline{\Cdh}[v]^{\invs},\bar{\mh}_1^{\gamma_i,j_i^*(\bt)}\right)$ induces an isomorphism $E_\ell^{p,q}(\overline{\mCdh}[v]^{\invs})\to E_\ell^{p,q}(\overline{\Cdh}[v]^{\invs})$.
\end{cor}
\begin{proof}
    Corollary~\ref{cor:C_mC_retract} and Theorem~\ref{thm:spectralquasi_iso_iso} imply the result. 
\end{proof}
In light of Corollary \ref{cor:spec_isom}, the spectral sequence associated with $(\overline{\Cdh},\bar{\mh}_1^{\gamma_i,j_i^*(\bt)},F_{\nu_S})$ 
is independent up to isomorphism of the choice of $i\in \{0,1\}$ for $\ell\ge 1$. Thus, by abuse of notation, we write $E_\ell^{p,q}(\overline{\Cdh})$ for both.

\begin{cor}\label{cor:cone_seq_exact_E_infty}
For the filtered cochain complex map \[(j_0^*,j_1^*):\left(\overline{\mCdh}[v]^{\invs},\bar{\mt}^{\gt,\bt}_1\right)\to \left(\overline{\Cdh}[v]^{\invs}\oplus\overline{\Cdh}[v]^{\invs},\bar{\mh}_1^{\gamma,j_0^*(\bt)}[v]\oplus \bar{\mh}_1^{\gamma',j_1^*(\bt)}[v]\right),\]
we have for $1\le \ell\le \infty,$
\[
E_{\ell}^{p,q+1}\left(\Cone((j^*_0, j_1^*))\right)\cong {E_{\ell}^{p,q}\left(\overline{\Cdh}[v]^{\invs}\right)}.\]
\end{cor}
\begin{proof}
    The filtration on $\overline{\mCdh}[v]$ and $\overline{\Cdh}[v]$ is bounded in each degree by Lemma~\ref{lm:S_filration_modu}, and hence the spectral sequence converges for each $p$ and $q$ for some $r$. By Corollary~\ref{cor:spec_isom}, the assumptions of Lemma~\ref{lm:cone_sequ_exact_generalized} are satisfied. Hence, there exists an exact sequence of modules
\[\begin{tikzcd}[column sep=scriptsize]
	0 & {E_{\ell}^{p,q}\left(\overline{\mCdh}[v]^{\invs}\right)} & {E_{\ell}^{p,q}\left(\overline{\Cdh}[v]^{\invs}\oplus\overline{\Cdh}[v]^{\invs}\right)} & {E_{\ell}^{p,q+1}\left(\Cone((j^*_0, j_1^*))\right)} & 0.
	\arrow[from=1-1, to=1-2]
	\arrow["{(j_0^*,j_1^*)}", from=1-2, to=1-3]
      \arrow[from=1-3, to=1-4]
	\arrow[from=1-4, to=1-5]
\end{tikzcd}\]
The results now follow from Corollary~\ref{cor:spec_isom}.
\end{proof}

\begin{lm}\label{lm:B_zero_cone_quasi_iso}
    The inclusion $\iota:\overline{\mB}_0[v]^{\invs}\to \Cone((j^*_0, j_1^*))$ is a filtered morphism.  Moreover, it induces an isomorphism $E_\ell^{p,q}\left(\overline{\mB}_0[v]^{\invs}\right)\cong E_\ell^{p,q}\left(\Cone((j^*_0, j_1^*))\right),$ for all $\ell\ge 1.$ 
\end{lm}
\begin{proof}
    Fix $p$ and let $\rho_t^*: A^{p,\bullet}(I\times L) \to A^{p,\bullet}(L)$ for $t=0,1$ denote the restriction maps. Consider the following commutative diagram.
    \[
    \begin{tikzcd}
        {A^{p,\bullet}(I\times L,\partial I\times L)} && {E_0^{p,\bullet}\left(\overline{\mB}_0\right)} \\
        \\
        {\Cone((\rho_0^*, \rho_1^*))} && {E_0^{p,\bullet}\left(\Cone((j^*_0, j_1^*))\right)}
        \arrow["{\iota_{p,\bullet}}", from=1-1, to=1-3]
        \arrow["i"', from=1-1, to=3-1]
        \arrow["i", from=1-3, to=3-3]
        \arrow["{\iota_{p,\bullet}[-1]\oplus\iota_{p,\bullet}[-1]\oplus\iota_{p,\bullet}}"', from=3-1, to=3-3]
    \end{tikzcd}
    \]
    It follows from Lemma~\ref{lm:E_0_isom_diff_forms} that the horizontal arrows are isomorphisms. The left vertical arrow is the standard quasi-isomorphism taking a relative form to the mapping cone of the restrictions. Consequently, the right vertical arrow is a quasi-isomorphism. The result follows by applying Theorem~\ref{thm:spectralquasi_iso_iso} to the morphism $i: \overline{\mB}_0\to \Cone((j^*_0, j_1^*))$.
\end{proof}
\begin{cor}\label{cor:B_zero_Cdh_spectral_iso}
    For all $1 \le \ell \le \infty$, there is an isomorphism
    \[
    E_\ell^{p,q}\left(\overline{\mB}_0[v]^{\invs}\right) \cong E_\ell^{p,q-1}\left(\overline{\Cdh}[v]^{\invs}\right).
    \]
\end{cor}
\begin{proof}
    This follows from Lemma~\ref{lm:B_zero_cone_quasi_iso} and Corollary~\ref{cor:cone_seq_exact_E_infty}.
\end{proof}
\begin{rem}
It is possible to calculate the spectral sequence for $(\overline{\mB}_0, \bar{\mt}^{\gt,\bt}_1)$ using Corollary~\ref{cor:spec_isom} and the short exact sequence $0 \to \overline{\mB}_0 \to \overline{\mCdh} \to \overline{\Cdh} \oplus \overline{\Cdh} \to 0.$ The boundary homomorphism of the associated long exact sequence reflects the twisted differential. We use the cone complex instead to avoid explicitly calculating this boundary homomorphism.
\end{rem}

\subsubsection{Relations between the spectral sequences of the pseudoisotopy}
\begin{lm}\label{lm:shrtexct_B_B_Z}
    We have a short exact sequence
\[\begin{tikzcd}[column sep=scriptsize]
	0 & {\overline{\mB}_0[v]^{\invs}} & {\overline{\mB}[v]^{\invs}} & {Z(S)\cdot e[v]^{\invs}\oplus Z(S)\cdot e[v]^{\invs}} & 0
	\arrow[from=1-1, to=1-2]
	\arrow["\varkappa", hook, from=1-2, to=1-3]
	\arrow["{(j^*_0,j^*_1)}", from=1-3, to=1-4]
	\arrow[from=1-4, to=1-5]
\end{tikzcd}\]
\end{lm}
\begin{proof}
By Corollary~\ref{cor:bar_fukaya_deformed_boundry}, we have $\bar{\mt}^{\gt,\bt}_1(\eta) = (\Id \otimes d)\eta -[\bar{\bt},\eta].$
    It follows from Corollary~\ref{cor:center_of_hat_algebra} that $Z(S)\cdot e\oplus Z(S)\cdot e$ is a subcomplex of $\overline{\Cdh}\oplus\overline{\Cdh}$. So, the short exact sequence of modules
    \[\begin{tikzcd}[column sep=scriptsize]
	0 & {\overline{\mB}_0^{\invs}} & {\overline{\mB}^{\invs}} & {Z(S)^{\invs}\cdot e\oplus Z(S)^{\invs}\cdot e} & 0,
	\arrow[from=1-1, to=1-2]
	\arrow["\varkappa", hook, from=1-2, to=1-3]
	\arrow["{(j^*_0,j^*_1)}", from=1-3, to=1-4]
	\arrow[from=1-4, to=1-5]
    \end{tikzcd}\]
    is a short exact sequence of cochain complexes.
\end{proof}

\begin{lm}\label{lm:even_bottom_cocycles}
Suppose that $p$ is even.
For every $\ell\ge1$ and $q\in\Z$, the differential
\[
    d_\ell:
    E_\ell^{p,q}(\overline{\mB}^{\kappa}[v])
    \longrightarrow
    E_\ell^{p+l,q+1-l}(\overline{\mB}^{\kappa}[v])
\]
is trivial. The same holds with $\overline{\mB}$ replaced by
$\overline{\mB}_{\not\e}$.
\end{lm}

\begin{proof}
For simplicity, we assume that $v=0$. The proof for $v\ne0$
is similar.
We define a splitting of the quotient map $F_{\nu_S}^p(\overline{\mCdh}^{\kappa})_{p+q} \to E_0^{p,q}(\overline{\mCdh}^{\kappa})$ by the following diagram.
\[
\begin{tikzcd}[column sep=small]
E_0^{p,q}(\overline{\mCdh}^{\kappa})
\arrow[r,hook]
\arrow[d,"\cong"',"{\text{Lemma~\ref{lm:E_0_isom_diff_forms}}}"]
&
F_{\nu_S}^p(\overline{\mCdh}^{\kappa})_{p+q}
\arrow[d,"\cong"',"{\text{Lemma~\ref{lm:Fpexplicit}}}"]
\\
(S^\kappa)_{-p}\otimes A^{2p+q}(I\times L)
\arrow[r,hook]
&
\bigl((S^\kappa)_{-p}\otimes A^{2p+q+\bullet}(I\times L)\bigr)
\oplus F_{\nu_S}^{p+1}(\overline{\mCdh}^{\kappa})_{p+q}
\end{tikzcd}
\]
Since $p$ is even, the commutator with $\bar{\bt}$ vanishes
on $\bigl((S^\kappa)_{-p}\otimes A^{2p+q}(I\times L)\bigr).$ Thus, by~\eqref{eq:m1bdt} the horizontal
arrows are cochain maps.

Recall the definition of $\mB$ from~\eqref{eq:defmB}. We can replace $\mCdh$ with $\mB$ in the above diagram because taking $S$-degree components commutes with $j_0^*,j_1^*$
and preserves $Z(\Rdh)\cdot\e$.
Recall the definition of $\mB_{\not\e}$ from~\eqref{eq:defnote}. We obtain a similar diagram with $\overline{\mB}_{\not\e}^{\kappa}$ in place of $\mCdh$ because taking $S$-degree components preserves
$((Z(\mRdh)\cap\ker d_{\mRdh})\cdot\e)^\kappa$.
The existence of cochain maps splitting the quotient maps implies that each element in $E_1^{p,q}(\overline{\mB})$ (resp. $E_1^{p,q}(\overline{\mB}_{\not \e})$) has a representative in $F^p(\overline{\mB})_{p+q}$ (resp. $F^p(\overline{\mB_{\not\e}})_{p+q}$)  
that is annihilated by $\bar{\mt}^{\gt,\bt}_1$, and thus the representative belongs to $Z_{\infty}^{p,q}(\overline{\mB})$ (resp. $Z_{\infty}^{p,q}(\overline{\mB}_{\not \e})$).
Therefore, $d_\ell=0$ for every $\ell\ge1$ in both cases.
\end{proof}
\begin{lm}\label{lm:B_zero_almost_iso_B}
     Let $\ell\ge 1$ be an integer.
      The map $\varkappa_\ell^{p,q}:E_\ell^{p,q}\left(\overline{\mB}_0[v]^{\invs}\right)\to E_\ell^{p,q}\left(\overline{\mB}[v]^{\invs}\right)$ is an isomorphism if $2p+q+v \ne 0,1$ or $p$ is odd. Moreover $E_{\ell}^{p,q}\left(\overline{\mB}[v]^{\invs}\right)=0$, if $p$ is even and $2p+q+v=1.$ If $p$ is even and $2p+q+v=0$, then $E_{\ell}^{p,q}\left(\overline{\mB}[v]^{\invs}\right)\simeq\left(S[v]^{\invs}\right)_{-p-v}\otimes 1$. 
\end{lm}
\begin{proof}
    For simplicity, we assume that $v=0$. The proof for $v\ne 0$ is similar. We proceed by induction on $\ell$.
    
    Base case $\ell=1$:
    We determine the behavior of $\varkappa_1^{p,q}$ using the long exact sequence associated with the short exact sequence in Lemma~\ref{lm:shrtexct_B_B_Z}:
    \[
    \begin{tikzcd}[column sep=small]
        \cdots \to E_{1}^{p,q-1}\left(K\right) \xrightarrow{\partial} E_{1}^{p,q}\left(\overline{\mB}_0^{\invs}\right) \xrightarrow{\varkappa_1} E_{1}^{p,q}\left(\overline{\mB}^{\invs}\right) \xrightarrow{(j^*_0,j^*_1)} E_{1}^{p,q}\left(K\right) \to \cdots
    \end{tikzcd}
    \]
    where $K = Z(S)^{\invs}\cdot e \oplus Z(S)^{\invs}\cdot e$. We analyze the map $\varkappa_1^{p,q}$ in four cases:
    \begin{itemize}
        \item If $p$ is odd: The terms $E_{1}^{p,q}(K)$ and $E_{1}^{p,q-1}(K)$ vanish by Lemma~\ref{lm:E_infty_page_center}. Thus, $\varkappa_1^{p,q}$ is an isomorphism.
        
        \item If $p$ is even and $2p+q \ne 0, 1$: The condition $2p+q \ne 0$ ensures $E_{1}^{p,q}(K)=0$ by Lemma~\ref{lm:E_infty_page_center}. Similarly, the condition $2p+q \ne 1$ implies  $2p+(q-1) \ne 0$, so $E_{1}^{p,q-1}(K)=0$. Consequently, $\varkappa_1^{p,q}$ is an isomorphism.
        
        \item If $p$ is even and $2p+q=0$: Lemma~\ref{lm:E_0_isom_diff_forms} implies $E_0^{p,q-1}\left(\overline{\mB}^{\invs}\right)=E_0^{p,q-1}\left(\overline{\mCdh}^\invs\right)=0$, so
        \[
        E_{1}^{p,q}\left(\overline{\mB}^{\invs}\right)=\ker d_0|_{\overline{\mB}}\subset \ker d_0=E_1^{p,q}\left(\overline{\mCdh}^\invs\right)\subset E_0^{p,q}\left(\overline{\mCdh}^\invs\right).
        \]
        By Lemma~\ref{lm:E_0_isom_diff_forms}, we have $E_1^{p,q}\left(\overline{\mCdh}^\invs\right)= \iota_{p,q}((S^{\invs})_{-p}\otimes 1) \subset E_1^{p,q}\left(\overline{\mB}^{\invs}\right)$. Thus, $E_{1}^{p,q}\left(\overline{\mB}^{\invs}\right)\simeq (S^{\invs})_{-p}\otimes 1$.
        
        \item If $p$ is even and $2p+q=1$: By Corollary~\ref{cor:B_zero_Cdh_spectral_iso}, and Lemma~\ref{lm:E_0_isom_diff_forms} we identify $E_{1}^{p,q}\left(\overline{\mB}_0^{\invs}\right)\cong (S^{\invs})_{-p}\otimes 1$ and $E_{1}^{p,q-1}\left(\overline{\mB}_0^{\invs}\right)=0$. By the result of the preceding case, $E_1^{p,q-1}\left(\overline{\mB}^{\invs}\right)\cong (S^{\invs})_{-p}\otimes 1$. By Lemma~\ref{lm:E_infty_page_center}, we have 
        \[
        E_1^{p,q-1}\left(K\right)\cong (S^{\invs})_{-p}\otimes 1 \oplus (S^{\invs})_{-p}\otimes 1
        \]
        and $E_1^{p,q}\left(K\right) =0.$ Substituting these calculations into the relevant segment of the exact sequence yields:
        \[0\to(S^{\invs})_{-p}\otimes 1 \xrightarrow{(j^*_0,j^*_1)} (S^{\invs})_{-p}\otimes 1 \oplus (S^{\invs})_{-p}\otimes 1 \xrightarrow{\partial} (S^{\invs})_{-p}\otimes 1 \xrightarrow{\varkappa_1} E_{1}^{p,q}\left(\overline{\mB}^{\invs}\right) \to 0.\] 
Exactness and finite dimensionality implies $E_{1}^{p,q}\left(\overline{\mB}^{\invs}\right)=0$.
    \end{itemize}
    This establishes the base case.

    Inductive step: We assume $2p+q \ge 0$, otherwise $E_{\ell}^{p,q}(\overline{\mB}_{0})=E_{\ell}^{p,q}(\overline{\mB})=0$ by Lemma~\ref{lm:E_0_isom_diff_forms} and Theorem~\ref{thm:filtered_spec_conv}. For $p$ even, the case $2p+q=1$ follows from
$E_1^{p,q}(\overline{\mB}^{\invs})=0$, while the case $2p+q=0$
follows from $E_{1}^{p,q}\left(\overline{\mB}^{\invs}\right)\simeq (S^{\invs})_{-p}\otimes 1$, the form of the differential~\eqref{eq:m1bdt} and the vanishing for $2p+q<0$.
    
    Assume by induction that $\varkappa_\ell^{p,q}$ is an isomorphism if $2p+q \ne 0, 1$ or if $p$ is odd. 
    By \cite[Lemma~1.2.3]{peschke2022exactcouplesspectralsequences}, $\varkappa_{\ell+1}^{p,q}$ is an epimorphism if $\varkappa_{\ell}^{p+\ell,q+1-\ell}$ is a monomorphism. By the induction hypothesis $\varkappa_{\ell}^{p+\ell,q+1-\ell}$ is a monomorphism if $2(p+\ell) + (q+1-\ell)\ne 0,1$. This is indeed the case since
    \[
    2(p+\ell) + (q+1-\ell) = 2p+q + 1 + \ell > 1.
    \]
    Thus, $\varkappa_{\ell+1}^{p,q}$ is an epimorphism.

    We now show that $\varkappa_{\ell+1}^{p,q}$ is a monomorphism.
    Let $(p', q') = (p-\ell, q-1+\ell)$. By \cite[Lemma~1.2.3]{peschke2022exactcouplesspectralsequences}, $\varkappa_{\ell+1}^{p,q}$ is a monomorphism if $\varkappa_{\ell}^{p',q'}$ is an epimorphism. By the induction hypothesis, this is the case if $2p'+q' \ne 0$ or $p'$ is odd. 
    
    Otherwise $p'$ is even and $2p'+q'=0$.
In this case,     $d_\ell:E_\ell^{p',q'}(\overline{\mB}^{\invs})
\to E_\ell^{p,q}(\overline{\mB}^{\invs})$ is trivial by Lemma~\ref{lm:even_bottom_cocycles}.
Since the differential entering $E_\ell^{p,q}$ vanishes, no new relations are introduced in $E_{\ell+1}^{p,q}$. Hence, $\varkappa_{\ell}^{p,q}$ is an isomorphism implies that $\varkappa_{\ell+1}^{p,q}$ must be a monomorphism.
 
\end{proof}
When using the spectral sequence to study the obstruction classes of Section~\ref{ssec:obs_theory},
we find that
$E_\infty^{p,q}\left(\overline{\mB}_{\not\e}\right)$ is unnaturally large. Indeed, recalling~\eqref{eq:defnote}, we have 
\[
\overline{\mB}_{\not\e} = \overline{\mB/\big(Z(\mRdh)\cap\ker d_{\mRdh}\big)\cdot\e}.
\]
It is shown in the proof of the following lemma that $\overline{\big(Z(\mRdh)\cap\ker d_{\mRdh}\big)\cdot\e}$ contains a copy of $Z(S) \otimes A^1(I)$ in form degree $1$ while it contains only $Z(S) \otimes 1$ in form degree $0.$ Consequently there is a copy of $Z(S) \otimes A^0(I)/\R$ in $\ker(\bar{\mt}^{\gt,\bt}_1) \subset \mB_{\not\e},$ which survives in $E^{p,q}_\infty(\mB_{\not\e})$ for $p$ even and $2p + q = 0.$
Nevertheless, since we will be interested in classes represented by cochains
that are already closed in $\overline{\mB}$, we can use the following lemma to bring the situation under control.
\begin{lm}\label{lm:B_not_e_almost_iso}
      Let $Q:\overline{\mB}[v]^{\invs}\to \overline{\mB}_{\not \e}[v]^{\invs}$ be the quotient map, and let $1\le \ell \le \infty$. Then, for $p$ odd, or $2p+q+v\ne 0$, the map $Q_{\ell}^{p,q}:E_\ell^{p,q}\left(\overline{\mB}[v]^{\invs}\right)\to E_\ell^{p,q}\left(\overline{\mB}_{\not \e}[v]^{\invs}\right)$ is an isomorphism. For $p$ even and $2p+q+v=0$, the map $Q_{\ell}^{p,q}:E_\ell^{p,q}\left(\overline{\mB}[v]^{\invs}\right)\to E_\ell^{p,q}\left(\overline{\mB}_{\not \e}[v]^{\invs}\right)$ is trivial.
\end{lm}
\begin{proof}
    For simplicity, we assume that $v=0$. The proof for $v\ne 0$ is similar.
    Recall that $S$ has a trivial differential and that $\overline{\mRdh}\simeq S\otimes A^*(I)$. Hence,
    \[
        \overline{\left(Z(\mRdh)\cap\ker d_{\mRdh}\right)\cdot\e}
        =\left(Z(S)\otimes\ker(d_{A^*(I)})\right)\cdot\e.
    \]
    Set $\mathcal{Z}_I:=Z(S)\otimes\ker(d_{A^*(I)})$.
    By definition and Lemma~\ref{lm:C_not_e_inf_alg}, we therefore have a short exact sequence of filtered complexes
    \[
        0\to \left(\mathcal{Z}_I\cdot\e\right)^\kappa
        \hookrightarrow \overline{\mB}^\kappa
        \overset{Q}{\longrightarrow} \overline{\mB}_{\not \e}^\kappa \to 0,
    \]
    where the differential on the leftmost complex is trivial.
    The action of $\kappa$ is trivial on $A^*(I)$. Since
    $Z(S)=\bigoplus_{i\in \N}S_{-2i}$, Theorem~\ref{thm:filtered_spec_conv}
    implies that for $0\le \ell\le \infty,$
    \begin{align}
        E_{\ell}^{p,q}\left(\left(\mathcal{Z}_I\cdot\e\right)^\kappa\right)
        &\cong E_{0}^{p,q}\left(\left(\mathcal{Z}_I\cdot\e\right)^\kappa\right) \nonumber\\
        &\cong \frac{F_{\nu_S}^p \left(Z(S)^\kappa\otimes \ker (d_{A^*(I)})\right)_{p+q}}{F_{\nu_S}^{p+1} \left(Z(S)^\kappa\otimes \ker (d_{A^*(I)})\right)_{p+q}} \nonumber\\
        &\cong\begin{cases}
            (S^\kappa)_{-p}\otimes 1, & p\in 2\N,\; 2p+q=0,\\
            (S^\kappa)_{-p}\otimes A^1(I), & p\in 2\N,\; 2p+q=1,\\
            0, & \text{otherwise}.
        \end{cases}\label{eq:ZI_calc}
    \end{align}

    We use this computation to analyze the map $Q_1^{p,q}$.
    The short exact sequence of complexes induces a long exact sequence on the $E_1$ page:
    \[
    \begin{tikzcd}[column sep=small]
        E_{1}^{p,q-1}\left(\overline{\mB}_{\not \e}^\kappa\right) \xrightarrow{\partial} E_{1}^{p,q}\left((\mathcal{Z}_I\cdot\e)^\kappa\right) \to E_{1}^{p,q}\left(\overline{\mB}^\kappa\right) \xrightarrow{Q_1^{p,q}} E_{1}^{p,q}\left(\overline{\mB}_{\not \e}^\kappa\right) \xrightarrow{\partial} E_{1}^{p,q+1}\left((\mathcal{Z}_I\cdot\e)^\kappa\right)
    \end{tikzcd}
    \]
    We analyze $Q_1^{p,q}$ by distinguishing four cases depending on $2p+q$ and the parity of $p.$
    \begin{itemize}
        \item Case 1: $2p+q<0$. By Lemma~\ref{lm:E_0_isom_diff_forms} and Theorem~\ref{thm:filtered_spec_conv} we have 
        \[
        E_{1}^{p,q}\left(\overline{\mB}_{\not \e}^\kappa\right)=E_{1}^{p,q}\left(\overline{\mB}^\kappa\right)=0.
        \]
        \item Case 2: $p$ is even and $2p+q=0$.
        From Lemma~\ref{lm:B_zero_almost_iso_B}, we know
        \[
        E_{1}^{p,q}(\overline{\mB}^\kappa)
        \cong (S^\kappa)_{-p}\otimes 1.
        \]
        Calculation~\eqref{eq:ZI_calc} gives
        $E_{1}^{p,q}((\mathcal{Z}_I\cdot\e)^\kappa)
        \cong (S^\kappa)_{-p}\otimes 1$ as well, while
        $E_{1}^{p,q-1}\left(\overline{\mB}_{\not \e}^\kappa\right)=0$
        because $2p+q-1<0$. Thus, the map from
        $E_{1}^{p,q}((\mathcal{Z}_I\cdot\e)^\kappa)$ to
        $E_{1}^{p,q}(\overline{\mB}^\kappa)$ is an isomorphism, which
        forces $Q_1^{p,q}$ to be the zero map.

        \item Case 3: $p$ is even and $2p+q=1$.
        We calculated in~\eqref{eq:ZI_calc} that $E_{1}^{p,q+1}((\mathcal{Z}_I\cdot\e)^\kappa)\cong 0$. So, $Q_1^{p,q}$ is surjective. From Lemma~\ref{lm:B_zero_almost_iso_B}, we have $E_{1}^{p,q}(\overline{\mB}^\kappa)=0$. Consequently,  $Q_1^{p,q}$ is an isomorphism of trivial modules. 
        
        \item Case 4: $p$ odd, or $p$ even with $2p+q >1$.
        In this range, calculation~\eqref{eq:ZI_calc} shows that $E_{1}^{p,q}((\mathcal{Z}_I\cdot\e)^\kappa) = 0$ and $E_{1}^{p,q+1}((\mathcal{Z}_I\cdot\e)^\kappa) = 0$.
        Thus, $Q_1^{p,q}$ is an isomorphism.
    \end{itemize}

    This establishes the result for $\ell=1$. The argument for general $\ell$ proceeds by induction as in the proof of Lemma~\ref{lm:B_zero_almost_iso_B}.

    Inductive step: If $2p+q<0$, the case 1 of the inductive step establish the isomorphism in all pages.
Thus, assume $2p+q\ge 0$.
For $p$ even and $2p+q=0$, the equality $Q_1^{p,q}=0$
implies $Q_\ell^{p,q}=0$ for every $\ell\ge 1$.

Assume by induction that the result holds for $\ell$, and
let $(p,q)$ satisfy $2p+q\ne 0$ or $p$ odd.
By \cite[Lemma~1.2.3]{peschke2022exactcouplesspectralsequences},
$Q_{\ell+1}^{p,q}$ is an epimorphism if
$Q_\ell^{p+\ell,q+1-\ell}$ is a monomorphism.
By the induction hypothesis, this is indeed the case since
\[
    2(p+\ell)+(q+1-\ell)=2p+q+1+\ell>1.
\]
Thus, $Q_{\ell+1}^{p,q}$ is an epimorphism.

We now show that $Q_{\ell+1}^{p,q}$ is a monomorphism.
Let $(p',q')=(p-\ell,q-1+\ell)$.
By \cite[Lemma~1.2.3]{peschke2022exactcouplesspectralsequences},
$Q_{\ell+1}^{p,q}$ is a monomorphism if $Q_\ell^{p',q'}$
is an epimorphism. By the induction hypothesis, this is
the case if $2p'+q'\ne 0$ or $p'$ is odd.

Otherwise, $p'$ is even and $2p'+q'=0$. So Lemma~\ref{lm:even_bottom_cocycles} applies and the differential
\[
    d_\ell:
    E_\ell^{p',q'}(\overline{\mB}_{\not\e}^{\kappa})
    \longrightarrow
    E_\ell^{p,q}(\overline{\mB}_{\not\e}^{\kappa})
\]
is trivial.
Since the differential entering
$E_\ell^{p,q}(\overline{\mB}_{\not\e}^{\kappa})$
vanishes, no new relations are introduced in
$E_{\ell+1}^{p,q}(\overline{\mB}_{\not\e}^{\kappa})$.
So, since $Q_\ell^{p,q}$ is an isomorphism,
$Q_{\ell+1}^{p,q}$ must be a monomorphism.
This completes the induction.
\end{proof}

\subsection{Proofs of main theorems}\label{subsec:proofs_of_main_theorems}

\begin{proof}[Proof of Theorem~\ref{thm:spectral_conv}]
From Theorem~\ref{thm:hat_spec_seq}, we obtain 
\[E_1^{p,q}\left(\overline{\Cdh}\right)\cong S_{-p}\otimes H^{q+2p}(L,d).\]
Furthermore, $E_1^{p,q} = E_{n-1}^{p,q}$, i.e., the differentials $d_r$ for $1 \leq r < n-1$ are all trivial, by Lemma~\ref{lm:energy_zero_degree_1} and Lemma~\ref{lm:differential_formula_ell_page}.

It remains to show that $E_n^{p,q} = E_\infty^{p,q}$, i.e., the spectral sequence degenerates at the $E_n$ page. Observe that for any pair $(p,q)$ where $q+2p \neq 0$, either $E_1^{p,q} = 0$ or $E_1^{p+n-1,q+2-n} = 0$. This is because $E_1^{p+n-1,q+2-n} \cong S_{-(p+n-1)} \otimes H^{q+2p+n}(L,d)$, which requires $0 \leq q+2p+n \leq n$ to be non-zero, implying $q+2p = 0$.

A similar argument shows that for any $r \geq n$, the differential $d_r: E_r^{p,q} \to E_r^{p+r,q-r+1}$ must be trivial since the target involves a cohomology degree of $q+2p+r+1 > n$, which exceeds the dimension of $L$. Hence $E_n^{p,q} = E_\infty^{p,q}$.

The critical step is computing the differential $d_{n-1}: E_{n-1}^{p,q} \to E_{n-1}^{p+n-1,q+2-n}$. For pairs $(p,q)$ where $2p+q \neq 0$, our earlier discussion shows that $d_{n-1}$ is trivial.

Let us focus on the case where $2p+q = 0$. From Theorem~\ref{thm:hat_spec_seq} and Corollary~\ref{cor:center_of_hat_algebra}, we can deduce that $d_{n-1}$ is trivial if either $n$ is odd or $p$ is even. For the remaining case (when $n$ is even and $p$ is odd), Lemma~\ref{lm:differential_formula_ell_page} and Corollary~\ref{cor:center_of_hat_algebra} tell us that the differential is given by:
\[d_{n-1}(a) = -2[\bar{b}] \cup a.\]

To determine whether this map is an isomorphism, we consider $h: S_{-p} \to S_{-p+1-n}$ defined by $a \mapsto a \cdot (-2\int_L \bar{b})$. This map can be expressed as a composition:
\[a \mapsto a \otimes 1 \xrightarrow{d_{n-1}} -2a \cdot 1 \cup [\bar{b}] \xrightarrow{\cap [L]} -2a \int_L \bar{b}.\]

Since multiplication by $s$ is bijective, all maps in this composition are isomorphisms except possibly $d_{n-1}$. Therefore, $h$ is an isomorphism if and only if $d_{n-1}$ is. We conclude that $d_{n-1}$ is an isomorphism precisely when $\int_L \bar{b} \neq 0$, and otherwise, it is the zero map.

Finally, note that the exact sequence 
\[0 \to Z(S) \cdot 1 \to \overline{\Cdh} \to \overline{\Cdh_{\not \e}} \to 0\]
satisfies the assumptions of Proposition~\ref{prop:spec_exact_sequ}. The result now follows from Lemma~\ref{lm:E_infty_page_center}.
\end{proof}

\begin{proof}[Proof of Theorem~\ref{thm:spectral_conv_psedu}]
Convergence follows from Theorem~\ref{thm:filtered_algebra_spectral_sequence} as by Lemma~\ref{lm:S_filration_modu} the filtration is bounded. Lemma~\ref{lm:B_zero_almost_iso_B} and Corollary~\ref{cor:B_zero_Cdh_spectral_iso} imply that the $E_r$ page, for $r\ge 1$, is given by:
\[
E_r^{p,q}\left(\overline{\mB}\right)\cong \begin{cases}
E_r^{p,q-1}\left(\overline{\Cdh}\right), & 2p+q\ne 0,1, \text{ or } p\in 2\N+1,\\
0,& 2p+q=1,p\in 2\N,\\
S_{-p}, & \text{otherwise.}
\end{cases}
.\]
Theorem~\ref{thm:spectral_conv} implies the result.
\end{proof}

\subsection{Corollaries}

\begin{cor}\label{cor:even_top_deg}
    There exists a natural isomorphism ${F_{\nu_S}^{n-h}H^h \left(\overline{\Cdh_{\not \e}},\bar{\mh}^{\gamma,b}_1\right)}\cong E^{n-h,2h-n}_{\infty}.$ Moreover, if $\int_L \bar{b}\ne0$, or $n$ is odd, then
    \[F_{\nu_S}^{n-h}H^h \left(\overline{\Cdh_{\not \e}},\bar{\mh}^{\gamma,b}_1\right)=\begin{cases}
        0,& h\in 2\Z\backslash\{n\},\\
        S_{h-n}\otimes H^{n}(L,d),& otherwise.
    \end{cases}\]   
\end{cor}
\begin{proof}
    We have an exact sequence 
    \[\begin{tikzcd}
	{F^{n-h+1}_{\nu_S}H^{h}\left(\overline{\Cdh_{\not \e}},\bar{\mh}^{\gamma,\bar{b}}_1\right)} && {F_{\nu_S}^{n-h}H^h \left(\overline{\Cdh_{\not \e}},\bar{\mh}^{\gamma,\bar{b}}_1\right)} && {E^{n-h,2h-n}_{\infty}.}
	\arrow[from=1-1, to=1-3]
	\arrow[from=1-3, to=1-5]
\end{tikzcd}\]
It follows from Lemma~\ref{lm:S_filration_modu} that ${F_{\nu_S}^{n-h}H^h \left(\overline{\Cdh_{\not \e}},\bar{\mh}^{\gamma,b}_1\right)}\cong E^{n-h,2h-n}_{\infty}.$ Observe that $n=2h-n+2(n-h)$.
If $n$ is odd then $S_{h-n}=0$ for $h\in 2\Z\backslash\{n\}$ Thus, Theorem~\ref{thm:spectral_conv} implies the  result, if $\int_L \bar{b}\ne0$, or $n$ is odd.
\end{proof}
\begin{lm}\label{lm:E_infty_explisit}
    Let $p,m$ be integers such that $0<p+m<n$, then
    \[E^{p,m-p}_\infty \cong S_{-p}\otimes H^{m+p}(L,d).\]
\end{lm}
\begin{proof}
	     Set $q:=m-p$. We have $q+2p=2p+m-p=p+m$. By the assumption $0<p+m<n$, and thus $0<q+2p<n$. It now follows from Theorem~\ref{thm:spectral_conv} that $E^{p,q}_{\infty}=S_{-p}\otimes H^{q+2p}(L,d).$ So, the result follows.
\end{proof}
\begin{lm}\label{lm:cohomo-asum-filt}
    If $H^i(L;\R)=0$ for $i\ne 0,n$, and either $n$ is odd or $\int_L \bar{b}\ne0$, then the natural map $F_{\nu_S}^{n-m}H^{m}(\overline{\Cdh_{\not \e}},\bar{\mh}^{\gamma,b}_1)\hookrightarrow H^{m}(\overline{\Cdh_{\not \e}},\bar{\mh}^{\gamma,b}_1)$ is an isomorphism.
\end{lm}
\begin{proof}
    Observe the exact sequence
    \[\begin{tikzcd}
	{F^{\ell+1}_{\nu_S}H^{m}(\overline{\Cdh_{\not \e}},\bar{\m}^{\gamma,b}_1)} && {F^{\ell}_{\nu_S}H^{m}(\overline{\Cdh_{\not \e}},\bar{\m}^{\gamma,b}_1)} && {E_{\infty}^{\ell,m-\ell}}.
	\arrow[from=1-1, to=1-3]
	\arrow[from=1-3, to=1-5]
\end{tikzcd}\]
By Lemma~\ref{lm:S_filration_modu} ${F^{-m}_{\nu_S}H^{m}(\overline{\Cdh_{\not \e}},\bar{\mh}^{\gamma,b}_1)}={H^{m}(\overline{\Cdh_{\not \e}},\bar{\mh}^{\gamma,b}_1)}$. So, to prove the result, it's enough to show that $E_{\infty}^{\ell,m-\ell}=0$, when $-m\le \ell<n-m.$

If $\ell=-m$, set $p=-m$ and $q=2m$. Then $q+2p=0$. If $\int_L \bar{b}\ne0$ and $n$ is even, Theorem~\ref{thm:spectral_conv} implies that $E_{\infty}^{-m,2m}=0$. Otherwise, by assumption, $n$ is odd. In this case $|s|=1-n$ is even, so if $S_m\ne0$ then $-m\in 2\N$. Thus Theorem~\ref{thm:spectral_conv} gives $E_1^{-m,2m}=0$; if $S_m=0$, the same conclusion follows from the description of the $E_1$ page. Since for $n$ odd the spectral sequence degenerates at $E_1$, we get $E_{\infty}^{-m,2m}=0$. If $-m< \ell<n-m$, then the cohomological assumption on $L$ and Lemma~\ref{lm:E_infty_explisit} imply that $E_{\infty}^{\ell,m-\ell}=0$.  
\end{proof}
\begin{cor}\label{cor:cohomo_asum_filt}
    If $H^i(L;\R)=0$ for $i\ne 0,n$, and either $n$ is odd or $\int_L \bar{b}\ne0$, then the natural map \[F_{\nu_S}^{n-(2+v)}H^{2}(\overline{\Cdh_{\not \e}}[v],(-1)^{v}\bar{\mh}^{\gamma,b}_1)\hookrightarrow H^{2}\left(\overline{\Cdh_{\not \e}}[v],(-1)^{v}\bar{\mh}^{\gamma,b}_1\right)\]
    is an isomorphism.
\end{cor}
\begin{proof}
    Immediate from Lemma~\ref{lm:cohomo-asum-filt}.
\end{proof}
\begin{cor}\label{cor:cohomo_asum_filt_dt}
    If $H^i(L;\R)=0$ for $i\ne 0,n$, and either $n$ is odd or $\int_L \bar{b}\ne0$, then the natural map \[F_{\nu_S}^{1+n-(2+v)}H^{2}(\overline{\mB}_{\not \e}[v],(-1)^{v}\bar{\mt}^{\gt,\bt}_1)\hookrightarrow F_{\nu_S}^{1-(2+v)}H^{2}\left(\overline{\mB}_{\not \e}[v],(-1)^{v}\bar{\mt}^{\gt,\bt}_1\right)\]
    is an isomorphism.
\end{cor}
\begin{proof}
    For simplicity assume that $v=0.$ For $-2<\ell<n-1$ we identify $E^{\ell,2-\ell}_\infty (\mB_{\not \e})$ with $E^{\ell,1-\ell}_\infty (\overline{\Cdh_{\not \e}})$ by using Theorem~\ref{thm:spectral_conv_psedu} and Lemma~\ref{lm:B_not_e_almost_iso}. Then we proceed just as in the proof of Lemma~\ref{lm:cohomo-asum-filt}.
\end{proof}
\begin{cor}\label{cor:dt_even_top_deg}
There exists a natural isomorphism \[{F_{\nu_S}^{1+n-h}H^h \left(\overline{\mB_{\not \e}},\bar{\mt}^{\gt,\bt}_1\right)}\cong E_{\infty}^{1+n-h,2h-n-1}\left(\overline{\mB_{\not \e}}\right).\]

\end{cor}
\begin{proof}
The first isomorphism is standard convergence from Theorem~\ref{thm:filtered_spec_conv} together with Lemma~\ref{lm:S_filration_modu}.
\end{proof}
\subsection{Real spectral sequence}
For this section we assume that we are in the real setting. Let $v\in2\Z$, and let $b\in \Cdh$ be real. Assume that $\bar{b}$ is a bounding cochain for the residue $A_\infty$ algebra $(\overline{\Cdh},\{\bar{\mh}^{\gamma}_k\}_{k\ge 0})$. It follows from Lemma~\ref{lm:phi_star_cochainmap} that $\overline{\Cdh}[v]^{-\phi^*}$ and $\overline{\Cdh}[v]^{\phi^*}$ are subcomplexes of $(\overline{\Cdh}[v],\bar{\mh}^{\gamma,b}_1[v]).$
We abuse the notation and write $ \frac{Id+\phi^*}{2}$ as a map from $S[v]$ to itself and as the extended map from $\overline{\Cdh}[v]$ to itself.
\begin{lm}\label{lm:real_exact_seq}
    For every $p,q\in\mathbb{Z}$, and $r\ge 0$, there exists an exact sequence
\[\begin{tikzcd}
	0 & {E^{p,q}_r\left(\overline{\Cdh}[v]^{-\phi^*} \right)}&& {E^{p,q}_r\left(\overline{\Cdh}[v]\right)} && {E^{p,q}_r\left(\overline{\Cdh}[v]^{\phi^*}\right)} & 0.
	\arrow[from=1-1, to=1-2]
	\arrow[hook, from=1-2, to=1-4]
	\arrow["{\frac{Id+\phi^*}{2}}", from=1-4, to=1-6]
	\arrow[from=1-6, to=1-7]
\end{tikzcd}\]
\end{lm}
\begin{proof}
    We apply Proposition~\ref{prop:spec_exact_sequ} to the short exact sequence of filtered differential graded modules,
\[\begin{tikzcd}
	0 & {\overline{\Cdh}[v]^{-\phi^*}}&& {\overline{\Cdh}[v]} && {\overline{\Cdh}[v]^{\phi^*}} & 0.
	\arrow[from=1-1, to=1-2]
	\arrow[hook, from=1-2, to=1-4]
	\arrow["{\frac{Id+\phi^*}{2}}", from=1-4, to=1-6]
	\arrow[from=1-6, to=1-7]
\end{tikzcd}\]   
First, we verify the surjectivity of the map $\frac{Id+\phi^*}{2}$. Let $\xi \in Z_r^{p,q}\left(\overline{\Cdh}[v]^{\phi^*}\right)$. Then, by definition, $\phi^*(\xi)=\xi$. So,  $\xi=\frac{Id+\phi^*}{2}\left(\xi\right)$.     
The filtration condition of Proposition~\ref{prop:spec_exact_sequ} is satisfied because the filtration on $\overline{\Cdh}[v]^{-\phi^*}$ is the filtration induced as a submodule of $\overline{\Cdh}[v].$
The lemma follows.
\end{proof}
\begin{cor}\label{cor:real_exact_not_e}
For every $p,q\in\mathbb{Z}$, and $\ell\ge 0$, there is an exact sequence 
\[\begin{tikzcd}
	0 & {E^{p,q}_\ell\left(Z(S)[v]^{\invs}\cdot1\right)}& {E^{p,q}_{\ell}\left(\overline{\Cdh}[v]^{\invs}\right)} & {E^{p,q}_{\ell}\left(\frac{\overline{\Cdh}[v]^{\invs} }{Z(S)[v]^{\invs}\cdot1}\right)} & 0
	\arrow[from=1-1, to=1-2]
	\arrow[hook, from=1-2, to=1-3]
	\arrow[twoheadrightarrow, from=1-3, to=1-4]
	\arrow[from=1-4, to=1-5]
\end{tikzcd},\]
where \(\invs\) is one of \(-\phi^*\), \(\phi^*\), and \(\mathrm{Id}\).
It follows that there is an exact sequence,
    \[\begin{tikzcd}
	0 & {E^{p,q}_{\ell}\left(\frac{\overline{\Cdh}[v]^{-\phi^*} }{Z(S)[v]^{-\phi^*}\cdot1} \right)}& {E^{p,q}_{\ell}\left(\frac{\overline{\Cdh}[v]}{Z(S)[v]\cdot 1}\right)} & {E^{p,q}_{\ell}\left(\frac{\overline{\Cdh}[v]^{\phi^*} }{Z(S)[v]^{\phi^*}\cdot 1} \right)} & 0
	\arrow[from=1-1, to=1-2]
	\arrow[hook, from=1-2, to=1-3]
	\arrow["{\frac{Id+\phi^*}{2}}", twoheadrightarrow, from=1-3, to=1-4]
	\arrow[from=1-4, to=1-5]
\end{tikzcd}.\]
\end{cor}
\begin{proof}
    Observe that by Corollary~\ref{cor:center_of_hat_algebra} the differentials on $Z(S)[v]\cdot 1$, $Z(S)[v]^{-\phi^*}\cdot 1$ and on $Z(S)[v]^{\phi^*}\cdot 1$ are trivial. It follows that for $\invs=\phi^*,-\phi^*, Id$, the map $Z_{\ell}^{p,q}\left(\overline{\Cdh}[v]^{\invs}\right)\to Z_{\ell}^{p,q}\left(\frac{\overline{\Cdh}[v]^{\invs} }{Z(S)[v]^{\invs}\cdot 1} \right)$ is surjective. Hence, the assumptions of Proposition~\ref{prop:spec_exact_sequ} hold, and the first claim follows. So, the rows in the following diagram are exact.
    \[\begin{tikzcd}[column sep=small, row sep=normal]
&0 \arrow[d] &0\arrow[d] &0 \arrow[d]& \\0 \arrow[r] & E_{\ell}^{p,q}(Z(S)[v]^{-\phi^*}\cdot 1) \arrow[r] \arrow[d] & 
E^{p,q}_{\ell}\left(\overline{\Cdh}[v]^{-\phi^*}\right) \arrow[r] \arrow[d] & 
E^{p,q}_{\ell}\left(\frac{\overline{\Cdh}[v]^{-\phi^*} }{Z(S)[v]^{-\phi^*}\cdot 1}\right) \arrow[r] \arrow[d] & 0 \\
0 \arrow[r] & E_{\ell}^{p,q}(Z(S)[v]\cdot 1) \arrow[r] \arrow[d] & 
E^{p,q}_{\ell}\left(\overline{\Cdh}[v]\right) \arrow[r] \arrow[d] & 
E^{p,q}_{\ell}\left(\frac{\overline{\Cdh}[v] }{Z(S)[v]\cdot 1}\right) \arrow[r] \arrow[d] & 0 \\
0 \arrow[r] & E_{\ell}^{p,q}(Z(S)[v]^{\phi^*}\cdot 1) \arrow[r] \arrow[d] & 
E^{p,q}_{\ell}\left(\overline{\Cdh}[v]^{\phi^*}\right) \arrow[r] \arrow[d] & 
E^{p,q}_{\ell}\left(\frac{\overline{\Cdh}[v]^{\phi^*} }{Z(S)[v]^{\phi^*}\cdot 1}\right) \arrow[r] \arrow[d] & 0 \\
& 0 & 0 & 0 & 
\end{tikzcd}\]
It follows from Lemma~\ref{lm:real_exact_seq} that the middle column is exact. The differentials on $Z(S)[v]\cdot 1$, $Z(S)[v]^{-\phi^*}\cdot 1$ and on $Z(S)[v]^{\phi^*}\cdot 1$ are trivial.  So, the left column is exact. Hence, by the nine lemma, the right column is also exact.
\end{proof}
Define $\overline{\Cdh_{\not \e}}[v]^{-\phi^*}:=\frac{\overline{\Cdh}[v]^{-\phi^*} }{Z(S)[v]^{-\phi^*}\cdot 1}.$
\begin{cor}\label{cor:overline_Real_spec}
For every $p,q\in\Z$ and every $\ell\ge0$,
\[
E_{\ell}^{p,q}\left(\overline{\Cdh_{\not \e}}[v]^{-\phi^*}\right)
=\begin{cases}
E_{\ell}^{p,q}\left(\overline{\Cdh_{\not \e}}[v]\right),& -p-v=2,3\mod 4,\\
0, &otherwise.
\end{cases}
\]
Consequently, the same formula holds for $E_\infty$.
\end{cor}
\begin{proof}
    It follows from Lemma~\ref{lm:Fpexplicit} with the fact that $\phi^*$ acts trivially on differential forms that
    \[E_{0}^{p,q}\left(\overline{\Cdh}[v]^{-\phi^*} \right)=\begin{cases}E^{p,q}_0\left(\overline{\Cdh}[v]\right),& -p-v=2,3\mod 4,\\
    0, &otherwise,        
    \end{cases}\]
    and that 
    \[E_{0}^{p,q}\left(\overline{\Cdh}[v]^{\phi^*}\right)=\begin{cases}E^{p,q}_0\left(\overline{\Cdh}[v]\right),& -p-v=0,1\mod 4,\\
    0, &otherwise.        
    \end{cases}\]
    Since a zero entry on the $E_0$ page remains zero on all later pages, if $-p-v=0,1\mod 4,$ then $E_{\ell}^{p,q}\left(\overline{\Cdh}[v]^{-\phi^*}\right)=0$, and consequently by  Corollary~\ref{cor:real_exact_not_e} above $E_{\ell}^{p,q}\left(\overline{\Cdh_{\not \e}}[v]^{-\phi^*}\right)=0$.  Otherwise, $E_{\ell}^{p,q}\left(\overline{\Cdh}[v]^{\phi^*}\right)=0$, and thus $E^{p,q}_{\ell}\left(\frac{\overline{\Cdh}[v]^{\phi^*} }{Z(S)[v]^{\phi^*}\cdot 1}\right)=0$. So, if $-p-v=2,3 \mod{4}$, the injective map that is given in the above corollary
    \[E_{\ell}^{p,q}\left(\overline{\Cdh_{\not \e}}[v]^{-\phi^*}\right) \hookrightarrow E_{\ell}^{p,q}\left(\overline{\Cdh_{\not \e}}[v]\right)\]
    is surjective. This proves the formula for every $\ell\ge0$. Hence, the boundedness of the filtration implies the same formula for $E_\infty$.
\end{proof}

\begin{lm}\label{lm:cohomo-asum-filt_real}
    Let $v\in 2\Z$. Assume $H^k(L;\R) = 0$ for $k \equiv 0,3\mod 4, \, k \neq 0,n$. Then,  the natural map $F_{\nu_S}^{n-(2+v)}H^{2}(\overline{\Cdh_{\not \e}}[v]^{-\phi^*},\bar{\mh}^{\gamma,b}_1)\hookrightarrow H^{2}(\overline{\Cdh_{\not \e}}[v]^{-\phi^*},\bar{\mh}^{\gamma,b}_1)$ is an isomorphism.
\end{lm}
\begin{proof}    
    By Lemma~\ref{lm:S_filration_modu}, we know that 
    \[ F_{\nu_S}^{-2-v}H^{2}(\overline{\Cdh_{\not \e}}[v]^{-\phi^*},\bar{\mh}^{\gamma,b}_1)= H^{2}(\overline{\Cdh_{\not \e}}[v]^{-\phi^*},\bar{\mh}^{\gamma,b}_1).\]
    Consider the exact sequence:
    \[\begin{tikzcd}
    {F^{\ell+1}_{\nu_S}H^{2}(\overline{\Cdh_{\not \e}}[v]^{-\phi^*},\bar{\mh}^{\gamma,b}_1)} \arrow[r] & 
    {F^{\ell}_{\nu_S}H^{2}(\overline{\Cdh_{\not \e}}[v]^{-\phi^*},\bar{\mh}^{\gamma,b}_1)} \arrow[r] & 
    {E_{\infty}^{\ell,2-\ell}\left(\overline{\Cdh_{\not \e}}[v]^{-\phi^*}\right).}
    \end{tikzcd}\]
    To establish our claim, we need to show that $E_{\infty}^{\ell,2-\ell}\left(\overline{\Cdh_{\not \e}}[v]^{-\phi^*}\right)=0$ for all $\ell$ in the range $-2-v\leq \ell<n-2-v$.
    Let us first examine the case when $\ell=-2-v$. Let $p=\ell$ and $q=2-\ell$. We have 
    \[E_{\infty}^{p,q}\left(\overline{\Cdh_{\not \e}}[v]\right)\cong E_{\infty}^{p,q+v}\left(\overline{\Cdh_{\not \e}}\right).\]
    
    For our choice of $p$ and $q$, we have $2p+q+v=0$. Since  $v\in 2\mathbb{Z}$, we know $p$ is even. Hence, by Theorem~\ref{thm:spectral_conv}, $E_{\infty}^{p,q}\left(\overline{\Cdh_{\not \e}}[v]\right)=0$. It now follows from Corollary~\ref{cor:overline_Real_spec} that $E_{\infty}^{\ell,2-\ell}\left(\overline{\Cdh_{\not \e}}[v]^{-\phi^*}\right)=0$ when $\ell=-2-v$.
    
    Now consider when $-2-v< \ell<n-2-v$. With $p=\ell$ and $q=2-\ell$ as before, the pair $(p,q)$ satisfies $0<q+2p+v<n$. Theorem~\ref{thm:spectral_conv} tells us:
    \begin{align*}
        E_{\infty}^{p,q}\left(\overline{\Cdh_{\not \e}}[v]\right)
        &\cong E_{\infty}^{p,q+v}\left(\overline{\Cdh_{\not \e}}\right)\\
        &= E_1^{p,q+v}\left(\overline{\Cdh_{\not \e}}\right)\\
        &\cong S_{-p}\otimes H^{2p+q+v}\left(L\right).
    \end{align*}
    Applying Corollary~\ref{cor:overline_Real_spec} and using the fact that $q=2-p$, we obtain:
    \begin{align*}
        \bigoplus_{-2-v<p<n-2-v}E_{\infty}^{p,q}\left(\overline{\Cdh_{\not \e}}[v]^{-\phi^*}\right)
        &=\bigoplus_{\substack{-2-v<p<n-2-v\\-p-v\equiv 2,3\mod{4}}}E_{\infty}^{p,q}\left(\overline{\Cdh_{\not \e}}[v]\right)\\
        &=\bigoplus_{\substack{-2-v<p<n-2-v\\-p-v\equiv 2,3\mod{4}}}S_{-p}\otimes H^{2p+q+v}(L)\\
        &=\bigoplus_{\substack{-2-v<p<n-2-v\\-p-v\equiv 2,3\mod{4}}}S_{-p}\otimes H^{p+2+v}(L).
    \end{align*}
    
    When $-p-v\equiv 2,3\mod{4}$, we have $p+2+v\equiv 0,3\mod{4}$. Our assumption states that $H^k(L;\mathbb{R})=0$ for $k\equiv 0,3\mod{4}, k\neq 0,n$. Therefore, all these cohomology groups vanish, and we have:
    \[
    \bigoplus_{\substack{-2-v<p<n-2-v\\-p-v\equiv 2,3\mod{4}}}S_{-p}\otimes H^{p+2+v}(L) = 0.
    \]
    
    This shows that $E_{\infty}^{\ell,2-\ell}\left(\overline{\Cdh_{\not \e}}[v]^{-\phi^*}\right)=0$ for all $\ell$ in the range $-2-v\leq \ell<n-2-v$. This completes the proof.
\end{proof}
\begin{lm}
\label{lm:cohomo-asum-filt_real_dt}
    Let $v\in 2\Z$. Assume $H^k(L;\R) = 0$ for $k \equiv 2,3\mod 4, \, k \neq n$. Then the natural map $F_{\nu_S}^{n+1-(2+v)}H^{2}(\overline{\mB}_{\not \e}[v]^{-\phi^*},\bar{\mt}^{\gt,\bt}_1)\hookrightarrow F_{\nu_S}^{1-(2+v)}H^{2}(\overline{\mB}_{\not \e}[v]^{-\phi^*},\bar{\mt}^{\gt,\bt}_1)$ is an isomorphism.
\end{lm}
\begin{proof}
  Consider the exact sequence:  
    \[\begin{tikzcd}
    {F^{\ell+1}_{\nu_S}H^{2}(\overline{\mB}_{\not \e}[v]^{-\phi^*},\bar{\mt}^{\gt,\bt}_1)} \arrow[r] & 
    {F^{\ell}_{\nu_S}H^{2}(\overline{\mB}_{\not \e}[v]^{-\phi^*},\bar{\mt}^{\gt,\bt}_1)} \arrow[r] & 
    {E_{\infty}^{\ell,2-\ell}\left(\overline{\mB}_{\not \e}[v]^{-\phi^*}\right).}
    \end{tikzcd}\]
    
    To establish our claim, we need to show that $E_{\infty}^{\ell,2-\ell}\left(\overline{\mB}_{\not \e}[v]^{-\phi^*}\right)=0$ for all $\ell$ in the range $-2-v< \ell<n+1-2-v$.
    
    Write $p=\ell$ and $q=2-\ell$. We have that the pair $(p,q)$ satisfies $0<q+2p+v<n+1$. 
    \begin{align*}
        \bigoplus_{-2-v<p<n+1-2-v}E_{\infty}^{p,q}\left(\overline{\mB}_{\not \e}[v]^{-\phi^*}\right)&\overset{\ref{lm:B_not_e_almost_iso}}{=}\bigoplus_{-2-v<p<n-1-v}E_{\infty}^{p,q}\left(\overline{\mB}[v]^{-\phi^*}\right)\\
        &\overset{\ref{lm:B_zero_almost_iso_B},\ref{lm:B_zero_cone_quasi_iso},\ref{cor:cone_seq_exact_E_infty}}{=}\bigoplus_{-2-v<p<n-1-v}E_{\infty}^{p,q-1}\left(\overline{\Cdh_{\not \e}}[v]^{-\phi^*}\right)\\
        &\overset{\ref{cor:overline_Real_spec}}{=}\bigoplus_{\substack{-2-v<p<n-1-v\\-p-v\equiv 2,3\mod{4}}}E_{\infty}^{p,q-1}\left(\overline{\Cdh_{\not \e}}[v]\right)\\
        &\overset{\ref{thm:spectral_conv}}{=}\bigoplus_{\substack{-2-v<p<n-1-v\\-p-v\equiv 2,3\mod{4}}}S_{-p}\otimes H^{2p+q-1+v}(L)\\
        &\overset{q+p=2}{=}\bigoplus_{\substack{-2-v<p<n-1-v\\-p-v\equiv 2,3\mod{4}}}S_{-p}\otimes H^{p+1+v}(L).
    \end{align*}
    When $-p-v\equiv 2,3\mod{4}$, we have $p+1+v\equiv 2,3\mod{4}$. Our assumption states that $H^k(L;\mathbb{R})=0$ for $k\equiv 2,3\mod{4}, k\neq n$. Since $p+1+v<n$, all these cohomology groups vanish, and we have:
    \[\bigoplus_{\substack{-2-v<p<n-1-v\\-p-v\equiv 2,3\mod{4}}}S_{-p}\otimes H^{p+1+v}(L) = 0. \]
    This shows that $E_{\infty}^{\ell,2-\ell}\left(\overline{\mB}_{\not \e}[v]^{-\phi^*}\right)=0$ for all $\ell$ in the range $-2-v< \ell<n+1-2-v$. This completes the proof.
\end{proof}
\section{Classification of bounding pairs}\label{sec:bd_chains}
\subsection{Existence of bounding cochains}\label{construct_bd_ch}
\subsubsection{Statement}
Recall the notion of a bounding pair $(\gamma,b)$  given in Definition~\ref{dfn_bd_pair}.
It is our objective to prove the following result.
\begin{prop}\label{prop:exist}
Consider the following two cases.
\begin{enumerate}[label=(\arabic*), ref=(\arabic*)]
    \item \label{cond:sphere_coho}
        We assume $H^{i}(L;\R) = 0$ for $i \neq 0,n$ and take $\kappa = \Id.$
    \item \label{cond:invo_coho}
        We assume $(X,L,\omega,\phi)$ is a real setting, $\s$ is a spin structure, $R$ is concentrated in even degree, $H^k(L;\R) = 0$ for $k \equiv 0,3 \mod 4, \, k \neq 0,n$, and take $\kappa = -\phi^*.$
\end{enumerate}
Then, for any closed $\gamma \in (\mI_QD)_2^{\invs}$ and any
$a \in (\langle s\rangle +\mI_{\Rh})_{1-n}^{\invs}$ such that
$a^1 \in s \cdot \hat{R}^*$, there exists a bounding cochain
$b\in(\Ch)^{\invs}$ for $\mgh$ such that
\[
\int_L b = a.
\]
\end{prop}
\begin{rem}
    When $\invs=-\phi^*$, the condition $a^1 \in s \cdot \hat{R}^*$ imposes restrictions on the dimension of $L$. Specifically, $-\phi^*(s) = s$ if and only if $n \equiv 2 \mod{4}$ or $n \equiv 3 \mod{4}$.
\end{rem}

The proof of Proposition~\ref{prop:exist} is based on the obstruction theory developed below building on ideas of~\cite{fukaya2010lagrangian, solomon2016differential}.

\subsubsection{Obstruction cochains}
Recall the definitions of $R, C$, and $S$, and the decomposition of tensor products of $S$ to odd and even parts from Section~\ref{subsubsec:a_infty}.
Also, recall the definition of $\nu$ from Section~\ref{ssec:gen_not}.
Fix a sababa sub-algebra $\Rd \subset R$ and the corresponding restriction of scalars $\Cd$ as in Section~\ref{sssec:sababa_property}.
By Lemma~\ref{lm:R_daimond_sababa_A_infty} $(\Cd, \mg_k, \langle\;,\,\rangle, 1 )$ is a sababa cyclic unital $A_\infty$ algebra. Let $S$ and $\nu_S$ be as in Section~\ref{ssec:the_ring_S}. Let $\left(\hat{C}^{\diamond},\left\{\hat{\mathfrak{m}}_k^\gamma\right\}_{k \geq 0},\langle,\rangle_F, 1\right)$ be the $(S,F)$ extension of $(\Cd, \mg_k, \langle\;,\,\rangle, 1 )$ of Section~\ref{sssec:Extension_scalars_fukaya}, which by Corollary~\ref{cor:fukaya:sf:pseudocomplete} is pseudo-complete with respect to $\nu_S\otimes \nu$. Let $\hat{g}:\Rd \otimes_{\bar{\Rd}}\overline{\Cdh}\to \Cdh $ be the induced $\bar{\Rd}$ structure on $\Cdh$ as in Definition~\ref{dfn:induced_R_structure}, and let $\hat{f}$ be its inverse. 
Recall from Section~\ref{sssec:sababa_property} that $\bar{\Rd}$ is isomorphic to $\R$. Henceforward, all tensor products will be taken over $\R.$ By Lemma~\ref{lm:cano_iso} and Proposition~\ref{prop:fukaya_infity_arises_dga}, we have $\overline{\Cdh} \simeq S\otimes A^*(L).$ Throughout this section, either $\kappa = \Id$ or $(X,L,\omega,\phi)$ is a real setting, $\s$ is a spin structure, $R$ is concentrated in even degree and $\kappa = -\phi^*.$

Let $E_0<E_1<...$ be defined as in \eqref{eq:energy_filtration}. For $i \in \N$ recall the definitions of $\Rd_{E_i}, \Rdh_{ E_i}, \Cdh_{ E_i}$ from Definition~\ref{dfn:frm}. Also recall the definition of $\pi_r : \Cdh_{E} \to \overline{\Cdh}[|r|]$ from ~\eqref{eq:pir}.

To prove Proposition~\ref{prop:exist}, we construct a bounding cochain $b\in\Cdh$ following the strategy of Remark~\ref{rem:induc_cons_desc}. Thus, the bounding cochain $b$ is defined as a limit of a sequence $b_{(\ell)}=b_{(0)}+\sum_{i=1}^{\ell}b_i$, where $b_{(\ell)}$ satisfies
\begin{equation*}
    \mghi{b_{(\ell)}}_0\equiv c_{(\ell)}\cdot 1
    \pmod{F^{>E_{\ell}}\Cdh},\qquad
    c_{(\ell)}\in\left(\mI_R\Rdh\cap Z(\Rdh)\right)_2.
\end{equation*}

Let $b\in \left({\Cdh}^{\invs}\right)_1$ with $\nu_S\otimes \nu(b)>0$ be a bounding cochain in $\frac{\Cdh}{F^{>{E_{l}}}\Cdh}$, as in Definition~\ref{dfn:bounding_cochain}.  Then, Lemma~\ref{lm:o_is_class_HCe} implies that $\mghb_0\in F^{E_{l+1}}\Cdh_{\not \e}$ and that $\mghb_{0}$ represents a cocycle in $(\Cdh_{\not \e, E_{l+1}})^{\invs}$, which we denote by $\mgie{b}{E_{l+1}}$. For the rest of this section, we abbreviate \[
\mathfrak{o} =
[\mgie{b}{E_{l+1}}]\in H^2\left((\Cdh_{\not \e, E_{l+1}})^{\invs}, \mghb_1\right).
\]
\begin{lm}\label{lm:obs_cohomo_S_filt}
    Let $r\in {\Rd_{E}}^{\vee}$ and $\alpha\in F^{E}\left({\Cdh_{\not \e}}^{\invs}\right)_m$ such that $\bar{\mh}^{\gamma,b}_1[|r|]\left(\pi_r(\a)\right)=0$. If $[\pi_r\left(\a\right)]\in F_{\nu_S}^{\ell}H^{m}(\overline{\Cdh_{\not \e}}[|r|]^{\invs},\bar{\mh}^{\gamma,b}_1[|r|])$, then there exists 
    $\eta \in F^{ E}\left({\Cdh_{\not \e}}^{\invs}\right)_{m-1}$ such that
    \[\pi_r\left(\a+\mghb_1(\eta)\right)\in F_{\nu_S}^{\ell}\left(\overline{\Cdh_{\not \e}}\right)[|r|]^{\invs}_m.\]
\end{lm}
\begin{proof}
    Assume without loss of generality that $r\ne 0$.
    By the assumption of the lemma, there exists $\bar\eta\in\left( \overline{\Cdh_{\not \e}}[|r|]^{\invs}\right)_{m-1}$ such that $\bar{\mh}_{1}^{\gamma,\bar{b}}[|r|](\bar\eta)+\pi_r\left(\a\right)\in F_{\nu_S}^{\ell} \left( \overline{\Cdh_{\not \e}}[|r|]^{\invs}\right)_{m}.$  By Lemma~\ref{lm:kappa_pi_r_chain}, choose a homogeneous element $\eta\in F^E\left({\Cdh_{\not \e}}^{\invs}\right)_{m-1}$ such that $\pi_r(\eta)=\bar{\eta}$.  Then, again by Lemma~\ref{lm:kappa_pi_r_chain}, we have
    \[\pi_r\left(\a+\mghb_1(\eta)\right)=\pi_r\left(\a\right)+\bar{\mh}_{1}^{\gamma,\bar{b}}[|r|]\left(\pi_r\left(\eta\right)\right)=\pi_r\left(\a\right)+\bar{\mh}_{1}^{\gamma,\bar{b}}[|r|](\bar\eta)\in F_{\nu_S}^{\ell} \left(\overline{\Cdh_{\not \e}}[|r|]^{\invs}\right)_{m}.
    \]

\end{proof}
\begin{lm}\label{lm:td}
Assume that $n$ is odd, or $\int \bar{b}\ne 0$.
Let $r\in {\Rd}^{\vee}_{E_{l+1}}$ be  homogeneous. If $(\pi_r)_*\left(\mathfrak{o}\right)$ is in the image of the inclusion 
\[
F_{\nu_S}^{n-(2+|r|)}H^{2}(\overline{\Cdh_{\not \e}}[|r|]^{\invs},\bar{\mh}^{\gamma,b}_1[|r|])\hookrightarrow H^{2}\left(\overline{\Cdh_{\not \e}}[|r|]^{\invs},\bar{\mh}^{\gamma,b}_1[|r|]\right),
\]
then $(\pi_r)_*\left(\mathfrak{o}\right)=0.$
\end{lm}
\begin{proof}
    Let $h=2+|r|$. It follows from Lemma~\ref{lm:obs_cohomo_S_filt} that we can choose $b^{\dagger}\in \left({\Cdh}^{\invs}\right)_1$ such that
    \begin{enumerate}
        \item $b^{\dagger}\equiv b\pmod{F^{>E_{l}}\Cdh_{\not \e}}$;
        \item $\pi_r\left(\mghi{b^{\dagger}}_0\right)\in F^{n-h}\overline{\Cdh_{\not \e}}[|r|]^{\invs}$\label{enum:td_2};
        \item  and $[\mgie{b^{\dagger}}{E_{l+1}}]=[\mgie{b}{E_{l+1}}]=\mathfrak{o}$ in $H^2\left(\Cdh_{\not \e, E_{l+1}},\mghb_1\right)$.  
    \end{enumerate}  
    Let $(E_r\ell^{p,q},d_\ell)$ be the spectral sequence of Theorem~\ref{thm:spectral_conv}. Corollary~\ref{cor:even_top_deg}, and Corollary~\ref{cor:primitive_exists} imply that 
    \[F_{\nu_S}^{n-h}H^{2}(\overline{\Cdh_{\not \e}}[|r|]^{\invs},\bar{\mh}^{\gamma,b}_1[|r|])\hookrightarrow F^{n-h}_{\nu_S}H^{2}(\overline{\Cdh_{\not \e}}[|r|],\bar{\mh}^{\gamma,b}_1[|r|])\cong
    F_{\nu_S}^{n-h}H^{h}(\overline{\Cdh_{\not \e}},\bar{\mh}^{\gamma,b}_1)\cong E_{\infty}^{n-h,2h-n}.
    \]
    Assume first $E_1^{n-h,2h-n}=E_{\infty}^{n-h,2h-n}$. 
    Then, by Theorem~\ref{thm:spectral_conv}, we have that $n$ is odd, or $n-h$ is odd, or $n-h=0$. 
    Additionally, $(\pi_r)_*\left(\mathfrak{o}\right)$ is trivial if $[\pi_r\left(\mgie{b^{\dagger}}{E_{l+1}}\right)]\in E_{1}^{n-h,2h-n}$ is trivial. 
    Lemma~\ref{lm:S_filration_modu} and \eqref{enum:td_2} imply that
    \begin{equation}\label{eq:pi_r_Im}
        \pi_r\left(\mgie{b^{\dagger}}{E_{l+1}}\right)\in S_{h-n}\otimes A^{n}(L).
    \end{equation}
    Equation \eqref{eq:pi_r_Im} together with Lemma~\ref{lm:E_0_isom_diff_forms} implies that 
    \[\int_L \pi_r\left(\mgie{b^{\dagger}}{E_{l+1}}\right)=0 \implies (\pi_r)_*\left(\mathfrak{o}\right)=0.\]
Since either $n$ is odd, in which case $S^{\mathrm{odd}}$ is trivial, or $h-n$ is zero or odd, in fact
    \[\int_L \pi_r\left(\mgie{b^{\dagger}}{E_{l+1}}\right)\equiv 0 \pmod{\left(S^{\mathrm{odd}}\right)^2} \implies (\pi_r)_*\left(\mathfrak{o}\right)=0.\]
    Recalling that $(\gamma|_{L})_n=0$, it follows from Proposition~\ref{no_top_deg} that 
    \begin{multline*}
        \left(\hat f(\mghi{b^\dagger}_0)\right)_n=\left(d\hat f(b^{\dagger})-\hat f(b^{\dagger})\cdot \hat f(b^{\dagger})\right)_n=\\
        \left(d\hat f(b^{\dagger})-\frac{1}{2}[\hat f(b^{\dagger}),\hat f (b^{\dagger})]\right)_n\equiv (d\hat f(b^{\dagger}) )_n\pmod{\left(S^{\mathrm{odd}}\right)^2},
    \end{multline*}
    where the equivalence follows from Corollary~\ref{cor:center_of_hat_algebra}. Thus, using $[\cdot]$ to denote the equivalence class in the quotient $\Rdh_{E_{l+1}} = F^{E_{l+1}}\Rdh/F^{>E_{l+1}}\Rdh,$ we have
    \[
        \int_L \pi_r\left(\mgie{b^{\dagger}}{E_{l+1}}\right)=(r\otimes \Id)\left(\left[\int_L \hat{f}\left(\mghi{b^\dagger}_0\right)\right]\right)\equiv
         (r\otimes \Id)\left(\left[\int_L d\hat{f}\left(b^{\dagger}\right)\right]\right)\pmod{\left(S^{\mathrm{odd}}\right)^2}.
    \]
   It follows from Stokes' Theorem
    that $(\pi_r)_*(\mathfrak{o})=0.$

    Now assume that $E_1^{n-h,2h-n}\ne E_{\infty}^{n-h,2h-n}$. Theorem~\ref{thm:spectral_conv} implies that  $E_{\infty}^{n-h,2h-n}=0$,
so $(\pi_r)_*\left(\mathfrak{o}\right)=0$.
\end{proof}

\begin{prop}\label{prop:obstruct_vanish}
In the two cases of Proposition~\ref{prop:exist}, if $n$ is odd or $\int_L \bar{b}\ne 0$, then $\mathfrak{o}=0.$ 
\end{prop}
\begin{proof}
    The assumptions of Corollary~\ref{cor:cohomo_asum_filt} are satisfied in case~\ref{cond:sphere_coho}, and the assumptions of Lemma~\ref{lm:cohomo-asum-filt_real} are satisfied in case~\ref{cond:invo_coho}. Hence, for every homogeneous element $r\in {\Rd_{E_{l+1}}}^{\vee}$, the image of the energy-$E_{l+1}$ obstruction class
    \[
    \mathfrak{o}\in H^2\left((\Cdh_{\not \e,E_{l+1}})^{\invs},\mghb_1\right)
    \]
    under $(\pi_r)_*$ lies in the image of
    \[
    F_{\nu_S}^{n-(2+|r|)}H^{2}(\overline{\Cdh_{\not \e}}[|r|]^{\invs},\bar{\mh}^{\gamma,b}_1[|r|])\hookrightarrow H^{2}\left(\overline{\Cdh_{\not \e}}[|r|]^{\invs},\bar{\mh}^{\gamma,b}_1[|r|]\right).
    \]
    Thus, the assumptions of Lemma~\ref{lm:td} are satisfied, and $(\pi_r)_*\left(\mathfrak{o}\right)=0$. It now follows from Lemma~\ref{lm:basis_exact} that $\mghb_{0,E_{l+1}}$ is exact.
    So,  $\mathfrak{o}=0.$
\end{proof}
\begin{cor}\label{cor:obs_vanish}
In the two cases of Proposition~\ref{prop:exist}, assume that $n$ is odd or $\int_L \bar{b}\ne 0$. Then there exists $\eta \in F^{E_{l+1}}\left({\Cdh}^{\invs}\right)_1$ such that $b+\eta\in \left({\Cdh}^{\invs}\right)_1$ is a solution to the Maurer-Cartan equation modulo $F^{>E_{l+1}}\Cdh_{\not \e}$, and $\int_L \hat{f}(b+\eta)=\int_L \hat{f}\left(b\right)$.
\end{cor}
\begin{proof}
    By Proposition~\ref{prop:obstruct_vanish} we are in the setting of Lemma~\ref{lm:obs_induction}, so we can choose $\xi \in F^{E_{l+1}}\left({\Cdh}^{\invs}\right)_1$ such that $b+\xi$ is a bounding cochain in $\frac{\Cdh}{F^{>{E_{l+1}}}\Cdh}$. 
    Let $P_n: \overline{\Cdh} \to S \otimes A^{n}(L)$ be the projection. Then by Corollary~\ref{cor:bar_fukaya_deformed_boundry},
    \[
    \ker\left(\bar{\m}_1^{\gamma,b}\circ P_n\right)=\overline{\Cdh}.
    \]
    Thus, by Lemma~\ref{lm:twisted_f_chain_iso}, the element
    \[
    \eta:=\hat{g}\left(
    \hat{f}(\xi)-(\Id \otimes P_n)\bigl(\hat{f}(\xi)\bigr)
    \right)
    \]
    is a primitive for $\mghb_{0,E_{l+1}}$. It follows that
    $ b+\eta$ is a solution to the Maurer-Cartan equation in $\frac{\Cdh_{\not \e}}{F^{>E_{l+1}}\Cdh_{\not \e}}$ , and 
    \[
    \begin{aligned}
    \int_L \hat{f}(b+\eta)
    &=\int_L \hat{f}\left(b\right)
      +\int_L\left[\hat{f}(\xi)-(\Id\otimes P_n)\bigl(\hat{f}(\xi)\bigr)\right]\\
    &=\int_L \hat{f}\left(b\right).
    \end{aligned}
    \]
\end{proof}
\begin{proof}[Proof of Proposition~\ref{prop:exist}]
Let $a \in (\langle s\rangle+\mI_{\hat{R}})_{1-n}$ such that $a^1 \in s \cdot \hat{R}^*$. Define $\Rd:=R(a)$ as in Notation~\ref{notn:n_decomp}.
Take $\bar{b}_0\in A^n(L)$ any representative of the Poincar\'e dual of a point, and let $b_{(0)}:=a\otimes \bar{b}_0$. By Lemma~\ref{lm:energy_zero_bounding_chain} $\bar{b}_{(0)}$ is a bounding cochain for the residue $A_\infty$ algebra $(\overline{\Cdh},\{\bar{\mh}^{\gamma}_k\}_{k\ge 0})$.
Write $a^1=su$ with $u\in\hat{R}^*$. The residue $\bar{u}$ is a unit in $\overline{\hat{R}}\simeq S$, so $\bar{a}^1=s\bar{u}\ne0$. Hence $\int_L\bar{b}_{(0)}=\bar{a}\ne0$. Every correction supplied by Corollary~\ref{cor:obs_vanish} has positive energy and therefore leaves the residue unchanged. Thus the hypothesis $\int_L\bar{b}\ne0$ holds at every step of the induction.
Applying Corollary~\ref{cor:obs_vanish} repeatedly, we get a sequence of elements $\{b_i\}_{i=1}^\infty\subset \left(\Cdh\right)_1$ such that $b_i \in F^{ E_i}\Cdh$, and  the element $b_{(\ell)}:=b_{(0)}+\sum_{i=1}^{\ell}b_i$  is a solution to the Maurer-Cartan equation in $\frac{\Cdh_{\not \e}}{F^{>E_{\ell}}\Cdh_{\not \e}}$, and
\[\int_L \hat{f}\left(b_{(\ell)}\right)=\int_L \hat{f}(b_{(0)})=a.\] Since $\lim_{i \to \infty}\nu(b_i)=\infty$, it follows from completeness of $\Cdh$ that there exists a unique limit $b:=\lim_{\ell \to \infty}b_{(\ell)}$ and $\int_L b = \lim_{\ell \to \infty} \int_L b_{(\ell)} = a.$ By construction $b$ is a solution to the Maurer-Cartan equation in $\Cdh_{\not \e}.$ Moreover, $\nu(\mghb_0)>0$ as $b$ is a solution to the Maurer-Cartan equation in $\overline{\Cdh}.$  We conclude that $b$ is a bounding cochain for $\gamma$.
\end{proof}

\subsection{Gauge equivalence of bounding pairs}\label{ssec:classification}
\subsubsection{Formulation}
In the following, we use the notation of Section~\ref{pseudoisot}.  In particular, we write $\mgp$ for operations defined using the almost complex structure $J'$ and a closed form $\gamma' \in \mI_QD$ with $|\gamma'|=2.$
Recall also the definition~\eqref{eq:mR} of $\mR=A^*(I;R)$. Denote by $\pi:I\times X\to X$ the projection. By abuse of notation, we write $\mgt$ instead of $\hmt^{\gt}$. 

\begin{prop}\label{prop:unique}
Consider the following two cases.
\begin{enumerate}[label=(\arabic*), ref=(\arabic*)]
    \item \label{cond:vanishing_cohomology} We assume $H^{i}(L ; \mathbb{R})=0$ for $i \neq 0, n$, and take $\kappa = \Id.$
    \item \label{cond:real_setting} We assume $(X,L,\omega,\phi)$ is a real setting, $\s$ is a spin structure, $R$ is concentrated in even degree, and $H^k(L;\R) = 0$ for $k \equiv 2,3 \mod 4$ and $k \neq 0,n$, and take $\kappa = -\phi^*.$
\end{enumerate}
Let $(\gamma, b)$ be a unit bounding pair with respect to $J$ and let $\left(\gamma^{\prime}, b^{\prime}\right)$ be a unit bounding pair with respect to $J^{\prime}$ such that $\varrho([\gamma, b])=\varrho\left(\left[\gamma^{\prime}, b^{\prime}\right]\right)$. Assume further that $\gamma, \gamma^\prime \in (\mI_QD)^{\invs}$ and $b, b^\prime \in (\Cdh)^{\invs}$. Then $(\gamma, b) \sim\left(\gamma^{\prime}, b^{\prime}\right)$.
\end{prop}
\begin{rem}
    When $\invs=-\phi^*$, the assumptions that $R$ is concentrated in even degree and $n$ is even imply that $\int_L b^0=0$ for every $b\in(\Cdh)_1$. Thus, for real bounding pairs, equality under $\varrho_\phi$ implies equality under $\varrho$.
\end{rem}
The proof of Proposition~\ref{prop:unique} is given toward the end of the section based on the construction detailed in the following.

\subsubsection{Obstruction theory}

Recall the definitions of $\mR$ and $\mC$ from
Section~\ref{pseudoisot}.
Fix a sababa sub-algebra $\Rd\subset R$, and define
$\mRd:=\Rd \otimes A^*(I)$ and
$\mCd:=\Rd \otimes A^*(I \times L)$.
Then, by Lemma~\ref{lm:R_daimond_sababa_A_infty} we have
$(\mCd,\{\mgt_k\}_{k\ge 0},\langle\langle\;,\,
\rangle\rangle,1)$ is a sababa cyclic unital $A_\infty$
algebra over $\Rd$ with values in $\mRd$.
Let $S$ and $\nu_S$ be as in Section~\ref{ssec:the_ring_S}.
Let $\mCdh$ be the $(S,F)$ extension of $\mCd$ as in
Section~\ref{sssec:Extension_scalars_fukaya} which by Corollary~\ref{cor:fukaya:sf:pseudocomplete} is pseudo-complete with respect to $\nu_S\otimes \nu$.
Recall the definition of $\mB$ from
Section~\ref{sssec:psedu_residu_obstruct}. 
Let $E_0<E_1<...$ be defined as in \eqref{eq:energy_filtration}.
For $i \in \N$ recall the definitions of
$\Rd_{E_i}, \mRdh_{E_i}, \mCdh_{E_i}$
from Definition~\ref{dfn:frm}. Denote by $\hat{g}:\Rd \otimes \overline{\mCdh} \to \mCdh$
the induced $\bar \Rd$ structure on $\mCdh$ from
Section~\ref{sssec:fukaya_scalar_ext_dfn} and by $\hat{f}$
its inverse.
Also recall the definition of
$\pi_r : \mCdh_E \to \overline{\mCdh}[|r|]$
from~\eqref{eq:pir}.
Throughout this section, either $\kappa = \Id$ or $(X,L,\omega,\phi)$ is a real setting, $\s$ is a spin structure, $R$ is concentrated in even degree and $\kappa = -\phi^*.$

To prove the gauge equivalence of Proposition~\ref{prop:unique}, we construct a bounding cochain $\bt \in\mCdh$ following the strategy of Remark~\ref{rem:induc_cons_desc}. Thus, the bounding cochain $\bt$ is defined as a limit of a sequence $\bt_{(\ell)} = \sum_{i=0}^{\ell} \bt_i,$ where $\bt_{(\ell)}$ satisfies  
\begin{equation*}
    \mgti{\bt_{(\ell)}}\equiv \ct_{(\ell)}\cdot 1 \pmod{F^{>E_l}\mCdh},\qquad \ct_{(\ell)}\in \left(\mI_R\mRdh\cap Z(\mRdh)\right)_2,
\end{equation*}
and $d\ct_{(\ell)}\equiv 0  \pmod{F^{>E_l}\mRdh}$. 

Let $b,b^\prime\in \Cdh$ be bounding cochains for the $A_\infty$ algebras $(\Cdh,\{\mgh_k\}_{k\ge 0},\langle\;,\,\rangle_F,1\otimes 1)$ and $\left(\Cdh, \mh^{\gamma^{\prime}},\langle\,,\,\rangle_F, 1\otimes 1\right)$ respectively. We assume that
\begin{equation}\label{eq:dt_integral_assumption}
\int_L \hat{f}\left(b\right)- \int_L \hat{f}\left(b'\right)\in \mI_{\hat{R}}\cap \left(\hat{R}^{odd}\right)^2.
\end{equation}
Let $\bt_0$ be given by the formula in the proof of Lemma~\ref{lm:bt_exist_base_case1}. By the definition of $F$ and assumption~\eqref{eq:dt_integral_assumption}, we have $\langle b,1\rangle_F=\langle b',1\rangle_F$. Thus,
\[
d_{\mRdh}\ll\bt_0,1\gg_F=0.
\]
By Lemma~\ref{lm:bt_exist_base_case}, $(\mB,\{\mgti{\bt_0}_k\}_{k\ge0},\ll\;,\,\gg_F,1)$ is a cyclic unital $A_\infty$ algebra. Let $\bt \in \left({\mB}^{\invs}\right)$ with $\nu(\bt)>0$ be a bounding cochain with respect to the induced $A_\infty$ algebra 
\[\left(\frac{\mB}{F^{>{E_{l}}}\mB},\{ \mgti{\bt_0}_k\}_{k\ge0},\ll\;,\,\gg_F,1\right). \]
It can be shown that $\left(\mgti{\bt_0}_0\right)^{\bt}=\mgti{\bt_0+\bt}_0$.
Then, the definition of a bounding cochain implies that $\left(\mgti{\bt_0}_0\right)^{\bt} \in F^{E_{l+1}}\mB_{\not \e}$. So,  Lemma~\ref{lm:o_is_class_HCe} implies that $\mgti{\bt_0+\bt}_0$ represents a cocycle in $(\mB_{\not \e, E_{l+1}})^{\invs}$, which we denote by $\mgti{\bt_0+\bt}_{0,E_{l+1}}$.
For the rest of this section, we abbreviate
\[
\mathfrak{o} =
[\mgti{\bt_0+\bt}_{0,E_{l+1}}]\in
H^2\left((\mB_{\not \e,E_{l+1}})^{\invs},
\mgti{\bt_0+\bt}_{1,E_{l+1}}\right).
\]

\begin{lm}\label{lm:dt_obs_cohomo_S_filt}
    Let $r$ be an element of ${\Rd_{E}}^{\vee},$ and $\alpha\in F^E\left(\mB_{\not \e}^{\invs}\right)_m$ such that $\mgti{\bt_0+\bt}_1\left(\alpha\right)=0$.
     If $[\pi_r\left(\a\right)]\in F_{\nu_S}^\ell H^{m}\left(\overline{\mB_{\not \e}}[|r|]^{\invs},\bar{\mt}^{\gt,\bt_0+\bt}_1[|r|]\right)$, then there exists $\eta \in F^{ E}\left(\mB_{\not \e}^{\invs}\right)_{m-1}$, such that
    \[
\pi_r\left(\a+\mgti{\bt_0+\bt}_1(\eta)\right)
\in F_{\nu_S}^{\ell}
\left(\overline{\mB_{\not \e}}[|r|]^{\invs}\right)_{m}.
\]
\end{lm}
\begin{proof}
    Same as the proof of Lemma~\ref{lm:obs_cohomo_S_filt}.
\end{proof}
\begin{lm}\label{lm:dt_o_vanish_bottom_filt}
\[\pi_r\left(\mathfrak{o}\right)\in F_{\nu_S}^{1-(2+|r|)}H^2\left(\overline{\mB_{\not \e}}[|r|]^{\invs},\bar{\mt}^{\gt,\bt_0}_1\right)\]    
\end{lm}
\begin{proof}
    By Lemma~\ref{lm:S_filration_modu} we have
    \[\pi_r(\mathfrak{o})\in F_{\nu_S}^{-(2+|r|)}H^2\left(\overline{\mB_{\not \e}}[|r|]^{\invs},\bar{\mt}^{\gt,\bt_0}_1\right).\]
    To obtain the stronger bound asserted by the lemma, it suffices to demonstrate that the class represented by $\pi_r(\mathfrak{o})$ vanishes in the corresponding graded component of the spectral sequence, specifically\[[\pi_r(\mathfrak{o})]=0 \in E_{\infty}^{-(2+|r|),4+|r|}\left(\overline{\mB_{\not \e}}[|r|]^{\invs}\right).\] 
Since $\bt$ is a bounding cochain in $\frac{\mB}{F^{>E_{l}}\mB}$, there exists $\ct \in Z(\mRdh)\cap\ker d_{\mRdh}$ such that $\mgti{\bt_0+\bt}_0-\ct\cdot1\in F^{E_{l+1}}\mB$. 
    Properties~\ref{it:a_infty} and~\ref{it:unit1} of the $A_\infty$ algebra $\left(\mB,\{\mgti{\bt_0+\bt}\}_{k\ge 0}\right)$ imply that $[\mgti{\bt_0+\bt}_0-\ct\cdot1]$ is a cocycle in $(\mB_{E_{l+1}})^{\invs}$.
    So, $\pi_r\left([\mgti{\bt_0+\bt}_0-\ct\cdot1]\right)$ is a cocycle in $\left(\overline{\mB}[|r|]^{\invs}\right)_2$. Let $Q:\overline{\mB}[|r|]^{\invs}\to \overline{\mB_{\not \e}}[|r|]^{\invs}$ and $Q^{\prime}:\mB \to \mB_{\not \e}$ be the quotient maps. We observe that 
    \begin{equation*}
        Q\circ \pi_r\left([\mgti{\bt_0+\bt}_0-\ct\cdot1]\right)= \pi_r\left( [Q^{\prime }\left(\mgti{\bt_0+\bt}_0-\ct\cdot1\right)]\right)
        =
        \pi_r\left( [Q^{\prime }\left(\mgti{\bt_0+\bt}_0\right)]\right)=\pi_r(\mgti{\bt_0+\bt}_{0,E_{l+1}}).
    \end{equation*}    
Consequently, the class $[\pi_r(\mathfrak{o})]
$ lies in the image of the map induced by $Q$ on the $E_{\infty}$ page
\[[\pi_r(\mathfrak{o})]\in Q_\infty^{-2-|r|,4+|r|}\left(E_\infty^{-2-|r|,4+|r|}\left(\overline{\mB}[|r|]^{\invs}\right)\right).\]
We set $p=-(2+|r|)$ and $q=4+|r|$. To complete the proof we distinguish two cases based on the parity of $p$:
\begin{itemize}
    \item Case $p$ is even. The indices satisfy $2p+q+|r|=0$. By Lemma~\ref{lm:B_not_e_almost_iso}, $Q^{p,q}_{\infty}$ is trivial, implying $[\pi_r(\mathfrak{o})]=0$.
    \item Case $p$ is odd. By Theorem~\ref{thm:spectral_conv} and Theorem~\ref{thm:spectral_conv_psedu} $E_{\infty}^{p,q}\left(\overline{\mB}[|r|]\right)=0,$ which forces $[\pi_r(\mathfrak{o})]=0$.
\end{itemize}
\end{proof}

\begin{lm}\label{lm:dt_td}
Assume that $n$ is odd or $\int_L \bar{b}\ne 0$.
Let $r\in {\Rd_{E_{l+1}}}^{\vee}$ be homogeneous. 
Assume $(\pi_r)_*\left(\mathfrak{o}\right)$ is in the image of $F_{\nu_S}^{1+n-(2+|r|)}H^{2}(\overline{\mB_{\not \e}}[|r|]^{\invs},\bar{\mt}^{\gt,\bt_0}_1)\hookrightarrow H^{2}(\overline{\mB_{\not \e}}[|r|]^{\invs},\bar{\mt}^{\gt,\bt_0}_1)$.
 Then, $(\pi_r)_*\left(\mathfrak{o}\right)=0$.
\end{lm}
\begin{proof}
    Let $h=2+|r|$. It follows from Lemma~\ref{lm:dt_obs_cohomo_S_filt} that we can choose $\bt^{\dagger}\in (\mB^{\invs})_1$ such that 
    \begin{enumerate}
    \item $\bt^{\dagger}\equiv \bt\pmod{F^{>E_{l}}\mB^{\invs}}$;
    \item $\pi_r\left(\mgti{\bt_0+\bt^{\dagger}}_{0, E_{l+1}}\right)
\in F_{\nu_S}^{1+n-h}
\overline{\mB_{\not \e}}[|r|]^{\invs}$;
\label{enum:dt_td_filt}
    \item  and $[\mgti{\bt_0 +\bt^{\dagger}}_{0,E_{l+1}}]=[\mgti{\bt_0 +\bt}_{0,E_{l+1}}]=\mathfrak{o}$ in $H^2\left((\mB_{\not \e})_{E_{l+1}},\mgti{\bt_0}_1\right)$.  
\end{enumerate}
    Let $\left(E_\ell^{p,q},d_\ell\right)$ be the spectral sequence of Theorem~\ref{thm:spectral_conv_psedu}. By Lemma~\ref{lm:B_not_e_almost_iso} we have 
    \[E_{\infty}^{n+1-h,2h-n-1}(\overline{\mB_{\not \e}})\cong E_{\infty}^{n+1-h,2h-n-1}.\]
    Thus, Corollary~\ref{cor:primitive_exists} and Corollary~\ref{cor:dt_even_top_deg} imply that
    \begin{multline*}
        F_{\nu_S}^{n+1-h}H^{2}(\overline{\mB_{\not \e}}[|r|]^{\invs},\mgti{\bt_0}_1)\hookrightarrow F^{n+1-h}_{\nu_S}H^{2}(\overline{\mB_{\not \e}}[|r|],\mgti{\bt_0}_1)\\ \cong
    F_{\nu_S}^{n+1-h}H^{h}(\overline{\mB_{\not \e}},\mgti{\bt_0}_1)\cong E_{\infty}^{n+1-h,2h-n-1}.
    \end{multline*}
    Assume first that $E_1^{n+1-h,2h-n-1}=E_\infty^{n+1-h,2h-n-1}$. Then, by Theorem~\ref{thm:spectral_conv_psedu}, we have $n$ is odd, or $n+1-h$ is odd, or $n+1-h=0$. 
    Additionally, $(\pi_r)_*\left(\mathfrak{o}\right)$ is trivial if and only if $[\pi_r\left(\mgti{\bt_0+\bt^{\dagger}}_{0,E_{l+1}}\right)]\in E_1^{n+1-h,2h-n-1}$ is trivial. Lemma~\ref{lm:S_filration_modu} with \eqref{enum:dt_td_filt} implies that
    \begin{equation}\label{eq:dt_pi_r_Im}
        \pi_r\left(\mgti{\bt_0+\bt^{\dagger}}_0\right)\in S_{h-n-1}\otimes A^{n+1}(I\times L).
    \end{equation}
    Equation \eqref{eq:dt_pi_r_Im} together with Lemma~\ref{lm:E_0_isom_diff_forms} implies that 
    \[\int_{I\times L} \pi_r\left(\mgti{\bt_0+\bt^{\dagger}}_{0, E_{l+1}}\right)=0 \implies (\pi_r)_*\left(\mathfrak{o}\right)=0.\]
    Since either $n$ is odd, in which case $S^{\mathrm{odd}}$ is trivial, or $h-n-1$ is zero or odd, in fact
    \[\int_{I\times L} \pi_r\left(\mgti{\bt_0+\bt^{\dagger}}_{0, E_{l+1}}\right)\equiv 0 \pmod{\left(S^{\mathrm{odd}}\right)^2} \implies (\pi_r)_*\left(\mathfrak{o}\right)=0.\]
    Recalling that $(\gt|_{I\times L})_{n+1}=0$, it follows from Proposition~\ref{cl:qt_no_top_deg} that 
        \begin{multline*}
        \left(\hat f(\mgti{\bt_0+\bt^{\dagger}}_0)\right)_{n+1}=\left(d\hat f(\bt_0+\bt^{\dagger})-\hat f(\bt_0+\bt^{\dagger})\cdot \hat f(\bt_0+\bt^{\dagger})\right)_{n+1}=\\
        \left(d\hat f(\bt_0+\bt^{\dagger})-\frac{1}{2}[\hat f(\bt_0+\bt^{\dagger}),\hat f (\bt_0+\bt^{\dagger})]\right)_{n+1}\\
        \equiv \left(d\hat f(\bt_0+\bt^{\dagger})\right)_{n+1}\pmod{\Rd \otimes \left(S^{\mathrm{odd}}\right)^2 \otimes A^{n+1}(I\times L)},
    \end{multline*}
    where the equivalence follows from Corollary~\ref{cor:center_of_hat_algebra}. 
    Thus, using $[\cdot]$ to denote the equivalence class in the quotient $\Rdh_{E_{l+1}} = F^{E_{l+1}}\Rdh/F^{>E_{l+1}}\Rdh,$ we have 
    \begin{multline*}
        \int_{I\times L} \pi_r\left(\mgti{\bt_0+\bt^{\dagger}}_{0, E_{l+1}}\right)=(r\otimes \Id)\left(\left[\int_{I\times L } \hat{f}\left(\mgti{\bt_0+\bt^{\dagger}}_{0}\right)\right]\right)\equiv \\\equiv (r\otimes \Id)\left(\left[\int_{I\times L} d\hat{f}\left((\bt_0+\bt^{\dagger})\right)\right]\right)\pmod{\left(S^{\mathrm{odd}}\right)^2}.
    \end{multline*}
   Recall that $n>0$ and that $j_i^*(\bt^{\dagger})\in Z(\Rdh)\cdot 1$ for $j=0,1$. Thus, recalling the sign convention of~\cite[Remark~2.3]{solomon2016differential}, Stokes' Theorem implies
    \[
\int_{I\times L} d\hat{f}(\bt_0+\bt^{\dagger})
=(-1)^{n+1}\left(\int_L\hat{f}(b^\prime)-\int_L\hat{f}(b)\right)
\equiv 0 \pmod{\mI_{\hat{R}}\cap\left(\hat{R}^{odd}\right)^2},
\]
    where the equivalence follows from assumption~\eqref{eq:dt_integral_assumption}.
     Let $q_S:S \to \frac{S}{\left(S^{\mathrm{odd}}\right)^2}$ be the quotient map of vector spaces. We factorize $q_S\circ (r\otimes \Id_S )$ through $\left(\frac{\Rdh}{\Rdh\cap(\Rh^{odd})^2}\right)_E$ as in the diagram,
\[\begin{tikzcd}
	{\Rdh_E} && {\left(\frac{\Rdh}{\Rdh\cap(\Rh^{odd})^2}\right)_E} \\
	\\
	S && {\frac{S}{\left(S^{\mathrm{odd}}\right)^2}}
	\arrow[from=1-1, to=1-3]
	\arrow["{r\otimes \Id_S}"', from=1-1, to=3-1]
	\arrow[dashed, from=1-3, to=3-3]
	\arrow["{q_S}"', from=3-1, to=3-3]
\end{tikzcd},\]
and deduce that 
\[q_S\circ (r\otimes \Id)\left(\left[\int_{I\times L} d\hat{f}\left((\bt_0+\bt^{\dagger})\right)\right]\right)=0.\]
    It follows that $(\pi_r)_*(\mathfrak{o})=0$.

    Now assume that $E_1^{n+1-h,2h-n-1}\not=E_\infty^{n+1-h,2h-n-1}$. By Theorem~\ref{thm:spectral_conv_psedu}, we have
    \[
     E_{\infty}^{n+1-h,2h-n-1}=0,
    \]
    so $(\pi_r)_*(\mathfrak{o})=0$.
\end{proof}
\begin{prop}\label{prop:dt_o_vanish}
In the two cases of Proposition~\ref{prop:unique}, if $n$ is odd or $\int_L \bar{b}\ne 0$, 
then $\mathfrak{o}=0.$ 
\end{prop}
\begin{proof}
    The assumptions of Corollary~\ref{cor:cohomo_asum_filt_dt} are satisfied in the case \ref{cond:vanishing_cohomology}. In case~\ref{cond:real_setting}, the assumptions of Lemma~\ref{lm:cohomo-asum-filt_real_dt} are satisfied. Hence, by Lemma~\ref{lm:dt_o_vanish_bottom_filt} for any homogeneous element $r\in {\Rd_{E}}^{\vee}$ the element $(\pi_r)_*\left(\mathfrak{o}\right)$ is in the image of $F_{\nu_S}^{1+n-(2+|r|)}H^{2}(\overline{\mB}_{\not \e}[|r|]^{\invs},\bar{\mt}^{\gt,\bt_{0}}_1)\hookrightarrow H^{2}(\overline{\mB}_{\not \e}[|r|]^{\invs},\bar{\mt}^{\gt,\bt_{0}}_1)$. Thus, the assumptions of Lemma~\ref{lm:dt_td} are satisfied, and $(\pi_r)_*\left(\mathfrak{o}\right)$ is an exact cochain for any homogeneous element $r\in {\Rd_{E}}^{\vee}$. It follows from Lemma~\ref{lm:dt_basis_exact} that $\mathfrak{o}=0.$
\end{proof}
\begin{cor}\label{cor:dt_obs_vanish}
    In the two cases of Proposition~\ref{prop:unique}, assume that $n$ is odd, or $\int_L \bar{b}\ne 0$. 
Then, there exists $\eta \in F^{E_{l+1}}\left(\mB^{\invs}\right)_1$ such that
\begin{enumerate}
    \item $\bt+\eta$ is a bounding cochain for $\frac{\mB}{F^{>E_{l+1}}\mB}$;
    \item $j_i^*(\bt+\eta)=j_i^*(\bt)$ for $i=0,1$.
\end{enumerate}
\end{cor}
\begin{proof}
    By Proposition~\ref{prop:dt_o_vanish} and Lemma~\ref{lm:obs_induction}, choose $\xi \in F^{E_{l+1}}\left(\mB^{\invs}\right)_1$ such that $\bt+\xi$ is bounding modulo $F^{>E_{l+1}}\mB$, and set
    \[\eta:=\xi-\left((1-t)j^*_0(\xi)+tj^*_1(\xi)\right).\]
    Then $\eta\in F^{E_{l+1}}(\mB^{\invs})_1$ and $j_i^*(\bt+\eta)=j_i^*(\bt)$. Moreover,
    \[d\left((1-t)j^*_0(\xi)+tj^*_1(\xi)\right)\in
    \left(Z(\mRdh)\cap\ker d_{\mRdh}\right)\cdot 1.\]
    Properties~\ref{it:lin} and~\ref{it:unit1} therefore imply that $\bt+\eta$ is bounding modulo $F^{>E_{l+1}}\mB$.
\end{proof}

\begin{proof}[Proof of Proposition~\ref{prop:unique}]
By assumption  $\varrho([\gamma, b])=\varrho\left(\left[\gamma^{\prime}, b^{\prime}\right]\right)$. Thus,
\[
\int_L \hat{f}\left(b\right)- \int_L \hat{f}\left(b'\right)\in \mI_{\Rh} \cap (\Rh^{odd})^2, 
\]
and we can define $\gt$ by equation~\eqref{eq:gt_psedu_def}. When $\kappa=-\phi^*$, we replace $\xi$ in that equation by
$(\xi-\phi^*\xi)/2$ so that $\gt$ is real. Define $\Rd:=R(b,b^\prime)$ as in Notation~\ref{notn:n_decomp}.

Since $\int_L \hat{f}\left(b\right)- \int_L \hat{f}\left(b'\right)\in \mI_{\Rh},$ Lemma~\ref{lm:bt_exist_base_case1} gives an element $\bt_0 \in (\mCdh)^\kappa$ with $\nu_S\otimes \nu(\bt_0)>0$, such that $\bar{\bt}_0$ is a bounding cochain for the residue $A_\infty$ algebra $(\overline{\mCdh},\{\bar{\mt}^{\gt}_k\}_{k\ge 0})$, and $j_0^*(\bt_0)=b$, $j_1^*(\bt_0)=b^\prime$. 
Recall that $b$ is a unit bounding cochain, so
\[
\int_L \bar{b}\ne 0.
\]
Applying Corollary~\ref{cor:dt_obs_vanish} successively gives $\bt_i\in F^{E_i}\mB^\kappa$ such that $\tilde{b}_{(\ell)}:=\sum_{i=1}^{\ell}\bt_i$ is bounding modulo $F^{>E_\ell}\mB$ and has zero endpoint pullbacks.
Since $\lim_{i \to \infty}\nu(\bt_i)=\infty$, it follows from completeness of $\mCdh$ that there exists a unique limit $\mathfrak{b}:=\bt_0+\lim_{\ell \to \infty}\tilde{b}_{(\ell)}$.
By construction $\mathfrak{b}$ is a solution to the Maurer-Cartan equation in $\mCdh_{\not \e}$, so $\mgti{\mathfrak{b}}_0=\ct\cdot1$ for some $\ct\in Z(\mRdh)\cap\ker d_{\mRdh}$. Moreover, $j_0^*(\mathfrak{b})=b$, $j_1^*(\mathfrak{b})=b^\prime$, and $\nu(\mgti{\mathfrak{b}}_0)>0$ as $\mathfrak{b}$ is a solution to the Maurer-Cartan equation in $\frac{\mCdh}{F^{>0}\mCdh}.$ Equation~\eqref{eq:deformed_curvature_unit_pairing} and $\ll1,1\gg_F=0$ give
\[
d_{\mRdh}\ll\mathfrak b,1\gg_F
=\ll\mgti{\mathfrak b}_0,1\gg_F
=\ct\,\ll1,1\gg_F=0.
\]
We conclude that $(\mgt,\mathfrak{b})$ is a pseudoisotopy from $(\mgh,b)$ to $(\hat{\m}^{\gamma^\prime},b^\prime)$.
\end{proof}
\begin{rem}\label{rem:pseudo_sababa}
Here it is important that $b$ and $b^{\prime}$ are elements of $\Ch$ and not only the $\nu_S \otimes \nu$ completion of $\hat{R}\otimes A^*(L)$ which we will denote by $\doublehat{C}$. Indeed, to construct the sababa algebra $R(b,b^{\prime})$, we decomposed the chains $b, b^{\prime}$ as in equation~\eqref{eq:n_decomp}, this decomposition does not always exist for elements of $\doublehat{C}$. For example $\sum_{k\in \N} \lambda_k \cdot 1$ where $\lambda_k\in \hat{R}$ are defined in Remark~\ref{rem:disaster_example} cannot be decomposed as in equation~\eqref{eq:n_decomp}. Working with the sababa $A_\infty$ algebra $\mCdh$ constructed from $\Rd$ is crucial for our construction of the bounding $\bt$ as explained in Remark~\ref{rem:induc_cons_desc}.
\end{rem}
\section{Superpotential}
\subsection{Invariance under pseudoisotopy}\label{subsec:invar}
Throughout this subsection, assume $n>0$ is even.
\begin{proof}[Proof of Theorem~\ref{thm_inv}]
Let $(\mgt,\bt)$ be a pseudoisotopy from $(\mgh,b)$ to $(\mh^{\gamma'},b')$ as in Definition~\ref{dfn_g_equiv}. In the following, we use $t$ to denote the parameter of the pseudoisotopy. To simplify the algebra of the proof, it is convenient to use an integral representation for the superpotential. Let $\rho\in[0,1]$. By Lemma~\ref{lm:deformed_operator_explicit_formula},
\[
\mgti{\rho\bt}_0=\sum_{k\ge0}\rho^k\mgt_k(\bt^{\otimes k}).
\]
By Lemma~\ref{lm:m_bar_complex_converge}, for every $E>0$ only finitely many summands of $\mgti{\rho\bt}_0$ are nonzero modulo $F^{>E}$. For $i=0,1$, Definition~\ref{dfn:cycunit}\ref{it:val2} implies that the paired series $\langle j_i^*\mgti{\rho\bt}_0,j_i^*\bt\rangle$ has the same finiteness property. Thus termwise integration gives
\[
\int_0^1\langle j_i^*\mgti{\rho\bt}_0,j_i^*\bt\rangle\,d\rho
=\sum_{k\ge0}\frac{1}{k+1}\langle j_i^*\mgt_k(\bt^{\otimes k}),j_i^*\bt\rangle.
\]
Using $j_0^*\bt=b$, $j_1^*\bt=b'$, and Lemma~\ref{lm:pseudo}, the definition of the superpotential gives
\[
\Omega_{J'}(\gamma',b')-\Omega_J(\gamma,b)
=\int_0^1\left(\langle j_1^*\mgti{\rho\bt}_0,j_1^*\bt\rangle-\langle j_0^*\mgti{\rho\bt}_0,j_0^*\bt\rangle\right)d\rho.
\]
Let $d^t$ denote the exterior derivative with respect to the pseudoisotopy parameter $t,$ while holding $\rho$ constant. Since $|\mgti{\rho\bt}_0|=2$, $|\bt|=1$, and $n$ is even, the sign in Lemma~\ref{lm:d_ll_gg} is $-1$. Thus, from Lemma~\ref{lm:pseudo}, Lemma~\ref{lm:d_ll_gg}, and Fubini's theorem, we get
\[
\Omega_{J'}(\gamma',b')-\Omega_J(\gamma,b)
= -\int_{\rho \in [0,1]} \left(\int_{t \in [0,1]} d^t\ll\mgti{\rho\bt}_0,\bt\gg \right) d\rho = -\int_{[0,1]^2} d^t\ll\mgti{\rho\bt}_0,\bt\gg \wedge d\rho.
\]
The $A_\infty$ relations give $\mgti{\rho\bt}_1(\mgti{\rho\bt}_0)=0$. Thus, property~\ref{it:cyclic} with $k=1$ implies
\begin{equation*}\label{eq:cycsym}
d^t\ll\mgti{\rho\bt}_0,\bt\gg
=-\ll\mgti{\rho\bt}_1(\bt),\mgti{\rho\bt}_0\gg .
\end{equation*}
The finiteness argument above also justifies differentiating in $\rho$ to obtain,
\begin{equation*}\label{eq:diffrho}
\frac{d}{d\rho}\mgti{\rho\bt}_0
=\sum_{k\ge1}k\rho^{k-1}\mgt_k(\bt^{\otimes k})
=\mgti{\rho\bt}_1(\bt).
\end{equation*}
Since $\mgti{\rho\bt}_0$ has degree $2$, the symmetry of the pairing gives
\begin{equation*}\label{eq:pairsym}
\frac{d}{d\rho}\ll\mgti{\rho\bt}_0,\mgti{\rho\bt}_0\gg
=2\ll\frac{d}{d\rho}\mgti{\rho\bt}_0,\mgti{\rho\bt}_0\gg .
\end{equation*}
Therefore,
\[
\Omega_{J'}(\gamma',b')-\Omega_J(\gamma,b) = \frac{1}{2}\int_{[0,1]^2} \frac{d}{d\rho}\ll\mgti{\rho\bt}_0,\mgti{\rho\bt}_0  \gg \wedge d\rho
= \frac{1}{2}\int_{t \in [0,1]} \left(\ll\mgtb_0,\mgtb_0\gg-\ll\mgt_0,\mgt_0\gg\right).
\]
Since $(\mgt_0)^1=0,$ it follows that $\ll\mgt_0,\mgt_0\gg = 0.$ Since $\bt$ is a bounding cochain, there exists $\ct\in Z(\mRdh)$ such that
\[
\mgtb_0=\sum_{k\ge 0}\mgt_k(\bt^{\otimes k})=\ct\cdot 1.
\]
Thus, since $\ct$ is central, property~\ref{it:plin} of Definition~\ref{dfn:cycunit} gives
\[
\Omega_{J'}(\gamma',b')-\Omega_J(\gamma,b)
=\frac{1}{2}pt_*\ll \ct\cdot 1,\ct\cdot 1\gg=\frac{1}{2} pt_* \ct^2 \ll 1,1\gg=0,
\]
where the last equality uses $n > 0.$
Hence the superpotential is invariant under gauge equivalence.
\end{proof}
\subsection{Comparison with Welschinger invariants}\label{ssec:comparison}
Proposition~\ref{prop:exist} in the real setting can be simplified in the case $n=2$.
\begin{prop}\label{prop:3d}
If $n=2$, and $\gamma \in (\mI_QD)_2$ is real and closed, any form $b \in \left(\langle s\rangle+\mI_{\hat{R}}\right)_{1-n}  \otimes A^{n}(L)$ is a real bounding cochain for $\mgh$.
\end{prop}
\begin{proof}
Since $1-n=-1\equiv3\pmod4,$ Lemma~\ref{lm:reR} shows that $\left(\langle s\rangle+\mI_{\hat{R}}\right)_{1-n}$ is real, so $b$ is real. Lemma~\ref{lm:energy_zero_bounding_chain} shows that $\bar b$ is a bounding cochain for the residue algebra.

Suppose $\mghb_0$ does not vanish in $\Cdh_{\not \e}$, and let $E_i>0$ be the first energy level where this occurs. In the notation of Section~\ref{ssec:obs_theory}, $\mghb_0$ defines a non-zero element $\mgie{b}{E_i}\in(\Cdh_{\not \e,E_i})^{-\phi^*}_2$. For $r\in{\Rd_{E_i}}^\vee$, recall the map $\pi_r$ defined in equation~\eqref{eq:pir}. By assumption, there exists a homogeneous element $r\in{\Rd_{E_i}}^\vee$ such that $\eta_r:=\pi_r(\mgie{b}{E_i})\ne 0$. Put $v=|r|$, which is even because $R$ is concentrated in even degrees. Lemmas~\ref{lm:o_is_class_HCe},~\ref{lm:pi_r_phi_equivariant}, and~\ref{lm:pi_r_chain} imply that $\eta_r$ is a cocycle in $\left(\overline{\Cdh_{\not\e}}[v]^{-\phi^*},  \bar{\mh}^{\gamma,b}_1 \right)$ of degree $2$. Applying Lemma~\ref{lm:cano_iso} and Lemma~\ref{prop:fukaya_infity_arises_dga}, 
\[
\left(\overline{\Cdh[v]}\right)_2 \simeq \bigoplus_{-2-v\le p\le-v} S_{-p}\otimes A^{2+v+p}(L).
\]
By Lemma~\ref{lm:reR} a real summand in the above decomposition satisfies $-p-v\equiv2,3\pmod4$ and thus $2 + v + p = 0,3 \pmod 4$. Consequently,
\[
\left(\overline{\Cdh}[v]^{-\phi^*}\right)_2 \simeq  S_{2+v}\otimes A^0(L).
\]
Since $v$ is even, Corollary~\ref{cor:center_of_hat_algebra} implies $[\cdot ,\eta_r]=0$. Therefore, Corollary~\ref{cor:bar_fukaya_deformed_boundry} gives 
\[
d\eta_r  = \bar{\mh}^{\gamma,b}_1(\eta_r)=0.
\]
As $L$ is connected, $\eta_r$ is constant and thus vanishes in $\overline{\Cdh_{\not\e}}[v]$, contradicting its choice. Hence, $b$ is a bounding cochain.
\end{proof}
\begin{proof}[Proof of Theorem~\ref{Welsch}]
Assume without loss of generality that $[\gamma_N]=A.$ Abbreviate $k=k_{d,l}.$ For $\chi_L(\beta)=d,$ it follows that $\mu(\beta)=c_1(X)(d).$ Thus,
\[
k=c_1(X)(d)-2l-1=\mu(\beta)-2l-1.
\]
Since $L$ is orientable, $\mu(\beta)$ is even, so $k$ is odd. Thus only odd powers of $s$ occur in the comparison below.

By Proposition~\ref{prop:3d} we may choose $b=s\cdot\bb$ where $\bb$ is a representative of the Poincar\'e dual of a point. Since only insertions of $\gamma_N$ are being counted, we may set to zero all formal variables in $\gamma$ except the variable $t_N$; that is, we take $\gamma=t_N\gamma_N.$ Then
\begin{align*}
\ogw_{\beta,k}(\gamma_N^{\otimes l})
&\,\,\,\,\,=
k!l!\cdot [T^\beta s^k t_N^l]\Omega(t_N\gamma_N,s\bb)\\
&\,\,\,\,\,\overset{\eqref{eq:superpotential}}{=}k!l!\cdot [T^\beta s^k t_N^l]
\sum_{j\ge 0}\frac{1}{2j+1}
\left\langle
\mh^{t_N\gamma_N}_{2j}\left((s\otimes\bb)^{\otimes 2j}\right),
s\otimes\bb
\right\rangle \\
&\overset{\eqref{eq:extscal},\eqref{eq:pairng_extension}}{=}k!l!\cdot [T^\beta s^k t_N^l]
\sum_{j\ge 0}\frac{(-1)^{j+1}s^{2j+1}}{2j+1}
\left\langle
\m^{t_N\gamma_N}_{2j}\left(\bb^{\otimes 2j}\right),
\bb
\right\rangle \\
&\,\,\,\,\overset{\eqref{eq:fukaya_operator}}{=}k!l!\left(
\frac{1}{l! k}\langle{evb_0^\beta}_* (\bigwedge_{j=1}^l(evi_j^\beta)^*\gamma_N\wedge
\bigwedge_{j=1}^{k-1} (evb_j^\beta)^*\bb
),\bb\rangle\right)\\
&\,\,\,\,\,=(k-1)! \int_{\M_{k,l}(\beta)} \wedge_{j=1}^l(evi^\beta_j)^*\gamma_N\wedge \wedge_{j=0}^{k-1} (evb^\beta_j)^*\bb.
\end{align*}

Let $\M_{k,l}^S(\beta)$ denote the moduli space of genus zero $J$-holomorphic open stable maps $u:(\Sigma,\d \Sigma) \to (X,L)$ of degree $\beta \in \sly$ with one boundary component, $k$ labeled boundary marked points with no requirement that their labels respect the cyclic order given by the orientation of $\d\Sigma$, and $l$ interior marked points. Equivalently, $\M_{k,l}^S(\beta)$ is the disjoint union of $(k-1)!$ copies of $\M_{k,l}(\beta)$ corresponding to the cyclic orders of the boundary labels.
The space $\M_{k,l}^S(\beta)$ comes with evaluation maps
$evb_j^\beta:\M_{k,l}^S(\beta)\to L$, $j=0,\ldots,k-1$, and $evi_j^\beta:\M_{k,l}^S(\beta)\to X$, $j=1,\ldots,l$.
The spin structure on $L$ induces a canonical orientation on $\M_{k,l}^S(\beta)$ as in~\cite[Chapter 8]{fukaya2010lagrangian} and~\cite{solomon2006intersection}. 

Recall the notation of Section~\ref{ssec:constr}. For a permutation $\sigma : \{0,\ldots,k-1\} \to \{0,\ldots,k-1\}$ with $\sigma(0) = 0,$ let $\varphi_\sigma : \M_{k,l}(\beta) \to \M_{k,l}^S(\beta)$ denote the open and closed embedding given by
\[
\varphi_\sigma(\Sigma,u,(z_0,\ldots,z_{k-1}),(w_1,\ldots,w_l)) = (\Sigma,u,(z_{\sigma(0)},\ldots,z_{\sigma(k-1)}),(w_1,\ldots,w_l)).
\]
Thus, $\varphi_\sigma$ identifies $\M_{k,l}(\beta)$ with the component of $\M_{k,l}^S(\beta)$ where the labels of the marked points are permuted by $\sigma$. In particular, $evb_j^\beta \circ \varphi_\sigma = evb_{\sigma(j)}^\beta.$
Following the orientation convention of~\cite[Definition 3.2]{solomon2006intersection} for $\dim L$ even, the maps $\varphi_\sigma$ preserve orientation. Since $|\bb| = 2,$ it follows that
\begin{align*}
\ogw_{\beta,k}(\gamma_N^{\otimes l}) &=\sum_{\sigma}\int_{\M_{k,l}(\beta)} \wedge_{j=1}^l(evi^\beta_j)^*\gamma_N\wedge \wedge_{j=0}^{k-1}(evb^\beta_{\sigma(j)})^*\bb  \\
&= \int_{\M_{k,l}^S(\beta)} \wedge_{j=1}^l(evi_j^\beta)^*\gamma_N\wedge \wedge_{j=0}^{k-1}(evb^\beta_j)^*\bb.
\end{align*}
After summing over $\chi_L(\beta)=d,$ the theorem follows from~\cite[Theorem 1.8]{solomon2006intersection}. The factor of $2^{1-l}$ arises from two independent sources.
First, as explained in \cite[Theorem 1.8]{solomon2006intersection}, each real holomorphic sphere corresponds to two holomorphic disks. This gives a factor of $2.$ Second, each interior constraint $\gamma_N$ is Poincar\'e dual to the homology class of a point. On the other 
hand, Welschinger~\cite{welschinger2005enumerative, welschinger2005spinor} considers constraints that are pairs of conjugate points, thus representing twice the homology class of a point. This gives an additional factor of $2^{-l}.$
\end{proof}
\bibliography{refs}

@book {bosch1984non,
    AUTHOR = {Bosch, S. and G\"untzer, U. and Remmert, R.},
     TITLE = {Non-{A}rchimedean analysis},
    SERIES = {Grundlehren der mathematischen Wissenschaften [Fundamental
              Principles of Mathematical Sciences]},
    VOLUME = {261},
      NOTE = {A systematic approach to rigid analytic geometry},
 PUBLISHER = {Springer-Verlag, Berlin},
      YEAR = {1984},
     PAGES = {xii+436},
      ISBN = {3-540-12546-9},
   MRCLASS = {32K10 (30G05 46P05)},
  MRNUMBER = {746961},
MRREVIEWER = {W.\ Bartenwerfer},
       DOI = {10.1007/978-3-642-52229-1},
       URL = {https://doi.org/10.1007/978-3-642-52229-1},
}

@article{solomon2016differential,
  AUTHOR = {Solomon, Jake P. and Tukachinsky, Sara B.},
     TITLE = {Differential forms, {F}ukaya {$A_\infty$} algebras, and
              {G}romov-{W}itten axioms},
   JOURNAL = {J. Symplectic Geom.},
  FJOURNAL = {The Journal of Symplectic Geometry},
    VOLUME = {20},
      YEAR = {2022},
    NUMBER = {4},
     PAGES = {927--994},
      ISSN = {1527-5256,1540-2347},
   MRCLASS = {53D12 (53D45)},
  MRNUMBER = {4573438},
MRREVIEWER = {Johan\ Asplund},
       DOI = {10.4310/jsg.2022.v20.n4.a5},
       URL = {https://doi.org/10.4310/jsg.2022.v20.n4.a5},
}

@book{fukaya2010lagrangian,
  AUTHOR = {Fukaya, Kenji and Oh, Yong-Geun and Ohta, Hiroshi and Ono,
              Kaoru},
     TITLE = {Lagrangian intersection {F}loer theory: anomaly and
              obstruction. {P}art {I}},
    SERIES = {AMS/IP Studies in Advanced Mathematics},
    VOLUME = {46.1},
 PUBLISHER = {American Mathematical Society, Providence, RI; International
              Press, Somerville, MA},
      YEAR = {2009},
     PAGES = {xii+396},
      ISBN = {978-0-8218-4836-4},
   MRCLASS = {53D40 (53D12 53D37)},
  MRNUMBER = {2553465},
MRREVIEWER = {Michael\ J.\ Usher},
       DOI = {10.1090/amsip/046.1},
       URL = {https://doi.org/10.1090/amsip/046.1},
}

@article{georgieva2016open,
  AUTHOR = {Georgieva, Penka},
     TITLE = {Open {G}romov-{W}itten disk invariants in the presence of an
              anti-symplectic involution},
   JOURNAL = {Adv. Math.},
  FJOURNAL = {Advances in Mathematics},
    VOLUME = {301},
      YEAR = {2016},
     PAGES = {116--160},
      ISSN = {0001-8708,1090-2082},
   MRCLASS = {53D45 (14N35)},
  MRNUMBER = {3539371},
MRREVIEWER = {Ritwik\ Mukherjee},
       DOI = {10.1016/j.aim.2016.06.009},
       URL = {https://doi.org/10.1016/j.aim.2016.06.009},
}

@article{solomon2006intersection,
  AUTHOR = {Solomon, Jake P.},
     TITLE = {Intersection theory on the moduli space of holomorphic curves
              with {L}agrangian boundary conditions},
      NOTE = {Thesis (Ph.D.)--Massachusetts Institute of Technology},
 PUBLISHER = {ProQuest LLC, Ann Arbor, MI},
      YEAR = {2006},
   MRCLASS = {99-05},
  MRNUMBER = {2717339},
       URL =
              {http://gateway.proquest.com/openurl?url_ver=Z39.88-2004&rft_val_fmt=info:ofi/fmt:kev:mtx:dissertation&res_dat=xri:pqdiss&rft_dat=xri:pqdiss:0809080},
}

@article{welschinger2005invariants,
  AUTHOR = {Welschinger, Jean-Yves},
     TITLE = {Invariants of real symplectic 4-manifolds and lower bounds in
              real enumerative geometry},
   JOURNAL = {Invent. Math.},
  FJOURNAL = {Inventiones Mathematicae},
    VOLUME = {162},
      YEAR = {2005},
    NUMBER = {1},
     PAGES = {195--234},
      ISSN = {0020-9910,1432-1297},
   MRCLASS = {53D45 (14N10 14N35 14P25)},
  MRNUMBER = {2198329},
MRREVIEWER = {Eugenii\ Shustin},
       DOI = {10.1007/s00222-005-0445-0},
       URL = {https://doi.org/10.1007/s00222-005-0445-0},
}

@article {solomon2016point,
    AUTHOR = {Solomon, Jake P. and Tukachinsky, Sara B.},
     TITLE = {Point-like bounding chains in open {G}romov-{W}itten theory},
   JOURNAL = {Geom. Funct. Anal.},
  FJOURNAL = {Geometric and Functional Analysis},
    VOLUME = {31},
      YEAR = {2021},
    NUMBER = {5},
     PAGES = {1245--1320},
      ISSN = {1016-443X,1420-8970},
   MRCLASS = {53D45 (14N10 14N35 53D12 53D37)},
  MRNUMBER = {4356703},
MRREVIEWER = {Cheng-Yong\ Du},
       DOI = {10.1007/s00039-021-00583-3},
       URL = {https://doi.org/10.1007/s00039-021-00583-3},
}

@article{solomon2019relative,
  AUTHOR = {Solomon, Jake P. and Tukachinsky, Sara B.},
     TITLE = {Relative quantum cohomology},
   JOURNAL = {J. Eur. Math. Soc. (JEMS)},
  FJOURNAL = {Journal of the European Mathematical Society (JEMS)},
    VOLUME = {26},
      YEAR = {2024},
    NUMBER = {9},
     PAGES = {3497--3573},
      ISSN = {1435-9855,1435-9863},
   MRCLASS = {53D37 (14N10 14N35 53D12 53D45)},
  MRNUMBER = {4767498},
MRREVIEWER = {Hsian-Hua\ Tseng},
       DOI = {10.4171/jems/1337},
       URL = {https://doi.org/10.4171/jems/1337},
}

@article{fukaya2010cyclic,
  AUTHOR = {Fukaya, Kenji},
     TITLE = {Cyclic symmetry and adic convergence in {L}agrangian {F}loer
              theory},
   JOURNAL = {Kyoto J. Math.},
  FJOURNAL = {Kyoto Journal of Mathematics},
    VOLUME = {50},
      YEAR = {2010},
    NUMBER = {3},
     PAGES = {521--590},
      ISSN = {2156-2261,2154-3321},
   MRCLASS = {53D37 (53D12 53D40)},
  MRNUMBER = {2723862},
MRREVIEWER = {Timothy\ Perutz},
       DOI = {10.1215/0023608X-2010-004},
       URL = {https://doi.org/10.1215/0023608X-2010-004},
}

@ARTICLE{welschinger2005enumerative,
       author = {{Welschinger}, Jean-Yves},
        title = "{Enumerative invariants of stongly semipositive real symplectic six-manifolds}",
      journal = {arXiv Mathematics e-prints},
         year = 2005,
        month = sep,
          eid = {math/0509121},
archivePrefix = {arXiv},
       eprint = {math/0509121},
 primaryClass = {math.AG},
       adsurl = {https://scixplorer.org/abs/2005math......9121W}
}

@article{welschinger2005spinor,
  AUTHOR = {Welschinger, Jean-Yves},
     TITLE = {Spinor states of real rational curves in real algebraic convex
              3-manifolds and enumerative invariants},
   JOURNAL = {Duke Math. J.},
  FJOURNAL = {Duke Mathematical Journal},
    VOLUME = {127},
      YEAR = {2005},
    NUMBER = {1},
     PAGES = {89--121},
      ISSN = {0012-7094,1547-7398},
   MRCLASS = {14N35 (14P25 53D45)},
  MRNUMBER = {2126497},
MRREVIEWER = {Eugenii\ Shustin},
       DOI = {10.1215/S0012-7094-04-12713-7},
       URL = {https://doi.org/10.1215/S0012-7094-04-12713-7},
}

@book {mccleary2001user,
    AUTHOR = {McCleary, John},
     TITLE = {A user's guide to spectral sequences},
    SERIES = {Cambridge Studies in Advanced Mathematics},
    VOLUME = {58},
   EDITION = {Second},
 PUBLISHER = {Cambridge University Press, Cambridge},
      YEAR = {2001},
     PAGES = {xvi+561},
      ISBN = {0-521-56759-9},
   MRCLASS = {55Txx (18G40)},
  MRNUMBER = {1793722},
MRREVIEWER = {Frank\ Neumann},
}

@ARTICLE{peschke2022exactcouplesspectralsequences,
       author = {{Peschke}, George},
        title = "{Exact couples and their spectral sequences}",
      journal = {arXiv e-prints},
         year = 2022,
        month = may,
          eid = {arXiv:2205.10705},
archivePrefix = {arXiv},
       eprint = {2205.10705},
 primaryClass = {math.KT},
       adsurl = {https://scixplorer.org/abs/2022arXiv220510705P}
}

@article {cho2008counting,
    AUTHOR = {Cho, Cheol-Hyun},
     TITLE = {Counting real {$J$}-holomorphic discs and spheres in dimension
              four and six},
   JOURNAL = {J. Korean Math. Soc.},
  FJOURNAL = {Journal of the Korean Mathematical Society},
    VOLUME = {45},
      YEAR = {2008},
    NUMBER = {5},
     PAGES = {1427--1442},
      ISSN = {0304-9914,2234-3008},
   MRCLASS = {53D45},
  MRNUMBER = {2439744},
MRREVIEWER = {Jean-Yves\ Welschinger},
       DOI = {10.4134/JKMS.2008.45.5.1427},
       URL = {https://doi.org/10.4134/JKMS.2008.45.5.1427},
}

@book {Liu2002Counting,
    AUTHOR = {Liu, Chiu-Chu Melissa},
     TITLE = {Moduli of {J}-holomorphic curves with {L}agrangian boundary
              conditions},
      NOTE = {Thesis (Ph.D.)--Harvard University},
 PUBLISHER = {ProQuest LLC, Ann Arbor, MI},
      YEAR = {2002},
     PAGES = {151},
      ISBN = {978-0493-65757-8},
   MRCLASS = {99-05},
  MRNUMBER = {2703393},
       URL =
              {http://gateway.proquest.com/openurl?url_ver=Z39.88-2004&rft_val_fmt=info:ofi/fmt:kev:mtx:dissertation&res_dat=xri:pqdiss&rft_dat=xri:pqdiss:3051225},
}

@ARTICLE{kedar2022ainftyorientorcalculus,
       author = {{Kedar}, Or and {Solomon}, Jake P.},
        title = "{$A_\infty$ relations in orientor calculus on moduli of stable disk maps}",
      journal = {arXiv e-prints},
         year = 2022,
        month = nov,
          eid = {arXiv:2211.05117},
archivePrefix = {arXiv},
       eprint = {2211.05117},
 primaryClass = {math.SG},
       adsurl = {https://scixplorer.org/abs/2022arXiv221105117K}
}

@ARTICLE{kedar2022fukaynonorientable,
       author = {{Kedar}, Or and {Solomon}, Jake P.},
        title = "{The Fukaya $A_\infty$ algebra of a non-orientable Lagrangian}",
      journal = {arXiv e-prints},
         year = 2022,
        month = nov,
          eid = {arXiv:2211.05439},
archivePrefix = {arXiv},
       eprint = {2211.05439},
 primaryClass = {math.SG},
       adsurl = {https://scixplorer.org/abs/2022arXiv221105439K}
}

@article {Itenberg2003enumerationrealcurves,
    AUTHOR = {Itenberg, Ilia and Kharlamov, Viatcheslav and Shustin,
              Eugenii},
     TITLE = {Welschinger invariant and enumeration of real rational curves},
   JOURNAL = {Int. Math. Res. Not.},
  FJOURNAL = {International Mathematics Research Notices},
      YEAR = {2003},
    NUMBER = {49},
     PAGES = {2639--2653},
      ISSN = {1073-7928,1687-0247},
   MRCLASS = {14P05 (14N10)},
  MRNUMBER = {2012521},
MRREVIEWER = {Thomas\ C.\ Craven},
       DOI = {10.1155/S1073792803131352},
       URL = {https://doi.org/10.1155/S1073792803131352},
}

@article {Itenberg2004LogEquiv,
    AUTHOR = {Itenberg, I. V. and Kharlamov, V. M. and Shustin, E. I.},
     TITLE = {Logarithmic equivalence of the {W}elschinger and the
              {G}romov-{W}itten invariants},
   JOURNAL = {Uspekhi Mat. Nauk},
  FJOURNAL = {Uspekhi Matematicheskikh Nauk},
    VOLUME = {59},
      YEAR = {2004},
    NUMBER = {6(360)},
     PAGES = {85--110},
      ISSN = {0042-1316,2305-2872},
   MRCLASS = {14P05 (14N10 14N35)},
  MRNUMBER = {2138469},
MRREVIEWER = {I.\ Dolgachev},
       DOI = {10.1070/RM2004v059n06ABEH000797},
       URL = {https://doi.org/10.1070/RM2004v059n06ABEH000797},
}

@article {Itenberg2007newLogWel,
    AUTHOR = {Itenberg, I. and Kharlamov, V. and Shustin, E.},
     TITLE = {New cases of logarithmic equivalence of {W}elschinger and
              {G}romov-{W}itten invariants},
   JOURNAL = {Tr. Mat. Inst. Steklova},
  FJOURNAL = {Trudy Matematicheskogo Instituta Imeni V. A. Steklova},
    VOLUME = {258},
      YEAR = {2007},
     PAGES = {70--78},
      ISSN = {0371-9685,3034-1809},
      ISBN = {978-5-02-036672-5; 978-5-02-035888-1},
   MRCLASS = {14N35 (14J26)},
  MRNUMBER = {2400524},
MRREVIEWER = {Luis\ Felipe\ Tabera},
       DOI = {10.1134/S0081543807030078},
       URL = {https://doi.org/10.1134/S0081543807030078},
}

@article {Brugalle2018SurgeryWelsch,
    AUTHOR = {Brugall\'e, Erwan},
     TITLE = {Surgery of real symplectic fourfolds and {W}elschinger
              invariants},
   JOURNAL = {J. Singul.},
  FJOURNAL = {Journal of Singularities},
    VOLUME = {17},
      YEAR = {2018},
     PAGES = {267--294},
      ISSN = {1949-2006},
   MRCLASS = {14N10 (14N35 14P05 14P25 53D12)},
  MRNUMBER = {3846537},
MRREVIEWER = {Eugenii\ Shustin},
       DOI = {10.5427/jsing.2018.17l},
       URL = {https://doi.org/10.5427/jsing.2018.17l},
}

@ARTICLE{Kosloff2026,
       author = {{Kosloff}, Elad},
        title = "{Open WDVV equations for non-commutative families of bounding cochains}",
      journal = {arXiv e-prints},
         year = 2026,
        month = oct,
          eid = {arXiv:2610.13209},
archivePrefix = {arXiv},
       eprint = {2610.13209},
 primaryClass = {math.SG},
       adsurl = {https://scixplorer.org/abs/2026arXiv261013209K}
}

@article {Fukaya2011Counting,
    AUTHOR = {Fukaya, Kenji},
     TITLE = {Counting pseudo-holomorphic discs in {C}alabi-{Y}au 3-folds},
   JOURNAL = {Tohoku Math. J. (2)},
  FJOURNAL = {The Tohoku Mathematical Journal. Second Series},
    VOLUME = {63},
      YEAR = {2011},
    NUMBER = {4},
     PAGES = {697--727},
      ISSN = {0040-8735,2186-585X},
   MRCLASS = {53D45 (32Q25 32Q65)},
  MRNUMBER = {2872962},
MRREVIEWER = {Michael\ J.\ Usher},
       DOI = {10.2748/tmj/1325886287},
       URL = {https://doi.org/10.2748/tmj/1325886287},
}

@article {Fukaya2010ToricI,
    AUTHOR = {Fukaya, Kenji and Oh, Yong-Geun and Ohta, Hiroshi and Ono,
              Kaoru},
     TITLE = {Lagrangian {F}loer theory on compact toric manifolds. {I}},
   JOURNAL = {Duke Math. J.},
  FJOURNAL = {Duke Mathematical Journal},
    VOLUME = {151},
      YEAR = {2010},
    NUMBER = {1},
     PAGES = {23--174},
      ISSN = {0012-7094,1547-7398},
   MRCLASS = {53D40 (53D12 53D45)},
  MRNUMBER = {2573826},
MRREVIEWER = {Timothy\ Perutz},
       DOI = {10.1215/00127094-2009-062},
       URL = {https://doi.org/10.1215/00127094-2009-062},
}

@article {Fukaya2011ToricII,
    AUTHOR = {Fukaya, Kenji and Oh, Yong-Geun and Ohta, Hiroshi and Ono,
              Kaoru},
     TITLE = {Lagrangian {F}loer theory on compact toric manifolds {II}:
              bulk deformations},
   JOURNAL = {Selecta Math. (N.S.)},
  FJOURNAL = {Selecta Mathematica. New Series},
    VOLUME = {17},
      YEAR = {2011},
    NUMBER = {3},
     PAGES = {609--711},
      ISSN = {1022-1824,1420-9020},
   MRCLASS = {53D12 (14M25 53D35 53D40)},
  MRNUMBER = {2827178},
MRREVIEWER = {Jelena\ Kati\'c},
       DOI = {10.1007/s00029-011-0057-z},
       URL = {https://doi.org/10.1007/s00029-011-0057-z},
}

@article {Fukaya2016Miror,
    AUTHOR = {Fukaya, Kenji and Oh, Yong-Geun and Ohta, Hiroshi and Ono,
              Kaoru},
     TITLE = {Lagrangian {F}loer theory and mirror symmetry on compact toric
              manifolds},
   JOURNAL = {Ast\'erisque},
  FJOURNAL = {Ast\'erisque},
    NUMBER = {376},
      YEAR = {2016},
     PAGES = {vi+340},
      ISSN = {0303-1179,2492-5926},
      ISBN = {978-2-85629-825-1},
   MRCLASS = {53D40 (14M25 53D37 53D45)},
  MRNUMBER = {3460884},
MRREVIEWER = {Christopher\ T.\ Woodward},
}

@article {Fukaya2017involution,
    AUTHOR = {Fukaya, Kenji and Oh, Yong-Geun and Ohta, Hiroshi and Ono,
              Kaoru},
     TITLE = {Antisymplectic involution and {F}loer cohomology},
   JOURNAL = {Geom. Topol.},
  FJOURNAL = {Geometry \& Topology},
    VOLUME = {21},
      YEAR = {2017},
    NUMBER = {1},
     PAGES = {1--106},
      ISSN = {1465-3060,1364-0380},
   MRCLASS = {53D40 (14J33 53D45)},
  MRNUMBER = {3608709},
MRREVIEWER = {Michael\ J.\ Usher},
       DOI = {10.2140/gt.2017.21.1},
       URL = {https://doi.org/10.2140/gt.2017.21.1},
}

@article {fukaya2024constructionkuranishistructuresmoduli,
    AUTHOR = {Fukaya, Kenji and Oh, Yong-Geun and Ohta, Hiroshi and Ono,
              Kaoru},
     TITLE = {Construction of {K}uranishi structures on the moduli spaces of
              pseudo-holomorphic disks: {II}},
   JOURNAL = {Adv. Math.},
  FJOURNAL = {Advances in Mathematics},
    VOLUME = {442},
      YEAR = {2024},
     PAGES = {Paper No. 109561, 63},
      ISSN = {0001-8708,1090-2082},
   MRCLASS = {32G07 (32Q65 53D30)},
  MRNUMBER = {4710872},
MRREVIEWER = {Kimio\ Miyajima},
       DOI = {10.1016/j.aim.2024.109561},
       URL = {https://doi.org/10.1016/j.aim.2024.109561},
}

@article {Fukaya2019Spectral,
    AUTHOR = {Fukaya, Kenji and Oh, Yong-Geun and Ohta, Hiroshi and Ono,
              Kaoru},
     TITLE = {Spectral invariants with bulk, quasi-morphisms and
              {L}agrangian {F}loer theory},
   JOURNAL = {Mem. Amer. Math. Soc.},
  FJOURNAL = {Memoirs of the American Mathematical Society},
    VOLUME = {260},
      YEAR = {2019},
    NUMBER = {1254},
     PAGES = {x+266},
      ISSN = {0065-9266,1947-6221},
      ISBN = {978-1-4704-3625-4; 978-1-4704-5325-1},
   MRCLASS = {53D40 (14M25 20F65 53D12 53D20 53D45)},
  MRNUMBER = {3986938},
MRREVIEWER = {Jun\ Zhang},
       DOI = {10.1090/memo/1254},
       URL = {https://doi.org/10.1090/memo/1254},
}

@book {Fukaya2020Kuranishi,
    AUTHOR = {Fukaya, Kenji and Oh, Yong-Geun and Ohta, Hiroshi and Ono,
              Kaoru},
     TITLE = {Kuranishi structures and virtual fundamental chains},
    SERIES = {Springer Monographs in Mathematics},
 PUBLISHER = {Springer, Singapore},
      YEAR = {[2020] \copyright 2020},
     PAGES = {xv+638},
      ISBN = {978-981-15-5562-6; 978-981-15-5561-9},
   MRCLASS = {53D45 (32M99 53D37 53D40)},
  MRNUMBER = {4179586},
MRREVIEWER = {Hsian-Hua\ Tseng},
       DOI = {10.1007/978-981-15-5562-6},
       URL = {https://doi.org/10.1007/978-981-15-5562-6},
}

@article {Fukaya1999Arnold,
    AUTHOR = {Fukaya, Kenji and Ono, Kaoru},
     TITLE = {Arnold conjecture and {G}romov-{W}itten invariant},
   JOURNAL = {Topology},
  FJOURNAL = {Topology. An International Journal of Mathematics},
    VOLUME = {38},
      YEAR = {1999},
    NUMBER = {5},
     PAGES = {933--1048},
      ISSN = {0040-9383},
   MRCLASS = {53D45 (37J10 37J45 53D40 57R17 57R58)},
  MRNUMBER = {1688434},
MRREVIEWER = {David\ E.\ Hurtubise},
       DOI = {10.1016/S0040-9383(98)00042-1},
       URL = {https://doi.org/10.1016/S0040-9383(98)00042-1},
}

@misc{li2014structures,
  title={${A}_\infty$-structures from Morse trees with pseudoholomorphic disks},
  author={Li, Jiayong and Wehrheim, Katrin},
  journal={preprint available at math. berkeley. edu/katrin/papers/disktrees. pdf},
  year={2014}
}

@article {Hofer2007Fredholm,
    AUTHOR = {Hofer, H. and Wysocki, K. and Zehnder, E.},
     TITLE = {A general {F}redholm theory. {I}. {A} splicing-based
              differential geometry},
   JOURNAL = {J. Eur. Math. Soc. (JEMS)},
  FJOURNAL = {Journal of the European Mathematical Society (JEMS)},
    VOLUME = {9},
      YEAR = {2007},
    NUMBER = {4},
     PAGES = {841--876},
      ISSN = {1435-9855,1435-9863},
   MRCLASS = {53D45 (47A53 58B15)},
  MRNUMBER = {2341834},
MRREVIEWER = {Tobias\ Ekholm},
       DOI = {10.4171/JEMS/99},
       URL = {https://doi.org/10.4171/JEMS/99},
}

@article {Hofer2009Fredholm,
    AUTHOR = {Hofer, Helmut and Wysocki, Krzysztof and Zehnder, Eduard},
     TITLE = {A general {F}redholm theory. {II}. {I}mplicit function
              theorems},
   JOURNAL = {Geom. Funct. Anal.},
  FJOURNAL = {Geometric and Functional Analysis},
    VOLUME = {19},
      YEAR = {2009},
    NUMBER = {1},
     PAGES = {206--293},
      ISSN = {1016-443X,1420-8970},
   MRCLASS = {53D42 (46T05 53D40 53D45 58B15 58C15)},
  MRNUMBER = {2507223},
MRREVIEWER = {Tobias\ Ekholm},
       DOI = {10.1007/s00039-009-0715-x},
       URL = {https://doi.org/10.1007/s00039-009-0715-x},
}

@article {Hofer2009functors,
    AUTHOR = {Hofer, Helmut and Wysocki, Kris and Zehnder, Eduard},
     TITLE = {A general {F}redholm theory. {III}. {F}redholm functors and
              polyfolds},
   JOURNAL = {Geom. Topol.},
  FJOURNAL = {Geometry \& Topology},
    VOLUME = {13},
      YEAR = {2009},
    NUMBER = {4},
     PAGES = {2279--2387},
      ISSN = {1465-3060,1364-0380},
   MRCLASS = {53D42 (46T99 47A53 53D40 53D45 57R17 58B15)},
  MRNUMBER = {2515707},
MRREVIEWER = {Tobias\ Ekholm},
       DOI = {10.2140/gt.2009.13.2279},
       URL = {https://doi.org/10.2140/gt.2009.13.2279},
}

@article {Hofer2010Integration,
    AUTHOR = {Hofer, H. and Wysocki, K. and Zehnder, E.},
     TITLE = {Integration theory on the zero sets of polyfold {F}redholm
              sections},
   JOURNAL = {Math. Ann.},
  FJOURNAL = {Mathematische Annalen},
    VOLUME = {346},
      YEAR = {2010},
    NUMBER = {1},
     PAGES = {139--198},
      ISSN = {0025-5831,1432-1807},
   MRCLASS = {53D42 (53D40 53D45 58B15)},
  MRNUMBER = {2558891},
MRREVIEWER = {Tobias\ Ekholm},
       DOI = {10.1007/s00208-009-0393-x},
       URL = {https://doi.org/10.1007/s00208-009-0393-x},
}

@ARTICLE{abouzaid2021complexcobordismhamiltonianloops,
       author = {{Abouzaid}, Mohammed and {McLean}, Mark and {Smith}, Ivan},
        title = "{Complex cobordism, Hamiltonian loops and global Kuranishi charts}",
      journal = {arXiv e-prints},
         year = 2021,
        month = oct,
          eid = {arXiv:2110.14320},
archivePrefix = {arXiv},
       eprint = {2110.14320},
 primaryClass = {math.SG},
       adsurl = {https://scixplorer.org/abs/2021arXiv211014320A}
}

@article {Abouzaid2024Gromov,
    AUTHOR = {Abouzaid, Mohammed and McLean, Mark and Smith, Ivan},
     TITLE = {Gromov-{W}itten invariants in complex and {M}orava-local
              {$K$}-theories},
   JOURNAL = {Geom. Funct. Anal.},
  FJOURNAL = {Geometric and Functional Analysis},
    VOLUME = {34},
      YEAR = {2024},
    NUMBER = {6},
     PAGES = {1647--1733},
      ISSN = {1016-443X,1420-8970},
   MRCLASS = {53D45 (14N35 55N15)},
  MRNUMBER = {4823210},
       DOI = {10.1007/s00039-024-00697-4},
       URL = {https://doi.org/10.1007/s00039-024-00697-4},
}

@ARTICLE{hirschi2025opencloseddelignemumfordfieldtheory,
       author = {{Hirschi}, Amanda and {Hugtenburg}, Kai},
        title = "{Open-closed Deligne-Mumford field theories: geometric foundations}",
      journal = {arXiv e-prints},
         year = 2025,
        month = jan,
          eid = {arXiv:2501.04687},
archivePrefix = {arXiv},
       eprint = {2501.04687},
 primaryClass = {math.SG},
       adsurl = {https://scixplorer.org/abs/2025arXiv250104687H}
}

@misc{stacks-project,
  author       = {{The Stacks project authors}},
  title        = {The {S}tacks project},
  howpublished = {\url{https://stacks.math.columbia.edu}},
  year         = {2025},
}

@article {solomon2020Involutions,
    AUTHOR = {Solomon, Jake P.},
     TITLE = {Involutions, obstructions and mirror symmetry},
   JOURNAL = {Adv. Math.},
  FJOURNAL = {Advances in Mathematics},
    VOLUME = {367},
      YEAR = {2020},
     PAGES = {107107, 52},
      ISSN = {0001-8708,1090-2082},
   MRCLASS = {53D37 (14J33 53D12 53D40 58J32)},
  MRNUMBER = {4078487},
MRREVIEWER = {Nathaniel\ Bottman},
       DOI = {10.1016/j.aim.2020.107107},
       URL = {https://doi.org/10.1016/j.aim.2020.107107},
}

@incollection {FukayaKurnishi2017I,
    AUTHOR = {Fukaya, Kenji and Oh, Yong-Geun and Ohta, Hiroshi and Ono,
              Kaoru},
     TITLE = {Construction of {K}uranishi structures on the moduli spaces of
              pseudo holomorphic disks: {I}},
 BOOKTITLE = {Surveys in differential geometry 2017. {C}elebrating the 50th
              anniversary of the {J}ournal of {D}ifferential {G}eometry},
    SERIES = {Surv. Differ. Geom.},
    VOLUME = {22},
     PAGES = {133--190},
 PUBLISHER = {Int. Press, Somerville, MA},
      YEAR = {2018},
      ISBN = {978-1-57146-361-6},
   MRCLASS = {32G07 (53D30)},
  MRNUMBER = {3838117},
}
\bibliographystyle{amsabbrvcnobysame} 
\end{document}